\documentclass[10.9pt,a4paper]{amsart}
\usepackage[english]{babel}
\usepackage[percent]{overpic}
\usepackage{xcolor}
\definecolor{lblue}{rgb}{52,219,252}
\newtheorem{theorem}{Theorem}[section]

\newtheorem{question}{Question}
\newtheorem{problem}{Problem}
\newtheorem{corollary}[theorem]{Corollary}
\newtheorem{lemma}[theorem]{Lemma}
\newtheorem{proposition}[theorem]{Proposition}

\newtheorem{claim}{Claim}
\newtheorem*{claim*}{Claim}
\theoremstyle{definition}
\newtheorem{definition}[theorem]{Definition}
\newtheorem{construction}[theorem]{Construction}

\newtheorem{example}[theorem]{Example}

\newtheorem{remark}[theorem]{Remark}
\numberwithin{equation}{section}

\usepackage[square,sort,comma,numbers]{natbib}

\usepackage[all]{xy}
\usepackage{pstricks}
\usepackage{enumerate}
\usepackage{amsfonts,amssymb,amsmath,eucal,pinlabel,array,hhline}
\usepackage{slashed}
\usepackage{tabulary}
\usepackage{fancyhdr}
\usepackage{color}
\usepackage{a4wide}
\usepackage{calrsfs,bbm}
\usepackage[position=b]{subcaption}
\usepackage[colorlinks,
    linkcolor={blue!50!black},
    citecolor={blue!50!black},
    urlcolor={blue!80!black}]{hyperref}

\newcommand{\Z}{\mathbb{Z}}
\newcommand{\Q}{\mathbb{Q}}
\newcommand{\R}{\mathbb{R}}
\newcommand{\C}{\mathbb{C}}

\newcommand{\F}{\mathbb{F}}

\newcommand{\bsm}{\left(\begin{smallmatrix}}
\newcommand{\esm}{\end{smallmatrix}\right)}
\newcommand{\id}{\operatorname{Id}}
\newcommand{\Bl}{\operatorname{Bl}}
\newcommand{\coker}{\operatorname{coker}}
\newcommand{\Iso}{\operatorname{Iso}}
\newcommand{\Homeo}{\operatorname{Homeo}}
\newcommand{\Surf}{\operatorname{Surf}}
\newcommand{\Aut}{\operatorname{Aut}}

\newcommand{\im}{\operatorname{im}}
\newcommand{\fr}{\operatorname{fr}}
\newcommand{\Hom}{\operatorname{Hom}}
\newcommand{\TB}{\operatorname{TB}}
\newcommand{\pd}{\operatorname{pd}}

\newcommand{\sH}{\mathcal{H}}
\newcommand{\Hi}{\mathcal{H}_i}

\newcommand{\ks}{\operatorname{ks}}
\newcommand{\sm}{\setminus}
\newcommand{\ol}{\overline}
\newcommand{\wt}{\widetilde}
\newcommand{\lk}{\ell k}

\newcommand{\tmfrac}[2]{\mbox{\large$\frac{#1}{#2}$}}
\newcommand{\unaryminus}{\scalebox{0.75}[1.0]{\( - \)}}

\DeclareSymbolFont{EulerScript}{U}{eus}{m}{n}
\DeclareSymbolFontAlphabet\mathscr{EulerScript}

\begin{document}

\title{$4$-manifolds with boundary and fundamental group $\Z$}

\begin{abstract}
We classify topological $4$-manifolds with boundary and fundamental group $\Z$,  under some assumptions on the boundary. We apply this to classify surfaces in simply-connected $4$-manifolds with $S^3$ boundary, where the fundamental group of the surface complement is $\Z$. We then compare these homeomorphism classifications with the smooth setting.
For manifolds, we show that every Hermitian form over $\Z[t^{\pm 1}]$ arises as the equivariant intersection form of a pair of exotic smooth 4-manifolds with boundary and fundamental group $\Z$.
For surfaces we have a similar result,  and in particular we show that every $2$-handlebody with $S^3$ boundary contains a pair of exotic discs.
%We also prove a similar and exotic surfaces with fundamental group $\Z$ for every reasonable equivariant intersection form.
%AC: It's not clear that Kyle's discs would stay exotic if put in a collar neighorhood of  $S^3$ boundary so this result is non-trivial, even as stated.
\end{abstract}
\author[A.~Conway]{Anthony Conway}
\address{The University of Texas at Austin, Austin TX 78712}
%%%
\email{anthony.conway@austin.utexas.edu}
\author[L.~Piccirillo]{Lisa Piccirillo}
\address{The University of Texas at Austin, Austin TX 78712}
\email{lisa.piccirillo@austin.utexas.edu}
%%%
\author[M.~Powell]{Mark Powell}
\address{School of  Mathematics and Statistics, University of Glasgow, United Kingdom}
\email{mark.powell@glasgow.ac.uk}

\maketitle
%\tableofcontents

In what follows a~$4$-manifold is understood to mean a compact, connected, oriented, topological~$4$-manifold.
Freedman classified closed~
%simply-connected~
%LP refered to fundamental group explicitly so that the group in the next sentence is better set up
$4$-manifolds with trivial fundamental group up to orientation-preserving homeomorphism.
Other groups~$\pi$ for which classifications of closed~$4$-manifolds with fundamental group~$\pi$ are known include~$\pi \cong \Z$,~\cite{FreedmanQuinn,WangThesis,StongWang}, ~$\pi$ a finite cyclic group~\cite{HambletonKreck},  and~$\pi$ a solvable Baumslag-Solitar group~\cite{HambletonKreckTeichner}.
%and aspherical with good fundamental group for which the high dimensional Borel conjecture is known~\cite{FreedmanQuinn}.
Complete classification results for manifolds with boundary essentially only include the simply-connected case~\cite{BoyerUniqueness, BoyerRealization}; see also~\cite{StongRealization}.

This paper classifies~$4$-manifolds with boundary and fundamental group~$\Z$,  under some extra assumptions on the boundary.
We give an informal statement now.
Fix a
%boundary
closed 3-manifold~$Y$,
%surjection
an epimorphism
%$\phi$
$\varphi \colon \pi_1(Y)\twoheadrightarrow\Z$,
%equivariant intersection form
a nondegenerate Hermitian form $\lambda$ over $\Z[t^{\pm1}]$, and an additional piece of data specifying how the Alexander module of $Y$ interacts with $\lambda$.
%the equivariant second homology of the filling.
Then up to homeomorphism fixing $Y,$ there exists a unique $4$-manifold $M$ filling $Y$ inducing the specified data.
%$\purple{\varphi} \colon \pi_1(Y)\to \pi_1(M)=\Z$ with equivariant intersection $\lambda$,  and the specified interaction data.
%
%for a nondegenerate Hermitian form~$\lambda$ over~$\Z[t^{\pm 1}]$ and a 3-manifold~$Y$,
%% (satisfying some assumptions)
%we define~$\mathcal{V}^0_\lambda(Y)$ to be the set of 4-manifolds~$M$ with a homeomorphism $\partial M \cong Y$, fundamental group~$\pi_1(M) \cong \Z$, equivariant intersection form~$\lambda$, and $\pi_1(Y) \twoheadrightarrow \pi_1(M)$ surjective, considered up to orientation-preserving homeomorphism rel.\ boundary; see Definition~\ref{def:V0lambdaY} for a precise definition of~$\mathcal{V}^0_\lambda(Y)$.
%The fact that $\lambda$ is nondegenerate implies that the Alexander module $H_1(Y;\Z[t^{\pm 1}])$ is torsion, with the coefficient system determined by the homomorphism $\pi_1(Y) \to \pi_1(M) \cong \Z = \langle t \rangle$.
%
%
%Our main result, Theorem \ref{thm:ClassificationRelBoundary},  provides a bijection
%\[ b \colon \mathcal{V}^0_\lambda(Y)\to \mathcal{A}(\lambda,Y) \]
%where~$\mathcal{A}(\lambda,Y)$ is a set defined algebraically in terms of~$\lambda$ and~$Y$; the description of $\mathcal{A}(\lambda,Y)$ can be found in Theorem~\ref{thm:ClassificationRelBoundary} (and Definition~\ref{def:presentation}) and the construction of the map $b$ can be found in Subsection~\ref{sec:MainTechnicalIntro}.
%%while we refer to Subsection~\ref{sec:MainTechnicalIntro} for the construction of the map $b$.
Uniqueness is a consequence of~\cite[Theorem 1.10]{ConwayPowell}. Existence is the main contribution of this paper, Theorem~\ref{thm:MainTechnicalIntro}.
%We give a similar classification of such~$M$ up to homeomorphism \textit{not} rel.\ boundary, Theorem~\ref{thm:Classification}.
We give a similar non-relative classification of such $M$ in Theorem~\ref{thm:Classification}.
%{MP: this didn't sound very relative before, so I thought it might not be clear to the reader what non-relative means. I added `fixing $Y$' instead to arrange}

A feature of  our classification, which we shall demonstrate in Section~\ref{sec:NonTrivialbAut}, is the existence of arbitrarily large sets of homeomorphism classes of such 4-manifolds, all of which have the same boundary $Y$ and the same form $\lambda$.
Recently, this was extended~\cite{CCP,ConwayDaiMiller}, using the results of this paper, to produce infinite sets of homeomorphism classes with this property.
%Thus in this paper we obtain the first classification of a significant subset of orientable 4-manifolds exhibiting the phenomenon of infinitely many 4-manifolds with the same intersection form cf.\ \cite{Jahren-Kwasik,BDK-07}.
Thus this paper leads to the first classification of infinite families of orientable $4$-manifolds, all with the same, nontrivial, equivariant intersection form. This can be compared with~\cite{Jahren-Kwasik,BDK-07} and \cite[Theorem~1.2]{Kwasik-Schultz}, which produced infinite families of manifolds homotopy equivalent to $\R P^4 \# \R P^4$ and $L(p,q) \times S^1$ respectively; note that in both cases $\pi_2=0$ and so there is no intersection form.

%{AC: This has become one big wall of text (it used to be broken up by a display of the bijection $b$.)}

We apply our results to study compact, oriented, locally flat, embedded surfaces in simply-connected~$4$-manifolds where the fundamental group of the exterior is infinite cyclic; we call these \emph{$\Z$-surfaces}.
The classification of closed surfaces in~$4$-manifolds whose exterior is simply-connected was carried out by Boyer~\cite{BoyerRealization}; see also~\cite{Sunukjian}.
Literature on the classification of discs in $D^4$ where the complement has fixed fundamental group includes~\cite{FriedlTeichner,ConwayPowellDiscs,Conway}.
%Literature for other fundamental groups includes~\cite{FriedlTeichner,ConwayPowellDiscs,Conway}, for the case of discs in~$D^4$.
%where the fundamental group of the complement is the Baumslag-Solitar group~$BS(1,2)$.
For surfaces in more general $4$-manifolds,~\cite{ConwayPowell} gave necessary and sufficient conditions for a pair of~$\Z$-surfaces to be equivalent.
In this work,  for a~$4$-manifold~$N$ with boundary~$S^3$ and a knot~$K \subset S^3$, we classify~$\Z$-surfaces in~$N$ with boundary~$K$ in terms of the equivariant intersection form of the surface exterior; see Theorem \ref{thm:SurfacesRelBoundaryIntro}.
An application to $H$-sliceness can be found in Corollary~\ref{cor:HSliceIntro}, while Theorem~\ref{thm:SurfacesClosedIntro} classifies closed~$\Z$-surfaces.

Finally, we compare these homeomorphism classifications with the smooth setting.
We demonstrate that for every Hermitian form~$\lambda$ over $\Z[t^{\pm 1}]$ there are pairs of smooth 4-manifolds with boundary, ~$\pi_1 \cong \Z$,  and equivariant intersection form~$\lambda$ which are homeomorphic rel.\ boundary but not diffeomorphic; see Theorem \ref{thm:exoticmanifolds}.
%{MP: unsure about the optimal description of this, or what exactly we prove. Update: I removed rel boundary from the diffeo part due to Akbulut-Ruberman. I didn't change anything about this anywhere else though.}
We also show in  Theorem~\ref{thm:exoticdiscs} that for every Hermitian form~$\lambda$ satisfying conditions which are conjecturally necessary, there is a smooth 4-manifold~$N$ with~$S^3$ boundary containing a pair of smoothly embedded~$\Z$-surfaces whose exteriors have equivariant intersection form $\lambda$ and which are topologically but not smoothly isotopic rel.\ boundary.

\section{Statement of results}
\label{sec:StatementIntro}

%LP I have not yet chased the Z-manifold definition through the document.

Before stating our main result, we introduce some terminology.
Our 3-manifolds~$Y$ will always be oriented and will generally come equipped with an epimorphism~$\varphi\colon \pi_1(Y) \twoheadrightarrow \Z$.

\begin{definition}
An oriented~$4$-manifold~$M$ together with an identification $\pi_1(M) \cong \Z$ is said to be a  \emph{$\Z$-manifold} if the inclusion induced map~$\pi_1(\partial M) \to \pi_1(M)$ is surjective.
\end{definition}
%{MP: I think it's important it's easy to find this terminology so I put it into a defn.}
When we say that a $\Z$-manifold~$M$
%with an identification~$\pi_1(M) \cong \Z$
has boundary~$(Y,\varphi)$,  we mean that $M$ comes equipped with a homeomorphism $\partial M \xrightarrow{\cong} Y$ such that the composition~$\pi_1(Y) \twoheadrightarrow \pi_1(M) \xrightarrow{\cong}~\Z$ agrees with~$\varphi$.
We will always assume that the Alexander module~$H_1(Y;\Z[t^{\pm 1}])$ is~$\Z[t^{\pm 1}]$-torsion; recall that the Alexander module is the first homology group of the infinite cyclic cover~$Y^\infty \to Y$ corresponding to $\ker(\varphi)$.
The action of the deck transformation group $\Z = \langle t \rangle$ makes the first homology into a~$\Z[t^{\pm 1}]$-module.

\subsection{The classification result}\label{sub:MainThm}

Our goal is to classify~$\Z$-manifolds~$M$
%with~$\pi_1(M)\cong \Z$
whose boundary~$\partial M \cong Y$ has $H_1(Y;\Z[t^{\pm 1}])$ torsion, up to orientation-preserving homeomorphism.
The isometry class of the \textit{equivariant intersection form}~$\lambda_M$ on~$H_2(M;\Z[t^{\pm 1}])$ is an invariant of such~$M$
%(whose
(this definition is recalled in Subsection~\ref{sub:HomologyIntersections}) and so, to classify such~$M$, it is natural to first
%let~$\lambda$ be
fix a nondegenerate Hermitian form $\lambda$ over~$\Z[t^{\pm 1}]$,
%and
and then to classify~$\Z$-manifolds~$M$ with
%$\pi_1(M)\cong \Z$, ribbon
boundary~$\partial M \cong Y$, and
%fixed
equivariant intersection form~$\lambda$.
%As mentioned above,
The fact that $\lambda$ is nondegenerate implies that the Alexander module $H_1(Y;\Z[t^{\pm 1}])$ is torsion.

For such a~$4$-manifold~$M,$
%AC: in this sentence it is understood that Y,M are as above.
%with boundary~$\partial M \cong Y$,
 the equivariant intersection form~$\lambda_M$ on~$H_2(M;\Z[t^{\pm 1}])$ \emph{presents} the \emph{Blanchfield form} on $H_1(Y;\Z[t^{\pm 1}])$ (see Subsection~\ref{sub:EquivariantLinking})
$$\Bl_Y \colon H_1(Y;\Z[t^{\pm 1}]) \times H_1(Y;\Z[t^{\pm 1}]) \to \Q(t)/\Z[t^{\pm 1}],$$
We make this algebraic notion precise next.
%{MP: Shortened this part to minimise repetition.}
%
%
%
%For such a~$4$-manifold~$M$
%%AC: in this sentence it is understood that Y,M are as above.
%%with boundary~$\partial M \cong Y$,
%there is a relationship between the equivariant intersection form~$\lambda_M$ on~$H_2(M;\Z[t^{\pm 1}])$ and the \emph{Blanchfield form}
%$$\Bl_Y \colon H_1(Y;\Z[t^{\pm 1}]) \times H_1(Y;\Z[t^{\pm 1}]) \to \Q(t)/\Z[t^{\pm 1}],$$
%whose definition is recalled in Subsection~\ref{sub:EquivariantLinking}.
%Thus
%%, to classify the~$4$-manifolds~$M$ with~$\pi_1(M)\cong \Z$ and ribbon boundary~$Y$,
%%MP: repetitive, so removed a bit.
%we can restrict attention to forms~$\lambda$ which \emph{present}~$(H_1(Y;\Z[t^{\pm 1}]),\Bl_Y)$, an algebraic notion which we make precise now.
%
If~$\lambda \colon H \times H \to \Z[t^{\pm 1}]$ is a nondegenerate Hermitian form on a finitely generated free $\Z[t^{\pm 1}]$-module (for short, a \emph{form}), then we write~$\widehat{\lambda} \colon H \to H^*$ for the linear map~$z \mapsto \lambda(-,z)$, and
%It can be checked that if~$\lambda$ is represented by a matrix~$A$, then so is~$\widehat{\lambda}$.
there is a short exact sequence
$$ 0 \to H \xrightarrow{\widehat{\lambda}} H^* \xrightarrow{} \coker(\widehat{\lambda}) \to 0.$$
Such a presentation induces a \emph{boundary linking form}~$\partial \lambda$ on~$\coker(\widehat{\lambda})$ in the following manner.
For~$[x] \in \coker(\widehat{\lambda})$ with~$x \in H^*$, since $\coker(\widehat{\lambda})$ is $\Z[t^{\pm 1}]$-torsion there exist elements~$z\in H$ and~$p\in\Z[t^{\pm 1}] \sm \{0\}$ such that~$\lambda(-,z)=px\in H^*$.
Then for~$[x],[y]\in \coker(\widehat{\lambda})$ with~$x,y\in H^*$, we define
$$\partial\lambda([x],[y]):=\frac{y(z)}{p}\in\Q(t)/\Z[t^{\pm 1}].$$
One can check that~$\partial \lambda$ is independent of the choices of~$p$ and $z$.

\begin{definition}
\label{def:presentation}
For~$T$ a torsion~$\Z[t^{\pm 1}]$-module with a linking form~$\ell \colon T \times T \to \Q(t)/\Z[t^{\pm 1}]$,  a nondegenerate Hermitian form~$(H,\lambda)$ \textit{presents}~$(T,\ell)$ if there is an isomorphism~$h\colon\coker(\widehat{\lambda})\to T$ such that~$\ell(h(x),h(y))=\partial\lambda(x,y)$.
Such an isomorphism~$h$ is called an \emph{isometry} of the forms, the set of isometries is denoted~$\Iso(\partial\lambda,\ell)$.
If~$(H,\lambda)$ presents~$(H_1(Y;\Z[t^{\pm 1}]), \unaryminus \Bl_Y)$ then we say~$(H,\lambda)$ \emph{presents }$Y$.
\end{definition}

This notion of a presentation is well known (see e.g.~\cite{RanickiExact,CrowleySixt}), and appeared in the classification of simply-connected $4$-manifolds with boundary in~\cite{BoyerUniqueness,BoyerRealization} and in~\cite{ConwayPowell} for $4$-manifolds with $\pi_1 \cong \Z$. See also~\cite{BorodzikFriedlClassical1,FellerLewarkBalanced}.
Presentations capture the geometric relationship between the linking form of a 3-manifold and the intersection form of a 4-manifold filling.
%More precisely,  as explained in the following paragraph, the form~$(H_2(M;\Z[t^{\pm 1}]),\lambda_M)$ presents~$\partial M$.
To see why the form~$(H_2(M;\Z[t^{\pm 1}]),\lambda_M)$ presents~$\partial M$, one first observes that
the long exact sequence of the pair~$(M, \partial M)$ with coefficients in~$\Z[t^{\pm 1}]$  reduces to the short exact sequence
\[0 \to H_2(M;\Z[t^{\pm 1}]) \to H_2(M,\partial M;\Z[t^{\pm 1}]) \to H_1(\partial M;\Z[t^{\pm 1}]) \to 0,\]
where $H_2(M;\Z[t^{\pm 1}])$ and $H_2(M,\partial M;\Z[t^{\pm 1}])$ are finitely generated free $\Z[t^{\pm 1}]$-modules~\cite[Lemma 3.2]{ConwayPowell}.
The left term of the short exact sequence supports the equivariant intersection form~$\lambda_M$  and the right supports~$\Bl_{\partial M}$.
As explained in detail in~\cite[Remark 3.3]{ConwayPowell}, some algebraic topology gives the following commutative diagram of short exact sequences, where the isomorphism~$D_M$ is defined so that the right-most square commutes:
\begin{equation}
\label{eq:SES}
\xymatrix@R0.5cm{
0 \ar[r]& H_2(M;\Z[t^{\pm 1}]) \ar[r]^{\widehat{\lambda}_M}\ar[d]^-{\id}_=& H_2(M;\Z[t^{\pm 1}])^* \ar[r]^-{}\ar[d]_\cong^-{\operatorname{ev}^{-1} \circ \operatorname{PD}}& \coker(\widehat{\lambda}_M) \ar[d]^{\operatorname{D_M}}_\cong\ar[r]& 0 \\
0 \ar[r]& H_2(M;\Z[t^{\pm 1}]) \ar[r]& H_2(M,\partial M;\Z[t^{\pm 1}]) \ar[r]^-{}& H_1(\partial M;\Z[t^{\pm 1}]) \ar[r]& 0.
}
\end{equation}
It then follows that~$(H_2(M;\Z[t^{\pm 1}]),\lambda_M)$ presents~$\partial M$, where the isometry~$\partial\lambda_M\cong \unaryminus \Bl_{\partial M}$ is given by~$D_M$. For details see~\cite[Proposition 3.5]{ConwayPowell}.

Thus to classify the~$\Z$-manifolds~$M$ with
%~$\pi_1(M)\cong \Z$ and ribbon
boundary ~$\partial M \cong Y$, it suffices to consider forms~$(H,\lambda)$ which present~$Y$.
%The set of self isometries of~$(H,\lambda)$ is denoted~$\Aut(\lambda)$.
In Section~\ref{sec:MainTechnicalIntro} we use $D_M$ to define an additional \emph{automorphism invariant}
\[b_M \in \Iso(\partial\lambda,\unaryminus\Bl_Y)/\Aut(\lambda).\]
%The action of an isometry $F \in \Aut(\lambda)$ on $h \in \Iso(\partial\lambda,\unaryminus\Bl_Y$ is v}
%\MP{The action of $F \in \Aut(\lambda)$ in the isometry group by pre-composition with the induced isometry $\partial F$ \purple{will be} defined precisely} in Equation~\eqref{eq:autaction} \purple{below.}\footnote{AC: Suggestion based on your comments.
Here, as we define precisely in Equation~\eqref{eq:autaction} below,  an isometry $F \in \Aut(\lambda)$ induces an isometry $\partial F$ of $\partial \lambda$,  and the action on $h \in \Iso(\partial\lambda,\unaryminus\Bl_Y)$ is then by $F \cdot h=h \circ \partial F^{-1}.$
%We will describe the action of~$\Aut(\lambda)$ on~$\Iso(\partial\lambda,\unaryminus\Bl_Y)$ in Equation~\eqref{eq:autaction} of Construction~\ref{cons:Invariant}, in \purple{Section}~\ref{sec:MainTechnicalIntro}.
%\purple{Section~\ref{sec:MainTechnicalIntro} also contains the definition the \emph{automorphism invariant}
%$$b_M \in \Iso(\partial\lambda,\unaryminus\Bl_Y)/\Aut(\lambda).$$}
Additionally,  recall that a Hermitian form $(H,\lambda)$ is \emph{even} if $\lambda(x,x)=q(x)+\overline{q(x)}$ for some $\Z[t^{\pm 1}]$-module homomorphism $q \colon H \to \Z[t^{\pm 1}]$ and is \emph{odd} otherwise.
 Our first classification now reads as follows.
%LP: moved to end of paragraph after the proof (modulo our main technical theorem) will be given in Section~\ref{sec:MainTechnicalIntro}.

\begin{theorem}
\label{thm:ClassificationRelBoundary}
Fix the following data:
\begin{enumerate}
\item a closed
%oriented
 3-manifold $Y$,
\item
%a surjection
an epimorphism $\varphi \colon \pi_1(Y)\twoheadrightarrow\Z$ with respect to which the Alexander module of $Y$ is torsion,
\item a
%sesquilinear
nondegenerate Hermitian form $\lambda \colon H\times H\to\Z[t^{\pm 1}]$ which presents $Y$,
\item if $\lambda$ is odd, $k \in \Z_2,$
%\Z/2\Z$,
\item
%an isometry $b$ in
a class $b  \in \Iso(\partial \lambda,\unaryminus\Bl_Y)/\Aut(\lambda)$.
\end{enumerate}
Up to homeomorphism rel.\  boundary, there exists a unique
% compact oriented
%%AC: I removed compact oriented because I think we saw we always assume this
$\Z$-manifold $M$ with boundary~$(Y,\varphi)$, equivariant intersection form $\lambda$, automorphism invariant $b$ and, in the odd case,  Kirby-Siebenmann invariant $k$.

%There exists a unique compact oriented $\Z$-manifold $M$ with boundary $Y$ representing this data up to homeomorphism rel. boundary.
%Let~$Y$ be a~$3$-manifold with an epimorphism~$\varphi \colon \pi_1(Y) \twoheadrightarrow \Z$ whose Alexander module is torsion,  and let~$(H,\lambda)$ be a nondegenerate Hermitian form over $\Z[t^{\pm 1}]$.
% Consider the set~$\mathcal{V}_\lambda^0(Y)$ of~$4$-manifolds with~$\pi_1(M)\cong \Z$, ribbon boundary~$\partial M\cong Y$, and~$\lambda_M \cong \lambda$, considered up to homeomorphism rel.\ boundary.
%
%\noindent If the form $(H,\lambda)$ presents $Y$,
%\color{purple}
%Equivalently, if the first three invariants are fixed, then~$\mathcal{V}_\lambda^0(Y)$ is nonempty and corresponds bijectively to
%%{MP: rel.\ boundary implies orientation-preserving.}
%%\begin{enumerate}
%\begin{itemize}
%\item $\Iso(\partial \lambda,\unaryminus\Bl_Y)/\Aut(\lambda)$, if~$\lambda$ is an even form;
% \item $\left( \Iso(\partial \lambda,\unaryminus\Bl_Y)/\Aut(\lambda)\right) \times \Z_2$ if~$\lambda$ is an odd form. The map to $\Z_2$ is given by the Kirby-Siebenmann invariant.
%\end{itemize}
% \color{black}
% \end{enumerate}
\end{theorem}

%%AC: I commented this
%In Subsection~\ref{sec:MainTechnicalIntro} we will make more precise what it means for a 4-manifold $M$ to represent such a set of data. We also give the proof of this theorem (modulo our main technical theorem) in Subsection~\ref{sec:MainTechnicalIntro}.

Here two 4-manifolds $M_0$ and $M_1$ with boundary $Y$ are \emph{homeomorphic rel.\  boundary} if there exists a homeomorphism $M_0 \xrightarrow{\cong} M_1$ such that the restriction composed with the given parametrisations of the boundary, $Y \cong \partial M_0 \xrightarrow{\cong} \partial M_1 \cong Y$ is the identity on $Y$.
The uniqueness part of the theorem (which follows from~\cite{ConwayPowell}) can be thought of as answering whether or not a given pair of parametrisations $Y \cong \partial M_i$ extend to a homeomorphism $M_0 \cong M_1$.
We refer to Remark~\ref{rem:UserGuide} for a guide to applying the uniqueness statement of Theorem~\ref{thm:ClassificationRelBoundary}.
We give the proof of Theorem \ref{thm:ClassificationRelBoundary} (modulo our main technical theorem) in Section~\ref{sec:MainTechnicalIntro}.

\begin{remark}
\label{rem:MainTheorem}
We collect a couple of further remarks about this result.
%these results.
%{AC: I commented the remark about the bijection and put it in Section 2
%
%I've also removed the nonspin/nontorsion Alexander module case mostly to shorten the intro, but also because it might please Pete....
%
%I also removed what used to be the first remark about the bijection being explicit...because the remark stating that there is a bijection has been moved to Section 2.  MP: all of this seems good. }
\begin{itemize}
%\item
%In both of these theorems,
%In Theorem~\ref{thm:ClassificationRelBoundary}}, the construction of the} 4-manifolds is explicit, \color{black}and we describe it in \purple{Section}~\ref{sec:MainTechnicalIntro}.
%Additionally, note that  since~$(H,\lambda)$ is assumed to present~$Y$, there is an isometry~$\partial \lambda \cong \unaryminus \Bl_Y$ and fixing a choice of one such isometry leads to a bijection
%$$\Iso(\partial \lambda,\unaryminus\Bl_Y)/\Aut(\lambda) \approx \Aut(\partial \lambda)/\Aut(\lambda),$$
%where $\Aut(\partial \lambda)$ denotes the group of self-isometries of $\partial  \lambda$.
%Note however that this bijection is not canonical as it depends on the choice of the isometry~$\partial \lambda \cong -\Bl_Y$.
%There are pairs~$(Y,\lambda)$ for which the set $\Iso(\partial \lambda,\unaryminus\Bl_Y)/(\Aut(\lambda) \times \Homeo^+_\varphi(Y))$ is arbitrarily large, as we will outline in Example~\ref{ex:LargeStableClassIntro} and show in detail in Section~\ref{sec:NonTrivialbAut}.\footnote{MP: this is mentioned elsewhere so I think we can remove this sentence now.  AC: Done.}
%MP: Let's just not mention the infinite thing. Then it won't seem so upsetting. Perhaps.
%infinite, but this will not be pursued in the present paper.
\item The \emph{automorphism invariant} that distinguishes $\Z$-manifolds with the same equivariant form is nontrivial to calculate in practice, as its definition typically involves choosing identifications of the boundary $3$-manifolds; see Section~\ref{sec:MainTechnicalIntro}.
\item
%These theorems
Theorem~\ref{thm:ClassificationRelBoundary} should be thought of as an extension of the work of Boyer~\cite{BoyerUniqueness,BoyerRealization} that classifies simply-connected $4$-manifolds with boundary and fixed intersection form and an extension
of the classification of closed  $4$-manifolds with $\pi_1=\Z$~\cite{FreedmanQuinn,StongWang}.
Boyer's main  statements are formulated using presentations instead of isometries of linking forms, but both approaches can be shown to agree when the $3$-manifold is a rational homology sphere~\cite[Corollary E]{BoyerRealization}.
By way of analogy,  rational homology 3-spheres are to 1-connected 4-manifolds with boundary as pairs~$(Y,\varphi)$ with torsion Alexander module are to $\Z$-manifolds.
    \item For $(Y,\varphi)$ as above, it is implicit in Theorem~\ref{thm:ClassificationRelBoundary} and in~\cite{ConwayPowell} that if $M_0$ and $M_1$ are spin~$4$-manifolds with $\pi_1(M_i) \cong \Z$,  boundary homeomorphic to~$(Y,\varphi)$,  isometric equivariant intersection form, and the same automorphism invariant, then their Kirby-Siebenmann invariants agree.
    The argument is given in Remark~\ref{rem:KSProof} below, whereas Section~\ref{sub:Example} shows that the assumption on the automorphism invariants cannot be dropped.
    % having identical automorphism invariant is a necessary condition for the
    %automorphism invariants
%    \purple{Kirby-Siebenmann invariants} to agree.
    We refer to~\cite[Proposition~4.1~(vi)]{BoyerUniqueness} for the analogous fact in the simply-connected setting.
\end{itemize}
\end{remark}

\begin{example}
\label{ex:LargeStableClassIntro}
We will show in Proposition~\ref{prop:LargeStableClass} that there are examples of pairs~$(Y,\varphi)$ for which the set of 4-manifolds with fixed boundary $Y$ and fixed (even) equivariant intersection form, up to homeomorphism rel.\ boundary, can have arbitrarily large cardinality (in the recent~\cite{CCP,ConwayDaiMiller} examples with infinite cardinality were obtained).
Details are given in Section~\ref{sec:NonTrivialbAut}, but we note that the underlying algebra is similar to that which was used in~\cite{CCPS-short} and \cite{CCPS-long} to construct closed manifolds of dimension~$4k \geq 8$ with nontrivial homotopy stable classes.
This arbitrarily large phenomenon also exists for simply-connected 4-manifolds bounding rational homology spheres, which can be deduced from Boyer's work \cite{BoyerUniqueness,BoyerRealization} with a similar proof. On the other hand in the simply-connected setting there can only ever be finite such families.
\end{example}

In Theorem~\ref{thm:ClassificationRelBoundary}, we fixed a parametrisation of the boundary. By changing the parametrisation by a homeomorphism of $Y$ that intertwines $\varphi$, we can change the invariant $b  \in \Iso(\partial \lambda,\unaryminus\Bl_Y)/\Aut(\lambda)$ by post-composition with the induced automorphism of $-\Bl_Y$. This leads to an absolute (i.e.\ non-rel.\ boundary) classification analogous to Theorem~\ref{thm:ClassificationRelBoundary}, which we will formalise in Theorem~\ref{thm:Classification}. For now we highlight the following example, which contrasts with Example~\ref{ex:LargeStableClassIntro}.
%, we give a similar non-relative classification of $\Z$-manifolds.
%We delay the precise statement to Section~\ref{sec:MainTechnicalIntro} in favour of some examples illustrating its content.

%\color{black}
\begin{example}\label{example:bdy-surface-x-S1}
  If $Y \cong \Sigma_g \times S^1$ and $\varphi \colon \pi_1(\Sigma_g \times S^1) \to \pi_1(S^1) \to \Z$ is induced by projection onto the second factor, then for a fixed non-degenerate Hermitian form~$\lambda$ that presents $Y$,  if~$\lambda$ is even there is a unique homeomorphism class of 4-manifolds with~$\pi_1\cong \Z$, boundary $Y$, and equivariant intersection form~$\lambda$, and if $\lambda$ is odd there are exactly two such homeomorphism classes.  Here we allow homeomorphisms to act nontrivially on the boundary.
The key input is that every automorphism of $\Bl_Y$ can be realised by a homeomorphism of $Y$ that intertwines $\varphi$~\cite[Proposition~5.6]{ConwayPowell}.
Therefore, given two 4-manifolds for which the rest of the data coincide, by re-parameterising $Y$ we can arrange for the automorphism invariants to agree.
%  as shown in \cite[Proposition~5.6]{ConwayPowell},
%  %the isometries of the Blanchfield form of $Y$ exactly coincide with the symplectic group of isometries of the intersection form of~$\Sigma_g$.\footnote{AC: Why are we re-explaining CP20? Also why did we include the word simplectic?}
%%Since every such isometry is realised by a self-homeomorphism of~$\Sigma_g$ \cite[Section~2.1]{FarbMargalit},  it follows that
% the action of $\Homeo^+_\varphi(Y)$ on~$\Aut(\Bl_Y)$ has one orbit, and therefore the quotients in Theorem~\ref{thm:Classification} consist of a single orbit. In other words,  for a fixed
%%intersection
%non-degenerate Hermitian form~$\lambda$ that presents $Y$,  if~$\lambda$ is even there is a unique homeomorphism class of 4-manifolds with $\pi_1\cong \Z$, boundary $Y$, and equivariant intersection form~$\lambda$, and if $\lambda$ is odd there are two such homeomorphism classes.
\end{example}

In Section~\ref{sec:MainTechnicalIntro}
we describe the automorphism invariant $b$ from Theorem \ref{thm:ClassificationRelBoundary},
% we describe the bijection in Theorem
 %{AC: I think it's ok to still call it a bijection? MP: How about this instead?  ``we describe the automorphism invariant $b$ from Theorem \ref{thm:ClassificationRelBoundary}, ''}
%\ref{thm:ClassificationRelBoundary}
%and \ref{thm:Classification}
%explicitly,
give the statement of our main technical theorem on realisation of the invariants by $\Z$-manifolds,
%outline the proof of Theorem \ref{thm:ClassificationRelBoundary}
%%It's not an outline, it's a proof.
and explain how Theorem \ref{thm:ClassificationRelBoundary} implies a non rel.\ boundary version of the result.
%Theorem \ref{thm:Classification}.
 But first, in Subsections~\ref{sub:SurfaceIntro} and~\ref{sub:exoticaintro}, we discuss some applications.

\subsection{Classification of~$\Z$-surfaces in simply-connected~$4$-manifolds with~$S^3$ boundary}
\label{sub:SurfaceIntro}
For a fixed simply-connected 4-manifold~$N$ with boundary $S^3$ and a fixed knot~$K \subset \partial  N=S^3$, we call two locally flat embedded compact surfaces~$\Sigma,\Sigma' \subset N$ with boundary~$K \subset S^3$ \emph{equivalent rel.\ boundary} if there is an orientation-preserving homeomorphism~$(N,\Sigma) \cong (N,\Sigma')$ that is pointwise the identity on~$S^3 \cong \partial N$.
We are interested in classifying the~$\Z$-surfaces in~$N$ with boundary~$K$ up to equivalence rel.\ boundary.

As for manifolds, first we inventory some invariants of $\Z$-surfaces.
%first observe that the classification naturally decomposes into more  accessible classification problems once we fix the appropriate invariants.
The genus of $\Sigma$ and the equivariant intersection form~$\lambda_{N_\Sigma}$ on~$H_2(N_\Sigma;\Z[t^{\pm 1}])$ are invariants of such a surface~$\Sigma$, where~$N_\Sigma$ denotes  the  exterior~$N\smallsetminus \nu (\Sigma)$.
%Thus it is natural to split~$\Z$-surfaces for~$K$ in~$N$ into the subsets
%\[
%\operatorname{Surf(g)}^0_\lambda(N,K) := \lbrace \text{genus~$g$~$ \Z-$surfaces $\Sigma \subset N$ for~$K$ with } \lambda_{N_\Sigma} \cong \lambda \rbrace/\text{ equivalence rel.~$\partial$}.
%\]
%Again as for manifolds,  we now describe some necessary conditions for the set $\operatorname{Surf(g)}^0_\lambda(N,K)$ to be nonempty.
Write~$E_K:=S^3  \setminus \nu(K)$ for the exterior of~$K$ and recall that the boundary of~$N_\Sigma$ has a natural identification
$$\partial N_\Sigma\cong E_K \cup_\partial(\Sigma_{g,1} \times S^1)=:M_{K,g}.$$
 As discussed in Subsection \ref{sub:MainThm}, there is a relationship between the equivariant intersection form~$\lambda_{N_\Sigma}$ on~$H_2(N_\Sigma;\Z[t^{\pm 1}])$ and the Blanchfield form~$\Bl_{M_{K,g}}$ on~$H_1(M_{K,g};\Z[t^{\pm 1}])$: the Hermitian form $(H_2(N_\Sigma;\Z[t^{\pm 1}]), \lambda_{N_\Sigma})$ presents~$M_{K,g}$.
% Thus it suffices to restrict our attention to the subsets $\operatorname{Surf(g)}^0_\lambda(N,K)$,  where~$(H,\lambda)$ is a nondegenerate Hermitian form over~$\Z[t^{\pm 1}]$ that presents~$M_{K,g}$.

There is one additional necessary condition for a given form~$(H,\lambda)$ to be isometric to the intersection pairing~$(H_2(N_\Sigma;\Z[t^{\pm 1}]), \lambda_{N_\Sigma})$ for some surface~$\Sigma$. Observe that we can reglue the neighborhood of~$\Sigma$ to~$N_\Sigma$ to recover~$N$.
This is reflected in the intersection form, as follows.
We write~$\lambda(1):=\lambda \otimes_{\Z[t^{\pm 1}]} \Z_\varepsilon$, where $\Z_\varepsilon$ denotes~$\Z$ with the trivial $\Z[t^{\pm 1}]$-module structure.
% recall that if~$A(t)$ is a matrix that represents~$\lambda$,  then~$A(1)$ represents~$\lambda(1)$.
If~$W$ is a~$\Z$-manifold, then~$\lambda_W(1) \cong Q_W$,  where~$Q_W$ denotes the standard intersection form of~$W$; see e.g.~\cite[Lemma 5.10]{ConwayPowell}.
Therefore, if~$\lambda \cong \lambda_{N_\Sigma}$, then we have the isometries
$$\lambda(1) \cong \lambda_{N_\Sigma}(1)=Q_{N_\Sigma} \cong Q_N \oplus  (0)^{\oplus 2g},$$
where the last isometry follows from a Mayer-Vietoris argument.
%Thus,  for the set~$\operatorname{Surf(g)}^0_\lambda(N,K)$ to be nonempty, it is necessary both that~$\lambda$ presents $M_{K,g}$ and that~$\lambda(1)\cong Q_N \oplus  (0)^{\oplus 2g}$.
The following theorem (which is stated slightly more generally in Theorem~\ref{thm:SurfacesRelBoundary} below)
%AC: This is our main theorem on surfaces so I think it's good if the reader can click on the link to find the proof.
 shows that these invariants, with these two necessary conditions, are in fact also sufficient once an automorphism invariant is fixed.

%and lead to a description of $\operatorname{Surf(g)}^0_\lambda(N,K)$.
%\begin{customthm}{\ref{thm:SurfacesRelBoundary}}
\begin{theorem}
\label{thm:SurfacesRelBoundaryIntro}
Fix the following data:

\begin{enumerate}
\item  a simply-connected~$4$-manifold $N$ with boundary~$S^3$,
\item an oriented knot $K \subset S^3$,
\item
%genus
an integer~$g \in  \mathbb{Z}_{\geq 0},$
\item a nondegenerate Hermitian form~$(H,\lambda)$ over~$\Z[t^{\pm 1}]$ which presents~$M_{K,g}$ and satisfies~$\lambda(1)\cong Q_N \oplus  (0)^{\oplus 2g}$,
\item
%an isometry $b$ in
a class $b \in \Aut(\Bl_K)/\Aut(\lambda)$.
\end{enumerate}
Up to equivalence rel.\ boundary, there exists a unique genus~$g$~$\Z$-surface $\Sigma \subset N$ with boundary~$K$ whose exterior $N_\Sigma$ has equivariant intersection form $\lambda$ and automorphism invariant $b$.
%realizing this data, up to equivalence rel.\ boundary

%
%Let~$N$ be a simply-connected~$4$-manifold with boundary~$S^3$ and let~$K \subset S^3$ be a knot.
%If a nondegenerate Hermitian form~$(H,\lambda)$ over~$\Z[t^{\pm 1}]$ presents~$M_{K,g}$ and satisfies~$\lambda(1)\cong Q_N \oplus  (0)^{\oplus 2g}$,  then the set~$\operatorname{Surf(g)}^0_\lambda(N,K)$ is nonempty and there is a bijective correspondence
% $$\operatorname{Surf(g)}^0_\lambda(N,K) \cong \Aut(\Bl_K)/\Aut(\lambda).$$
%~$\Iso(\partial \lambda,\unaryminus\Bl_{M_{K,g}})/\Aut(\lambda).$
\end{theorem}
%\end{customthm}

The action of the group~$\Aut(\lambda)$ on the set~$\Aut(\Bl_K)$ arises by restricting the action of $\Aut(\lambda)$ on~$\Aut(\partial \lambda) \cong \Aut(\Bl_{M_{K,g}}) \cong \Aut(\Bl_K) \oplus \operatorname{Sp}_{2g}(\Z)$ to the first summand.
Here the (non-canonical) isomorphism~$\Aut(\partial \lambda) \cong  \Aut(\Bl_{M_{K,g}})$ holds because the form~$\lambda$ presents $M_{K,g}$,
%AC: And \Aut(\Bl_{M_{K,g}))=\Aut(-\Bl_{M_{K,g}))
while the isomorphism $\Aut(\Bl_{M_{K,g}}) \cong \Aut(\Bl_K) \oplus \operatorname{Sp}_{2g}(\Z)$ is a consequence of~\cite[Propositions 5.6 and 5.7]{ConwayPowell}.

Again,  the construction is explicit.
%we give an explicit description of the bijection.
%we describe the bijection explicitly.
The idea is that the set of topological surfaces (up to equivalence rel.\ boundary) is in bijection with the set of surface complements (up to homeomorphism rel.\ boundary).
So this theorem can be recovered from Theorem \ref{thm:ClassificationRelBoundary} by taking $Y$ to be $M_{K,g}$. We detail this in Section~\ref{sec:Discs} where we state the outcome as a bijection between $ \Aut(\Bl_K)/\Aut(\lambda)$ and the set of rel.\ boundary isotopy classes of $\Z$-surfaces $\Sigma \subset N$ with boundary $K$ and equivariant intersection form $\lambda_{N_\Sigma} \cong \lambda$.
%\color{black}
% we relate $\operatorname{Surf(g)}^0_\lambda(N,K)$ with a particular set of 4-manifolds $\mathcal{V}^0_\lambda(M_{K,g})$, and then we relate this to~$\Aut(\Bl_K)/\Aut(\lambda)$ via Theorem~\ref{thm:ClassificationRelBoundary}.
%Construction~\ref{cons:Invariant}.
%MP:{I liked that we said that the bijection is explicitly given earlier, just the existence of a bijection is less impressive.}
Finally, we note that when~$N=D^4$,  equivalence rel.\ boundary can be upgraded to isotopy rel.\ boundary via the Alexander trick.  See also \cite[Theorem~F]{Orson-Powell-MCG} for more cases when equivalence can be upgraded to isotopy.

%{AC: In the manifold case, I moved the bijection statement to Section 2. Here I just deleted it altogether; see the new purple text above Theorem~\ref{thm:SurfacesRelBoundary}.  MP: Ok! }

\begin{remark}
\label{rem:Discs}
%We briefly discuss the genus zero case, as it is often of independent interest.
%If the surface is a  disc, then~$M_{K,g}=M_K$.
% and therefore~$\Bl_{M_{K,g}}=\Bl_K$
%{MP: With the new formulation this doesn't contain so much information, we already have $\Bl_K$ there. }AC: That's true.  So I removed the ~$M_{K,g}=M_K$.
Previous classification results of locally flat discs in $4$-manifolds include $\Z$-discs in~$D^4$~\cite{FreedmanQuinn,ConwayPowellDiscs},  $BS(1,2)$-discs in~$D^4$~\cite{FriedlTeichner,ConwayPowellDiscs} and $G$-discs in $D^4$ (under some  assumptions on the group~$G$)~\cite{FriedlTeichner,Conway}. In the latter case it is not known whether there are groups satisfying the assumptions other than $\Z$ and $BS(1,2)$.
%Since
Our result is the first classification of discs with non simply-connected exteriors in 4-manifolds other than~$D^4$.
\end{remark}

%%%Don't delete just yet.
%\color{purple}
%\begin{remark}
%\label{rem:BijectionSurface}
%%As in Remark~\ref{rem:BijectionRelBoundary},
%%and~\ref{rem:BijectionNotRelBoundary}
%We can formulate Theorem~\ref{thm:SurfacesRelBoundaryIntro} as saying that if a Hermitian nondegenerate form presents $M_{K,g}$ and satisfies~$\lambda(1)\cong Q_N \oplus  (0)^{\oplus 2g}$,  then the set
%\[
%\operatorname{Surf(g)}^0_\lambda(N,K) := \lbrace \text{genus~$g$~$ \Z-$surfaces $\Sigma \subset N$ for~$K$ with } \lambda_{N_\Sigma} \cong \lambda \rbrace/\text{ equivalence rel.~$\partial$}.
%\]
%is nonempty and corresponds bijectively to $\Aut(\Bl_K)/\Aut(\lambda)$.
%\end{remark}
%\color{black}

Before continuing with $\Z$-surfaces, we mention an application of Theorem~\ref{thm:SurfacesRelBoundaryIntro} to $H$-sliceness.
A knot~$K$ in~$\partial N$ is said to be (topologically) \emph{$H$-slice} if~$K$ bounds a locally flat, embedded disc~$D$ in~$N$ that represents the trivial class in~$H_2(N,\partial N)$. The study of $H$-slice knots has garnered some interest recently because of its potential applications towards producing small closed exotic 4-manifolds~\cite{ConwayNagel, ManolescuMarengonSarkarWillis,
ManolescuMarengonPiccirillo,
IidaMukherjeeTaniguchi,ManolescuPiccirillo,KjuchukovaMillerRaySakalli}.
Since~$\Z$-slice knots are $H$-slice (see e.g.~\cite[Lemma~5.1]{ConwayPowell}), Theorem~\ref{thm:SurfacesRelBoundaryIntro} therefore gives a new criterion for topological~$H$-sliceness.
Our results also apply in higher genus.
When~$N=D^4$, this is reminiscent of the combination of~\cite[Theorems 2 and 3]{FellerLewarkOnClassical} and~\cite[Theorem 1.1]{BorodzikFriedlLinking} (and for $g=0$ it is Freedman's theorem that Alexander polynomial one knots bound $\Z$-discs~\cite{Freedman:1984-1,FreedmanQuinn}).
%(and no doubt implies it,  but we will not pusue this here).
%Feller-Lewark: g_Z \leq g_alg \leq  u_a
%Borodzik-Friedl u_a=n(K)
In connected sums of copies of~$\C P^2$, this is closely related to~\cite[Theorem~1.3]{KjuchukovaMillerRaySakalli}. Compare also~\cite[Theorem~1.10]{FellerLewarkBalanced}, which applies in connected sums of copies of $\C P^2 \# \overline{\C P}^2$ and $S^2 \times S^2$.

\begin{corollary}
\label{cor:HSliceIntro}
Let~$N$ be a simply-connected~$4$-manifold with boundary~$S^3$ and let~$K \subset S^3$ be a knot.
If~$\Bl_{M_{K,g}}$ is presented by a nondegenerate Hermitian matrix~$A(t)$ such that~$A(1)$ is congruent to~$Q_N \oplus (0)^{\oplus 2g}$, then~$K$ bounds a genus~$g$~$\Z$-surface in~$N$.
In particular, when~$g=0$,~$K$ is~$H$-slice in~$N$.
\end{corollary}

We also study~$\Z$-surfaces up to equivalence (instead of equivalence rel.\  boundary).
Here an additional technical requirement is needed on the knot exterior $E_K:=S^3 \setminus \nu(K)$.
%AC: We defined E_K before but a reminder might heplful for the reader who is not reading the paper sequentially.

%\begin{customthm}{\ref{thm:SurfacesWithBoundary}}
\begin{theorem}
\label{thm:SurfacesWithBoundaryIntro}
 Let ~$K$ be a knot in~$S^3$ such that every isometry of~$\Bl_K$ is realised by an orientation-preserving homeomorphism~$E_K \to E_K$.
 % that is the identity on the boundary $\partial E_K$.
 %{AC: We must decide if the ``identity on the boundary" assumption can be omitted. Most likely it can.
 If a nondegenerate Hermitian form~$(H,\lambda)$ over~$\Z[t^{\pm 1}]$  presents~$M_{K,g}$ and satisfies~$\lambda(1)\cong Q_N \oplus  (0)^{\oplus 2g}$, then up to equivalence, there exists a unique genus~$g$ surface~$\Sigma \subset N$ with boundary~$K$ and whose exterior  has equivariant intersection form~$\lambda$.
 \end{theorem}
% \end{customthm}

The classification of closed~$\Z$-surfaces then follows from Theorem \ref{thm:SurfacesWithBoundaryIntro}.
To state the result, given a closed simply-connected $4$-manifold $X$,  we use $X_\Sigma$ to denote the exterior of a surface $\Sigma \subset X$
%write
%$$\operatorname{Surf(g)}_\lambda(X)=\lbrace \text{genus~$g$~$ \Z-$surfaces $\Sigma \subset X$  with } \lambda_{X_\Sigma} \cong \lambda \rbrace/\text{ equivalence}$$
%as well as
and~$N:=X \setminus \mathring{D}^4$ for the  manifold obtained by puncturing $X$.
The details are presented in Section~\ref{sub:Closed}. The idea behind the  proof is that %for~$K=U$ the unknot, the sets~$\operatorname{Surf(g)}_\lambda(N,U)$ and~$\operatorname{Surf(g)}_\lambda(X)$ are
closed surfaces are in bijective correspondence, with surfaces with boundary $U$, so we can apply Theorem~\ref{thm:SurfacesWithBoundaryIntro}.

\begin{theorem}
\label{thm:SurfacesClosedIntro}
Let~$X$ be a closed simply-connected~$4$-manifold.
If a nondegenerate Hermitian form~$(H,\lambda)$ over~$\Z[t^{\pm 1}]$ presents~$\Sigma_g \times S^1$ and satisfies~$\lambda(1)\cong Q_X \oplus  (0)^{\oplus 2g}$, then there exists a unique $($up to equivalence$)$ genus~$g$ surface~$\Sigma \subset X$ whose exterior has equivariant intersection form~$\lambda$.
 \end{theorem}

Note that the boundary 3-manifold in question here, $\Sigma_g \times S^1$, is the same one that appeared in Example~\ref{example:bdy-surface-x-S1}.
 We conclude with a couple of remarks on Theorems~\ref{thm:SurfacesRelBoundaryIntro},~\ref{thm:SurfacesWithBoundaryIntro},  and~\ref{thm:SurfacesClosedIntro}.
Firstly,  we note that for
% the bijection in
 each theorem,
%case the bijections mentioned in these results,
the uniqueness statements follow from~\cite{ConwayPowell}.
 Our contributions in this work are the existence statements.
Secondly, we note that similar results were obtained for closed surfaces with simply-connected complements by Boyer~\cite{BoyerRealization}.
%%Commented out the Gordon-Luecke thing.
%Thirdly, note that in Theorem~\ref{thm:SurfacesWithBoundaryIntro} (and in~\cite[Theorem 1.3]{ConwayPowell}),  the Gordon-Luecke theorem~\cite{GordonLuecke} implies that the requirement that the homeomorphism~$E_K \to E_K$ be the identity on $\partial E_K$ can most likely be omitted, but we will not pursue this improvement here.
Some open questions concerning $\Z$-surfaces are discussed in Subsection~\ref{sub:OpenQuestions}.

\subsection{Exotica for all equivariant intersection forms }
\label{sub:exoticaintro}
So far, we have seen that the data in Theorems~\ref{thm:ClassificationRelBoundary} and~\ref{thm:SurfacesRelBoundaryIntro} determine  the topological type of $\Z$-manifolds and $\Z$-surfaces respectively.
%So far, we have seen that the equivariant intersection form and the automorphism invariant from Theorems~\ref{thm:ClassificationRelBoundary} and~\ref{thm:Classification} determine the topological type of our $4$-manifolds (and $\Z$-surfaces).
% and that under appropriate restrictions, every form arises in this way.
In what follows, we investigate the smooth failure of these statements.
%this statement.

One of the driving questions in smooth 4-manifold topology is whether every smoothable simply-connected closed 4-manifold admits multiple smooth structures. This question has natural generalisations to 4-manifolds with boundary and with other fundamental groups; we set up these generalisations with the following definition.

\begin{definition}
\label{def:ExoticallyRealisableRel}
For a 3-manifold~$Y$, a (possibly degenerate) symmetric form~$Q$ over~$\Z$ (resp.\ Hermitian form~$\lambda$ over~$\Z[t^{\pm 1}]$) is \textit{exotically realisable rel.~$Y$}
if there exists a pair of smooth simply-connected
%(resp.~$\pi_1 \cong \Z$)
4-manifolds~$M$ and~$M'$ with boundary $Y$ (resp.\ $\Z$-manifolds with boundary $Y$)
%AC: I guess you could ask for ribbon boundary even in the simply-connected case as everyone surjects to the trivial group but asking for it seems more confusing than anything.
 and intersection form~$Q$ (resp.\  equivariant intersection form~$\lambda$) such that there is an orientation-preserving homeomorphism~$F \colon M\to M'$ (for $\pi_1 \cong \Z$, we additionally require that $F$ respects the identifications of $\pi_1(M)$ and $\pi_1(M')$ with $\Z$) but no diffeomorphism~$G \colon M\to M'$.
 %with~$F|_\partial\cong G|_\partial$.
%AC: Requiring isometry on the forms forces it to commute with the maps on \pi_1. But perhaps too slick for our own good.
%AC: Ribbon boundary and \pi_1(M)=Z defines maps \pi_1(Y) -->>Z...but anyway, we do not need the obstruction here.
%AC: A priori the map could induce an anti-isometry on the \Lambda H_2.  But in fact asking for it to induce an isometry means that the homeo commutes with the identifications \pi_1(M)=Z=\pi_1(M'). So compatible coefficient systems.
%The maps from pi_1 of the boundary to Z are asked by definition the maps to \pi_1(M) and then to Z.
%Not also that we are not fixing the Y so can't fix the \varphi.
\end{definition}
%\begin{remark}
%One could instead make an absolute version of this definition, where one demanded the stronger condition that there be no diffeomorphism~$G:X\to X'$, one might define such a form to be \emph{absolutely exotically realisable rel.~$M$}.
%\end{remark}

In this language, the driving question above becomes (a subquestion of) the following: which symmetric bilinear forms over~$\Z$ are exotically realisable rel.~$S^3$?
%AC: This question asks two questions: realisability by smooth manifolds and then exotic realisability.
%This is one reason why the driving question is a subquestion.
%More to the point,  closed exotica for a form implies rel.\ S3 exotica for this form (remove a ball). (Converse most likely true). It is in this sense foremost that it is a subquestion.
There is substantial literature demonstrating that some forms are exotically realisable rel.~$S^3$ (we refer to~\cite{AkhmedovPark,AkhmedovPark2} both for the state of the art and for a survey of results on the topic) but there remain many forms, such as definite forms or forms with~$b_2<3$, for which determining exotic realisability rel.~$S^3$ remains out of reach.  For more general 3-manifolds, the situation is worse;  in fact it is an open question whether for every integer homology sphere $Y$ there exists \textit{some} symmetric  form $Q$ that is exotically realisable rel.~$Y$ \cite{EMM3manifolds}.
%there are many 3-manifolds~$Y$, including Brieskorn homology spheres, for which it is not known whether there exists a form over~$\Z$ that is exotically realisable rel.~$Y$ \cite{EMM3manifolds, GompfSteinembeddings}.

Presently there only seems to be traction on exotic realisability of intersection forms if one relinquishes control of the homeomorphism type of the boundary.
\begin{definition}\label{def:relexotic}
A symmetric form~$Q$ over~$\Z$ (resp.\ a Hermitian form~$\lambda$ over~$\Z[t^{\pm 1}]$) is \textit{exotically realisable} if there exists pair of smooth simply-connected
%(resp.~$\pi_1 \cong \Z$)
4-manifolds~$M$ and~$M'$ with intersection form~$Q$ (resp.\ $\Z$-manifolds with equivariant intersection form~$\lambda$)
%AC: I guess you could ask for ribbon boundary even in the simply-connected case as everyone surjects to the trivial group but asking for it seems more confusing than anything.
such that there is an orientation-preserving homeomorphism~$F \colon M\to M'$ (for $\pi_1 \cong \Z$, we additionally require that $F$ respects the identifications of $\pi_1(M)$ and $\pi_1(M')$ with $\Z$) but no diffeomorphism~$G \colon M\to M'$.
%AC: A priori the map could induce an anti-isometry on the \Lambda H_2.  But in fact asking for it to induce an isometry means that the homeo commutes with the identifications \pi_1(M)=Z=\pi_1(M').
%The maps on the boundary to Z are by definition the maps to \pi_1(M) and then to Z.
%Not also that we are not fixing the Y so can't fix the \varphi.
\end{definition}

The following theorem, which appears in \cite{AR16} for $n=0$ and \cite{AkbulutYasui} for $n>1$,  shows that contrarily to the closed setting, \emph{every} symmetric bilinear form over $\Z$ is exotically realisable.

\begin{theorem}[{Akubulut-Yasui~\cite{AkbulutYasui} and Akbulut-Ruberman~\cite{AR16}}]
\label{thm:exoticsimplyconn}
Every symmetric bilinear form~$(\Z^n,Q)$ over~$\Z$ is exotically realisable.
%there exists a pair of simply-connected smooth 4-manifolds~$M$ and~$M'$ with boundary~$Y$ and intersection form isometric to $Q$ that are homeomorphic but not diffeomorphic.
%For every symmetric bilinear form~$(\Z^n,Q)$ over~$\Z$, there exists a pair of simply-connected smooth 4-manifolds~$M$ and~$M'$ with the same boundary such that:
%\begin{enumerate}
%\item  there is a homeomorphism~$F \colon M\to M'$;
%\item  the equivariant intersection forms~$Q_M$ and~$Q_{M'}$ are isometric to~$Q$;
%\item  there is no diffeomorphism from~$M$ to~$M'$.
%\end{enumerate}
%In other words, every symmetric bilinear form~$(\Z^n,Q)$ over~$\Z$ is exotically realisable.
\end{theorem}

%Theorem \ref{thm:exoticsimplyconn} can be proved using a simplification of the argument we use to prove Theorem~\ref{thm:exoticmanifolds}, so we do not include a separate proof.\footnote{AC: Lisa do you think this statement should remain? Both answers are fine.}

Following our classification of $\Z$-manifolds with fixed boundary and fixed equivariant intersection form~$\lambda$ it is natural to ask which Hermitian forms~$\lambda$ are exotically realisable, with or without fixing a parametrisation of the boundary 3-manifold.
We resolve the latter.

\begin{theorem}\label{thm:exoticmanifolds}
Every Hermitian form $(H,\lambda)$ over~$\Z[t^{\pm 1}]$ is exotically realisable.
%there exist two smooth $\Z$-manifolds~$M$ and~$M'$ with boundary~$(Y,\varphi)$ and intersection form isometric to $\lambda$ that are homeomorphic but not diffeomorphic.
%For every Hermitian form~$(H,\lambda)$ over~$\Z[t^{\pm 1}]$ there exists a pair of smooth~$\Z$-manifolds~$M$ and~$M'$
%%with ribbon boundary and fundamental group~$\Z$,
%such that:
%\begin{enumerate}
%%\item  The equivariant intersection forms~$\lambda_M$ and~$\lambda_{M'}$ are isometric to~$\lambda$.
%\item  there is a homeomorphism~$F \colon  M\to M'$;
%%inducing some~$f \colon \partial M\to \partial M'$;
%\item   $F$ induces an isometry $\lambda_M \cong \lambda_{M'}$, and both forms are isometric to~$\lambda$;
%%AC: This condition in fact ensures that $F$ intertwines the identifications which is what we asked for.
%%This condition holds for our examples because they are built using a cork-twist.
%\item  there is no diffeomorphism from~$M$ to~$M'$.
%\end{enumerate}
%In other words, every Hermitian form~$(H,\lambda)$ over~$\Z[t^{\pm 1}]$ is exotically realisable.
\end{theorem}

4-manifold topologists are also interested in finding smooth surfaces which are topologically but not smoothly isotopic.
While literature in the closed case includes~\cite{FinashinKreckViro,FintushelStern,
KimModifying,KimRubermanSmooth,
KimRubermanTopological,Mark,HoffmanSunukjian}
there has been a recent surge of interest in the relative setting on which we now focus~\cite{JuhaszMillerZemke,
Hayden,
HaydenKjuchukovaKrishnaMillerPowellSunukjian,
HaydenSundberg, DaiMallickStoffregen}; see also~\cite{AkbulutZeeman}.
Most relevant to us are the exotic ribbon discs from~\cite{Hayden}.
In order to prove that his discs in $D^4$ are topologically isotopic, Hayden showed that their exteriors have group $\Z$ and appealed to~\cite{ConwayPowellDiscs}.
From the perspective of this paper and~\cite{ConwayPowell},  any two $\Z$-ribbon discs are isotopic rel.\ boundary because their exteriors are aspherical and therefore have trivial equivariant intersection form.
To generalise Hayden's result to other forms than the trivial one, we introduce some terminology.

\begin{definition}
\label{def:realisedByExoticSurfaces}
For a fixed smooth simply-connected 4-manifold~$N$, with boundary $S^3$, a form~$\lambda$ over~$\Z[t^{\pm 1}]$ is \emph{realised by exotic $\Z$-surfaces in~$N$} if there exists a pair of smooth properly embedded~$\Z$-surfaces~$\Sigma$ and~$\Sigma'$ in~$N$,  with the same boundary,  whose exteriors have equivariant intersection forms isometric to~$\lambda$,  and which are topologically but not smoothly isotopic rel.\ boundary.
%AC: Since the homeo is rel.\ boundary and these are Z-suraces, one can check that the homeo necessarily commutes with the maps pi_1(disc exteriors)to Z and therefore induces a isometry (and not an anti-isometry) between the forms. So we do not need the same extra condition as in the previous exotic definition.
\end{definition}

Using this terminology,  Hayden's result states that the trivial form is realised by exotic $\Z$-discs (in~$D^4$).
The next result shows that in fact \emph{every} form is realised by exotic $\Z$-discs.

\begin{theorem}\label{thm:exoticdiscs}
Every Hermitian form~$(H,\lambda)$ over~$\Z[t^{\pm 1}]$,  such that~$\lambda(1)$ is realised as the intersection form of a smooth simply-connected 4-dimensional 2-handlebody~$N$ with boundary $S^3$, is realised by exotic $\Z$-discs in $N$.
%there exists a pair of smooth~$\Z$-discs~$D$ and~$D'$ in~$N$ with the same boundary and the following properties:
 %AC: Note that since $\lambda(1)$ is non singular, $\lambda$ must be nondegenerate.
% \begin{enumerate}
% \item the equivariant intersection forms~$\lambda_{N_D}$ and~$\lambda_{N_{D'}}$ are isometric to~$\lambda$;
% \item $D$ is topologically isotopic to~$D'$ rel.\ boundary;
% \item $D$ is not smoothly equivalent to~$D'$ rel.\ boundary.
% \end{enumerate}
% In other words,  for every~$\lambda, N$ satisfying the hypothesis, $\lambda$ can be realised by exotic $\Z$-discs in~$N$.
\end{theorem}

\begin{remark}
\label{rem:Smooth}
We make a couple of remarks on Theorems~\ref{thm:exoticmanifolds} and~\ref{thm:exoticdiscs}.
%\footnote{MP: This applies to the other bullet point remark too: we could try to shorten them?  I'm worried about deleting something important, but one of the things you notice when scanning the intro is these rather long remarks. Feel free to disregard this if you like.  }
\begin{itemize}
\item The~$11/8$ conjecture predicts that every integer intersection form which is realisable by a  smooth~$4$-manifold with~$S^3$ boundary is realisable by a smooth 4-dimensional 2-handlebody with~$S^3$ boundary, thus our hypothesis on the realisability of~$\lambda(1)$ by 2-handlebodies is likely not an additional restriction (a nice exposition on why this follows from the~$11/8$ conjecture is given in~\cite[page 24]{HLSX}).
%AC: Basically all possible forms could then be realised as combinations of CP2, \bar{CP2}, K3s,  S2xS2 and those are 2-handlebodies (with no 3-handles, whence the assertion below).
%It is likely that with a little more care the proof of Theorem~\ref{thm:exoticdiscs} could be upgraded to prove the result under the milder assumption that~$\lambda(1)$ be realisable by a 4-manifold with~$S^3$ boundary and no 3-handles. Since the~$11/8$ conjecture again predicts that this is the same set of forms included in the present hypothesis, we do not pursue this upgrade.
\item The handlebody $N$ is very explicit: it can be built from $D^4$ by attaching $2$-handles according to $\lambda(1)$.
In particular,  when $\lambda$ is the trivial form, then $N=D^4$ and so Theorem~\ref{thm:exoticdiscs} demonstrates that there are exotic discs in $D^4$. This was originally proved in~\cite{Hayden}, and we note that our proof relies on techniques developed there.
\item The proof of Theorem~\ref{thm:exoticdiscs} also shows that every smooth $2$-handlebody with $S^3$ boundary contains a pair of exotic $\Z$-discs.
 We expand on this above the statement of Theorem~\ref{thm:ExoticDiscsMain}.
\end{itemize}
\end{remark}

We briefly mention the idea of the proof of Theorem~\ref{thm:exoticmanifolds}.
For a given Hermitian form $(H,\lambda)$ over~$\Z[t^{\pm 1}]$, we construct a Stein $4$-manifold $M$ with $\pi_1(M) \cong \Z$ and $\lambda_M \cong \lambda$ that contains a cork.
Twisting along this cork produces the $4$-manifold $M'$ and the homeomorphism $F \colon M \cong M'$.
%whose restriction to the boundaries is the required $f$.
We show that if $F|_\partial$ extended to a diffeomorphism $M \cong M'$,  two auxiliary $4$-manifolds $W$ and $W'$ (obtained from $M$ and $M'$ by adding a single $2$-handle) would be diffeomorphic.
We show this is not the case by proving that $W$ is Stein whereas $W'$ is not using work of Lisca-Matic~\cite{LiscaMatic}.  This proves that $M$ and $M'$ are non-diffeomorphic  rel.\ $F|_\partial$. We then use a result of \cite{AR16} to show that there exists a pair of smooth manifolds $V$ and $V'$, which are homotopy equivalent to~$M$ and~$M'$ respectively, and which are homeomorphic but not diffeomorphic to each other.
The proof of Theorem~\ref{thm:exoticdiscs} uses similar ideas.

\subsection*{Organisation}

In Section~\ref{sec:MainTechnicalIntro} we describe our main technical result and how it implies Theorem~\ref{thm:ClassificationRelBoundary}. %and~\ref{thm:Classification}.}
In Section~\ref{sec:Prelim}, we recall and further develop the theory of equivariant linking numbers.
In Section~\ref{sec:reidemeister-torsion} we review the facts we will need on Reidemeister torsion.
Section~\ref{sec:ProofMainTechnical}, we prove our main technical result, Theorem~\ref{thm:MainTechnicalIntro}.
Section~\ref{sec:Discs} is concerned with our applications to surfaces and in particular, we prove Theorems~\ref{thm:SurfacesRelBoundaryIntro},~\ref{thm:SurfacesWithBoundaryIntro}
and~\ref{thm:SurfacesClosedIntro}.
%Finally,
Our results in the smooth category, namely Theorems~\ref{thm:exoticmanifolds} and~\ref{thm:exoticdiscs}, are proved in Section~\ref{sec:ubiq}.
Finally,  Section~\ref{sec:NonTrivialbAut}  exhibits the arbitrarily large collections promised in Example~\ref{ex:LargeStableClassIntro}
%shows that the sets $\mathcal{V}_\lambda^0(Y)$ and $\mathcal{V}_\lambda^0(Y)$ can be arbitrarily large.
%{MP: changed this sentence, since it referred to $\mathcal{V}(Y)$ sets that have now not yet been defined. }

\subsection*{Conventions}
\label{sub:Conventions}

In Sections~\ref{sec:MainTechnicalIntro}-\ref{sec:Discs}
%In Sections~\ref{sec:Prelim}-\ref{sec:Discs}
 and~\ref{sec:NonTrivialbAut}, we work in the topological category with locally flat embeddings unless otherwise stated.
In Section~\ref{sec:ubiq},  we work in the smooth category.

From now on, all manifolds are assumed to be compact, connected, based and oriented; if a manifold has a nonempty boundary, then the basepoint is assumed to be in the boundary.

If $P$ is manifold and $Q \subseteq P$ is a submanifold with closed tubular neighborhood $\ol{\nu}(Q) \subseteq P$, then~$P_Q := P \setminus \nu(Q)$ will always denote the exterior of $Q$ in $P$, that is the complement of the open tubular  neighborhood.
The only exception to this use of notation is that the exterior of a knot~$K$ in $S^3$ will be denoted $E_K$ instead of~$S^3_K$.

We write~$p \mapsto \overline{p}$ for the involution on~$\Z[t^{\pm 1}]$ induced by~$t \mapsto t^{-1}$.
Given a~$\Z[t^{\pm 1}]$-module~$H$, we write~$\overline{H}$ for the~$\Z[t^{\pm 1}]$-module whose underlying abelian group is~$H$ but with module structure given by~$p \cdot h=\overline{p}h$ for~$h \in H$ and~$p \in \Z[t^{\pm 1}]$.
We write $H^*:=\overline{\Hom_{\Z[t^{\pm 1}]}(H,\Z[t^{\pm 1}])}$.

If a pullback map $F^*$ is invertible we shall abbreviate $(F^*)^{-1}$ to $F^{-*}$.
Similarly, for an invertible square matrix $A$ we write $A^{-T} := (A^T)^{-1}$.

\subsection*{Acknowledgments}
We thank the referee of a previous draft of this paper for helpful comments on the exposition.
L.P.\ was supported in part by a Sloan Research Fellowship and a Clay Research Fellowship. L.P.\ thanks the National Center for Competence in Research (NCCR) SwissMAP of the Swiss National Science Foundation for their hospitality during a portion of this project.  M.P.\ was partially supported by EPSRC New Investigator grant EP/T028335/2 and EPSRC New Horizons grant EP/V04821X/2.

\section{The main technical realisation statement}
\label{sec:MainTechnicalIntro}
%\color{red} Move V and S set definitions here
The goal of this section is to formulate our main technical theorem,  to explain how it implies Theorem~\ref{thm:ClassificationRelBoundary} from the introduction, and to formulate its non-relative analogue.
Along the way we also define the automorphism invariant in more detail.
We begin by defining a set of $\Z$-manifolds $\mathcal{V}_\lambda^0(Y)$ with boundary $Y$ and intersection form $\lambda$.
%{MP: modified this to again account for the fact that the $V$ thing is not defined until a little further down the page.}
 Then we describe a map~$b\colon \mathcal{V}_\lambda^0(Y)\to\Iso(\partial \lambda,\unaryminus\Bl_Y)/\Aut(\lambda)$.
Theorem~\ref{thm:ClassificationRelBoundary} (as formulated in Remark~\ref{rem:BijectionRelBoundary}) then reduces to the statement that~$b$ is a bijection. As we will explain, the injectivity of~$b$ follows from~\cite[Theorem~1.10]{ConwayPowell}. The main technical result of this paper is Theorem \ref{thm:MainTechnicalIntro}, which gives the surjectivity of~$b$ (and thus implies Theorem \ref{thm:ClassificationRelBoundary}).
We also prove in this section that Theorem \ref{thm:Classification}, our absolute (i.e. non- rel. boundary) homeomorphism classification result,  follows from Theorem \ref{thm:ClassificationRelBoundary}.
We finish the section with an outline of the proof of Theorem~\ref{thm:MainTechnicalIntro}.

We start by describing the set $\mathcal{V}_\lambda^0(Y)$ from Theorem~\ref{thm:ClassificationRelBoundary} more carefully.

\begin{definition}
\label{def:V0lambdaY}
Let~$Y$ be a~$3$-manifold with an epimorphism~$\varphi \colon \pi_1(Y) \twoheadrightarrow \Z$ whose Alexander module is torsion, and let~$(H,\lambda)$ be a Hermitian form presenting~$Y$.
Consider the set~$S_\lambda(Y)$ of pairs~$(M,g)$, where
\begin{itemize}
\item $M$ is a~$\Z$-manifold with a fixed identification $\pi_1(M) \xrightarrow{\cong} \Z$, equivariant intersection form isometric to~$\lambda$, and boundary homeomorphic to~$Y$;
\item $g \colon  \partial M \xrightarrow{\cong} Y$ is an orientation-preserving homeomorphism such that~$Y \xrightarrow{g^{-1},\cong} \partial M \to M$ induces~$\varphi$ on fundamental groups.
\end{itemize}
Define~$\mathcal{V}_\lambda^0(Y)$ as the quotient of~$S_\lambda(Y)$ in which two pairs~$(M_1,g_1), (M_2,g_2)$ are deemed equal if and only if there is a homeomorphism~$\Phi \colon M_1 \cong M_2$ such that~$ \Phi|_{\partial M_1}=g_2^{-1} \circ g_1$.   Note that such a homeomorphism is necessarily orientation-preserving because $g_1$ and $g_2$ are.
%Use the long exact sequence of the pair (M_i,\partial M_i)
For conciseness, we will say that~$(M_1,g_1)$ and~$(M_2,g_2)$ are \emph{homeomorphic rel.\ boundary} to indicate the existence of such a homeomorphism~$\Phi$.
%AC: Note that the data of $(M,g)$ determines an identification $\pi_1(M)=\Z$
\end{definition}

\begin{remark}
\label{rem:BijectionRelBoundary}
%We highlight a useful re-formulation of Theorem~\ref{thm:ClassificationRelBoundary}.
%Let $Y$ be a closed $3$-manifold
%%, let $\varphi \colon \pi_1(M) \twoheadrightarrow \Z$ be an epimorphism,
%and let $\lambda \colon H \times H \to \Z[t^{\pm 1}]$ be a nondegenerate Hermitian form.
%%We state an equivalent formulation of Theorem~\ref{thm:ClassificationRelBoundary} in terms of the set~$\mathcal{V}^0_\lambda(Y)$ of 4-manifolds~$M$ with a homeomorphism $\partial M \cong Y$, fundamental group~$\pi_1(M) \cong \Z$, equivariant intersection form~$\lambda$, and $\pi_1(Y) \twoheadrightarrow \pi_1(M)$ surjective, considered up to orientation-preserving homeomorphism rel.\ boundary; see Definition~\ref{def:V0lambdaY} for a precise definition of~$\mathcal{V}^0_\lambda(Y)$.
%Define~$\mathcal{V}^0_\lambda(Y)$ to be the set of 4-manifolds~$M$ with a homeomorphism $\partial M \cong Y$, \MP{an identification}~$\pi_1(M) \cong \Z$, equivariant intersection form~$\lambda$, and $\pi_1(Y) \twoheadrightarrow \pi_1(M)$ surjective, considered up to orientation-preserving homeomorphism rel.\ boundary; see Definition~\ref{def:V0lambdaY} for a precise definition of~$\mathcal{V}^0_\lambda(Y)$.
%\color{blue} In this language,
Using Definition~\ref{def:V0lambdaY}, Theorem~\ref{thm:ClassificationRelBoundary} is equivalent to the following statement.  \emph{If $\lambda$ presents $Y$, then~$\mathcal{V}_\lambda^0(Y)$ is nonempty and corresponds bijectively to}
\begin{itemize}
\item \emph{$\Iso(\partial \lambda,\unaryminus\Bl_Y)/\Aut(\lambda)$, if~$\lambda$ is an even form;}
 \item \emph{$\left( \Iso(\partial \lambda,\unaryminus\Bl_Y)/\Aut(\lambda)\right) \times \Z_2$ if~$\lambda$ is an odd form. The map to $\Z_2$ is given by the Kirby-Siebenmann invariant.}
\end{itemize}
%{MP: put the statement in italics. }
%In Theorem~\ref{thm:ClassificationRelBoundary}, the construction of \purple{the} 4-manifolds is explicit, \color{black}and we describe it in \purple{Section}~\ref{sec:MainTechnicalIntro}.
The bijection is explicit and will be constructed in Construction~\ref{cons:EmbVBijection}.

Additionally, note that  since~$(H,\lambda)$ is assumed to present~$Y$, there is an isometry~$\partial \lambda \cong \unaryminus \Bl_Y$ and fixing a choice of one such isometry leads to a bijection
$$\Iso(\partial \lambda,\unaryminus\Bl_Y)/\Aut(\lambda) \approx \Aut(\partial \lambda)/\Aut(\lambda),$$
where $\Aut(\partial \lambda)$ denotes the group of self-isometries of $\partial  \lambda$.
Note however that this bijection is not canonical as it depends on the choice of the isometry~$\partial \lambda \cong -\Bl_Y$.
\end{remark}

\begin{construction}
\label{cons:Invariant}
[Constructing the map~$b \colon \mathcal{V}_\lambda^0(Y)\to\Iso(\partial \lambda,\unaryminus\Bl_Y)/\Aut(\lambda)$.]
\label{cons:PresentationAssociatedToManifold}
Let~$Y$ be a~$3$-manifold with an epimorphism~$\varphi \colon \pi_1(Y) \twoheadrightarrow \Z$ whose corresponding Alexander module is torsion, and let~$(H,\lambda)$ be a form presenting~$Y$.
Let~$(M,g)$ be an element of~$\mathcal{V}^0_\lambda(Y)$, i.e.\ $M$ is a~$\Z$-manifold with
%~$\pi_1(M)\cong \Z$,
equivariant intersection form isometric to~$\lambda$ and $g \colon \partial M  \cong Y$ is a homeomorphism as in Definition~\ref{def:V0lambdaY}.

In the text preceding Theorem \ref{thm:ClassificationRelBoundary}, we showed how~$M$ determines an isometry~$D_M  \in \Iso(\partial \lambda_M, \unaryminus \Bl_{\partial M})$.
%$D_M\colon \coker(\widehat{\lambda}_M)\to H_1(\partial M;\Z[t^{\pm 1}])$.
Morally, one should think that this isometry~$D_M$ is the invariant we associate to~$M$.
For this to be meaningful however,  we instead need an isometry that takes value in a set defined in terms of just the 3-manifold~$Y$ and the form~$(H,\lambda)$, without referring to $M$ itself.
We resolve this by composing~$D_M$ with other isometries, so that our invariant is ultimately an element of~$\Iso(\partial \lambda,\unaryminus\Bl_Y)$. Once we have built the invariant, we will show it is well defined up to an action by~$\Aut(\lambda)$.

%First, we describe an isometry $g_*\colon \Bl_{\partial M} \cong  \Bl_Y$.
%~$ H_1(\partial M;\Z[t^{\pm 1}])\to H_1(Y;\Z[t^{\pm 1}])$.
%The fact that fillings are required to respect the chosen epimorphism~$\varphi\colon \pi_1(M)\twoheadrightarrow \Z$ implies that there exists some homeomorphism~$g \colon \partial M \cong Y$ such that
%Since the composition~$Y \xrightarrow{g^{-1}} \partial M \to M$ induces~$\varphi \colon \pi_1(Y) \twoheadrightarrow \Z$ on fundamental groups.
%Note that there is no preferred choice of~$g$.
We first use $g$ to describe an isometry  $\Bl_{\partial M} \cong \Bl_Y$.
Since on the level of fundamental groups~$g$ intertwines the maps to $\Z$, \cite[Proposition 3.7]{ConwayPowell} implies that~$g$ induces an isometry
$$g_* \colon \Bl_{\partial M} \cong \Bl_Y.$$
%We will describe how to associate an element of~$\Iso(\partial \lambda,\unaryminus\Bl_Y)$ to~$M$.

Next we describe an isometry $\partial \lambda \cong \partial \lambda_M$.
% isomorphism~$\partial F\colon\coker(\widehat{\lambda})\to \coker(\widehat{\lambda}_M)$.
The assumption that~$M$ has equivariant intersection form~$\lambda$ means by definition that there is an isometry  $F \colon \lambda  \cong \lambda_M$,  i.e.\ an isomorphism~$F \colon H\to H_2(M;\Z[t^{\pm 1}])$ that intertwines the forms~$\lambda$ and~$\lambda_M$. Note that there is no preferred choice of~$F$.
Any such~$F$ induces an isometry~$\partial F \in \Aut(\partial \lambda,\partial \lambda_M)$ as follows:
~$F \colon H\to H_2(M;\Z[t^{\pm 1}])$ gives an isomorphism~$(F^{*})^{-1} \colon H^*\to H_2(M;\Z[t^{\pm 1}])^*$ that descends to an isomorphism~$\coker(\widehat{\lambda})\cong \coker(\widehat{\lambda}_M)$ and is in fact an isometry; this is by definition
$$\partial F := (F^{*})^{-1} \colon \partial \lambda \cong \partial \lambda_M.$$
This construction is described in greater generality in~\cite[Subsection 2.2]{ConwayPowell}.
We shall henceforth abbreviate $(F^*)^{-1}$ to $F^{-*}$.

We are now prepared to associate  an isometry in~$\Iso(\partial \lambda,\unaryminus\Bl_Y)$ to~$(M,g)\in \mathcal{V}_\lambda^0(Y)$ as follows: choose an isometry~$F \colon \lambda_M \cong \lambda$ and consider the isometry
$$b_{(M,g,F)}:=g_* \circ D_M \circ \partial F \in \Iso(\partial \lambda,\unaryminus\Bl_Y).$$
We are not quite done, because we need to ensure that our invariant is independent of the choice of~$F$ and that $b$ defines a map on $\mathcal{V}_\lambda^0(Y)$.

First, we will make our invariant independent of the choice of~$F$. We require the following observation.
Given a Hermitian form $(H,\lambda)$ and linking form $(T,\ell)$,  there is a natural left action~$\Aut(\lambda) \curvearrowright \Iso(\partial \lambda,\ell)$ defined via
\begin{equation}\label{eq:autaction}
 G \cdot h  :=h \circ \partial G^{-1} \text{ for } G\in\Aut(\lambda) \text{ and } h \in   \Iso(\partial \lambda,\ell).
\end{equation}
In particular, we can consider~$$b_{(M,g)}:=g_* \circ D_M \circ \partial F \in \Iso(\partial \lambda,\unaryminus\Bl_Y)/\Aut(\lambda).$$
It is now not difficult to check that~$b_{(M,g)}$ is independent of the choice of~$F$.
%AC: if we have $F_i \colon \lambda \cong \lambda_M$ for i=1,2, then g_*D_M \partial F_1 \equiv g_*D_M \partial F_1 \partial (F_1^{-1} \circ F_2)=g_*D_M \partial F_2.

The fact that if~$(M_0,g_0)$ and~$(M_1,g_1)$ are homeomorphic rel.\ boundary (recall  Definition~\ref{def:V0lambdaY}), then $b_{(M_0,g_0)}=b_{(M_1,g_1)}$ follows fairly quickly.
%AC: one has to use the square (12) from~\cite[proof of Proposition 3.11]{ConwayPowell}.  There is a note on this in the dropbox.
From now on we omit the boundary identification~$g \colon  \partial M \cong Y$ from the notation, writing~$b_M$ instead of~$b_{(M,g)}$. This concludes the construction of our automorphism invariant.
\end{construction}

We are now ready to state our main technical theorem.

\begin{theorem}
\label{thm:MainTechnicalIntro}
Let~$Y$ be a~$3$-manifold with an epimorphism~$\varphi \colon \pi_1(Y) \twoheadrightarrow \Z$ whose Alexander module is torsion, and let~$(H,\lambda)$ be a nondegenerate Hermitian form presenting~$Y$.
If~$b \in \Iso(\partial \lambda,\unaryminus\Bl_Y)/\Aut(\lambda)$ is an isometry, then there is a~$\Z$-manifold~$M$ with equivariant intersection form~$\lambda_M \cong \lambda$,  boundary~$Y$ and~$b_M=b$.
If the form is odd, then~$M$ can be chosen to have either~$\ks(M)=0$ or~$\ks(M)=1$.
\end{theorem}

We now describe how to obtain Theorem~\ref{thm:ClassificationRelBoundary} (as formulated in Remark~\ref{rem:BijectionRelBoundary}) by combining this result with~\cite{ConwayPowell}.

\begin{proof}[Proof of Theorem~\ref{thm:ClassificationRelBoundary} assuming Theorem~\ref{thm:MainTechnicalIntro}]
First, notice that Theorem~\ref{thm:MainTechnicalIntro} implies the surjectivity portion of the statement in
Theorem~\ref{thm:ClassificationRelBoundary}.
It therefore suffices to prove that the assignment~$ \mathcal{V}_\lambda^0(Y) \to \Iso(\partial \lambda,\unaryminus\Bl_Y)/\Aut(\lambda)$ which sends~$M$ to~$b_M$ is injective for~$\lambda$ even, and that the assignment~$\mathcal{V}^0_\lambda(Y) \to\left( \Iso(\partial \lambda,\unaryminus\Bl_Y)/\Aut(\lambda)\right) \times~\Z_2$ which sends~$M$ to~$(b_M,\ks(M))$ is injective for~$\lambda$ odd.

Let~$(M_0,g_0),(M_1,g_1)$ be two pairs representing elements in~$\mathcal{V}_\lambda^0(Y)$.
Each~$4$-manifold~$M_i$ comes with an isometry~$F_i\colon (H, \lambda)\to (H_2(M_i;\Z[t^{\pm 1}]),\lambda_{M_i})$ and for $i=0,1$, the homeomorphisms $g_i \colon \partial M_i \to Y$ are as in Definition~\ref{def:V0lambdaY}.
We then get epimorphisms
$$(g_i)_* \circ D_{M_i}\circ \partial F_i\circ \pi\colon H^* \twoheadrightarrow H_1(Y; \Z[t^{\pm 1}]).$$
Here $\pi \colon H^* \to \coker(\widehat{\lambda})$ denotes the canonical projection.
We assume that~$b_{M_0} = b_{M_1}$ and, if~$\lambda$ is odd, then we additionally assume that~$\ks(M_0)=\ks(M_1)$. The fact that~$b_{M_0} = b_{M_1}$ implies that there is an isometry~$F \colon (H,\lambda) \cong (H,\lambda)$ that makes the following diagram commute:
$$
\xymatrix @C+1.5cm{
0\ar[r] &H \ar[r]^{\widehat{\lambda}}\ar[d]_F & H^* \ar[r]^-{(g_0)_* \circ  D_{M_0} \circ \partial F_0 \circ \pi}\ar[d]_-{F^{-*}}& H_1(Y;\Z[t^{\pm 1}])\ar[d]^=  \ar[r]&0 \\
0\ar[r] &H  \ar[r]^{\widehat{\lambda}}& H^* \ar[r]^-{(g_1)_* \circ D_{M_1} \circ \partial F_1 \circ \pi}& H_1(Y;\Z[t^{\pm 1}]) \ar[r]&0.
}
$$
But now, by considering the isometry~$G \colon \lambda_{M_0} \cong \lambda_{M_1}$ defined by~$G:=F_1 \circ F \circ F_0^{-1}$, a quick verification
%$$\partial_1 G^{-*}=\partial_{M_1} \circ F_1^{*} \circ F^{-*} \circ F_0^{-*}=(\id_Y)* \circ \partial_{M_0}~$$
shows that~$(G,\id_Y)$ is a compatible pair in the sense of~\cite{ConwayPowell}.
%AC: I wrote up a note on this.  Essentially you write out all the maps and note that the appropriate diagram for compatiblity commutes excatly because of the existence of F as above.
Consequently~\cite[Theorem 1.10]{ConwayPowell} shows that there is a homeomorphism~$M_0 \cong M_1$ extending~$\id_Y$ and inducing~$G$; in particular~$M_0$ and~$M_1$ are homeomorphic rel.\ boundary.
\end{proof}

\begin{remark}
\label{rem:KSProof}
For $(Y,\varphi)$ as in Theorem~\ref{thm:MainTechnicalIntro}, we explain the fact (already mentioned in Remark~\ref{rem:MainTheorem}) that if $M_0$ and $M_1$ are spin $4$-manifolds with $\pi_1(M_i) \cong \Z$,  boundary homeomorphic to~$(Y,\varphi)$, isometric equivariant intersection form, and the same automorphism invariant, then their Kirby-Siebenmann invariants agree.
As explained during the proof of Theorem~\ref{thm:ClassificationRelBoundary},  these assumptions ensure the existence of a compatible pair~$(G,\id_Y)$.
This in turn implies that $M:=M_0 \cup_{g_0 \circ g_1^{-1}} M_1$ is spin and has fundamental group~$\Z $~\cite[Theorem 3.12]{ConwayPowell}.
The assertion now follows from additivity of $\ks$ and Novikov additivity of the signature:
$$\ks(M_0)+\ks(M_1) = \ks(M) \equiv \frac{\sigma(M)}{8}=\frac{\sigma(M_0)-\sigma(M_1)}{8}=0 \quad \pmod 2.$$
We also use that the signatures of $M$, $M_0$, and $M_1$ can be obtained from the respective equivariant intersection forms by specialising to $t=1$ and taking the signature.

In Section~\ref{sub:Example}, we exhibit examples of spin 4-manifolds with boundary homeomorphic to $\unaryminus L(8,1) \# (S^1 \times S^2)$ and isometric equivariant intersection form that have different Kirby-Siebenmann invariants, demonstrating that the automorphism invariant was needed in the argument of this remark.
\end{remark}

%We conclude
%the introduction
%this section by
Next we outline the strategy of the proof of Theorem~\ref{thm:MainTechnicalIntro}.

\begin{proof}[Outline of the proof of Theorem~\ref{thm:MainTechnicalIntro}]
\label{pf:ProofStrategy}
The idea is to perform surgeries on~$Y$ along a set of generators of~$H_1(Y;\Z[t^{\pm 1}])$ to obtain a~$3$-manifold~$Y'$ with~$H_1(Y';\Z[t^{\pm 1}])=0$.
% (i.e.~$Y'$ has Alexander polynomial one).
The verification that $H_1(Y';\Z[t^{\pm 1}])=0$ uses Reidemeister torsion.
We then use surgery theory to show that this~$Y'$ bounds a~$4$-manifold~$B$ with~$B \simeq S^1$; this step relies on Freedman's work in the topological category~\cite{Freedman, FreedmanQuinn,DET}.
The~$4$-manifold~$M$ is then obtained as the union of the trace of these surgeries with~$B$.
To show that in the odd case both values of the Kirby-Siebenmann invariant are realised, we use the star construction~\cite{FreedmanQuinn,StongRealization}.
The main difficulty of the proof is to describe the correct surgeries on~$Y$ to obtain~$Y'$; this is where the fact that~$\lambda$ presents~$\Bl_Y$ comes into play: we show that generators of~$H_1(Y;\Z[t^{\pm 1}])$ can be represented by a framed link~$\widetilde{L}$ with equivariant linking matrix equal to minus the transposed inverse of a matrix representing~$\lambda$.
\end{proof}

This is a strategy similar to the one employed in  Boyer's classification of simply-connected~$4$-manifolds with a given boundary~\cite{BoyerRealization}.
The argument is also reminiscent of~\cite[Theorem~2.9]{BorodzikFriedlClassical1}, where Borodzik and Friedl obtain bounds (in terms of a presentation matrix for~$\Bl_K$) on the number of crossing changes required to turn~$K$ into an Alexander polynomial one knot: they perform surgeries on the zero-framed surgery~$Y=M_K$ to obtain~$Y'=M_{K'}$,  where~$K'$ is an Alexander polynomial one knot.

\begin{remark}
\label{rem:HomotopyEquivalence}
As we mentioned in Construction~\ref{cons:Invariant}, if~$M_0$ and~$M_1$
are
%orientation-preserving  AC: It's automatic.
homeomorphic rel.\ boundary, then~$b_{M_0}=b_{M_1}$ in~$\Iso(\partial \lambda,\unaryminus\Bl_Y)/\Aut(\lambda)$.
In fact the same proof shows more.
If two $4$-manifolds~$M_0$ and~$M_1$ that represent elements
of $\mathcal{V}^0_\lambda(Y)$ are \emph{homotopy equivalent} rel.\  boundary, then~$b_{M_0}=b_{M_1}$ in~$\Iso(\partial \lambda,\unaryminus\Bl_Y)/\Aut(\lambda)$.
\end{remark}

Next, we describe how the classification in the case where the homeomorphisms need not fix the boundary pointwise
%(i.e.\ Theorem \ref{thm:Classification})
follows from Theorem \ref{thm:ClassificationRelBoundary}.
To this effect,  we use~$\Homeo^+_\varphi(Y)$ to denote the orientation-preserving homeomorphisms of~$Y$ such that the induced map on~$\pi_1$ commutes with~$\varphi \colon \pi_1(Y) \twoheadrightarrow \Z$ and
%we
we
describe the set of homeomorphism classes of $\Z$-manifolds that we will be working with.
%the set $\mathcal{V}_\lambda(Y)$ from Theorem~\ref{thm:Classification} more precisely, similarly to  the way we defined~$\mathcal{V}_\lambda^0(Y)$ in Definition~\ref{def:V0lambdaY}.

\begin{definition}
\label{def:VlambdaY}
For $Y$ and $(H,\lambda)$ as in Definition~\ref{def:V0lambdaY},  define~$\mathcal{V}_\lambda(Y)$ as the quotient of~$S_\lambda(Y)$ in which two pairs~$(M_1,g_1), (M_2,g_2)$ are deemed equal  if and only if there is a homeomorphism~$\Phi \colon M_1 \cong M_2$ such that~$ \Phi|_{\partial M_1}=g_2^{-1} \circ f \circ g_1$ for some~$f \in \Homeo^+_\varphi(Y)$; note that such a homeomorphism $\Phi$ is necessarily orientation-preserving.
%Use the long exact sequence of the pair (M_i,\partial M_i)
%AC: This is introduced just above now so I don't think we need it here anymore.
%Here, recall  that~$\Homeo^+_\varphi(Y)$ denotes the set of those orientation-preserving homeomorphisms~$f \colon Y \cong Y$ such that~$\varphi \circ f_* =\varphi \colon \pi_1(Y) \twoheadrightarrow \Z$.
\end{definition}

%%%%%Don't delete: This used to be in the intro
%Use~$\Homeo^+_\varphi(Y)$ to denote the orientation-preserving homeomorphisms of~$Y$ such that the induced map on~$\pi_1$ commutes with~$\varphi \colon \pi_1(Y) \twoheadrightarrow \Z$.
%As we show in \purple{Section}~\ref{sec:MainTechnicalIntro} below,  Theorem~\ref{thm:ClassificationRelBoundary}  leads to a description of \color{blue}
%%~$\mathcal{V}_\lambda(Y)$, the analogous
%the set of fillings of~$Y$ with the rel.\ boundary condition omitted.  We will describe the action of~$\Aut(\lambda) \times \Homeo^+_\varphi(Y)$ on~$\Iso(\partial\lambda,\unaryminus\Bl_Y)$ in detail in~\eqref{eq:autaction2}, just below Definition~\ref{def:VlambdaY}.
%The non-relative classification statement
%%for~$\mathcal{V}_\lambda(Y)$
%reads as follows.
%%%

We continue to set up notation to describe how
%Theorem \ref{thm:Classification}
the non relative classification
follows from Theorem \ref{thm:ClassificationRelBoundary}.
Observe that the group~$\Homeo^+_\varphi(Y)$ acts on~$\mathcal{V}_\lambda^0(Y)$ by setting~$f \cdot (M,g):=(M,f \circ g)$ for~$f\in \Homeo^+_\varphi(Y)$. Further, observe that
\begin{equation}
\label{eq:NotRelBoundary}
 \mathcal{V}_\lambda(Y)=\mathcal{V}_\lambda^0(Y)/\Homeo^+_\varphi(Y).
 \end{equation}
Recall that  any~$f \in \Homeo^+_\varphi(Y)$ induces an isometry~$f_*$ of the Blanchfield form~$\Bl_Y$.
Thus the group~$\Homeo^+_\varphi(Y)$ acts on~$\Iso(\partial \lambda,\unaryminus \Bl_Y)$ by~$f \cdot h:=f_* \circ h$.
Finally, there is a natural left action~$\Aut(\lambda) \times \Homeo^+_\varphi(Y)$  on~$\Iso(\partial \lambda,\unaryminus \Bl_Y)$ defined via
\begin{equation}\label{eq:autaction2}
(F,f) \cdot h:=f_* \circ h \circ \partial F^{-1}.
\end{equation}

The non-relative classification statement reads as follows.

\begin{theorem}
\label{thm:Classification}
%%%%Don't delete: new version.
%Fix the following data:
%\begin{enumerate}
%\item a closed
%%oriented
%3-manifold $Y$,
%\item
%%a surjection
%\purple{an epimorphism} $\purple{\varphi} \colon \pi_1(Y)\purple{\twoheadrightarrow}\Z$ with respect to which the Alexander module of $Y$ is torsion,
%\item a
%%sesquilinear
%\purple{nondegenerate} Hermitian form $\lambda \colon H\times H\to\Z[t^{\pm 1}]$ which presents $Y$;
%\item if $\lambda$ is odd, $k \purple{\in} \Z_2$;
%\item
%%an isometry $b$ in
%\purple{a class} $\purple{b  \in} \Iso(\partial \lambda,\unaryminus\Bl_Y)/(\Aut(\lambda) \times \Homeo^+_\varphi(Y))$.
%\end{enumerate}
%
%\purple{Up to homeomorphism, there exists a unique
%%compact oriented
%%%AC: I removed compact oriented because I think we saw we always assume this
%$\Z$-manifold $M$ with boundary~$(Y,\varphi)$, equivariant intersection form $\lambda$, automorphism invariant $b$ and, in the odd case,  Kirby-Siebenmann invariant~$k$.}
%%%%%%%End of don't delete.
Let~$Y$ be a~$3$-manifold with an epimorphism~$\pi_1(Y) \twoheadrightarrow \Z$ whose Alexander module is torsion,  let~$(H,\lambda)$ be a nondegenerate Hermitian form over $\Z[t^{\pm 1}]$.
 Consider the set~$\mathcal{V}_\lambda(Y)$ of~$\Z$-manifolds
 %with~$\pi_1(M)\cong \Z$, ribbon
with boundary~$\partial M\cong Y$, and~$\lambda_M \cong \lambda$, considered up to orientation-preserving homeomorphism.
%AC: The word orientation-preserving is automatic here also, I think but removing it is perhaps a bit too slick at this point. See the definition of V_\lambda(Y); since it must restrict to something ori-pres on the boundary you basically get it for free.
%Really the 3-manifold is oriented etc.

 \noindent  If the form $(H,\lambda)$ presents $Y$, then~$\mathcal{V}_\lambda(Y)$ is nonempty and corresponds bijectively to
\begin{enumerate}
\item~$\Iso(\partial \lambda,\unaryminus\Bl_Y)/(\Aut(\lambda) \times \Homeo^+_\varphi(Y))$, if~$\lambda$ is an even form;
 \item $\left(  \Iso(\partial \lambda,\unaryminus\Bl_Y)/(\Aut(\lambda) \times \Homeo^+_\varphi(Y))\right) \times \Z_2$ if~$\lambda$ is an odd form. The map to $\Z_2$ is given by the Kirby-Siebenmann invariant.
 \end{enumerate}
%%In the even case, the bijection is induced by~$b$ while in the odd case, it is induced by~$b \times ks$.
\end{theorem}

%\color{purple}
%
%\begin{remark}
%\label{rem:BijectionNotRelBoundary}
%As in Remark~\ref{rem:BijectionRelBoundary} we can reformulate Theorem~\ref{thm:Classification} as a bijection involving the set of fillings $\mathcal{V}_\lambda(Y)$ of $Y$, but now with the rel.\ boundary condition omitted.
%The statement is that if $\lambda$ presents $Y$, then~$\mathcal{V}_\lambda(Y)$ is nonempty and corresponds bijectively to
%\begin{itemize}
%\item $\Iso(\partial \lambda,\unaryminus\Bl_Y)/(\Aut(\lambda) \times \Homeo^+_\varphi(Y))$, if~$\lambda$ is an even form;
% \item $\left(
%\Iso(\partial \lambda,\unaryminus\Bl_Y)/(\Aut(\lambda) \times \Homeo^+_\varphi(Y))
% \right) \times \Z_2$ if~$\lambda$ is an odd form. The map to $\Z_2$ is given by the Kirby-Siebenmann invariant.
%\end{itemize}
%\end{remark}

%%%From arxiv.
%We now explain how Theorem~\ref{thm:Classification}
%%\purple{(as formulated in Remark~\ref{rem:BijectionNotRelBoundary})},
%which gives a description of $\mathcal{V}_\lambda(Y)$   in terms of the orbit set~$\Iso(\partial \lambda,\unaryminus \Bl_Y)/(\Aut(\lambda) \times \Homeo^+_\varphi(Y))$, follows from Theorem~\ref{thm:ClassificationRelBoundary}, which describes the set~$\mathcal{V}_\lambda^0(Y)$ in  terms of the orbit set~$\Iso(\partial \lambda,\unaryminus \Bl_Y)/\Aut(\lambda).$
%%where  the manifolds in question were considered up to homeomorphism rel.\ boundary.

\begin{proof}
%[Proof of Theorem \ref{thm:Classification} assuming Theorem \ref{thm:ClassificationRelBoundary}]
Thanks to Theorem~\ref{thm:ClassificationRelBoundary}  (as formulated in Remark~\ref{rem:BijectionRelBoundary}) and~\eqref{eq:NotRelBoundary}, it suffices to prove that the map~$b$ respects the~$\Homeo^+_\varphi(Y)$ actions, i.e.\ that~$b_{f \cdot (M,g)}=f \cdot b_{(M,g)}$, where~$g \colon \partial M \cong Y$ is a homeomorphism as in Definition~\ref{def:V0lambdaY} and $f \in \Homeo^+_\varphi(Y)$.
%AC: This proves that the map b descends to the quotients. It then follows purely formally that it is a bijection. Wrote a note about this.
This now follows from the following formal calculation:
$b_{f \cdot (M,g)}=b_{(M,f \circ g)}=f_* \circ g_* \circ D_M \circ \partial F=f \cdot b_{(M,g)},$
 where~$F \colon \lambda_M \cong \lambda$ is an isometry and we used the definitions of the~$\Homeo^+_\varphi(Y)$ actions and of the map~$b$.
 %This is really formal.
\end{proof}

\begin{remark}
\label{rem:UserGuide}
To make the results as user friendly as possible,  we spell out how to apply them in practice.
% we spell out how to decide whether two
% %compact, orientable
%$\purple{\Z}$-manifolds $M_0$ and $M_1$ with
%% $\pi_1(M_i) \cong \Z$ and
%$\partial M_i \cong Y$ \purple{are orientation-preserving homeomorphic rel.\ boundary.}
%% for $i=0,1$
%%Fix an orientation on $Y$.
%The 4-manifolds $M_0$ and $M_1$ are homeomorphic if they have the same Kirby-Siebenmann invariants, and the following hold.}
Fix an oriented $3$-manifold $Y$ with torsion Alexander module.
Two orientable $\Z$-manifolds $M_0$ and $M_1$ with boundary $Y$ are homeomorphic if and only if they have the same Kirby-Siebenmann invariants, and the following hold.
  \begin{enumerate}
    \item There are identifications $\psi_i \colon \pi_1(M_i) \xrightarrow{\cong} \Z$, for $i=0,1$, and
    \item there are homeomorphism $g_i \colon Y \xrightarrow{\cong} \partial M_i$, for $i=0,1$, and a surjection $\pi_1(Y) \to \Z$, such that $\psi_i \circ \operatorname{incl}_i \circ g_i = \varphi$ for $i=0,1$,
    %$H_1(Y;\Z[t^{\pm 1}])$ is torsion,
    and such that
    \item using the coefficient system induced by the $\psi_i$, and the orientations induced by the $g_i$ to define the intersection forms, there is an isometry
    \[F \colon (H_2(M_0;\Z[t^{\pm 1}]),\lambda_{M_0}) \cong (H_2(M_1;\Z[t^{\pm 1}]),\lambda_{M_1}), \text{ and }\]
    \item with respect to this isometry we have that $b_{M_0} = b_{M_1} \in \Iso(\partial \lambda_{M_0},\unaryminus\Bl_Y)/(\Aut(\lambda_{M_0}) \times \Homeo^+_\varphi(Y))$ or,  equivalently, there exists an isometry $F \colon \lambda_{M_0} \cong \lambda_{M_1}$ whose algebraic boundary $\partial F \colon \partial \lambda_{M_0} \cong \partial \lambda_{M_1}$ is induced by some orientation-preserving homeomorphism $f \colon Y \to Y$ that intertwines $\varphi$.
    In~\cite{ConwayPowell} such a pair $(f,F)$ was called \emph{compatible}.
  \end{enumerate}
\end{remark}

The next few sections are devoted to proving Theorem~\ref{thm:MainTechnicalIntro}.

\section{Equivariant linking and longitudes}
\label{sec:Prelim}

We collect some preliminary notions that we will need later on.
In Subsection \ref{sub:HomologyIntersections} we fix our notation for twisted homology and equivariant intersections.
In Subsection~\ref{sub:EquivariantLinking},  we collect some facts about linking numbers in infinite cyclic covers, while in Subsection~\ref{sub:Parallels},  we define an analogue of integer framings of a knot in~$S^3$ for knots in infinite cyclic covers.

\subsection{Covering spaces and twisted homology}
\label{sub:HomologyIntersections}
We fix our conventions on twisted homology and recall some facts about equivariant intersection numbers.
We refer the reader interested in the intricacies of transversality in the topological category to~\cite[Section 10]{FriedlNagelOrsonPowell}.
\medbreak

We first introduce some notation for infinite cyclic covers.
%Given a CW complex~$X$
Given a space~$X$ that has the homotopy type of a finite CW complex,
 together with an epimorphism~$\varphi \colon \pi_1(X) \twoheadrightarrow \Z$, we write~$p\colon X^\infty \to X$ for the infinite cyclic cover corresponding to~$\ker(\varphi)$.
If~$A \subset X$ is a subspace
%subcomplex,
then we set~$A^\infty :=p^{-1}(A)$ and often write~$H_*(X,A;\Z[t^{\pm 1}])$ instead of~$H_*(X^\infty,A^\infty)$.
Similarly,  since~$\Q(t)$ is flat over~$\Z[t^{\pm 1}]$, we often write~$H_*(X,A;\Q(t))$ or~$H_*(X,A;\Z[t^{\pm 1}]) \otimes_{\Z[t^{\pm 1}]} \Q(t)$ instead of~$H_*(X^\infty,A^\infty) \otimes_{\Z[t^{\pm 1}]} \Q(t)$.

\begin{remark}
\label{rem:AlexanderPolynomial}
%Assume $H_1(X;\Z[t^{\pm 1}])$ is finitely generated.
%In this case,
The \emph{Alexander polynomial} of $X$, denoted $\Delta_X$ is the order of the \emph{Alexander module}~$H_1(X;\Z[t^{\pm 1}])$.
%AC: R is a Noetherian ring, then every finitely generated R-module is finitely presented.
While we refer to Remark~\ref{rem:AlexPoly} below for some recollections on orders of modules, here we simply note that $\Delta_X$ is a Laurent polynomial that is well defined up to multiplication by~$\pm t^k$ with~$k \in \Z$ and that if~$X=M_K$ is the $0$-framed surgery along a knot~$K$, then $\Delta_X$ is the Alexander polynomial of $K$.
%More generally,  if $X$ is a closed $3$-manifold, then $\Delta_X$ admits a symmetric representative, i.e.\ one that satisfies~$p(t)=p(t^{-1})$~\cite[Lemma 14.10]{TuraevIntroductionTo}.
\end{remark}

Next, we move on to equivariant intersections in covering spaces.

\begin{definition}
\label{def:EquivariantIntersection}
Let~$M$ be an~$n$-manifold (with possibly nonempty boundary) with an epimorphism~$\pi_1(M)\twoheadrightarrow \Z$.
For a~$k$-dimensional closed submanifold ~$A \subset M^\infty$ and an~$(n-k)$-dimensional closed submanifold~$A' \subset M^\infty$ such that $A$ and $t^jA'$ intersect transversely for all $j \in \Z$, we define the \emph{equivariant intersection}~$A\cdot_{\infty,M}A' \in \Z[t^{\pm1} ]$ as
$$A\cdot_{\infty, M}A'=\sum_{j\in \Z} (A\cdot_{M^\infty} (t^jA'))t^{-j},$$
 where~$\cdot_{M^\infty}$ denotes the usual (algebraic)  signed count of points of intersection.
If the boundary of~$M$ is nonempty and~$A' \subset M$ is properly embedded, then we can make the same definition and also write~$A\cdot_{\infty, M}A'  \in \Z[t^{\pm1} ]$.
\end{definition}

\begin{remark}
\label{rem:EquivariantIntersections}
We collect a couple of observations about equivariant intersections.
\begin{enumerate}
\item Equivariant intersections are well defined on homology and in fact~$A\cdot_{\infty, M}A'=\lambda([A'],[A])$, where~$\lambda$ denotes the equivariant intersection form
$$ \lambda \colon H_k(M;\Z[t^{\pm 1}]) \times H_{n-k}(M;\Z[t^{\pm 1}]) \to \Z[t^{\pm 1}].$$
The reason for which $A\cdot_{\infty, M}A'$ equals $\lambda([A'],[A])=\overline{\lambda([A],[A'])}$ instead of $\lambda([A],[A'])$ is due to the fact that we are following the conventions from~\cite[Section 2]{ConwayPowell} in which the adjoint of a Hermitian form $\lambda \colon H \times H \to \Z[t^{\pm 1}]$ is defined by the equation $\widehat{\lambda}(y)(x)=\lambda(x,y)$.
With these conventions $\lambda$ is linear in the first variable and anti-linear in the second, whereas~$\cdot_{\infty,M}$ is linear in the second variable and anti-linear in the first.
\item When~$\partial M \neq \emptyset$ and~$A \subset M$ is a properly embedded submanifold with boundary,  then again~$A\cdot_{\infty, M}A'=\lambda^\partial([A'],[A])$ where this time~$\lambda^\partial$ denotes the pairing
$$ \lambda^\partial \colon H_k(M;\Z[t^{\pm 1}]) \times H_{n-k}(M,\partial M;\Z[t^{\pm 1}]) \to \Z[t^{\pm 1}].$$
As previously $\lambda^\partial$ is linear in the first variable and anti-linear in the second.
\item The definition of the pairings~$\lambda$  and~$\lambda^\partial$ can be made with arbitrary twisted  coefficients.
In order to avoid extraneous generality, we simply mention that there are~$\Q(t)$-valued pairings~$\lambda_{\Q(t)}$ and~$\lambda_{\Q(t)}^\partial$ defined on homology with~$\Q(t)$-coefficients and that if~$A,B \subset M^\infty$ are closed submanifolds of complementary dimension, then~$\lambda_{\Q(t)}([A],[B])=\lambda([A],[B])$ and similarly for properly embedded submanifolds with boundary.
\end{enumerate}
\end{remark}

\subsection{Equivariant linking}
\label{sub:EquivariantLinking}
We recall definitions and properties of equivariant linking numbers.
Other papers that feature discussions of the topic include~\cite{PrzytyckiYasuhara, BorodzikFriedlLinking,KimRuberman}.
\medbreak

We assume for the rest of the section that~$Y$ is a~$3$-manifold and that~$\varphi \colon \pi_1(Y) \twoheadrightarrow \mathbb{Z}$ is an epimorphism such that the corresponding Alexander module~$H_1(Y;\Z[t^{\pm 1}])$ is torsion, i.e.\ ~$H_*(Y;\Q(t))=0$.
We also write~$p \colon Y^\infty \to Y$ for the infinite cyclic  cover corresponding to~$\ker(\varphi)$ so that~$H_1(Y;\Z[t^{\pm 1}])=H_1(Y^\infty)$.
Given a simple closed curve~$\widetilde{a} \subset Y^\infty$,
 we write~$a^\infty:=\bigcup_{k \in \Z} t^k \widetilde{a}$ for the union of all the translates of~$\widetilde{a}$ and~$a:=p(\widetilde{a}) \subset Y$ for the projection of~$\widetilde{a}$ down to~$Y$.
This way,  the covering map~$p \colon Y^\infty \to Y$ restricts to a covering map
$$ Y^\infty \setminus \nu(a^\infty) \to Y\setminus \nu (a)=:Y_a.$$
Since the Alexander module of~$Y$ is torsion, a short Mayer-Vietoris argument shows that the vector space~$H_*(Y_a;\Q(t))=\Q(t)$ is generated by~$[\widetilde{\mu}_a]$, the class of a meridian of~$\widetilde{a} \subset Y^\infty$.
\begin{definition}
\label{def:EquivariantLinking}
The \emph{equivariant linking number} of two disjoint simple closed curves~$\widetilde{a},\widetilde{b} \subset Y^\infty$ is the unique rational function~$\ell k_{\Q(t)}(\widetilde{a},\widetilde{b}) \in \Q(t)$ such that
$$ [\widetilde{b}]=\ell k_{\Q(t)}(\widetilde{a},\widetilde{b})[\widetilde{\mu}_a] \in H_1(Y \setminus \nu(a);\Q(t)).$$
\end{definition}

Observe that this linking number is only defined for \emph{disjoint} pairs of simple closed curves.
We give a second, more geometric, description of the equivariant linking number.

\begin{remark}\label{rem:torsionsurface}
Since~$H_1(Y;\Z[t^{\pm 1}])$ is torsion, for any simple closed curve $\widetilde{a}$ in $Y^\infty$, there is some polynomial~$p(t)=\sum_i c_it^i$ such that~$p(t)[\widetilde{a}]=0.$
Thus there is a surface~$F\subset Y^\infty\smallsetminus \nu(a^\infty)$ with boundary consisting of the disjoint union of~$c_i$ parallel copies of~$t^i\cdot \widetilde{a}'$ and $d_j$ meridians of $t^j\cdot \widetilde{a}'$ where~$\widetilde{a}'$ is some pushoff of~$\widetilde{a}$ in~$\partial \overline{\nu}(\widetilde{a})$ and $j \neq i$; we abusively write $\partial F=p(t)\widetilde{a}$.
%AC: You need the meridians because the translates might get in the way.  So a surface with boundary \widetilde{a} exists but it might they intersect some neighbooring surfaces; drill out holes at intersections and get additional meridians as boundary.
\end{remark}

%Don t delete.
%The surface can be taken in the complement: the parallels can be taken to be very close to a and take a surface in Y) and then take the neighborhood to be slightly bigger than the furthest parallel.
%The restriction of the surface to the complement now bounds parallels of the original parallels.

\begin{proposition}\label{prop:EquivariantLinkingDefinitions}
Let~$Y$ be a~$3$-manifold, let~$\varphi \colon \pi_1(Y) \twoheadrightarrow \Z$ be an epimorphism such that the Alexander module~$H_1(Y;\Z[t^{\pm 1}])$ is torsion, and let~$\widetilde{a},\widetilde{b} \subset Y^\infty$ be disjoint simple closed curves.

Let $F$ and $p(t)$ be respectively a surface and a polynomial associated to $\widetilde{a}$ as in Remark \ref{rem:torsionsurface}.
The equivariant linking of~$\widetilde{a}$ and~$\widetilde{b}$ can be written as
 \begin{equation}
 \label{eq:EquivariantLinkingGeometric}
\ell k_{\Q(t)}(\widetilde{a} ,\widetilde{b} )=\frac{1}{p(t^{-1})}\sum_{k \in \Z} (F  \cdot t^k \widetilde{b}) t^{-k}=\frac{1}{p(t^{-1})}(F\cdot_{\infty,Y_a} \widetilde{b}).
\end{equation}
In particular, this expression is independent of the choices of $F$ and $p(t)$.
\end{proposition}

\begin{proof}
As in Subsection~\ref{sub:HomologyIntersections}, write~$\lambda^\partial$ for the (homological) intersection  pairing~$H_1(Y_a;\Z[t^{\pm 1}])  \times H_2(Y_a,\partial Y_a;\Z[t^{\pm 1}]) \to \Z[t^{\pm 1}]$ and~$\lambda^\partial_{\Q(t)}$ for the pairing involving~$\Q(t)$-homology.

Write~$\ell:=\ell k(\widetilde{a},\widetilde{b})$ so that~$[\widetilde{b}]=\ell [\widetilde{\mu}_a]  \in H_1(Y_a;\Q(t))$.
From this and Remark~\ref{rem:EquivariantIntersections}, for a surface~$F$ as in the statement, we obtain
$$ F \cdot_{\infty,Y_a} \widetilde{b}
=\lambda^\partial([\widetilde{b}],[F])
=\lambda_{\Q(t)}^\partial([\ell \widetilde{\mu}_a],[F])
=\ell \lambda_{\Q(t)}^\partial([\widetilde{\mu}_a],[F])
=\ell(F \cdot_{\infty,Y_a} \widetilde{\mu}_a)
=\ell p(t^{-1}).$$
The last equality here follows from inspection; since $F\hookrightarrow Y^\infty\smallsetminus \nu (a^\infty)$ has boundary along $c_i$ copies of $t^i\cdot \widetilde{a'}$ and $d_j$ copies of $t^j\widetilde{\mu}_a$, each meridian $t^i\cdot \mu_{\widetilde{a}}$ intersects $F$ in $c_i$ points.
The result now follows after dividing out by~$p(t^{-1})$.
\end{proof}

Just as for linking numbers in rational homology spheres, the equivariant linking number is not well defined on homology, unless the target is replaced by $\Q(t)/\Z[t^{\pm 1}]$.
To describe the resulting statement, we briefly recall the definition of the Blanchfield form.

\begin{remark}
\label{rem:Needp(t)Symmetric}
Using the same notation and assumptions as in Proposition~\ref{prop:EquivariantLinkingDefinitions},  the Blanchfield form is a nonsingular sesquilinear, Hermitian pairing that can be defined as
\begin{align}
\label{eq:BlanchfieldGeom}
\Bl_Y \colon H_1(Y;\Z[t^{\pm 1}]) \times H_1(Y;\Z[t^{\pm 1}]) &\to \Q(t)/\Z[t^{\pm 1}]  \nonumber \\
([\widetilde{b}],[\widetilde{a}]) &\mapsto \left[\frac{1}{p(t)}(F\cdot_{\infty,Y_a} \widetilde{b})\right].
\end{align}
We refer to~\cite{PowellBlanchfield,FriedlPowell} for further background  and homological definitions of this pairing.
\end{remark}

%We refer to~\cite{PowellBlanchfield,FriedlPowell} for further background concerning this non singular sesquilinear, Hermitian pairing but note that the switch in variables is again due to the conventions on adjoints chosen in~\cite{FriedlPowell,ConwayPowell}.
%In what follows, the reader can think of~$\Bl_Y([\widetilde{b}],[\widetilde{a}])$ as being defined as the reduction mod~$\Z[t^{\pm 1}]$ of~$\ell k_{\Q(t)}(\widetilde{a} ,\widetilde{b} )$.
We summarise this discussion and collect another property of equivariant linking in the next proposition.

\begin{proposition}\label{prop:Linkingprop}
Let~$Y$ be a~$3$-manifold and let~$\varphi \colon \pi_1(Y) \twoheadrightarrow \Z$ be an epimorphism such that the Alexander module~$H_1(Y;\Z[t^{\pm 1}])$ is torsion.
For disjoint simple closed curves~$\widetilde{a},\widetilde{b} \subset Y^\infty$, the equivariant linking number satisfies the following properties:
\begin{enumerate}
\item sesquilinearity:~$\ell k_{\Q(t)}(p \widetilde{a} ,q  \widetilde{b} )=\overline{p}q\ell k_{\Q(t)}(\widetilde{a} ,\widetilde{b} )$ for all~$p,q \in \Z[t^{\pm 1}]$;
\item symmetry:~$\ell k_{\Q(t)}(\widetilde{a} ,\widetilde{b} )=\overline{\ell k_{\Q(t)}(\widetilde{b} ,\widetilde{a} )}$;
\item relation to the Blanchfield form:~$[\ell k_{\Q(t)}(\widetilde{a} ,\widetilde{b} )]=\Bl_Y([\widetilde{b}],[\widetilde{a}]) \in \Q(t)/\Z[t^{\pm 1}]$.
\end{enumerate}
\end{proposition}
\begin{proof}
The first property follows from~\eqref{eq:EquivariantLinkingGeometric}.
Before proving the second and third properties,we note that in~\eqref{eq:EquivariantLinkingGeometric} and~\eqref{eq:BlanchfieldGeom}, we can assume that $p(t)=p(t^{-1})$.
Indeed, both formulae are independent of the choice of $p(t)$ and if $q(t)$ satisfies $q(t)[\widetilde{a}]=0$, then so does $p(t):=q(t)q(t^{-1})$.
The proof of the second assertion now follows as in~\cite[Lemma 3.3]{BorodzikFriedlLinking}, whereas the third follows by inspecting~\eqref{eq:EquivariantLinkingGeometric} and~\eqref{eq:BlanchfieldGeom}.
\end{proof}

The reader will have observed that the formulas in Proposition~\ref{prop:EquivariantLinkingDefinitions}  and~\ref{prop:Linkingprop} depend heavily on conventions chosen for adjoints, module structures, equivariant intersections and twisted homology.
It is for this reason that the formulas presented here might differ (typically up to switching variables) from others in the literature.
%Ultimately the choice of these conventions does not affect the main results of this paper.
%As the paper progresses, the reader will no doubt notice that these conventions are ultimately irrelevant as far as the main results are concerned as long as
%One just puts tranposes or not at various points.
%Also we just work with properties not the actual formulas.

\subsection{Parallels, framings, and longitudes}
\label{sub:Parallels}

Continuing with the notation and assumptions from the previous section, we fix some terminology regarding parallels and framings in infinite cyclic covers. The goal is to be able to describe a notion of integer surgery for appropriately nullhomologous knots in the setting of infinite cyclic covers.
Our approach is inspired by~\cite{BoyerLines,BoyerRealization}.

\begin{definition}
\label{def:ParallelLongitude}
Let~$\widetilde{K} \subset Y^\infty$ be a knot,  let~$p \colon Y^\infty \to Y$ be the covering map, and denote~$K:=p(\widetilde{K})\subset Y$ the projection of~$\widetilde{K}$.
\begin{enumerate}
\item A \emph{parallel} to~$\widetilde{K}$ is a simple closed curve~$\pi \subset \partial \overline{\nu}(\widetilde{K})$ that is isotopic to~$\widetilde{K}$ in~$\overline{\nu}(\widetilde{K})$.
\item Given any parallel~$\pi$ of~$\widetilde{K}$, we use~$\overline{\nu}_\pi(\widetilde{K})$ to denote the parametrisation~$S^1\times D^2\xrightarrow{\cong} \overline{\nu}(\widetilde{K})$ which sends~$S^1\times\{x\}$ to~$\pi$ for some~$x\in\partial D^2$.
\item A \emph{framed link} is a link~$\widetilde{L} \subset Y^\infty$ together with a choice of a parallel for each of its components.
\item We say that the knot~$\widetilde{K}$ \emph{admits framing coefficient}~$r(t) \in \Q(t)$ if there is a parallel~$\pi$ with~$\ell k_{\Q(t)}(\widetilde{K},\pi)=r(t)$.
We remark that, unlike in the setting of homology with integer coefficients where every knot~$K$ admits any integer~$r$ as a framing coefficient, when we work with~$\Z[t^{\pm 1}]$-homology,  a fixed knot~$\widetilde{K}$ will have many~$r(t) \in \Q(t)$ (in fact even in~$\Z[t^{\pm 1}]$) which it does not admit as a framing coefficient.
We will refer to~$\pi$ as a \emph{framing curve} of~$\widetilde{K}$ with framing~$r(t)$.
\item A framed $n$-component link~$\widetilde{L}$ which admits framing coefficients~$\mathbf{r}(t):=(r_i(t))_{i=1}^n$, together with a choice of parallels realising those framing coefficients, is called an  $\mathbf{r}(t)$-framed link.
\item
The \emph{equivariant linking matrix} of an~$\mathbf{r}(t)$-framed link~$\widetilde{L}$ is the matrix~$A_{\widetilde{L}}$ with diagonal term~$(A_{\widetilde{L}})_{ii}=r_i(t)$ and off-diagonal terms~$(A_{\widetilde{L}})_{ij}=\ell k_{\Q(t)}(\widetilde{K}_i,\widetilde{K}_j)$ for~$i \neq j$.
\item For a link~$\widetilde{L}$ in~$Y^\infty$,  we define~$L^\infty$ to be the set of all the translates of~$\widetilde{L}$.
We also set
$$L:=p(\widetilde{L}).$$
We say that $\wt{L}$ is in \emph{covering general  position} if the map $p \colon L^{\infty} \to L$ is a trivial $\Z$-covering isomorphic to the pullback cover
\[\xymatrix @R0.5cm @C0.5cm{L^{\infty} \ar[r] \ar[d] & \R \ar[d] \\ L \ar[r]^{c} & S^1}\]
where $c$ is a constant map. In particular each component of $L^{\infty}$ is mapped by $p$, via a homeomorphism, to some component of $L$. From now on we will always assume that our links $\wt{L}$ are in covering general position. This assumption is to avoid pathologies, and  holds generically.
\item For an $n$-component link~$\widetilde{L}$ which admits framing coefficients~$\mathbf{r}(t):=(r_i(t))_{i=1}^n$, the \emph{$\mathbf{r}(t)$-surgery} along~$\widetilde{L}$ is the covering space~${Y}^\infty_{\mathbf{r}(t)}(\widetilde{L}) \to Y_{\mathbf{r}}(L)$ defined by Dehn filling~$Y^\infty\setminus \nu(L^\infty)$ along all the translates of all the parallels~$\pi_1^\infty,\ldots,\pi_n^\infty$  as follows:
$$ {Y}^\infty_{\mathbf{r}(t)}(\widetilde{L})=Y^\infty \setminus \Big( \bigcup_{k \in \Z} \bigcup_{i=1}^n \left(  t^k\overline{\nu}_{\pi_i}(\widetilde{K}_i\right) \Big) \cup \Big( \bigcup_{k \in \Z} \bigcup_{i=1}^n \left( D^2 \times S^1 \right)\Big).$$
%AC: This means that the framing curve \pi_1 now bounds the D^2 factor and that the meridian pt x \partial D^2 is identified with the glued on pt x S^1.
\noindent
%Some covering space theory together with the fact that
%~$Y^\infty$ is a~$\Z$-cover
Since $\widetilde{L}$ is in covering general position,
%implies that
for all~$\widetilde{K}_i$ the covering map~$p|_{\widetilde{K}_i} \colon \widetilde{K}_i\to K_i$ is a homeomorphism,  so $p|_{\overline{\nu}(\widetilde{K}_i)} \colon \overline{\nu}(\widetilde{K}_i)\to \nu(K_i)$ is a homeomorphism.
%AC: The restriction of the cover to \nu(\widetilde{K}_i) is a covering of the torus and there are as many of these as there are of coverings of the circle. Since we already know that it s degree one on the circle, it must also be the trivial cover here.
Thus any parallel~$\pi_i$ of~$\widetilde{K}_i$ projects to a parallel of~$K$, so we may also define~$\mathbf{r}$-surgery along~$L$ downstairs:

$${Y}_{\mathbf{r}}(L)=Y \setminus \Big( \bigcup_{i=1}^n \overline{\nu}_{p(\pi_i)}(p(\widetilde{K}_i))\Big)  \cup  \Big( \bigcup_{i=1}^n (D^2 \times S^1) \Big).$$

\noindent Observe that there is a naturally induced cover ~${Y}^\infty_{\mathbf{r}(t)}(\widetilde{L}) \to {Y}_{\mathbf{r}}(L)$ obtained by restricting~$p \colon Y^\infty \to Y$ to the link exterior and then extending it to the trivial disconnected $\Z$-cover over each of the surgery solid tori.
\item
The \emph{dual framed link}~$\widetilde{L'}\subset {Y}^\infty_{\mathbf{r}(t)}(\widetilde{L})$ associated to a framed link~$\widetilde{L} \subset Y^\infty$ is defined as follows:
\begin{itemize}
\item the~$i$-th component~$\widetilde{K}_i'$ of the underlying link~$\widetilde{L}' \subset {Y}^\infty_{\mathbf{r}(t)}(\widetilde{L})$ is obtained by considering the core of the~$i$-th surgery solid torus~$D^2 \times S^1$.
%=t^0(D^2 \times S^1)$;
\item The framing of~$\widetilde{K}_i'$ is given by the~$S^1$-factor~$S^1 \times \lbrace \operatorname{pt} \rbrace$ of the parametrised solid torus used to define~$\widetilde{K}_i'$.
\end{itemize}
\item We also define analogues of these notions (except $(6)$ and $(7)$) for a link $L$ in the $3$-manifold~$Y$, without reference to the cover.
\end{enumerate}
\end{definition}

The next lemma provides a sort of analogue for the Seifert longitude of a knot in~$S^3$; it is inspired by~\cite[Lemma 1.2]{BoyerLines}.
The key difference with the Seifert longitude is that in our setting this class, which we denote by~$\lambda_{\widetilde{K}}$, is just a homology class in $H_1(\partial \overline{\nu}(\widetilde{K});\Q(t))$; it will frequently not be represented by a simple closed curve.

\begin{lemma}
\label{lem:SimpleClosedCurve}
For every knot~$\widetilde{K} \subset Y^\infty$, there is a unique homology class~$\lambda_{\widetilde{K}} \in H_1(\partial \overline{\nu}( \widetilde{K});\Q(t))$ called the \emph{longitude} of~$\widetilde{K}$ such that the following two conditions hold.
\begin{enumerate}
\item The algebraic equivariant intersection number of~$[\mu_{\widetilde{K}}]$ and~$\lambda_{\widetilde{K}}$ is one:
$$\lambda_{\partial \overline{\nu}(K),\Q(t)}([\mu_{\widetilde{K}}],\lambda_{\widetilde{K}})=1.$$
%AC:\lambda_{\partial \overline{\nu}(K) intentionally has no \tilde on the K.
\item The class~$\lambda_{\widetilde{K}}$ maps to zero in~$H_1(Y_K;\Q(t))$.
\end{enumerate}
For any parallel~$\pi$ of~$\widetilde{K}$, this class satisfies
$$ \lambda_{\widetilde{K}} =[\pi]-\ell k_{\Q(t)}(\widetilde{K},\pi)[\mu_{\widetilde{K}}].$$
%Additionally,~$\lambda_{\widetilde{K}}$ only depends on the isotopy class of~$\widetilde{K}$: if~$\widetilde{K}'$ is isotopic to~$\widetilde{K}$, then there is a canonical isomorphism~$H_1(\partial \overline{\nu}(\widetilde{K});\Q(t)) \cong H_1(\partial \overline{\nu} \widetilde{K'};\Q(t))~$ taking~$\lambda_{\widetilde{K}}$ to~$\lambda_{\widetilde{K'}}$.
\end{lemma}
\begin{proof}
We first prove existence and then uniqueness.
For existence, pick any parallel~$\pi$ to~$\widetilde{K}$, i.e.\ any curve in~$\partial \overline{\nu}(\widetilde{K})$ that is isotopic to~$\widetilde{K}$ in~$ {\overline{\nu}(\widetilde{K})}$ and define
$$ \lambda_{\widetilde{K}}:=[\pi]-\ell k_{\Q(t)}(\widetilde{K},\pi)[\mu_{\widetilde{K}}].$$
Here recall that the equivariant linking~$r:=\ell k_{\Q(t)}(\widetilde{K},\pi)$ is the unique element of~$\Q(t)$ such that~$[\pi]=r[\mu_{\widetilde{K}}]$ in~$H_1(Y_K;\Q(t))$.
The two axioms now follow readily.

For uniqueness, we suppose that~$\lambda_{\widetilde{K}}$ and~$\lambda_{\widetilde{K}}'$ are two homology classes as in the statement of the lemma.
Choose a parallel~$\pi$ of~$\widetilde{K}$ and base~$H_1(\partial \overline{\nu}(K);\Q(t))$ by the pair~$(\mu_{\widetilde{K}},\pi)$.
This way, we can write~$\lambda_{\widetilde{K}}=r_1[\mu_{\widetilde{K}}]+r_2[\pi]$ and~$\lambda_{\widetilde{K}}'=r_1'[\mu_{\widetilde{K}}]+r_2'[\pi]$.
%{If we had chosen a different parallel~$\widehat{\pi}$, we would have~$[\widehat{\pi}]=m[\mu_{\widetilde{K}}]+[\pi]$ for~$m \in \Z$, and we can continue in the same manner by simply replacing~$\widehat{r_1}$ by~$r_1+m$ (and similarly for~$r_1'$), thus our choice of~$\pi$ is inconsequential.}
The first condition on~$\lambda_{\widetilde{K}}$ now promptly implies that~$r_2=r_2'=1$; formally
$$1=\lambda_{\partial \overline{\nu}(K),\Q(t)}([\mu_{\widetilde{K}}],\lambda_{\widetilde{K}})=r_2\lambda_{\partial \overline{\nu}(K),\Q(t)}([\mu_{\widetilde{K}}],[\pi])=r_2$$
and similarly for~$r_2'$.
To see that~$r_1=r_1'$, observe that since~$r_2=r_2'$, we have that~$\lambda_{\widetilde{K}}=\lambda_{\widetilde{K}}'+(r_1'-r_1)[\mu_{\widetilde{K}}]$. Recall that~$[\mu_{\widetilde{K}}]$ is a generator of the vector space~$H_1(Y_K;\mathbb{Q}(t))=\Q(t)$ and that~$\lambda_{\widetilde{K}}', \lambda_{\widetilde{K}}'$ are zero in~$H_1(Y_K;\mathbb{Q}(t))$.
We conclude that~$(r_1'-r_1)=0$, as required.
%Old longer proof.
%But now if we look at the Mayer-Vietoris sequence for~$\widetilde{Y}$, then we get the isomorphism
%$$ H_1(\partial Y_K;\Q(t)) \xrightarrow{\cong} H_1(Y_K;\Q(t)) \oplus H_1({\overline{\nu}(\widetilde{K})};\Q(t))=\Q(t)[\mu_{\widetilde{K}}] \oplus \Q(t)[\pi]$$
%and we know that~$\lambda_{\widetilde{K}},\lambda_{\widetilde{K}}'$ are mapped to~$(r_1,1),(r_1',1)$ on the right hand side.
%From the second condition on~$\lambda_{\widetilde{K}}$, we deduce that~$r_1=0$ and~$r_1'=0$.
%Since the map from the Mayer-Vietoris sequence is an isomorphism, we deduce that~$\lambda_{\widetilde{K}}=\lambda_{\widetilde{K}}'$.
%To see that~$\lambda_{\widetilde{K}}$ only depends on the isotopy class of~$\widetilde{K}$, note that if~$\widetilde{K}'$ is isotopic to~$\widetilde{K}$, then~$\mu_{K}$
\end{proof}

%AC: Usually we don't write \Z coefficients, but here we'll levae it to avoid confusion.
As motivation, observe that for a link $L=K_1 \cup \cdots \cup K_n \subset S^3$,  the group $H_1(E_L;\Z)$ is freely generated by the meridians $\mu_{K_i}$ and, if $L$ is framed with integral linking matrix $A$, then the framing curves $\pi_i$ can be written in this basis as $[\pi_i]=\sum_{j=1}^n A_{ij}[\mu_{K_j}] \in H_1(E_L;\Z)$.
The situation is similar in our setting.

%\begin{remark}
%\label{rem:FramingCurvesInExterior}
\begin{proposition}\label{prop:relating-pi_i_and_meridians}
Let $\widetilde{L} \subset Y^\infty$ be an $n$-component framed link in covering general position whose components have framing curves~$\pi_1,\ldots,\pi_n$.
Recall that~$H_1(Y_L;\Q(t))=\Q(t)^n$ is generated by the homology classes of the meridians~$\mu_{\widetilde{K}_1},\ldots,\mu_{\widetilde{K}_n}$.
The homology classes of the~$\pi_i$ in~$H_1(Y_L;\Q(t)) \cong \Q(t)^n$ are related to the meridians by the formula
\[[\pi_i]=\sum_{j=1}^n (A_{\widetilde{L}})_{ij} [\mu_{\widetilde{K}_j}] \in H_1(Y_L;\Q(t)).\]
\end{proposition}

\begin{proof}
By definition of the equivariant linking matrix $A_{\wt{L}}$, we must prove that
\begin{equation}
\label{eq:ForClaim}
[\pi_i]
%=\sum_{j=1}^n (A_{\widetilde{L}})_{ij}[\widetilde{\mu}_{K_j}]
=\ell k_{\Q(t)}(\widetilde{K}_i,\pi_i)[\mu_{\widetilde{K}_i}]+\sum_{j \neq i} \ell k_{\Q(t)}(\widetilde{K}_i,\widetilde{K}_j)[\mu_{\widetilde{K}_j}] \in H_1(Y_L;\Q(t))
\end{equation}
for each $i$.
Since the sum of the inclusion induced maps give rise to an isomorphism \[H_1(Y_L;\Q(t)) \cong \bigoplus_{j=1}^n H_1(Y_{K_j};\Q(t))\] it suffices to prove the equality after applying the inclusion map $H_1(Y_L;\Q(t)) \to H_1(Y_{K_j};\Q(t))$, for each~$j$.
%AC: used to say i.
Since~$\pi_i$ is a parallel of~$\widetilde{K}_i$,  applying Lemma~\ref{lem:SimpleClosedCurve}, we have
$$ [\pi_i]=\ell k_{\Q(t)}(\widetilde{K}_i,\pi_i)[\mu_{\widetilde{K}_i}]+\lambda_{\widetilde{K}_i} \in H_1(\partial Y_{K_i};\Q(t)).$$
We consider the image of this homology class in~$H_1(Y_{K_j};\Q(t))$ for~$j=1,\dots,n$.
In the vector space~$H_1(Y_{K_i};\Q(t))=\Q(t)[\mu_{\widetilde{K}_i}]$, the longitude class~$\lambda_{\widetilde{K}_i}$ vanishes (again by Lemma~\ref{lem:SimpleClosedCurve}).
For~$j \neq i$,  the class~$[\mu_{\widetilde{K}_i}]$ vanishes in~$H_1(Y_{K_j};\Q(t))$; thus the image of~$[\pi_i]$ in~$H_1(Y_{K_j};\Q(t))$ is~$\ell k_{\Q(t)} (\pi_i,\widetilde{K}_j)[\mu_{\widetilde{K}_j}] =\ell k_{\Q(t)} (\widetilde{K}_i,\widetilde{K}_j)[\mu_{\widetilde{K}_j}]$.
This concludes the proof of~\eqref{eq:ForClaim}.
\end{proof}

% and therefore concludes the proof of the claim.
%\end{remark}
%\color{black}

From now on, we will be working with $\Z[t^{\pm 1}]$-coefficient homology both for $Y$ and for the result $Y':=Y_{\mathbf{r}(t)}(L)$ of surgery on a framed link $L \subset Y$. Let $W$ denote the trace of the surgery from $Y$ to $Y'$.
We therefore record a fact about the underlying coefficient systems for later reference.

%\begin{proposition}
%\label{prop:CoefficientSystemY'}
%The epimorphism $\varphi \colon \pi_1(Y) \twoheadrightarrow \Z$ extends to an epimorphism $\pi_1(W) \twoheadrightarrow \Z $ where~$W$ denotes the trace of the surgery,  and so $\varphi$ induces an epimorphism $\varphi' \colon \pi_1(Y') \twoheadrightarrow \Z$ that vanishes on the knots $K_i' \subset Y$ dual to the original $K_i \subset Y$.
%\end{proposition}
%\begin{proof}
%First, note that~$\pi_1(W)$ is obtained from~$\pi_1(Y)$ by adding additional relators that kill each of the~$[K_i] \in \pi_1(Y)$ (indeed $W$ is obtained by adding~$2$-handles to $Y \times [0,1]$ along the~$K_i$).
%Since~$\varphi$ is trivial on the~$K_i \subset Y$ (because they lift to $Y^\infty$),  we deduce that $\varphi$ descends to an epimorphism on~$\pi_1(W)$.
%The composition~$\pi_1(Y') \to \pi_1(W) \twoheadrightarrow \Z$ is also surjective because~$\pi_1(W)$ is obtained from~$\pi_1(Y')$  by adding relators that kill each of the~$[K_i'] \in \pi_1(Y')$ (indeed~$W$ is obtained by adding~$2$-handles to $Y' \times [0,1]$ along the dual knots $K_i'$).
%In particular the homomorphism $\varphi' \colon \pi_1(Y') \twoheadrightarrow \Z$ vanishes on the knots $K_i' \subset Y$ dual to the original $K_i \subset Y$.
%\end{proof}

\begin{lemma}\label{lem:coeff-system}
The epimorphism $\varphi \colon \pi_1(Y) \twoheadrightarrow \Z$ extends to an epimorphism $\pi_1(W) \twoheadrightarrow \Z$, which by precomposition with the inclusion map induces an epimorphism $\varphi' \colon \pi_1(Y') \twoheadrightarrow \Z$.
\end{lemma}
%where $W$ denotes the trace of the surgery and,

\begin{proof}
Note that~$\pi_1(W)$ is obtained from~$\pi_1(Y)$ by adding relators that kill each of the~$[K_i] \in \pi_1(Y)$ (indeed $W$ is obtained by adding~$2$-handles to $Y \times [0,1]$ along the $K_i$).
Since~$\varphi$ is trivial on the~$K_i \subset Y$ (because they lift to $Y^\infty$),  we deduce that $\varphi$ descends to an epimorphism on $\pi_1(W)$.

The composition $\pi_1(Y') \to \pi_1(W) \twoheadrightarrow \Z$ is also surjective because $\pi_1(W)$ is obtained from $\pi_1(Y')$  by adding relators that kill each of the~$[K_i'] \in \pi_1(Y')$; indeed $W$ is obtained by adding~$2$-handles to $Y' \times [0,1]$ along the dual knots $K_i'$.
\end{proof}

\begin{remark}\label{rem:CoefficientSystemY'}
In particular note from the proof of Lemma~\ref{lem:coeff-system} that the homomorphism $\varphi' \colon \pi_1(Y') \twoheadrightarrow \Z$ vanishes on the knots $K_i' \subset Y$ dual to the original $K_i \subset Y$.
\end{remark}

%The next lemma proves an infinite cyclic cover analogue of the following familiar statement:
%given a rational homology sphere $Y$, performing surgery on a knot $K \subset Y$ that represents a non-zero class in $H_1(Y)$ results in another rational homology sphere.
%AC: It used to be written ``non-zero surgery on a knot in a rational homology sphere produces a rational homology sphere." That wasn't good because zero doesn't necessarily make sense
%%AC: We had initially agreed on this:
%non-zero surgery on a knot in an integer homology sphere produces a rational homology sphere.
%But I think it's really not the correct analogue because in that case the linking form is trivial...
%All in all, I do not think it's about the framing as much as it is about the homology class in the presence of a nondegenerate linking form.

The next lemma proves an infinite cyclic cover analogue of the following familiar statement:
performing surgery on a framed link $L \subset S^3$ whose linking matrix is invertible over $\Q$ results in a rational homology sphere.

\begin{lemma}
\label{lem:surgQsphere}
Let $Y$ be a 3-manifold and let $\varphi \colon \pi_1(Y) \twoheadrightarrow \Z$ be an epimorphism such that the  Alexander module $H_1(Y;\Z[t^{\pm 1}])$ is torsion.
If $\widetilde{L} \subset Y^\infty$ is an $n$-component framed link in covering general position, whose equivariant linking matrix $A_{\widetilde{L}}$ is invertible over $\Q(t)$, then the result $Y'$ of surgery on $L$  satisfies $H_1(Y';\Q(t))=0$.
\end{lemma}

\begin{proof}
The result will follow by studying the portion
$$ \cdots \to H_2(Y,Y_L;\Q(t)) \xrightarrow{\partial} H_1(Y_L;\Q(t)) \to H_1(Y';\Q(t)) \to H_1(Y',Y_L;\Q(t))$$
of the long exact sequence sequence of the pair $(Y,Y_L)$ with $\Q(t)$-coefficients, and arguing that $H_1(Y',Y_L;\Q(t))=0$ and that  $\partial$ is an isomorphism.

The fact that $H_1(Y',Y_L;\Q(t)) =0$ can be deduced from excision, replacing $(Y',Y_L)$ with the pair $(\sqcup^n S^1 \times D^2, \sqcup^n S^1 \times S^1)$. For the same reason, the vector space $H_2(Y,Y_L;\Q(t))=\Q(t)^n$ is based by the classes of the discs $(D^2 \times \lbrace \operatorname{pt} \rbrace)_i \subset (D^2 \times S^1)_i$ whose boundaries are the framing curves~$\pi_i$.
To conclude that $\partial$ is indeed an isomorphism,  note that $H_1(Y_L;\Q(t))=\Q(t)^n$ is generated by the $[\mu_{\widetilde{K}_i}]$ (because the Alexander module of~$Y$ is torsion) and use Proposition~\ref{prop:relating-pi_i_and_meridians} to deduce that with respect to these bases,  $\partial$ is represented by the equivariant linking matrix~$A_{\widetilde{L}}$.
Since this matrix is by assumption invertible over $\Q(t)$, we deduce that  $\partial$ is an isomorphism.  It follows that~$H_1(Y';\Q(t))=0$, as desired.
\end{proof}

The next lemma describes the framing on the dual of a framed link.
The statement ressembles~\cite[Lemma 1.5]{BoyerLines} and~\cite[Theorem 1.1]{PrzytyckiYasuhara}.

\begin{lemma}
\label{lem:InverseMatrix}

Let $Y$ be a 3-manifold and let $\varphi \colon \pi_1(Y) \twoheadrightarrow \Z$ be an epimorphism such that the  Alexander module $H_1(Y;\Z[t^{\pm 1}])$ is torsion.
If $\widetilde{L} \subset Y^\infty$ is a framed link in covering general position whose equivariant linking matrix $A_{\widetilde{L}}$ is invertible over $\Q(t)$, then the equivariant linking matrix of the dual framed link~$\widetilde{L}'$ is
$$A_{\widetilde{L}'}=-A_{\widetilde{L}}^{-1}.$$
%Let~$A_{\widetilde{L}}$ be the equivariant linking matrix for~$\widetilde{L} \subset Y^\infty$ and let~$A_{\widetilde{L}'}$ be the equivariant linking matrix for the dual link~$\widetilde{L}' \subset {Y'}^\infty$.
%Then both~$A_{\widetilde{L}}$ and~$A_{\widetilde{L}'}$ are invertible over~$\Q(t)$ and
%$$A_{\widetilde{L}'}=-A_{\widetilde{L}}^{-1}.$$
%In particular, if~$\widetilde{K}_i$ is~$(A_{\widetilde{L}})_{ii}$-framed, then~$\widetilde{K}_i'$ is~$-(A_{\widetilde{L}}^{-1})_{ii}$-framed.
\end{lemma}
%\color{black}
\begin{proof}
Consider the exterior~$Y_L=Y'_{L'}$ and recall that~$H_1(Y_L;\Q(t))=\Q(t)^n$ is generated by the meridians~$\mu_{\widetilde{K}_1},\ldots,\mu_{\widetilde{K}_n}$ of the link~$\widetilde{L}$ because we assumed that~$H_1(Y;\Q(t))=0$.
Since we assumed that $H_1(Y;\Q(t))=0$ and $\det(A_{\widetilde{L}})\neq 0$ we can apply Lemma~\ref{lem:surgQsphere} to deduce that~$H_1(Y';\Q(t))=0$ and
%By Lemma~\ref{lem:surgQsphere}, it is also the case that
hence~$H_1(Y_L;\Q(t))=H_1(Y'_{L'};\Q(t))$ is also generated by the meridians~$\mu_{\widetilde{K}_1'},\ldots,\mu_{\widetilde{K}_n'}$ of the link~$\widetilde{L}'$.

Thus the vector space~$H_1(Y_L;\Q(t))=\Q(t)^n$ has bases both~$\boldsymbol{\mu}=([\mu_{\widetilde{K}_1}],\ldots,[\mu_{\widetilde{K}_n}])$ and~$\boldsymbol{\mu}'=([\mu_{\widetilde{K}'_1}],\ldots,[\mu_{\widetilde{K}'_n}])$, and we let~$B$ be the change of basis matrix between these two bases so that~$B\boldsymbol{\mu}=\boldsymbol{\mu}'$.
Here and in the remainder of this proof,  we adopt the following convention: if~$C$ is a matrix over~$\Q(t)^n$ and if~$\boldsymbol{x}=(x_1,\ldots,x_n)$ is a collection of~$n$ vectors in~$\Q(t)^n$,  then we write~$C\boldsymbol{x}$ for the collection of~$n$ vectors~$Cx_1,\ldots,Cx_n$.

Recall that for~$i=1,\ldots,n$,  the framing curves of the~$\widetilde{K}_i$ and~$\widetilde{K}_i'$ are respectively denoted by~$\pi_i \subset Y^\infty$ and~$\pi_i' \subset {Y'}^\infty$.
Slightly abusing notation, we also write~$[\pi_i]$ for the class of~$\pi_i$ in~$H_1(Y_{K_i};\Q(t))$.
We set~$\boldsymbol{\pi}=([\pi_1],\ldots,[\pi_n])$ and~$\boldsymbol{\pi}'=([\pi_1'],\ldots,[\pi_n'])$ and
% claim that~$\boldsymbol{\pi}=A_{\widetilde{L}}\boldsymbol{\mu}$
%and~$\boldsymbol{\pi'}=A_{\widetilde{L}'}\boldsymbol{\mu'}$.
%To establish this claim, we must prove that
%\begin{equation}
%\label{eq:ForClaim}
%[\pi_i]
%=\sum_{j=1}^n (A_{\widetilde{L}})_{ij}[\widetilde{\mu}_{K_j}]
%=\ell k_{\Q(t)}(\widetilde{K}_i,\pi_i)[\mu_{\widetilde{K}_i}]+\sum_{j \neq i} \ell k_{\Q(t)}(\widetilde{K}_i,\widetilde{K}_j)[\mu_{\widetilde{K}_j}] \in H_1(Y_L;\Q(t)).
%\end{equation}
%Since~$\pi_i$ is a parallel of~$\widetilde{K}_i$,  applying Lemma~\ref{lem:SimpleClosedCurve}, we have
%$$ [\pi_i]=\ell k_{\Q(t)}(\widetilde{K}_i,\pi_i)[\mu_{\widetilde{K}_i}]+\lambda_{\widetilde{K}_i} \in H_1(\partial Y_{K_i};\Q(t)).$$
%We consider the image of this homology class in~$H_1(Y_{K_j};\Q(t))$ for~$j=1,\dots,n$.
%In the vector space~$H_1(Y_{K_i};\Q(t))=\Q(t)[\mu_{\widetilde{K}_i}]$, the longitude class~$\lambda_{\widetilde{K}_i}$ vanishes (again by Lemma~\ref{lem:SimpleClosedCurve}).
%For~$j \neq i$,  the class~$[\mu_{\widetilde{K}_i}]$ vanishes in~$H_1(Y_{K_j};\Q(t))$; thus the image of~$[\pi_i]$ in~$H_1(Y_{K_j};\Q(t))$ is~$\ell k_{\Q(t)} (\pi_i,\widetilde{K}_j)[\mu_{\widetilde{K}_j}] =\ell k_{\Q(t)} (\widetilde{K}_i,\widetilde{K}_j)[\mu_{\widetilde{K}_j}]$.
%This concludes the proof of~\eqref{eq:ForClaim} and therefore concludes the proof of the claim.
%Summarising,  the claim implies that we have the relations
and use Proposition~\ref{prop:relating-pi_i_and_meridians} to deduce that
\begin{align*}
\boldsymbol{\pi}= A_{\widetilde{L}}\boldsymbol{\mu}, \ \ \ \ \ \ \ \
\boldsymbol{\pi'}=  A_{\widetilde{L}'}\boldsymbol{\mu'}
\end{align*}
%and,
Inspecting the surgery instructions, we also have the relations
\begin{align*}
\boldsymbol{\mu'}=-\boldsymbol{\pi} \ \ \ \ \ \ \ \
\boldsymbol{\mu}=\boldsymbol{\pi'}.
\end{align*}
We address the sign in Remark \ref{rem:sign} below.
Combining these equalities, we obtain
\begin{align*}
\boldsymbol{\mu}&=\boldsymbol{\pi'}= A_{\widetilde{L}'}\boldsymbol{\mu'}=A_{\widetilde{L}'}B\boldsymbol{\mu}, \\
\boldsymbol{\mu'}&=-\boldsymbol{\pi}= -A_{\widetilde{L}}\boldsymbol{\mu}=-A_{\widetilde{L}}B^{-1}\boldsymbol{\mu'}.
\end{align*}
Unpacking the equality~$A_{\widetilde{L}'}B\boldsymbol{\mu}=\boldsymbol{\mu}$, we deduce that~$A_{\widetilde{L}'}B[\mu_{\widetilde{K}_i}]=[\mu_{\widetilde{K}_i}]$ for~$i=1,\ldots,n$.
But since the~$[\mu_{\widetilde{K}_1}],\ldots,[\mu_{\widetilde{K}_n}]$ form a basis for~$\Q(t)^n$, this implies that~$A_{\widetilde{L}'}B=I_n$.
The same argument shows that~$-A_{\widetilde{L}}B^{-1}=I_n$ and therefore both matrices~$A_{\widetilde{L}}$ and~$A_{\widetilde{L}'}$ are invertible, with~$-A_{\widetilde{L}}=B=A_{\widetilde{L}'}^{-1}$.
%{LP: Now i am really convinced that there is noting more to say for the final line.~$\widetilde{K}_i$ is~$(A_{\widetilde{L}})_{ii}$ framed by an honest curve~$\gamma$ by setup, and  the dual comes with a framing curve~$\gamma_i$ by definition and the definition of~$(A_{\widetilde{L}}^{-1})_{ii}$ is~$lk_{\widetilde Y}(\widetilde{K}_i, \gamma_i)$ which is also the definition of the framing coefficient on~$\widetilde{K}_i$. }
%%%DONT DELETE
%%%AC: What I had originally written to justify the last sentence.
%Let~$\pi_i$ be the framing curves for the~$\widetilde{K}_i$ with framing coefficient~$(A_{\widetilde{L}})_{ii}=\ell k_{Y,\Q(t)}(\pi_i,\widetilde{K}_i)$.
%By definition, the dual knot~$\widetilde{K}_i' \subset {Y'}^\infty$ is framed by the parallel~$\pi_i'$ that arises by thinking of the meridian~$\mu_{\widetilde{K}_i} \subset Y^\infty$ of~$\widetilde{K}_i \subset Y^\infty$ as a curve in~${Y'}^\infty$.
%Therefore, again by definition, its framing coefficient is~$(A_{\widetilde{L}'})_{ii}=\ell k_{Y',\Q(t)}(\pi_i',\widetilde{K}_i')$.
%This shows that the second assertion of the lemma follows from the first.
\end{proof}

\begin{remark}\label{rem:sign}
In the above proposition, we were concerned with the relationship between the curves~$(\boldsymbol{\mu},\boldsymbol{\pi})$ and $(\boldsymbol{\mu'},\boldsymbol{\pi'})$, all of which represent classes in $H_1(\partial Y_L,\Q(t))$.
We know from the surgery instructions that~$g(\boldsymbol{\mu})=\boldsymbol{\pi'}$.
We are free to choose the collection of curves $g(\boldsymbol{\pi})$ so long as we choose each~$g(\pi_i)$ to intersect $\pi_i'$ geometrically once (as unoriented curves).
%As unoriented curves,  $\boldsymbol{\pm \mu'}$ satisfy this condition.
We choose the unoriented curves $\boldsymbol{\pm \mu'}$.
Since we know that the surgery was done to produce an oriented manifold, it must be the case that the gluing transformation $g\colon \partial Y_L\to \partial Y_L$ is orientation-preserving.
The fact that $g$ is orientation-preserving implies that it preserves intersections numbers,
we deduce that
$ \delta_{ij}=\mu_i \cdot \pi_j=g(\mu_i) \cdot g(\pi_j)=\pi_j' \cdot (\pm \mu_i').$
This forces $g(\boldsymbol{\pi})=-\boldsymbol{\mu'}$.
\end{remark}

\section{Reidemeister torsion}
\label{sec:reidemeister-torsion}

We recall the definition of the Reidemeister torsion of a based chain complex as well as the corresponding definition for CW complexes.
This will be primarly used in Subsection~\ref{sub:Step2}.
References on Reidemeister torsion include~\cite{TuraevIntroductionTo, TuraevReidemeisterTorsionInKnotTheory, ChaFriedl}.
\medbreak
Let~$\mathbb{F}$ be a field.
Given two bases~$u,v$ of a~$r$-dimensional~$\F$-vector space, we write~$\det(u/v)$ for the determinant of the matrix taking~$v$ to~$u$, i.e.\ the determinant of the matrix~$A=(A_{ij})$ that satisfies~$v^i=\sum_{j=1}^r A_{ij}u^j$.
A \emph{based chain complex} is a finite chain complex
$$C=\left( 0 \to C_m \xrightarrow{\partial_{m-1}} C_{m-1} \xrightarrow{\partial_{m-2}} \cdots \xrightarrow{\partial_2} C_1 \xrightarrow{\partial_0} C_0  \to 0\right)$$
of~$\F$-vector spaces together with a basis~$c_i$ for each~$C_{i+1}$.
Given a based chain complex, fix a basis~$b_i$ for~$B_i=\im(\partial_{i+1})$ and pick a lift~$\widetilde{b}_i$ of~$b_i$ to~$C_i$.
Additionally, fix a basis~$h_i$ for each homology group~$H_i(C)$ and let~$\widetilde{h}_i$ be a lift of~$h_i$ to~$C_i$.
One checks that that~$(b_i,\widetilde{h}_i,\widetilde{b}_{i-1})$ forms a basis of~$C_i$.

\begin{definition}
\label{def:ReidemeisterTorsion}
Let~$C$ be a based chain complex over~$\F$ and let~$\mathcal{B}=\lbrace h_i \rbrace$ be a basis for~$H_*(C)$.
The \emph{Reidemeister torsion} of~$(C,\mathcal{B})$ is defined as
$$ \tau(C,\mathcal{B})=\frac{\prod_i  \det((b_{2i+1},\widetilde{h}_{2i+1},\widetilde{b}_{2i})|c_{2i+1})}{\prod_i \det((b_{2i},\widetilde{h}_{2i},\widetilde{b}_{2i-1})|c_{2i})} \in \F\setminus \lbrace 0\rbrace.~$$
Implicit in this definition is the fact that~$\tau(C,\mathcal{B})$ depends neither on the choice of the basis~$b_i$, nor on the choice of the lifts~$\widetilde{b}_i$, nor on the choice of the lifts~$\widetilde{h}_i$ of the~$h_i$. It does depend on $\mathcal{B}= \{h_i\}$.

When~$C$ is acyclic, we drop~$\mathcal{B}$ from the notation and simply write~$\tau(C)$.
\end{definition}

Note that we are following Turaev's sign convention~\cite{TuraevIntroductionTo,TuraevReidemeisterTorsionInKnotTheory}; Milnor's convention~\cite{MilnorDualityTheorem}
%(which is also used in the influential paper~\ref{KirkLivingston})
yields the multiplicative inverse of~$\tau(C,\mathcal{B})$~\cite[Remark 1.4 item 5]{TuraevIntroductionTo}.
The next result collects two properties of the torsion that will be used later on.

\begin{proposition}
\label{thm:ReidemeisterTorsion}
~
\begin{enumerate}
\item Suppose that~$0 \to C' \to C \to C'' \to 0$ is a short exact sequence of based chain complexes and that~$\mathcal{B}',\mathcal{B}$, and $\mathcal{B}''$ are bases for~$H_*(C'),H_*(C)$ and~$H_*(C'')$ respectively.
If we view the associated homology long exact sequence as an acyclic complex~$\mathcal{H}$, based by~$\mathcal{B},\mathcal{B}'$, and $\mathcal{B}''$ respectively, then
$$\tau(C,\mathcal{B})=\tau(C',\mathcal{B}')\tau(C'',\mathcal{B}'')\tau(\mathcal{H}).$$
\item If~$C=(0 \to C_1 \xrightarrow{\partial_{0}} C_{0} \to 0)$ is an isomorphism between~$n$-dimensional vector spaces, so that~$C$ is an acyclic based chain complex, then
$$\tau(C)=\det(A)^{-1}$$
where~$A$ denotes the~$n \times n$-matrix which represents~$\partial_0$ with respect to the given bases.
\end{enumerate}
\end{proposition}
\begin{proof}
The multiplicativity statement is proved in~\cite{MilnorDualityTheorem},
The second statement follows from Definition~\ref{def:ReidemeisterTorsion}; details are in~\cite[Remark 1.4, item 3]{TuraevIntroductionTo}.
\end{proof}

We now recall the definition of the torsion of a pair of CW complexes.
We focus on the case where the spaces come with a map of their fundamental group to~$\Z$. This is a special case of an analogous general theory for the case of an arbitrary group~\cite{TuraevIntroductionTo}, and for more general twisted coefficients~\cite{FriedlVidussiSurvey}.

Let~$(X,A)$ be a finite CW pair, let~$\varphi \colon \pi_1(X) \to \Z$ be a homomorphism, and let~$\mathcal{B}$ be a basis for the~$\Q(t)$-vector space~$H_*(X,A;\Q(t))$.
Write~$p \colon X^\infty \to X$ for the cover corresponding to~$\ker(\varphi)$ and set~$A^\infty:=p^{-1}(A)$.
The chain complex~$C_*(X^\infty,A^\infty)$ can be based over~$\Z[t^{\pm 1}]$ by choosing a lift of each cell of~$(X,A)$ and orienting it; this also gives a basis of~$C_*(X,A;\Q(t))= C_*(X^\infty,A^\infty)  \otimes_{\Z[t^{\pm 1}]} \Q(t)$.
Let $\mathcal{E}$ denote the resulting choice of basis for $C_*(X,A;\Q(t))$.
We then define the torsion of~$(X,A,\varphi)$ as
$$ \tau(X,A,\mathcal{B},\mathcal{E}):=\tau(C_*(X,A;\Q(t)),\mathcal{B},\mathcal{E})\in \Q(t)\setminus \lbrace 0\rbrace.$$
%, we factor out $\Q(t)\sm \{0\}$ by multiplication by $\pm t^k$, for $k \in \Z$.
Given~$p(t),q(t) \in \Q(t)$, we write~$p(t)\doteq q(t)$ to indicate that~$p(t)$ and~$q(t)$ agree up to multiplication by~$\pm t^k$, for some $k \in \Z$.
This will enable us to obtain an invariant that does not depend on the choice of $\mathcal{E}$.
We write
\[\tau(X,A,\mathcal{B}) := [\tau(X,A,\mathcal{B},\mathcal{E})] \in (\Q(t)\sm \{0\})/\doteq, \]
for some choice of $\mathcal{E}$.  It is known that~$\tau(X,A,\mathcal{B})$ is well defined
%up to multiplication by~$\pm t^k$ with $k \in \Z$
and is invariant under simple homotopy equivalence preserving~$\mathcal{B}$~\cite[Theorem 9.1]{TuraevIntroductionTo}.
We drop the~$\mathcal{B}$ from the notation if~$H_*(X,A;\Q(t))=0$.

Additionally, Chapman proved that~$\tau(X,A,\mathcal{B})$ only depends on the underlying homeomorphism type of~$(X,A)$~\cite{Chapman}, and not on the particular CW structure.
In particular, when~$(M,N)$ is a manifold pair, we can define~$\tau(M,N,\mathcal{B})$ for any finite CW-structure on~$(M,N)$,  We will only consider the Reidemeister torsion of 3-manifolds, and so every pair $(M,N)$ we consider will admit a CW structure.  It will not be relevant in this paper, but we note that it is possible to define Reidemeister torsion for topological $4$-manifolds not known to admit a CW structure; see~\cite[Section 14]{FriedlNagelOrsonPowell} for a discussion.

\begin{remark}
\label{rem:AlexPoly}
The reason we consider Reidemeister torsion is its relation with Alexander polynomials; see Subsection~\ref{sub:Step2} below.
To this effect, we recall some relevant algebra.
Let $P$ be a~$\Z[t^{\pm 1}]$-module with presentation
\[\Z[t^{\pm 1}]^m \xrightarrow{f} \Z[t^{\pm 1}]^n \to P \to 0.\]
Consider elements of the free modules $\Z[t^{\pm 1}]^m$ and $\Z[t^{\pm 1}]^n$ as row vectors and represent $f$ by an~$m \times n$ matrix $A$, acting on the right of the row vectors.  By adding rows of zeros, corresponding to trivial relations, we may assume that $m \geq n$.
The \emph{$0$-th elementary ideal}~$E_0(P)$ of a finitely presented~$\Z[t^{\pm 1}]$-module~$P$ is the ideal of~$\Z[t^{\pm 1}]$ generated by all~$n \times n$ minors of~$A$. This definition is independent of the choice of the presentation matrix~$A$.
The \emph{order} of~$P$, denoted~$\Delta_P$,  is then by definition a generator of the smallest principal ideal containing~$E_0(P)$, i.e.\ the greatest common divisor of the minors.
The order of~$P$ is well defined up to multiplication by units of~$\Z[t^{\pm 1}]$ and if~$P$ admits a square presentation matrix, then~$\Delta_P\doteq\det(A)$, where~$A$ is some square presentation matrix for~$P$.
It follows that for a~$\Z[t^{\pm 1}]$-module~$P$ which admits a square presentation matrix,  one has~$P=0$ if and only if~$\Delta_P \doteq 1$.
For more background on these topics,
%presentations matrices, elementary ideals and the order of a module,
we refer the reader to~\cite[Section~1.4]{TuraevIntroductionTo}.
\end{remark}

\section{Proof of Theorem~\ref{thm:MainTechnicalIntro}.}
\label{sec:ProofMainTechnical}

Now we prove Theorem~\ref{thm:MainTechnicalIntro} from the introduction.
For the reader's convenience, we recall the statement of this result.

\begin{theorem}
\label{thm:MainTechnical}
Let~$Y$ be a~$3$-manifold with an epimorphism~$\varphi \colon \pi_1(Y) \twoheadrightarrow \Z$ whose Alexander module is torsion, and let~$(H,\lambda)$ be a nondegenerate Hermitian form over $\Z[t^{\pm 1}]$ presenting~$Y$.
If~$b \in \Iso(\partial \lambda,\unaryminus\Bl_Y)/\Aut(\lambda)$ is an isometry, then there is a~$\Z$-manifold~$M$ with equivariant intersection form~$\lambda_M \cong \lambda$,  boundary~$Y$ and with~$b_M=b$.
If the form is odd, then~$M$ can be chosen to have either~$\ks(M)=0$ or~$\ks(M)=1$.
\end{theorem}

For the remainder of the section, we let~$Y$ be a~$3$-manifold, let~$\varphi \colon \pi_1(Y) \twoheadrightarrow \Z$ be an epimorphism, and let~$p \colon Y^\infty \to Y$ be the infinite cyclic cover associated to~$(Y,\varphi)$.
We assume that~$H_1(Y;\Z[t^{\pm 1}]):=H_1(Y^\infty)$ is~$\Z[t^{\pm 1}]$-torsion.
We first describe the strategy of the proof and then carry out each of the steps successively.

\subsection{Plan}
\label{sub:Plan}

Let~$b \colon (\coker(\widehat{\lambda}),\partial \lambda) \to (H_1(Y;\Z[t^{\pm 1}]),\unaryminus \Bl_Y)$ be an isometry.
Precompose $b$ with the projection~$H^* \twoheadrightarrow \coker(\widehat{\lambda})$ to get an epimorphism~$\pi \colon H^* \twoheadrightarrow H_1(Y;\Z[t^{\pm 1}])$.
In particular,~$0 \to H \xrightarrow{\widehat{\lambda}} H^* \xrightarrow{\varpi} H_1(Y;\Z[t^{\pm 1}]) \to 0$ is a presentation of~$Y$.
Pick generators~$x_1,\ldots,x_n$ for~$H$ and endow~$H^*$ with the dual basis~$x_1^*,\ldots,x_n^*$.
Write~$Q$ for the matrix of~$\lambda$ in this basis. Note that $Q = \ol{Q}^T$ since $\lambda$ is Hermitian.
The strategy to prove Theorem~\ref{thm:MainTechnical} is as follows.

\begin{itemize}
\item Step 1: Prove that one can represent the classes~$\pi(x_1^*),\cdots, \pi(x_n^*)$ by an~$n$-component framed link~$\widetilde{L} = \widetilde{K}_1 \cup \cdots \cup \widetilde{K}_n$  with equivariant linking matrix~$A_{\widetilde{L}}=-Q^{-T}$.
\item Step 2: Argue that the result~$Y'$ of surgery on~$L=p(\widetilde{L})$ satisfies~$H_1(Y';\Z[t^{\pm 1}]) = 0$.
\item Step 3: There is a topological~$4$-manifold~$B \simeq S^1$ with boundary~$Y'$ following~\cite[Section~11.6]{FreedmanQuinn}.
%{AC: say more; relation to~$\varphi$.} Meh not needed now.
%If it's not Z, just kill the kernel by surgeries.
\item Step 4: Argue that the equivariant intersection form of the~$4$-manifold~$M$ defined below with boundary~$Y$ is represented by~$Q$ and prove that~$b_M = b$.
Here, the~$4$-manifold~$M$ and its infinite cyclic cover~$M^\infty$ are defined via
\begin{align*}
-M^\infty&:=\Big( (Y^\infty \times [0,1]) \cup \bigcup_{i=1}^n \bigcup_{j_i \in\Z} t^{j_i} h_i^{(2)} \Big) \cup_{{Y'}^\infty} -B^\infty \\
-M&:=\Big( (Y \times [0,1]) \cup \bigcup_{i=1}^n h_i^{(2)} \Big) \cup_{Y'} -B,
\end{align*}
where upstairs the~$2$-handles~$h_i^{(2)}$ are attached along the link ~$L^\infty$; downstairs, one attaches the 2-handles along the projection~$L=p(L^\infty)$ of this link.
\item Step 5: If~$\lambda$ is odd,  then we use the star construction~\cite{FreedmanQuinn,StongUniqueness} to show that both values of the Kirby-Siebenmann invariant can occur.
\end{itemize}

\subsection{Step 1: constructing a link with the appropriate equivariant linking matrix}
\label{sub:Step1}

We continue with the notation from the previous section.
In particular,  we have a presentation $0 \to H \xrightarrow{\widehat{\lambda}} H^* \xrightarrow{\varpi} H_1(Y;\Z[t^{\pm 1}]) \to 0$ and a basis $x_1,\ldots,x_n$ for $H$ with dual basis $x_1^*,\ldots,x_n^*$ for $H^*$.
The aim of this section is to prove that it is possible to represent the generators~$\pi(x_1^*),\ldots,\pi(x_n^*)$ of~$H_1(Y;\Z[t^{\pm 1}])$ by a framed link~$\widetilde{L}=\widetilde{K}_1 \cup \cdots \cup \widetilde{K}_n \subset Y^\infty$ whose transposed equivariant linking matrix agrees with~$-Q^{-1}$; see Proposition~\ref{prop:Step1}.
In other words, we must have
$$ \ell k_{\Q(t)} (\widetilde{K}_j,\widetilde{K}_i)=-(Q^{-1})_{ij} \ \ \ \text{ and } \ \ \  \ell k_{\Q(t)} (\widetilde{K}_i,\pi_i)=-(Q^{-1})_{ii},$$
where~$\pi_i$ is the framing curve of~$\widetilde{K}_i$.
Since the Blanchfield form~$\Bl_Y$ is represented by the~$\Q(t)$-coefficient matrix~$-Q^{-1}$~\cite[Section 3]{ConwayPowell}, we know from Proposition~\ref{prop:Linkingprop} that any link representing the~$\pi(x_i^*)$ must satisfy these relations up to adding a polynomial in $\Z[t^{\pm 1}]$.
Most of this section therefore concentrates on showing that the equivariant linking (resp.\  framing) of an arbitrary framed link in~$Y^\infty$ can be changed by any polynomial (resp.\  symmetric polynomial) in~$\Z[t^{\pm 1}]$, without changing the homology classes defined by the components of this link.
%and framing of an arbitrary framed link in~$Y^\infty$ can be changed by any polynomial in~$\Z[t^{\pm 1}]$, without changing the homology classes defined by the components of this link.
\medbreak
We start by showing how to modify the equivariant linking between distinct components of a link, without changing the homology class of the link.

%note that a similar result is stated in Lemma 2.3 of Kim-Ruberman on spines.
\begin{lemma}
\label{lem:Step1}
Let~$\widetilde{L}=\widetilde{K}_1 \cup \cdots  \cup \widetilde{K}_n \subset Y^\infty$ be an~$n$-component framed link in covering general position, with parallels $\pi_1,\dots,\pi_n$.
For every distinct~$i,j$ and every polynomial~$p(t) \in \Z[t^{\pm 1}]$, there  is a framed link~$\widetilde{L}':=\widetilde{K}_1 \cup \cdots \cup  \widetilde{K}_{i-1} \cup \widetilde{K}_i' \cup \widetilde{K}_{i+1} \cup \cdots \cup \widetilde{K}_n$, also  in covering general position, such that:
\begin{enumerate}
\item the knot~$\widetilde{K}_i'$ is isotopic to~$\widetilde{K}_i$ in~$Y^\infty$. In particular,~$[\widetilde{K}_i']=[\widetilde{K}_i]$ in~$H_1(Y;\Z[t^{\pm 1}])$;
\item the equivariant linking between~$\widetilde{K}_i$ and~$\widetilde{K}_j$ is changed by~$p(t)$, i.e.\
$$\ell k_{\Q(t)}(\widetilde{K}_i',\widetilde{K}_j)=\ell k_{\Q(t)}(\widetilde{K}_i,\widetilde{K}_j)+p(t);$$
\item the equivariant linking between~$\widetilde{K}_i$ and~$\widetilde{K}_\ell$ is unchanged for~$\ell \neq i,j$;
\item the framing coefficients are unchanged; that is, there is a parallel $\gamma_i$ for $\wt{K}'_i$ such that
\[\ell k_{\Q(t)}(\widetilde{K}_i',\gamma_i)=\ell k_{\Q(t)}(\widetilde{K}_i,\pi_i).\]
\end{enumerate}
%In particular, there is a link~$\widetilde{L}'$ representing~$\pi(x_1^*),\ldots,\pi(x_n^*)$ so that~$\ell k_{\Q(t)} (\widetilde{K}_i,\widetilde{K}_j)=(Q^{-1})_{ij}$.
\end{lemma}

\begin{proof}
Without loss of generality we can assume that~$p(t)=mt^k$ for~$m,k \in \Z$.
The new knot~$\widetilde{K}_i'$ is then obtained by band summing~$\widetilde{K}_i$ with~$m$ meridians of~$t^{-k} \widetilde{K}_j$, framed using the bounding framing induced by meridional discs.
The first, third, and fourth properties of~$\widetilde{K}_i'$ are immediate: clearly the linking of~$\widetilde{K}_i$ with~$\widetilde{K}_\ell$ is unchanged for~$\ell \neq i,j$ and since the aforementioned meridians bound discs in~$Y^\infty$ over which the framing extends, we see that~$\widetilde{K}_i'$ is framed isotopic (and in particular homologous) to~$\widetilde{K}_i$ in~$Y^\infty$. It follows that the framing coefficient is unchanged.

The second property is obtained from a direct calculation using the sesquilinearity of equivariant linking numbers:
\[ \ell k_{\Q(t)}(\widetilde{K}_i',\widetilde{K}_j)=\ell k_{\Q(t)}(\widetilde{K}_i,\widetilde{K}_j)+m \ \ell k_{\Q(t)}(t^{-k}\mu_{\widetilde{K}_j},\widetilde{K}_j)=\ell k_{\Q(t)}(\widetilde{K}_i,\widetilde{K}_j)+ mt^k.  \qedhere \]
%To see how to deduce the last statement of the lemma, note that~$\ell k_{\Q(t)} (\widetilde{K}_i,\widetilde{K}_j)$ and~$(Q^{-1})_{ij}$ both reduce mod~$\Z[t^{\pm 1}]$ to~$\Bl_Y(\pi(x_i^*),\pi(x_j^*))$, thus~$\ell k_{\Q(t)} (\widetilde{K}_i,\widetilde{K}_j)$ agrees with~$(Q^{-1})_{ij}$ up by adding of a Laurent polynomial.
\end{proof}

Next, we show how to modify the framing of a framed link component by a symmetric polynomial $p=\ol{p}$, without changing the homology class of the link.

\begin{lemma}
\label{lem:ModifyFraming}
Let~$\widetilde{L}=\widetilde{K}_1 \cup \cdots \cup \widetilde{K}_n \subset Y^\infty$ be an~$n$-component framed link in covering general position.
Fix a parallel~$\pi_i$ for~$\widetilde{K}_i$.
For each~$i=1,\ldots,n$ and every symmetric polynomial~$p(t) = p(t^{-1})$,  there exists a knot~$\widetilde{K}_i' \subset Y^\infty$ and a parallel~$\gamma_i$ of~$\widetilde{K}_i'$ such that
%there  is a link~$\widetilde{L}':=\widetilde{K}_1 \cup \cdots \cup  K_{i-1} \cup \widetilde{K}_i' \cup K_{i+1} \cup \cdots \cup \widetilde{K}_n$ such that
\begin{enumerate}
\item the knot~$\widetilde{K}_i'$ is isotopic to~$\widetilde{K}_i$ in~$Y^\infty \sm \cup_{j \neq i} \wt{K}_j$, and in particular,~$[\widetilde{K}_i']=[\widetilde{K}_i]$ in~$H_1(Y;\Z[t^{\pm 1}])$;
\item the framing coefficient of~$\widetilde{K}_i$ is changed by~$p(t)$, i.e.\
$$\ell k_{\Q(t)}(\widetilde{K}_i' ,\gamma_i)=\ell k_{\Q(t)}(\widetilde{K}_i ,\pi_i)+p(t);$$
\item the other linking numbers are unchanged:~$\ell k_{\Q(t)}(\widetilde{K}_i',\widetilde{K}_j)=\ell k_{\Q(t)}(\widetilde{K}_i,\widetilde{K}_j)$ for all~$j\neq i$.
\end{enumerate}
\end{lemma}
\begin{proof}
We first prove the lemma when~$p(t)$ has no constant term.
In this case, it suffices to show how to change the self-linking number by~$m(t^k+t^{-k})$ for~$k \neq 0$.
To achieve this, band sum~$\widetilde{K}_i$ with~$m$ meridians of~$t^k\widetilde{K}_i$.
As in the proof of Lemma~\ref{lem:Step1}, the first and third properties of~$\widetilde{K}_i$ are clear.
To define~$\gamma_i$ and prove the second property, define ~$\mu_{\widetilde{K}_i}'$ to be a parallel of~$\mu_{\widetilde{K}_i}$ with~$\ell k_{\Q(t)}(\mu_{\widetilde{K}_i},\mu_{\widetilde{K}_i}')=0$ in~$Y^\infty$. Define~$\gamma_i$ to be the parallel of~$\widetilde{K}_i'$ obtained by banding~$\pi_i$ to~$m$ copies of~$t^k\mu_{\widetilde{K}_i}'$, using bands which are push-offs of the bands used to define~$\widetilde{K}_i'$, and parallel copies of the meridian chosen with the zero-framing with respect to the framing induced by the associated meridional disc.
Using the sesquilinearity of equivariant linking numbers, we obtain
\begin{align*}
\ell k_{\Q(t)}(\widetilde{K}_i',\gamma_i)
&=\ell k_{\Q(t)}(\widetilde{K}_i,\pi_i)+m \ \ell k_{\Q(t)}(t^k\mu_{\widetilde{K}_i},\pi_i)+m\ \ell k_{\Q(t)}(\widetilde{K}_i,t^k\mu_{\widetilde{K}_i}')+\ell k_{\Q(t)}(\mu_{\widetilde{K}_i},\mu_{\widetilde{K}_i}') \\
&=\ell k_{\Q(t)}(\widetilde{K}_i,\pi)+m(t^k+t^{-k}).
\end{align*}
We have therefore shown how to modify the self-linking within a fixed homology class by a symmetric polynomial with no constant term.

The general case follows: thanks to the previous paragraph, it suffices to describe how to change the self-linking by a constant, and this can be arranged by varying the choice of the parallel~$\gamma_i$ i.e.\ by additionally winding an initial choice of~$\gamma_i$ around the appropriate number of meridians of~$\widetilde{K}_i'$.
\end{proof}

By combining the previous two lemmas, we can now prove the main result of this section.

\begin{proposition}
\label{prop:Step1}
Let~$0 \to H \xrightarrow{\widehat{\lambda}} H^* \xrightarrow{\varpi} H_1(Y;\Z[t^{\pm 1}]) \to 0$ be a presentation of~$Y$.
Pick generators~$x_1,\ldots,x_n$ for~$H$ and endow~$H^*$ with the dual basis~$x_1^*,\ldots,x_n^*$.
Let~$Q$ be the matrix of~$\lambda$ with respect to these bases.
The classes~$\pi(x_1^*),\ldots,\pi(x_n^*)$ can be represented by simple closed curves~$\widetilde{K}_1,\ldots,\widetilde{K}_n \subset Y^\infty$ such that $\widetilde{L}=\widetilde{K}_1 \cup \cdots \cup\widetilde{K}_n$ is in covering general position and satisfies the following properties:
\begin{enumerate}
\item the equivariant linking of the~$\widetilde{K}_i$ satisfy
$\ell k_{\Q(t)}(\widetilde{K}_j,\widetilde{K}_i)=\unaryminus(Q^{-1})_{ij}$ for~$i \neq j$;
\item there exist parallels~$\gamma_1,\ldots,\gamma_n$ of~$\widetilde{K}_1,\ldots,\widetilde{K}_n$ such that~$\ell k_{\Q(t)}(\widetilde{K}_i,\gamma_i)=\unaryminus(Q^{-1})_{ii}$.
%\item the parallel~$\gamma_i$ represents the homology class~$(Q^{-1})_{ii}[\mu_{\widetilde{K}_i}]+\lambda_{\widetilde{K}_i} \in H_1(\partial \overline{\nu} (\widetilde{K}_i);\Q(t))$.
\end{enumerate}
In particular the parallel~$\gamma_i$ represents the homology class~$\unaryminus(Q^{-1})_{ii}[\mu_{\widetilde{K}_i}]+\lambda_{\widetilde{K}_i} \in H_1(\partial \overline{\nu} (K_i);\Q(t))$ and the transpose of the equivariant linking matrix of~$\widetilde{L}$ equals~$-Q^{-1}$.
\end{proposition}
\begin{proof}
Represent the classes~$\pi(x_1^*),\ldots,\pi(x_n^*)$ by an~$n$-component link in~$Y^\infty$ that can be assumed to be in covering general position.
Use~$\widetilde{J}_1,\ldots,\widetilde{J}_n$ to denote the components of this link.
Thanks to Lemma~\ref{lem:Step1}, we can assume that the equivariant linking numbers of these knots coincide with the off-diagonal terms of~$Q^{-1}$; we can apply this lemma because for $i \neq j$ the rational functions~$\ell k_{\Q(t)} (\widetilde{J}_j,\widetilde{J}_i)$ and the corresponding~$-(Q^{-1})_{ij}$ both reduce mod~$\Z[t^{\pm 1}]$ to~$\Bl_Y(\pi(x_i^*),\pi(x_j^*))$ and thus differ by a Laurent polynomial~$p(t) \in \Z[t^{\pm 1}]$.

We arrange the framings and last assertion simultaneously.
For brevity,  from now on we write
$$r_i:=-(Q^{-1})_{ii}.$$
By Lemma~\ref{lem:SimpleClosedCurve}, for each $i$,
%~$i=1,\ldots,n$,
the class~$r_i[\mu_{\widetilde{J}_i}]+\lambda_{\widetilde{J}_i}$ can be rewritten as~$(r_i-\ell k_{\Q(t)}(\widetilde{J}_i,\pi_i))[\mu_{\widetilde{J}_i}]+[\pi_i]$ for any choice of parallel~$\pi_i$ for~$\widetilde{J}_i$.
Note that~$r_i-\ell k_{\Q(t)}(\widetilde{J}_i,\pi_i)$ is a Laurent polynomial: indeed both~$r_i$ and~$\ell k_{\Q(t)}(\widetilde{J}_i,\pi_i)$ reduce mod $\Z[t^{\pm 1}]$ to~$\Bl_Y(\pi([x_i^*]),\pi([x_i^*]))$.
\begin{claim*}
The polynomial~$r_i-\ell k_{\Q(t)}(\widetilde{J}_i,\pi_i)$ is symmetric.
\end{claim*}
\begin{proof}
We first assert that if~$\sigma$ is a parallel of~$\widetilde{J}_i$,  then~$\ell k_{\Q(t)}(\sigma,\widetilde{J}_i)$ is symmetric.
The rational function~$\ell k_{\Q(t)}(\sigma,\widetilde{J}_i)$ is symmetric if and only if~$\ell k_{\Q(t)}(\sigma,\widetilde{J}_i)=\overline{\ell k_{\Q(t)}(\sigma,\widetilde{J}_i)}$.
By the symmetry property of the equivariant linking form mentioned in Proposition~\ref{prop:Linkingprop}, this is equivalent to the equality~$\ell k_{\Q(t)}(\sigma,\widetilde{J}_i)=\ell k_{\Q(t)}(\widetilde{J}_i,\sigma)$ and in turn this equality holds because~the ordered link~$(\sigma,\widetilde{J}_i)$ is isotopic to the ordered link~$(\widetilde{J}_i,\sigma)$ in~$Y^\infty$.
This concludes the proof of the assertion that~$\ell k_{\Q(t)}(\sigma,\widetilde{J}_i)$ is symmetric.
%%Don t delete
%More words from LP: The definition of \ell k_{\Q(t)}(\sigma,\widetilde{J}_i) involves looking at $J$ in the complement of \sigma. Now when you want to consider \ell k_{\Q(t)}(\widetilde{J}_i,\sigma you can first isotope the link to exchange the compnents, and then the defition is again tto consider $J$ in the complement of $\sigma$.

%Indeed it is a parallell.
%It suffices to show that for all~$k\in\mathbb{Z}\smallsetminus\{0\}$\{Actually now i think~$k=0$ is fine too, the concern was the third equality},~$\ell k_{\Q}(\sigma,t^k\widetilde{K}_i)=\ell k_{\Q}(\sigma,t^{-k}\widetilde{K}_i)$.
%\{AC: Can you explain more why this suffices?
%We have~$\ell k_{\Q(t)}(x,y)=\frac{1}{p} \sum_{k \in \Z} (x \cdot t^k F)t^{-k}$ where~$py=\partial F$ for some polynomial~$p$; let's assume that~$p$ can be chosen to be symmetric (though you would have to convince me).
%You are claiming that~$x \cdot t^k F=\ell_\Q(x,t^k y)$.
%I'm not sure about this because~$t^k y$ does not bound~$t^k F$. Also you can't use~$p$ to make it bound here, because it is a polynomial and not an element of~$\Z$ as required by the definition of~$\ell_\Q$.
%}
%We observe that indeed
%$$\ell k_{\Q}(\sigma,t^k\widetilde{K}_i)=\ell k_{\Q}(t^{-k}\sigma,\widetilde{K}_i)=\ell k_{\Q}(\widetilde{K}_i,t^{-k}\sigma)=\ell k_{\Q}(\sigma,t^{-k}\widetilde{K}_i)$$
%where the first claim is by sesquiliniarity of the~$\Q(t)$ linking form, the second by symmetry of the~$\Q$ linking form, and the third by the observation that the link~$(\widetilde{K}_i,t^{-k}\sigma)$ is pairwise isotopic to the link~$(\sigma,t^{-k}\widetilde{K}_i)$.

We conclude the proof of the claim.
Thanks to the assertion, it now suffices to prove that~$r_i$ is symmetric.
To see this, note that since the matrix~$Q^{-1}$ is Hermitian (because~$Q$ is) we have~$r_i(t^{-1})=-(\overline{Q^{-1}})_{ii}=-(\overline{Q^{-T}})_{ii}=-(Q^{-1})_{ii}=r_i(t)$, as required.
\end{proof}
We can now apply Lemma~\ref{lem:ModifyFraming} to~$p(t):=r_i-\ell k_{\Q(t)}(\widetilde{J}_i,\pi_i)$ (which is symmetric by the claim) to isotope the~$\widetilde{J}_i$ to knots~$\widetilde{K}_i$ (without changing the equivariant linking) and to find parallels~$\gamma_1,\ldots,\gamma_n$ of~$\widetilde{K}_1,\ldots, \widetilde{K}_n$ that satisfy the equalities~$\unaryminus (Q^{-1})_{ii}=r_i=\ell k_{\Q(t)}(\widetilde{K}_i,\gamma_i)$.
This proves the second item of the proposition and the assertions in the last sentence follow because~$r_i[\mu_{\widetilde{K}_i}]+\lambda_{\widetilde{K}_i}=[\gamma_i]$ (by Lemma~\ref{lem:SimpleClosedCurve})
%AC explanation
%Since \gamma is a parallel of K_i, we have \lambda_{K_i}
%=[\pi_i]-\ell k (K_i,\gamma_i)[\mu_{K_i}]
%=[\pi_i]-r_i[\mu_{K_i}]
%and then rearrange the equation.
 and from the definition of the equivariant linking matrix.
\end{proof}

\subsection{Step 2: the result of surgery is a~$\Z[t^{\pm 1}]$-homology~$S^1 \times S^2$}
\label{sub:Step2}
Let $\widetilde{L} \subset Y^\infty$ be a framed link in covering general position.
Let~$Y'$ be the effect of surgery on the framed link~$L=p(\widetilde{L})$ with equivariant linking matrix $A_{\widetilde{L}}$ over $\Q(t)$.
We assume throughout this subsection that $\det(A_{\widetilde{L}})\neq 0$.
Our goal is to calculate the Alexander polynomial~$\Delta_{Y'}$ in terms of~$\Delta_Y$ and of the equivariant linking matrix of~$\widetilde{L} \subset Y^\infty$.
In Theorem~\ref{thm:OrderOfEffectOfSurgery} we will show that
\begin{equation}
\label{eq:AlexGoal}
\Delta_{Y'} \doteq \Delta_Y\det(A_{\widetilde{L}}).
\end{equation}
We then apply this to the framed link~$\widetilde{L} \subset Y^\infty$ that we built in Proposition~\ref{prop:Step1}; this framed link satisfies $\det(A_{\widetilde{L}})=\det(Q^{-T})\neq 0$.
Continuing with the notation from that proposition, we have~$\det(A_{\widetilde{L}})=\det(-Q^{-T}) \doteq \frac{1}{\Delta_Y}$ (because~$Q$ presents~$H_1(Y;\Z[t^{\pm 1}])$) so in this case~\eqref{eq:AlexGoal} implies that~$\Delta_{Y'} \doteq 1$, which in turn implies that~$Y'$  is a~$\Z[t^{\pm 1}]$-homology~$S^1 \times S^2$; see Remark~\ref{rem:AlexPoly} and Proposition~\ref{prop:Step2}.

 \medbreak
We start by outlining the proof of~\eqref{eq:AlexGoal}, which will be later recorded as Theorem ~\ref{thm:OrderOfEffectOfSurgery}.
\begin{proof}[Outline of proof of Theorem~\ref{thm:OrderOfEffectOfSurgery}]
\label{rem:Step3IdeaOfProof}
Our plan is to compute the Reidemeister torsion~$\tau(Y')$ in terms of the Reidemeister torsion~$\tau(Y)$, and then, for $Z=Y,Y'$ to use the relation
\begin{equation}\label{eq:Alextotorsion}
\Delta_{Z}=\tau(Z)(t-1)^2
\end{equation}
from~\cite[Theorem 1.1.2]{TuraevReidemeisterTorsionInKnotTheory} to derive~\eqref{eq:AlexGoal}. We note that in our setting we are allowed to write~$\tau(Y)$ and~$\tau(Y')$ for the Reidemeister torsions without having to choose bases~$\mathcal{B}$; this is because both~$H_*(Y;\Q(t))=0$ and~$H_*(Y';\Q(t))=0$, recall Lemma \ref{lem:surgQsphere} and Section~\ref{sec:reidemeister-torsion}; here note that we can apply Lemma \ref{lem:surgQsphere} because we are assuming that $\det(A_{\widetilde{L}})\neq 0$.

We will calculate~$\tau(Y')$ from~$\tau(Y)$ by studying the long exact sequence of the pairs~$(Y,Y_L)$ and~$(Y',Y_L)$ with~$\Q(t)$ coefficients.
More concretely, in Construction~\ref{cons:Bases}, we endow the $\Q(t)$-vector spaces~$H_*(Y,Y_L;\Q(t))$,~$H_*(Y',Y_L;\Q(t))$, and~$H_*(Y_L;\Q(t))$ with bases that we denote by~$\mathcal{B}_{Y,Y_L},\mathcal{B}_{Y',Y_L}$, and~$\mathcal{B}_{Y_L}$ respectively.
%Then in Lemma~\ref{lem:LESSimple}, we prove a statement that in particular implies that if~$H_*(Y;\Q(t))=0$ then~$H_*(Y';\Q(t))=0$.
In Lemma~\ref{lem:MultiplicativityTorsion}, we then show
%that~$\tau(Y,Y_L,\mathcal{B}_{Y,Y_L})=1$ and~$\tau(Y',Y_{L},\mathcal{B}_{Y,Y_{L}})=1$ and deduce from the multiplicativity of the torsion of Theorem~\ref{thm:ReidemeisterTorsion}
that
$$
\tau(Y)\tau(\mathcal{H}_L)^{-1}\doteq\tau(Y_L,\mathcal{B}_{Y_L})\doteq \tau(Y')\tau(\mathcal{H}_{L'})^{-1},
$$
where~$\mathcal{H}_L$ and~$\mathcal{H}_{L'}$ respectively denote the long exact sequences in~$\Q(t)$-homology of the pairs~$(Y,Y_L)$ and~$(Y',Y_L)$.
Finally, we prove that~$\tau(\mathcal{H}_L) \doteq 1$ and~$\tau(\mathcal{H}_{L'}) \doteq \det(A_{\widetilde{L}})$.
From~\eqref{eq:Alextotorsion} and the previous equation we then deduce
\[\frac{\Delta_Y}{(t-1)^2 \cdot 1} \doteq  \tau(Y)\tau(\mathcal{H}_L)^{-1}\doteq \tau(Y')\tau(\mathcal{H}_{L'})^{-1} \doteq \frac{\Delta_{Y'}}{(t-1)^2 \cdot \det(A_{\widetilde{L}})}.
\]
The equality~$\Delta_{Y'} \doteq \Delta_Y\det(A_{\widetilde{L}})$ follows promptly.
\end{proof}

We start filling in the details with our choice of bases for the previously mentioned~$\Q(t)$-homology vector spaces.

\begin{construction}
\label{cons:Bases}
We fix bases for $H_*(Y,Y_L;\Q(t))$, $H_*(Y',Y_L;\Q(t))$, and $H_*(Y_L;\Q(t))$, that we will respectively denote by~$\mathcal{B}_{Y,Y_L},\mathcal{B}_{Y',Y_L}$ and~$\mathcal{B}_{Y_L}$.
\begin{itemize}
\item
We base the~$\Q(t)$-vector spaces~$H_*(Y,Y_L;\Q(t))$ and~$H_*(Y',Y_L;\Q(t))$.
Excising~$\mathring{Y}_L$, we obtain~$H_i(Y,Y_L;\Q(t))=\bigoplus_{i=1}^n H_i(D^2 \times S^1,S^1 \times S^1;\Q(t))$ where~$n$ is the number of components of~$L$.
%Here the point is that the coefficient system is trivial on the link components (they lift!) whence the freeness.
Similarly, by excising~$\mathring{Y}_L \cong \mathring{Y}_{L'}$, we have~$H_i(Y',Y_L;\Q(t))=\bigoplus_{i=1}^n H_i(S^1 \times D^2,S^1 \times S^1;\Q(t))$.
Since the map $\pi_1(S^1) \to \Z$ determining the coefficients is trivial,  \[\bigoplus_{i=1}^n H_i(S^1 \times D^2,S^1 \times S^1;\Q(t)) \cong \bigoplus_{i=1}^n H^{3-i}(S^1;\Q(t)) \cong \bigoplus_{i=1}^n H^{3-i}(S^1;\Z) \otimes \Q(t).\]  These homology vector spaces are only non-zero when~$i=2,3$.  in which case they are isomorphic to~$\Q(t)^n$.

We now pick explicit generators for these vector spaces.
Endow~$S^1 \times S^1$ with its usual cell structure,  with one~$0$-cell,  two~$1$-cells and one~$2$-cell~$e^2_{S^1 \times S^1}$.
Note that~$D^2 \times S^1$ is obtained from~$S^1 \times S^1\times I$ by additionally attaching a 3-dimensional~$2$-cell~$e^2_{D^2 \times S^1}$ and~$3$-cell, ~$e^3_{D^2 \times S^1}$, where  on  the chain level~$\partial e^3_{D^2 \times S^1}=e^2_{D^2 \times S^1}+e^2_{S^1\times S^1}-e^2_{D^2 \times S^1}=e^2_{S^1\times S^1}$.
We now fix once and for all lifts of these cells to the covers.
It follows that for~$k=2,3$:
\begin{align*}
H_k(Y,Y_L;\Q(t))&=C_k(Y,Y_L;\Q(t))=C_k(D^2 \times S^1,S^1 \times S^1;\Q(t))=\bigoplus_{i=1}^n \Q(t) (\widetilde{e}_{D^2 \times S^1}^k)_i \\
H_k(Y',Y_L;\Q(t))&=C_k(Y',Y_L;\Q(t))=C_k(S^1 \times D^2,S^1 \times S^1;\Q(t))=\bigoplus_{i=1}^n  \Q(t) (\widetilde{e}_{S^1 \times D^2}^k)_i.
\end{align*}
\item We now base~$H_*(Y_L;\Q(t))$.
Since~$H_*(Y;\Q(t))=0$,  a Mayer-Vietoris argument shows that~$H_1(Y_L;\Q(t)) \cong \Q(t)^n$,  generated by the meridians~$\mu_{\widetilde{K}_i}$ of~$\widetilde{L}$. Mayer-Vietoris also shows that the inclusion of the boundary induces an isomorphism~$\Q(t)^n=H_2(\partial Y_L;\Q(t)) \cong H_2(Y_L;\Q(t))$.
We can then base~$H_2(Y_L;\Q(t))$ using fixed lifts of the aforementioned~$2$-cells~$(e^2_{S^1 \times S^1})_i$ generating each of the torus factors of~$\partial Y_L$.
Summarising, we have
\begin{align*}
 H_1(Y_L;\Q(t))&=\bigoplus_{i=1}^n \Q(t)\mu_{\widetilde{K}_i},\\
  H_2(Y_L;\Q(t))&=\bigoplus_{i=1}^n \Q(t)(\widetilde{e}^2_{S^1 \times S^1})_i.
\end{align*}
\end{itemize}
\end{construction}

The next lemma reduces the calculation of~$\Delta_{Y'}$ to the calculation of~$\tau(\mathcal{H}_L)$ and~$\tau(\mathcal{H}_{L'})$.
Here, recall that~$\tau(\mathcal{H}_L)$ and~$\tau(\mathcal{H}_{L'})$ denote the torsion of the long exact sequences~$\mathcal{H}_L$ and~$\mathcal{H}_{L'}$ of the pairs $(Y,Y_L)$ and $(Y',Y_L)$, viewed as based acyclic complexes with bases~$\mathcal{B}_{Y_L},\mathcal{B}_{Y,Y_L}$,  and~$\mathcal{B}_{Y',Y_L}$.

\begin{lemma}
\label{lem:MultiplicativityTorsion}
If~$H_1(Y;\Q(t))=0$ and $\det(A_{\widetilde{L}})\neq 0$, then we have
\begin{align*}
\tau(Y)&\doteq \tau(Y_L,\mathcal{B}_{Y_L})\cdot \tau(\mathcal{H}_L), \\
 \tau(Y')&\doteq\tau(Y_L,\mathcal{B}_{Y_L})\cdot \tau(\mathcal{H}_{L'}).
\end{align*}
In particular, we have
$$\Delta_{Y'}\cdot \tau(\mathcal{H}_{L})\doteq\Delta_Y\cdot \tau(\mathcal{H}_{L'}).$$
\end{lemma}
\begin{proof}
We start by proving that the last statement follows from the first.
First note that since the vector spaces~$H_1(Y;\Q(t))$ and~$H_1(Y';\Q(t))$ vanish (for the latter we use Lemma~\ref{lem:surgQsphere} which applies since $\det(A_{\widetilde{L}})\neq 0$), the Alexander polynomials of~$Y$ and~$Y'$ are nonzero.
Next,~\cite[Theorem 1.1.2]{TuraevReidemeisterTorsionInKnotTheory} implies that~$\tau(Y)(t-1)^2=\Delta_Y$ and similarly for~$Y'$.  Therefore
$\Delta_{Y'}/\Delta_Y=\tau(Y')/\tau(Y).$
The first part of the lemma implies that $\tau(Y')/\tau(Y)=\tau(\mathcal{H}_{L'})/\tau(\mathcal{H}_{L})$. Combining these equalities,
%torsions are non-zero by definition
$$\frac{\Delta_{Y'}}{\Delta_Y}=\frac{\tau(Y')}{\tau(Y)}=\frac{\tau(\mathcal{H}_{L'})}{\tau(\mathcal{H}_{L})}, $$
from which the required statement follows immediately.
%Kind of overkill: you can just simplify the $t-1$....

To prove the first statement of the lemma, it suffices to prove that~$\tau(Y,Y_L,\mathcal{B}_{Y,Y_L})=1$ as well as~$\tau(Y',Y_L,\mathcal{B}_{Y',Y_L})=1$: indeed, the required equalities then follow by applying the multiplicativity of Reidemeister torsion (the first item of Proposition~\ref{thm:ReidemeisterTorsion}) to the short exact sequences
\[0 \to C_*(Y_L;\Q(t)) \to C_*(Y;\Q(t)) \to C_*(Y,Y_L;\Q(t)) \to 0,\]
leading to $\tau(Y) = \tau(Y_L) \cdot \tau(Y,Y_L,\mathcal{B}_{Y,Y_L}) \cdot \tau(\mathcal{H}_{L}) = \tau(Y_L) \cdot 1 \cdot \tau(\mathcal{H}_{L})$ as desired. And similarly for the pair~$(Y',Y_L)$.

We use Definition~\ref{def:ReidemeisterTorsion} to prove that~$\tau(Y,Y_L,\mathcal{B}_{Y,Y_L})=1$; again the proof for~$L'$ is analogous.
We endow~$Y$ and~$Y_L$ with cell structures for which~$Y_L$ and $\partial Y_L$ are subcomplexes of~$Y$, and~$Y$ is obtained from~$Y_L$ by attaching~$n$ solid tori to $\partial Y_L$.
By definition of the relative chain complex, we have~$C_*(Y,Y_L;\Q(t))=C_*(Y;\Q(t))/C_*(Y_L;\Q(t))$.
Since we are working with cellular chain complexes we deduce that
$$C_*(Y,Y_L;\Q(t))=C_*(Y;\Q(t))/C_*(Y_L;\Q(t))=\bigoplus_{i=1}^n C_*(D^2 \times S^1;\Q(t))/C_*(S^1 \times S^1;\Q(t)).$$
Using the cell structures described in Construction~\ref{cons:Bases},~$D^2 \times S^1$ is obtained from~$S^1 \times S^1$ by attaching a~$2$-cell and a~$3$-cell.
By the above sequence of isomorphisms, this shows that~$C_i(Y,Y_L;\Q(t))=0$ for~$i\neq 2,3$ and gives a basis for~$C_2(Y,Y_L;\Q(t))$ and~$C_3(Y,Y_L;\Q(t))$.
In fact, this also implies that~$C_i(Y,Y_L;\Q(t))=H_i(Y,Y_L;\Q(t))$ and that the differentials in the chain complex are zero, as was mentioned in Construction~\ref{cons:Bases}.
Thus,  the basis of~$C_*(Y,Y_L;\Q(t))$ corresponds exactly to the way we based~$H_*(Y,Y_L;\Q(t))$ in Construction~\ref{cons:Bases}.
Therefore the change of basis matrix is the identity and so the torsion is equal to~$1$.
This concludes the proof of the lemma.
\end{proof}

%Next, we prove that if~$H_*(Y;\Q(t))=0$, and~$H_*(Y';\Q(t))=0$.
%In fact we will prove a more general statement that describes the long exact sequences~$\mathcal{H}_L$ and~$\mathcal{H}_{L'}$, as we will need them later on anyways.
Our goal is now to show that~$\tau(\mathcal{H}_L) \doteq 1$ and~$\tau(\mathcal{H}_{L'}) \doteq \det(A_{\widetilde{L}})$.  We start by describing the long exact sequences~$\mathcal{H}_L$ and~$\mathcal{H}_{L'}$.

\begin{lemma}
\label{lem:LESSimple}
Assume that~$H_1(Y_L;\Q(t))=0$ and $\det(A_{\widetilde{L}})\neq 0$.
The only nontrivial portions of the long exact sequence of the pairs~$(Y,Y_L)$ and~$(Y,Y_{L'})$ with~$\Q(t)$-coefficients are of the following form:
\begin{align*}
\mathcal{H}_L=&\, \Big( 0 \to H_3(Y,Y_L;\Q(t)) \xrightarrow{\partial_3^L } H_2(Y_L;\Q(t)) \to 0 \to H_2(Y,Y_{L};\Q(t)) \xrightarrow{\partial_2^L } H_1(Y_{L};\Q(t)) \to 0 \Big), \\
\mathcal{H}_{L'}=&\, \Big( 0 \to H_3(Y',Y_L;\Q(t)) \xrightarrow{\partial_3^{L'}} H_2(Y_L;\Q(t)) \to 0 \to H_2(Y',Y_{L};\Q(t)) \xrightarrow{\partial_2^{L'} } H_1(Y_L;\Q(t)) \to 0 \Big).
\end{align*}
Additionally,  with respect to the bases of Construction~\ref{cons:Bases},
\begin{itemize}
\item the homomorphism~$\partial_2^{L'}$ is represented by ~$-A_{\widetilde{L}}^{-1}$, i.e.~minus the inverse of the equivariant linking matrix for~$\widetilde{L}$;
\item the homomorphisms $\partial_2^L$, $\partial_3^L$, and~$\partial_3^{L'}$ are represented by identity matrices.
\end{itemize}
\end{lemma}
\begin{proof}
Since~$Y^\infty$ and~${Y'}^\infty$ are connected,
%, the relative homology group~$H_0(Y,Y_L;\Z[t^{\pm 1}])$ and~$H_0(Y',Y_L;\Z[t^{\pm 1}])$ vanish, while the absolute ones vanish because
we have~$H_0(Y;\Z[t^{\pm 1}])=\Z$ and~$H_0(Y';\Z[t^{\pm 1}])=\Z$, so~$H_0(Y;\Q(t))=0$ and~$H_0(Y';\Q(t))=0$.
Since we are working with field coefficients, Poincar\'e duality and the universal coefficient theorem imply that~$H_3(Y;\Q(t))=0$ and~$H_3(Y';\Q(t))=0$.
As observed in Construction~\ref{cons:Bases} above, by excision, the only non-zero relative homology groups of~$(Y,Y_L)$ and~$(Y',Y_L)$ are
\begin{align*}
H_i(Y,Y_L;\Q(t))=\Q(t)^n \ \ \ \ &\text{ and } \ \ \ \ H_i(Y',Y_L;\Q(t))=\Q(t)^n
%\\
%H_i(Y,Y_L;\Z[t^{\pm 1}])=\Z[t^{\pm 1}]^n \ \ \ \ &\text{ and } \ \ \ \ H_i(Y',Y_L;\Z[t^{\pm 1}])=\Z[t^{\pm 1}]^n
\end{align*}
for~$i=2,3$.
%%Here the point is that the coefficient system is trivial on the link components (they lift!) whence the freeness.
Next, since by assumption~$H_1(Y;\Q(t))=0$, duality and the universal coefficient theorem imply that~$H_2(Y;\Q(t))=0$.
Since we proved in Lemma~\ref{lem:surgQsphere} that~$H_1(Y';\Q(t))=0$, (here we used $\det(A_{\widetilde{L}}) \neq 0$) the same argument shows that~$H_2(Y';\Q(t))=0$.
This establishes the first part of the lemma.

We now prove the statement concerning~$\partial_2^L$ and~$\partial_2^{L'}$.
Recall from Construction~\ref{cons:Bases} that we based the vector spaces~$H_2(Y,Y_L;\Q(t))$ and~$H_2(Y',Y_L;\Q(t))$ by meridional discs to the~$\widetilde{K}_i$ and~$\widetilde{K}_i'$ respectively.
The map~$\partial_2^L$ takes each disc to its boundary, the meridian~$\mu_{\widetilde{K}_i}$; since these meridians form our chosen basis for~$H_1(Y_L;\Q(t))$, we deduce that~$\partial_2^L$ is represented by the identity matrix.
The map~$\partial_2^{L'}$ also takes each meridional disc to its boundary, the meridian~$\widetilde{\mu}_{K_i'}$ to the dual knot.
It follows that~$\partial_2^{L'}$ is represented by the change of basis matrix~$B$ such that~$\boldsymbol{\mu}'=B\boldsymbol{\mu}$.
But during the proof of Lemma~\ref{lem:InverseMatrix} we saw that~$B=-A_{\widetilde{L}}^{-1}$ .

Finally, we prove that~$\partial_3^L$ and~$\partial_3^{L'}$ are represented by identity matrices.
In Construction~\ref{cons:Bases},  we based~$H_3(Y,Y_L;\Q(t))$ and~$H_3(Y',Y_L;\Q(t))$ using respectively (lifts of) the~$3$-cells of the~$(D^2 \times S^1)_i$ and~$(S^1 \times D^2)_i$.
Now both~$\partial_3^L$ and~$\partial_3^{L'}$ take these~$3$-cells to their boundaries.
But as we noted in Construction~\ref{cons:Bases}, these boundaries are (algebraically) the~$2$-cells~$(e^2_{S^1 \times S^1})_i$.
In other words both~$\partial_3^L$ and~$\partial_3^{L'}$ map our choice of ordered bases to our other choice of ordered bases, and are therefore represented in these bases by  identity matrices, as required.
This concludes the proof of  Lemma~\ref{lem:MultiplicativityTorsion}.
\end{proof}

%\begin{lemma}
%\label{lem:TorsionRelative}
%We have~$\tau(Y,Y_L,\mathcal{B}_{Y,Y_L})=1$ and~$\tau(Y',Y_L,\mathcal{B}_{Y',Y_L})=1$
%\end{lemma}
%\begin{proof}
%Since~$C_*(Y,Y_L)=C_*(Y)/C_*(Y_L)$ is isomorphic to~$\bigoplus_{i=1}^n C_*(D^2 \times S^1)/C_*(S^1 \times S^1)$, we know that~$C_*(Y,Y_L)$ has~$n$~$2$-cells and~$n$~$3$-cells and no other cell.
%This is exactly how we based the homology of~$H_*(Y,Y_L;\Q(t))$ and so the change of basis matrix is the identity and so the torsion is equal to~$1$.
%The same reasoning applies to~$L'$.
%\end{proof}

%Thanks to Lemma~\ref{lem:MultiplicativityTorsion}, we deduce that the only remaining task is to calculate~$\tau(\mathcal{H}_L)$ and~$\tau(\mathcal{H}_{L'})$.

As we now understand the exact sequences~$\mathcal{H}_L$ and~$\mathcal{H}_{L'}$ we can calculate their torsions, leading to the proof of the main result of this subsection.

\begin{theorem}
\label{thm:OrderOfEffectOfSurgery}
If~$H_1(Y_L;\Q(t))=0$ and $\det(A_{\widetilde{L}})\neq 0$, then we have
$$\Delta_{Y'}\doteq \det(A_{\widetilde{L}})\Delta_Y.$$
\end{theorem}

\begin{proof}
Use the bases from Construction~\ref{cons:Bases}.
Combine the second item of Proposition~\ref{thm:ReidemeisterTorsion} with Lemma~\ref{lem:LESSimple} to obtain:
\begin{align*}
\tau(\mathcal{H}_L) \doteq \frac{\det(\partial_3^L)}{\det(\partial_2^L)} \doteq 1\text{ and }
\tau(\mathcal{H}_{L'})\doteq \frac{\det(\partial_3^{L'})}{\det(\partial_2^{L'})} \doteq \det(A_{\widetilde{L}}).
\end{align*}
We deduce that
$\tau(\mathcal{H}_{L'})/\tau(\mathcal{H}_{L}) \doteq \det(A_{\widetilde{L}}).$
Apply Lemma~\ref{lem:MultiplicativityTorsion} to obtain
$$\frac{\Delta_{Y'}}{\Delta_Y} \doteq \frac{\tau(\mathcal{H}_{L'})}{\tau(\mathcal{H}_{L})}\doteq \det(A_{\widetilde{L}}).$$
Rearranging yields the desired equality.
\end{proof}

As a consequence, we complete the second step of the plan from Subsection~\ref{sub:Plan}.
\begin{proposition}
\label{prop:Step2}
Let~$0 \to H \xrightarrow{\widehat{\lambda}} H^* \xrightarrow{\varpi} H_1(Y;\Z[t^{\pm 1}]) \to 0$ be a presentation of~$Y$.
Pick generators~$x_1,\ldots,x_n$ for~$H$ and endow~$H^*$ with the dual basis~$x_1^*,\ldots,x_n^*$.
Let~$Q$ be the matrix of~$\lambda$ with respect to these bases.
The classes~$\pi(x_1^*),\ldots,\pi(x_n^*)$ can be represented by a framed link~$\widetilde{L}$ in covering general position with equivariant linking matrix~$A_{\widetilde{L}}=-Q^{-T}$.
In addition, the~$3$-manifold~$Y'$ obtained by surgery on~$Y$ along $L$  satisfies~$H_1(Y';\Z[t^{\pm 1}])=0$.
\end{proposition}

\begin{proof}
The existence of~$\widetilde{L}$ representing the given
%basis
generators and with equivariant linking matrix~$A_{\widetilde{L}}=-Q^{-T}$ is proved in Proposition \ref{prop:Step1}.
Since~$Q^{T}$ presents~$H_1(Y;\Z[t^{\pm 1}])$, we have~$\det(Q) \doteq \Delta_Y$ and therefore~$\det(A_{\widetilde{L}}) \doteq \frac{1}{\Delta_Y}$.
%AC: Det(Q)=Det(Q^T)=1/Delta_Y, really.
Theorem~\ref{thm:OrderOfEffectOfSurgery} now implies that~$\Delta_{Y'} \doteq 1$.

A short argument is now needed to use Remark~\ref{rem:AlexPoly} in order to conclude~$H_1(Y';\Z[t^{\pm 1}])=0$: we require that this torsion module admits a square presentation matrix, i.e.\ has projective dimension at most~$1$, denoted~$\pd (H_1(Y';\Z[t^{\pm 1}])) \leq 1$.
Here recall that that a $\Z[t^{\pm 1}]$-module~$P$ \emph{has projective dimension at most~$k$} if~$\operatorname{Ext}^i_{\Z[t^{\pm 1}]}(P;V)=0$ for every~$\Z[t^{\pm 1}]$-module~$V$ and every~$i\geq k+1$,
and that for a short exact sequence $0 \to A \to B \to C \to 0$ of $\Z[t^{\pm 1}]$-modules, the associated long exact sequence in $\operatorname{Ext}(-;V)$ groups implies that:
\begin{enumerate}[(a)]
\item if $\pd (C) \leq 1$ and $A$ is free, then $\pd (B) \leq 1$;
\item if $\pd (B) \leq 1$ and $A$ is free, then $\pd (C) \leq 1$.
\end{enumerate}
The  following paragraph proves that~$\pd (H_1(Y';\Z[t^{\pm 1}])) \leq 1$.
As~$H_1(Y;\Z[t^{\pm 1}])$ and~$H_1(Y';\Z[t^{\pm 1}])$ are torsion (for the latter recall Lemma~\ref{lem:surgQsphere}),  a duality argument implies that~$H_2(Y;\Z[t^{\pm 1}])=\Z$ and~$H_2(Y';\Z[t^{\pm 1}])=\Z$ (see e.g.  the first item of~\cite[Lemma 3.2]{ConwayPowell}).
%AC: Note that we were able to use Lemma~\ref{lem:surgQsphere} because $\det A_{\widetilde{L})=-Q^{-T}$ is invertible.
Since these modules are torsion and since excision implies that
\begin{align*}
H_2(Y,Y_L;\Z[t^{\pm 1}])=\Z[t^{\pm 1}]^n \ \ \ \ &\text{ and } \ \ \ \ H_2(Y',Y_L;\Z[t^{\pm 1}])=\Z[t^{\pm 1}]^n \\
H_1(Y,Y_L;\Z[t^{\pm 1}])=0\ \ \ \ &\text{ and } \ \ \ \ H_1(Y',Y_L;\Z[t^{\pm 1}])=0,
\end{align*}
we deduce that the maps~$H_2(Y;\Z[t^{\pm 1}]) \to  H_2(Y,Y_L;\Z[t^{\pm 1}])$ and~$H_2(Y';\Z[t^{\pm 1}]) \to  H_2(Y',Y_L;\Z[t^{\pm 1}])$ are both trivial leading to the short exact sequences
\begin{align*}
& 0  \to H_2(Y,Y_L;\Z[t^{\pm 1}]) \to H_1(Y_L;\Z[t^{\pm 1}]) \to H_1(Y;\Z[t^{\pm 1}]) \to 0, \\
& 0  \to H_2(Y',Y_L;\Z[t^{\pm 1}]) \to H_1(Y_L;\Z[t^{\pm 1}]) \to H_1(Y';\Z[t^{\pm 1}]) \to 0.
\end{align*}
Next we apply the facts (a) and (b) on projective dimension given above.
Since the torsion module~$H_1(Y;\Z[t^{\pm 1}])$ is presented by~$(H,\lambda)$,  it has projective dimension at most~$1$ and since $H_2(Y,Y_L;\Z[t^{\pm 1}])$ is  free, the first short exact sequence implies that $H_1(Y_L;\Z[t^{\pm 1}])$ has projective dimension at most~$1$.
Since~$ H_2(Y',Y_L;\Z[t^{\pm 1}])$ is free, the second short exact sequence now implies that~$\pd (H_1(Y';\Z[t^{\pm 1}])) \leq 1$ as required.
%See https://stacks.math.columbia.edu/tag/00O2 but can also be proved by hand.

As explained above, since~$\pd (H_1(Y';\Z[t^{\pm 1}])) \leq 1$ and~$\Delta_{Y'} \doteq 1$,  Remark~\ref{rem:AlexPoly} now allow us to conclude that~$H_1(Y';\Z[t^{\pm 1}])=0$, as required.
\end{proof}

\subsection{Step 3: every~$\Z[t^{\pm 1}]$-homology~$S^1 \times S^2$ bounds a homotopy  circle.}
\label{sub:Step3}

The goal of this subsection is to prove the following theorem, which is a generalisation of a key step in the proof that Alexander polynomial one knots are topologically slice.

\begin{theorem}
\label{thm:Step3}
Let~$Y$ be a~$3$-manifold with an epimorphism~$\pi_1(Y) \twoheadrightarrow \Z$ whose Alexander module vanishes, i.e.~$H_1(Y;\Z[t^{\pm 1}])=0$.
Then there exists a~$4$-manifold~$B$ with a homotopy equivalence~$g \colon B \xrightarrow{\simeq} S^1$ so that~$\partial B \cong Y$ and~$\pi_1(Y) \twoheadrightarrow \pi_1(B) \xrightarrow{g_*}\pi_1(S^1) \cong \Z$ agrees with~$\varphi.$
\end{theorem}
\begin{proof}
This proof can be deduced by combining various arguments from~\cite[Section~11.6]{FreedmanQuinn}, so we only outline the main steps.
%We start by describing the general strategy.
%First we use framed bordism theory to find some 4-manifold $W$ whose boundary is~$Y$, with a map to~$S^1$ realising~$\varphi$.
%This map might not be a homotopy equivalence, but  we then we use surgery theory to show that~$W$ is bordant rel.\  boundary to a homotopy circle.
First we use framed bordism to find some 4-manifold whose boundary is~$Y$, with a map to~$S^1$ realising~$\varphi$,  as in~\cite[Lemma 11.6B]{FreedmanQuinn}.
This map might not be a homotopy equivalence, but  we then we will use surgery theory to show that~$W$ is bordant rel.\  boundary to a homotopy circle.

To start the first step, recall that every oriented 3-manifold admits a framing of its tangent bundle.
Using the axioms of a generalised homology theory, we have
\[\Omega_3^{\fr}(B\Z) \cong \Omega_3^{\fr} \oplus \Omega_2^{\fr} \cong \Z/24 \oplus \Z/2.\]
We consider the image of~$(Y,\varphi)$ in~$\Omega_3^{\fr}(B\Z)$.
The first summand can be killed by changing the choice of framing of the tangent bundle of~$Y$; see~\cite[proof of Lemma 11.6B]{FreedmanQuinn} for details.
 The second summand is detected by an Arf invariant, which vanishes thanks to the assumption that~$H_1(Y;\Z[t^{\pm 1}])=0$; details are again  in~\cite[proof of Lemma 11.6B]{FreedmanQuinn}.
  % The second summand is detected by an Arf invariant,
% It turns out that this is determined by the Alexander polynomial~$\Delta_Y$ of~$Y$, evaluated at~$t=-1$.
% Since~$H_1(Y;\Z[t^{\pm 1}]) =0$, the Arf invariant vanishes and so~$(Y,\varphi) = 0 \in \Omega_3^{\fr}(B\Z)$.
Therefore there exists a framed 4-manifold~$W$ with framed boundary~$Y$, such that the map~$Y \to S^1$ associated with~$\varphi$ extends over~$W$.

Now we use surgery theory to show that $W$ is bordant rel.\  boundary to a homotopy circle.
Consider the mapping cylinder
\begin{equation}
\label{eq:MappingCylinder}
X := \mathcal{M}(Y \xrightarrow{\varphi} S^1).
\end{equation}
We claim that~$(X,Y)$ is a Poincar\'{e} pair.
The argument is similar to~\cite[Proposition~11.C]{FreedmanQuinn}.
As~$X \simeq S^1$,  the connecting homomorphism from the exact sequence of the pair $(X,Y)$ gives an isomorphism~$\partial \colon H_4(X,Y) \cong H_3(Y) \cong \Z$.
We then define the required fundamental class as~$[X,Y]:=\partial^{-1}([Y]) \in H_4(X,Y)$.
Using~$H_1(Y;\Z[t^{\pm 1}])=0$,  one can now use the same argument as in~\cite[Lemma 3.2]{FriedlTeichner} to show that the following cap product is an isomorphism:
%AC: I wrote it out better in the lecture notes of my surgery theory course.
%First,  note that since~$X \simeq S^1$,  the only nontrivial homology of~$X$ with~$\Z[t^{\pm 1}]$ coefficients is~$H_0(X;\Z[t^{\pm 1}]) \cong \Z$.
%Similarly the relative homology~$H_*(X,Y;\Z[t^{\pm 1}])$ vanishes apart from~$H_3(X,Y;\Z[t^{\pm 1}])=\Z$.
%We can also compute that the cohomology of~$X$ is~$H^1(X;\Z[t^{\pm 1}]) \cong \Z$ and is otherwise zero.
$$- \cap [X,Y] \colon H^i(X,Y;\Z[t^{\pm 1}]) \to H_{4-i}(X;\Z[t^{\pm 1}]).$$
This concludes the proof of the fact that~$(X,Y)$ is a Poincar\'{e} pair.
%AC: For the relative one, $Y$ is already a manifold so all is well.

The end of the argument follows from the exactness of the surgery sequence for $(X,Y)$ as in~\cite[Proposition 11.6A]{FreedmanQuinn} but we outline some details for the reader unfamiliar with surgery theory.
Since~$(X,Y)$ is a Poincar\'e pair,  we can consider its set~$\mathcal{N}(X,Y)$ of normal invariants.
The set~$\mathcal{N}(X,Y)$ consists of normal bordism classes of degree one normal maps to~$X$ that restrict to a homeomorphism on the boundary, where a bordism restricts to a product cobordism homeomorphic to~$Y \times I$ between the boundaries.
The next paragraph uses the map~$W \to S^1$ to define an element of~$\mathcal{N}(X,Y)$.

Via the homotopy equivalence~$X \simeq S^1$,  the map~$Y \to S^1 \simeq X$ extends to~$F \colon W \to S^1 \simeq X$.
It then follows from the naturality of the long exact sequence of the pairs~$(W,Y)$ and~$(X,Y)$ that~$F$ has degree one.
%AC: Use that the connecting homomorphisms $\partial H_4(M,Y) \to H_3(Y)$ are isomorphisms for $M=W,X$
We therefore obtain a degree one map $(F,\id_Y) \colon (W,Y) \to (X,Y)$.
To upgrade~$(F,\id_Y)$ to a degree one normal map, take a trivial (stable)  bundle $\xi \to X$ over the codomain.   Normal data is determined by a (stable) trivialisation of $TW \oplus F^*\xi$.  The framing of~$W$ provides a trivialisation for the first summand, while any choice of trivialisation for $F^*\xi$ can be used for the second summand. We therefore have a degree one normal map
\[\big( (F,\id_Y) \colon (W,Y) \to (X,Y)\big) \in \mathcal{N}(X,Y).\]
%AC: Again details in my surgery theory lecture notes.

Our goal is to change $W$ to $W\#^{\ell} Z$, where $Z = E_8$, and then to do surgery on the interior of the domain~$(W \#^\ell Z,Y)$ to convert $F$ into a homotopy equivalence $(F',\id_Y) \colon (B,Y) \to (X,Y)$.
Since the fundamental group~$\Z$ is a good group, surgery theory says that this is possible
%which implies that the structure set~$\mathcal{S}(X,Y)$ is nonempty,
if and only if $\ker(\sigma)$
%AC: It used to be written ~$\sigma^{-1}(\{0\})$ but in the topological category, the surgery exact sequence consists of groups and group homomorphisms, so might as well write $\ker$.
 is nonempty~\cite[Section 11.3]{FreedmanQuinn}.
Here
\[\sigma \colon \mathcal{N}(X,Y) \to L_4(\Z[t^{\pm 1}])\]
is the surgery obstruction map.
Essentially,
%AC: This is true true because  $X$ is aspherical so the surgery kernel agrees with $H_2(W;\Lambda)$.
it takes the intersection pairing on~$H_2(W;\Z[t^{\pm 1}])$ and considers it in the Witt group of nonsingular, Hermitian, even forms over~$\Z[t^{\pm 1}]$ up to stable isometry, where stabilisation is by hyperbolic forms \[\left(\Z[t^{\pm 1}] \oplus \Z[t^{\pm 1}],\begin{pmatrix}
  0 & 1 \\ 1 & 0
\end{pmatrix}\right).\]
Shaneson splitting~\cite{ShanesonSplitting} implies that
$L_4(\Z[t^{\pm 1}]) \cong L_4(\Z) \oplus L_3(\Z) \cong L_4(\Z) \cong 8\Z.$
The last isomorphism is given by taking the signature.
We take the connected sum of~$W \to X$ with copies of~$(E_8 \to S^4)$ or~$(-E_8 \to S^4)$, to arrange that the signature becomes zero.
Then the resulting normal map~$W \#^{\ell} Z \to X$ has trivial surgery obstruction in~$L_4(\Z[t^{\pm 1}])$ (i.e.\ lies in $\ker(\sigma)$) and therefore is normally bordant to a homotopy equivalence
$(F',\id_Y) \colon (B,Y) \to (X,Y)$, as desired.
Since the mapping cylinder $X$ from~\eqref{eq:MappingCylinder} is a homotopy circle, so is $B$. This concludes the proof of the theorem.
\end{proof}

\subsection{Step 4: constructing a~$4$-manifold that induces the given boundary isomorphism}
\label{sub:Step4}

We begin by recalling the notation and outcome of Proposition~\ref{prop:Step2}.
Let~$b \in \Iso(\partial \lambda,\unaryminus\Bl_Y)$ be an isometry of linking forms.
Pulling this back to~$H$, we obtain a presentation
$$0 \to H \xrightarrow{\widehat{\lambda}} H^* \xrightarrow{\varpi} H_1(Y;\Z[t^{\pm 1}]) \to 0$$
of~$Y$.
Pick generators~$x_1,\ldots,x_n$ for~$H$ and endow~$H^*$ with the dual basis~$x_1^*,\ldots,x_n^*$.
Let~$Q$ be the matrix of~$\lambda$ with respect to these bases.
By Propositions~\ref{prop:Step1} and~\ref{prop:Step2}, the classes~$\pi(x_1^*),\ldots,\pi(x_n^*)$ can be represented by a framed link~$\widetilde{L} \subset Y^\infty$ in covering general position with transposed equivariant linking matrix~$-Q^{-1}$ and the~$3$-manifold~$Y'$ obtained by surgery on~$L=p(\widetilde{Y})$ satisfies~$H_1(Y';\Z[t^{\pm 1}])=0$.
Applying Theorem~\ref{thm:Step3}, there is a topological~$4$-manifold~$B$ with boundary~$Y'$ and such that~$B \simeq S^1$.

We now define a~$4$-manifold~$M$ with boundary~$Y$ as follows: begin with~$Y\times I$ and attach 2-handles to~$Y\times \{1\}$ along the framed link~$L:=p(\widetilde{L})$~(here recall that $p \colon Y^\infty \to Y$ denotes the covering map), so that the resulting boundary is~$Y'$.
Call this 2-handle cobordism~$W$, and observe that $\partial^-W=-Y$.
We can now cap $\partial^+W\cong Y'$ with~$-B$.
Since $W\cup -B$ has boundary $-Y$, we define $M$ to be $-W\cup B$. We can then consider the corresponding~$\Z$-cover:
\begin{align*}
-M^\infty&:=\Big( (Y^\infty \times [0,1]) \cup \bigcup_{i=1}^n \bigcup_{j_i \in\Z} t^{j_i} h_i^{(2)} \Big) \cup_{{Y'}^\infty} -B^\infty =W^\infty \cup_{{Y'}^\infty} -B^\infty\\
%%%%%
-M&:=\Big( (Y \times [0,1]) \cup \bigcup_{i=1}^n h_i^{(2)} \Big) \cup_{{Y'}} -B=:W \cup_{{Y'}} -B,
\end{align*}
in which the~$2$-handles are attached along the framed link~$\widetilde{L}$ upstairs and its framed projection~$L$ downstairs.

We begin by verifying some properties of $M$.

\begin{lemma}
\label{lem:Pi1Z}
The $\Z$-manifold $M$ has boundary $Y$.
\end{lemma}
\begin{proof}
We first prove that $\pi_1(M)\cong \Z$.
A van Kampen argument shows that~$\pi_1(M)$ is obtained from~$\pi_1(B)$ by modding out the~$[\iota(\widetilde{K}'_i)]$ where~$\widetilde{K}_1',\ldots,\widetilde{K}_n'$ denote the components of the framed link dual to~$\widetilde{L}$ and where~$\iota \colon \pi_1(Y') \to \pi_1(B)$ is the inclusion induced map.
%AC: Munkres second edition page 433 Exercise 2 shows that f: \pi_1(Y') ->\pi_1(W) is surjective, \pi_1(M)=\pi_1(B)/<< image in \pi_1(B) of \ker(f)>>=\pi_1(B)/\iota(\widetilde{K}_i' because W is obtained from Y' by adding handles to kill the \widetilde{K}_i'.
Recall from Lemma~\ref{lem:coeff-system} and Remark~\ref{rem:CoefficientSystemY'} that the epimorphism~$\varphi \colon \pi_1(Y) \twoheadrightarrow \Z$ induces an epimorphism~$ \varphi' \colon \pi_1(Y') \twoheadrightarrow \Z$ and that~$\varphi'([K_i'])=0$ for~$i=1,\ldots,n$.
Since Theorem~\ref{thm:Step3} ensures that~$\iota$ agrees with $\varphi'$, we deduce that the classes~$[\iota(\widetilde{K}'_i)]$ are trivial and therefore~$\pi_1(M)\cong \pi_1(B) \cong \Z$.

Next we argue that as a $\Z$-manifold $M$ has boundary $Y$.
Since the inclusion induced map $\pi_1(Y) \to \pi_1(W)$ is surjective, it suffices to prove that the inclusion induced map $\pi_1(W) \to \pi_1(M)$ is surjective.
This follows from van Kampen's theorem: as $\pi_1(Y') \to \pi_1(B)$ is surjective, so is $\pi_1(W) \to \pi_1(M)$.
%AC: \pi_1(Y)-->\pi_1(W) is surjective because W is obtained from Y by adding 2-handles. \pi_1(W)-->\pi_1(M) is surjective because \pi_1(M) is an amalgamated product and \iota=\varphi' is surjective.
\end{proof}

%\begin{remark}

%It is a general fact for compact, orientable $4$-manifolds with fundamental groups $\Z$ and ribbon boundary~\cite[Lemma~3.2]{ConwayPowell}, that  $H_2(M;\Z[t^{\pm 1}])$ is finitely generated and free.

%We compute that the rank is $n$ and find explicit generators in this homology.
%\end{remark}

It is not too hard to compute, as we will do in Proposition~\ref{prop:BasisH2} below, that $H_2(M;\Z[t^{\pm 1}])$ is f.g. free of rank $n$.
To complete step 4, we must prove the following two claims.
\begin{enumerate}
\item The equivariant intersection form~$\lambda_M$ of~$M$ is represented by~$Q$; i.e.~$\lambda_M$ is isometric to~$\lambda$.
\item The~$4$-manifold~$M$ satisfies~$b_M =b \in \Iso(\partial \lambda,\unaryminus\Bl_Y)/\Aut(\lambda)$.
\end{enumerate}

The proof of the first claim follows a standard outline; for the hasty reader we will give the outline here, and for the record we provide a detailed proof at the end of the subsection.
%Throughout the rest of the section, for a knot~$K$ in a manifold~$Y$, we use $r(K)$ to denote reversing the strand orientation of~$K$.

\begin{proof}[Proof outline of claim (1)]
Since by setup the transposed equivariant linking matrix of the framed link~$\widetilde{L}$ is~$-Q^{-1},$ Proposition~\ref{lem:InverseMatrix} shows that the transposed equivariant linking matrix of the dual link~$\widetilde{L}'$ is~$Q$. Thus, it suffices to show that~$\lambda_M$ is presented by the transposed equivariant linking matrix of~$\widetilde{L}'$.

While it was natural initially to build~$W^\infty$ by attaching 2-handles to~$Y^\infty\times I$, in what follows it will be more helpful to  view~$-W^\infty$ as being obtained from~$Y'\times I$ by attaching~$2$-handles to the framed link~$\widetilde{L}'$ dual to~$\widetilde{L}$. In particular, the components of~$\widetilde{L}'$ bound the cores of the~$2$-handles.

Recall that~$H_1(Y';\Z[t^{\pm 1}])=0$ by Proposition~\ref{prop:Step2} and that~$H_2(B;\Z[t^{\pm 1}])=0$ by Proposition~\ref{thm:Step3}.
Let~$\Sigma_i$ denote a surface in~${Y'}^\infty$ with boundary~$\widetilde{K}'_i$, and let~$F_i$ be the surface in~$M$ formed by~$\Sigma_i$ capped with the core of the 2-handle attached along~$\widetilde{K}'_i$.
The proof that~$H_2(M;\Z[t^{\pm 1}])$ is freely generated by the~$[F_i]$ and that the equivariant intersection form~$\lambda_M$ is represented by the transposed equivariant linking matrix of~$\widetilde{L'}$ (which we showed above is~$Q$), is now routine; the details are expanded in Propositions~\ref{prop:BasisH2} and~\ref{thm:IntersectionForm} below.
\end{proof}

As promised, the section now concludes with a detailed proof of the claims.
Firstly in Construction~\ref{cons:BasisH2}, we give the detailed construction of the surfaces~$F_i$ that were mentioned in the proof outline.
Secondly, in Proposition~\ref{prop:BasisH2} we show that these surfaces lead to a basis of~$H_2(M;\Z[t^{\pm 1}])$.
Thirdly, in Proposition~\ref{thm:IntersectionForm} we conclude the proof of the first claim by showing that with respect to this basis,~$\lambda_M$ is represented by the transposed equivariant linking matrix of~$\widetilde{L'}$.
Finally, in Proposition~\ref{prop:step4}, we prove the second claim.

\begin{construction}
\label{cons:BasisH2}
For~$i=1,\ldots,n$, we define the closed surfaces~$F_i \subset -W^\infty \subset M^\infty$ that were mentioned in the outline.
As~$H_1(Y';\Z[t^{\pm 1}])=0$ (by Step 2), each component~$\widetilde{K}_i'$ of~$\widetilde{L}'$ bounds a surface~$\Sigma_i \subset {Y'}^\infty$.
Additionally, each~$\widetilde{K}'_i$ (considered in $Y' \times \lbrace 1 \rbrace$) bounds the core of one of the (lifted) 2-handles in the dual handle decomposition of~$-W$. Define the surface~$F_i \subset -W^\infty \subset M^\infty$ by taking the union of~$\Sigma_i$ with this core.
\end{construction}

%\begin{construction}
%\label{cons:BasisH2}
%For~$i=1,\ldots,n$, we define surfaces~$F_i \subset -W^\infty \subset M^\infty$ and~$F_i' \subset M^\infty$ as follows.
%\begin{enumerate}
%\item We start by building the surfaces~$F_i$.
%Since~$H_1(Y';\Z[t^{\pm 1}])=0$ (by Step 2), each component~$\widetilde{K}_i'$ of~$\widetilde{L}'$ bounds a surface~$\Sigma_i \subset {Y'}^\infty$.
%Additionally, each~$\widetilde{K}'_i$ (considered in $Y' \times \lbrace 1 \rbrace$) bounds the core of one of the 2-handles in this dual handle decomposition of~$-W$. Define the surface~$F_i \subset -W^\infty \subset M^\infty$ by taking the union of~$\Sigma_i$ with this core.
%\item The construction of the surfaces~$F_i'$ is analogous but we push the interiors of the~$\Sigma_i$ into~$B^\infty$ before capping them off with the cores.
%%Namely, we let~${Y'}^\infty \times [0,\varepsilon]$ be a collar of~${Y'}^\infty=\partial B^\infty$ in~$B^\infty$, and let~$\Sigma_i' \subset B^\infty$ be the properly embedded surface obtained as
%%$$ \Sigma_i'=(\partial \Sigma_i \times [0,\varepsilon]) \cup (\Sigma_i \times \lbrace \varepsilon \rbrace)$$
%%i.e.\ by pushing the boundary of~$\Sigma_i$ radially into the collar and capping off the result with a parallel copy of~$\Sigma_i$.
%Define~$F_i' \subset B^\infty$ by taking the union of~$\Sigma_i'$ with the previously mentioned core.
%\end{enumerate}
%Note that by construction, for~$i=1,\ldots,n$, the surface~$F_i$ is isotopic to the surface~$F_i'$ in~$M^\infty$.
%\end{construction}

The next proposition shows that the surfaces~$F_i'$ give a basis for~$H_2(M;\Z[t^{\pm 1}])$.
It is with respect to this basis that we will calculate~$\lambda_M$ in Proposition~\ref{thm:IntersectionForm} below.
%AC: Recall from Remark~\ref{rem:CoefficientSystemY'} that $\varphi$ extends over the trace..

\begin{proposition}
\label{prop:BasisH2}
The following isomorphisms hold:
\begin{align*}
H_2(-W;\Z[t^{\pm 1}])=\Z \oplus \bigoplus_{i=1}^n \Z[t^{\pm 1}] [F_i], \ \ \ \text{ and } \ \ \  H_2(M;\Z[t^{\pm 1}])=\bigoplus_{i=1}^n \Z[t^{\pm 1}] [F_i].
\end{align*}
\end{proposition}
\begin{proof}
%AC: Twice Mayer-Vietoris. Calulate for trace~$W=Y' \times [0,1] \cup \text{2h}$ and then for~$M=W \cup B$.
These follow by standard arguments using Mayer-Vietoris, which we outline now.

The first equality follows from the observation that~$-W^\infty$ is obtained from~${Y'}^\infty\times [0,1]$ by attaching the dual 2-handles to the $h^{(2)}_i$. Morally, since~$H_1(Y';\Z[t^{\pm 1}])=0$ (Step 2), each dual 2-handle contributes a free generator.
 The additional~$\Z$ summand comes from~$H_2(Y' \times [0,1];\Z[t^{\pm 1}]) \cong\mathbb{Z}$.
  More formally, one applies Mayer-Vietoris with~$\Z[t^{\pm 1}]$-coefficients to the decomposition of $W$ as the union of $Y' \times [0,1]$ with the dual 2-handles, which since the dual 2-handles are contractible and $H_1(Y';\Z[t^{\pm 1}])=0$ yields  the short exact sequence:
$$  0 \to H_2(Y' \times [0,1];\Z[t^{\pm 1}]) \to  H_2(-W;\Z[t^{\pm 1}]) \xrightarrow{\partial} H_1(\overline{\nu}(L');\Z[t^{\pm 1}]) \to 0.$$
%
%AC: (recall from Definition \ref{def:ParallelLongitude} that~$\overline{\nu}_{\pi'}(\widetilde{L'})$ refers to the normal bundle of~$\widetilde{L'}$ parametrized by the framing.)
Since $\varphi'([L']) =0$, $H_1(\overline{\nu}(L');\Z[t^{\pm 1}]) \cong \bigoplus_{i=1}^n \Z[t^{\pm 1}]$, generated by the $[K_i']$. Mapping each~$[K_i']$ to~$[F_i]$ determines a splitting.

%AC: %This exact sequence is split because the rightmost~$\Z[t^{\pm 1}]$-module is free; in fact~$H_1(\widetilde{L}' \times D^2;\Z[t^{\pm 1}])=\Z[t^{\pm 1}]^n$ is freely generated by the components~$\widetilde{K}_i'$ for~$i=1,\ldots,n$.
%Use~$s$ to denote a splitting of~$\partial$.
%By definition of the connecting homomorphism~$\partial$ in the Mayer-Vietoris exact sequence, we deduce that closed surfaces representing the~$s([\widetilde{K}_i'])$ are obtained by making the~$\widetilde{K}_i$ bound surfaces both in~${Y'}^\infty \times [0,1]$ and in the~$2$-handles and then gluing the two resulting surfaces together along~$\widetilde{K}_i'$ for~$i=1,\ldots,n$.
%Since the surfaces~$F_i$ described in Construction~\ref{cons:BasisH2} are obtained precisely by this method, we deduce from the previous split exact sequence that
%$$ H_2(W;\Z[t^{\pm 1}]) \cong H_2(Y';\Z[t^{\pm 1}]) \oplus s(H_1(\widetilde{L}' \times D^2;\Z[t^{\pm 1}])) \cong \Z \oplus \bigoplus_{i=1}^n \Z[t^{\pm 1}] [F_i].$$
%Next, we move on to the second homology of~$M^\infty$.
For the second equality, note that since~$B$ is a homotopy circle and $g_* \colon \pi_1(B) \to \Z$ is an isomorphism, $B$ has no (reduced)~$\Z[t^{\pm 1}]$-homology.
The Mayer-Vietoris exact sequence associated to the decomposition~$M=-W \cup_{Y' \times \{1\}} B$ therefore yields the short exact sequence
$$  0 \to H_2(Y';\Z[t^{\pm 1}]) \to  H_2(-W;\Z[t^{\pm 1}]) \to H_2(M;\Z[t^{\pm 1}]) \to 0.$$
Appealing to our computation of~$H_2(-W;\Z[t^{\pm 1}])$, we deduce that~$H_2(M;\Z[t^{\pm 1}])$ is freely generated by the~$[F_i]$.
\end{proof}

%Since the~$F_i'$ are isotopic to the~$F_i$ in~$M^\infty$, we can just as well use the~$[F_i']$ as generators for~$H_2(M;\Z[t^{\pm 1}])$.

%We will use the following lemma during the proof of Proposition~\ref{thm:IntersectionForm} below: since $H_2(B;\Z[t^{\pm 1}])=0$, the equivariant intersection form of two surfaces can be calculated using any other surfaces with the same boundary.

%\begin{lemma}\label{lem:intonhomology}
%For surfaces~$\Sigma$ and~$\Sigma'$ embedded in~$B^\infty$ with boundary a common knot in ${Y'}^{\infty}$, and a properly embedded surface~$G$ in~$B^\infty$ with boundary disjoint from $\partial \Sigma = \partial \Sigma'$, we have that
%$$ \Sigma \cdot_{\infty,B} G=\Sigma'\cdot_{\infty,B} G.$$
%\end{lemma}
%\begin{proof}
%Observe that~$\Sigma \cdot_{\infty,B}G=\Sigma'\cdot_{\infty,B} G$ if and only if~$(\Sigma\cup -\Sigma')\cdot_{\infty,B} G=0$.
%Using both Remark~\ref{rem:EquivariantIntersections} and the fact that $(\Sigma\cup -\Sigma')$ determines a class $[(\Sigma\cup -\Sigma')]$ in $H_2(B;\Z[t^{\pm 1}])=0$, we get $(\Sigma\cup -\Sigma')\cdot_{\infty,B} G=\lambda^\partial_B ([G],[\Sigma \cup -\Sigma'])=0$, as required.
%\end{proof}

Now we prove the first claim of the previously mentioned outline.%\footnote{\purple{AC: I deleted the lemma with the $G$s.}}

\begin{proposition}
\label{thm:IntersectionForm}
With respect to the basis of~$H_2(M;\Z[t^{\pm 1}])$ given by the~$[F_1],\ldots,[F_n]$, the equivariant intersection form~$\lambda_M$ of~$M$ is given by the transposed equivariant linking matrix of the framed link~$\widetilde{L}'$ dual to~$\widetilde{L}$.
\end{proposition}

\begin{proof}
Recall from Construction~\ref{cons:BasisH2} that for $i=1,\ldots,n$, the surface $F_i \subset -W^\infty \subset M^\infty$ was obtained as the union of a surface $\Sigma_i \subset {Y'}^\infty$ whose boundary is~$\widetilde{K}_i'$ with the core of a (lifted) $2$-handle in the dual handle decomposition of $W$.
For $i=1,\ldots,n$, define $F_i'$ to be a surface isotopic to $F_i$ obtained by pushing the interior of $\Sigma_i$ into $B^\infty$. Let $\Sigma_i'$ be such a push-in.
Since $F_i$ and $F_i'$ are isotopic for every $i=1,\dots,n$, we can use the $F_i'$ to calculate~$\lambda_M$.
Fix real numbers $0<s_1 < \cdots < s_n <1$.  We model $\Sigma_i'$ in the coordinates of a collar neighborhood $\partial B \times [0,1]$ as \[\Sigma_i' :=  (\partial \Sigma_i \times [0,s_i]) \cup (\Sigma_i \times \{s_i\}).\]

%For~$i=1,\ldots,n$, the surface~$F_i'$ was obtained as the union of the core of the~$2$-handle with a push-in~$\Sigma_i' \subset B^\infty$ of a surface~$\Sigma_i \subset {Y'}^\infty$ with boundary~$\widetilde{K}_i'$.
We start by calculating the equivariant intersection form~$ \lambda_M([F_i'],[F_j'])$ for~$i \neq j$.
Since the aforementioned cores of the dual 2-handles are pairwise disjoint, we obtain
$$\overline{\lambda_M([F_i'],[F_j'])}=F_i'\cdot_{\infty,M} F_j'=\Sigma_i'\cdot_{\infty, B} \Sigma_j'.$$
Recall that we use~$A_{\widetilde{L}'}$ to be the linking matrix of the framed link~$L'$.
It therefore remains to show that~$\Sigma_i'\cdot_{\infty, B} \Sigma_j'=(A_{\widetilde{L}'})_{ij}$.
Assume without loss of generality that $i>j$, and so $s_i > s_j$.
Also note that $\partial \Sigma_i \cap \partial \Sigma_j = \emptyset$.
By inspecting the locations of the intersections, it follows that
\[\Sigma_i'\cdot_{\infty,B}\Sigma_j' = (\partial\Sigma_i \times [0,s_i]) \cdot_{\infty,B} (\Sigma_j \times \{s_j\}) = \partial \Sigma_i\cdot_{\infty,\partial B}\Sigma_j=\ell k_{\Q(t)}(\widetilde{K}_i',\widetilde{K}_j'),\]
where the last equality uses the definition of the equivariant linking number in~$\partial B=Y'$.
For~$i \neq j$, we have therefore proved that
$$\lambda_M([F_j'],[F_i'])=\Sigma_i'\cdot_{\infty, B} \Sigma_j'=\ell k_{\Q(t)}(\widetilde{K}_i',\widetilde{K}'_j).$$
It remains to prove that~$\lambda_M([F_i'],[F_i'])=(A_{\widetilde{L}'})_{ii}$.
%%Here it is unncessary to transpose because Hermitian means $\overline{p}=p$ on the diagonal.
%%%
%As above, since there are no intersections$\Sigma_i'\cdot_\infty\Sigma_i'=(A_{\widetilde{L}'})_{ii}$.AC: More accurately, it remains to prove that~$\lambda_M([F_i'],[F_i'])=(A_{\widetilde{L}'})_{ii}$. Gotta to make the perturbation first before writing~$\Sigma_i'\cdot_{\infty, B} \Sigma_j'$.
% OLD ARXIV: Recall that by definition,  the dual knot~$\widetilde{K}_i'$  is framed by~$\pi_i'=-\mu_{\widetilde{K}_i}$, which means that~$(A_{\widetilde{L}'})_{ii}=\ell k_{\Q(t)}(\widetilde{K}_i',\pi_i')$.
By definition of the dual framed knot~$\widetilde{K}_i'$,  we have~$(A_{\widetilde{L}'})_{ii}=\ell k_{\Q(t)}(\widetilde{K}_i',\pi_i')$, where~$\pi_i'$ denotes the framing curve of $\widetilde{K}_i'$.

Perform a small push-off of the surface~$\Sigma_i' \subset B^\infty$ to obtain a surface~$\Sigma_i'' \subset B^\infty$ isotopic to~$\Sigma_i' \subset B^\infty$ with boundary~$\partial \Sigma_i''=\pi_i'$.
Cap off $\Sigma_i''$ with a parallel disjoint copy of the cocore of the 2-handle, yielding a closed surface $F_i''$ that is isotopic to $F_i'$, and such that all the intersections between the two occur between $\Sigma_i'$ and $\Sigma_i''$.
As in the~$i\neq j$ case, we then have
%~$\lambda_M([F_i'],[F_i'])=\Sigma_i'\cdot_\infty\Sigma_i''$ and
$$\lambda_M([F_i'],[F_i'])=\Sigma_i'\cdot_{\infty,B}\Sigma_i''=\ell k_{\Q(t)}(\widetilde{K}_i',\pi_i').$$
We have therefore shown that the equivariant intersection form of~$M$ is represented by the transposed linking matrix~$A_{\widetilde{L}'}^T$ and this concludes the proof of the proposition.
\end{proof}

Finally, we prove the second claim of our outline, thus completing step 4.

\begin{proposition}\label{prop:step4}
Let~$Y$ be a~$3$-manifold with an epimorphism~$\varphi \colon \pi_1(Y) \twoheadrightarrow \Z$ whose Alexander module is torsion, and let~$(H,\lambda)$ be a nondegenerate Hermitian form presenting~$Y$.
If~$b \in \Iso(\partial \lambda,\unaryminus\Bl_Y)/\Aut(\lambda)$ is an isometry, then there is a~$\Z$-manifold~$M$ with equivariant intersection form~$\lambda_M\cong \lambda$,  boundary~$Y$ and with~$b_M=b$.
\end{proposition}

\begin{proof}
Let~$M$ be the 4-manifold with boundary~$Y$ constructed as described above.
The manifold~$M=-W \cup_{Y'} B$ comes with a homeomorphism~$g \colon \partial M \cong Y$,  because~$-W$ is obtained from~$Y \times [0,1]$ by adding~$2$-handles.
We already explained why~$M$ has intersection form isometric to~$\lambda$ but we now make the isometry more explicit.
Define an isomorphism~$F \colon H \to H_2(M;\Z[t^{\pm 1}])$ by mapping~$x_i$ to~$[F_i]$, where the~$F_i \subset M^\infty$ are the surfaces built in Construction~\ref{cons:BasisH2}.
This is an isometry because, by combining Proposition~\ref{thm:IntersectionForm} with Lemma~\ref{lem:InverseMatrix}, we get
$$\lambda_M([F_i],[F_j])=(A_{\widetilde{L}'})_{ji}=-(A_{\widetilde{L}}^{-1})_{ji}
=Q_{ij}=\lambda(x_i,x_j).$$
We now check that~$b_M = b$ by proving that~$b=g_* \circ D_M \circ \partial F$.
This amounts to proving that the bottom square of the following diagram commutes (we refer to Construction~\ref{cons:PresentationAssociatedToManifold}  if a refresher on the notation is needed):
$$
\xymatrix@R0.5cm@C1.7cm{
H^* \ar[r]^-{F^{-*},\cong} \ar@{->>}[d]^-{\operatorname{proj}}&H_2(M;\Z[t^{\pm 1}])^* \ar[r]^-{\operatorname{PD} \circ \operatorname{ev}^{-1},\cong}\ar@{->>}[d]^-{\operatorname{proj}}&H_2(M,\partial M;\Z[t^{\pm 1}])\ar@{->>}[d]^-{\delta_M} \\
\coker(\widehat{\lambda}) \ar[r]^-{\partial F,\cong}\ar[d]^-{=}&  \coker(\widehat{\lambda}_M) \ar[r]^-{D_M,\cong}& H_1(\partial M;\Z[t^{\pm 1}])\ar[d]^-{g_*,\cong} \\
\coker(\widehat{\lambda}) \ar[rr]^-{b,\cong}&&H_1(Y;\Z[t^{\pm 1}]).
}
$$
The top squares of this diagram commute by definition of~$\partial F$ and~$D_M$.
Since the top vertical maps are surjective, the commutativity of the bottom square is now equivalent to the commutativity of the outer square.
It therefore remains to prove that~$g_* \circ \delta_M \circ (\operatorname{PD} \circ \operatorname{ev}^{-1}) \circ F^{-*}=\pi$; (recall that by definition~$\pi=b \circ \operatorname{proj}$).
%~\cite[Remark 3.3]{ConwayPowell}
In fact, it suffices to prove this on the~$x_i^*$ as they form a basis of~$H^*$.
Writing~$c_i$ for the core of the 2-handles attached to~$Y \times [0,1]$,  union a product of their attaching circles with $[0,1]$ in $Y \times [0,1]$, note that the $c_i$ intersects $F_j$ in $\delta_{ij}$ points, since $F_j$ is built from a surface in ${Y'}^{\infty}$ union the cocore of the $j$th 2-handle.
%So that it has boundary in $Y \times \{0 \}$
We have
$$g_* \circ \delta_M \circ (\operatorname{PD} \circ \operatorname{ev}^{-1}) \circ F^{-*}(x_i^*)
=g_* \circ  \delta_M \circ (\operatorname{PD} \circ \operatorname{ev}^{-1})([F_i]^*)
=g_* \circ  \delta_M ([\widetilde{c}_i])
=[\widetilde{K}_i]=\pi(x_i^*).
$$
%AC: F_i is built using the core of the dual handle while c_i is the core of the original handle so over all they will be Poincaré dual.
Here we use successively the definition of~$F$, the geometric interpretation of~$\operatorname{PD} \circ \operatorname{ev}^{-1}$, the fact that~$\widetilde{g}(\partial \widetilde{c}_i)=\widetilde{K}_i$ and the definition of the~$\widetilde{K}_i$.
Therefore the outer square commutes as asserted.
This concludes the proof that~$b=g_* \circ D_M \circ \partial F$ and therefore~$b_M = b$, as required.
\end{proof}

\subsection{Step 5: fixing the Kirby-Siebenmann invariant and concluding}
\label{sub:Step5}

The conclusion of Theorem~\ref{thm:MainTechnical} will follow promptly from Proposition~\ref{prop:step4} once we recall how, in the odd case,  it is possible to modify the Kirby-Siebenmann invariant of a given $4$-manifold with fundamental group~$\Z$.
This is achieved using the star construction, a construction which we now recall following~\cite{FreedmanQuinn} and~\cite{StongRealization}.
In what follows, $*\C P^2$ denotes the Chern manifold, i.e.\ the unique simply-connected topological $4$-manifold homotopy equivalent to~$\C P^2$ but with~$\ks(*\C P^2)=1$.
\medbreak

Let~$M$ be a topological~$4$-manifold with (potentially empty) boundary,  good fundamental group~$\pi$ and such that the second Stiefel-Whitney class of the universal cover $w_2(\wt{M})$ is nontrivial.
There is a~$4$-manifold~$*M$, called the \emph{star partner of~$M$} that is rel.\ boundary homotopy equivalent to~$M$ but has the opposite Kirby-Siebenmann invariant from that of~$M$~\cite[Theorem 10.3~(1)]{FreedmanQuinn}.  See~\cite{teichner-star} or~\cite[Propostion~5.8]{KPR-counterexamples} for a more general condition under which a star partner exists.

\begin{remark}
For fundamental group $\Z$, every non-spin 4-manifold has $w_2(\wt{M}) \neq 0$.  To see this, we use the exact sequence
\[0 \to H^2(B\pi;\Z/2) \to H^2(M;\Z/2) \xrightarrow{p^*} H^2(\wt{M};\Z/2)^\pi,\]
where $\pi := \pi_1(M)$. This can be deduced from the Leray-Serre spectral sequence for the fibration $\wt{M} \to M \to B\pi$; see e.g.\ \cite[Lemma~3.17]{KLPT}.  For $\pi =\Z$ the first term vanishes, so $p^*$ is injective. By naturality, $p^*(w_2(M)) = w_2(\wt{M})$, so $w_2(M) \neq 0$ implies $w_2(\wt{M}) \neq 0$ as desired.  It follows that for a non-spin 4-manifold $M$ with fundamental group $\Z$, \cite[Theorem 10.3]{FreedmanQuinn} applies and there is a star partner.
\end{remark}

To describe $*M$, consider the~$4$-manifold~$W:=M \# (*\C P^2)$ and note that the inclusions $M \hookrightarrow W$ and~$* \C P^2 \hookrightarrow W$ induce a splitting
\begin{equation}
\label{eq:StarSplitting}
\pi_2(M) \oplus (\pi_2(*\C P^2) \otimes_\Z \Z[\pi]) \xrightarrow{\cong}  \pi_2(W).
 \end{equation}
By~\cite[Theorem 10.3~(1)]{FreedmanQuinn} (cf.\ \cite[Proposition~5.8]{KPR-counterexamples})  there exists a $4$-manifold~$*M$ and an orientation-preserving homeomorphism
%MP: Email : take interior # with *CP^2, realise the generator of pi_2 by  an embedding (there is an odd sphere which does not intersect it, so it is not s-characteristic, so can be embedded). Then V # *CP^2 cong *V #  CP^2 for some *V, and that's the one with the other ks.  This whole construction took place in the interior, so boundary parametrisations can be assumed well behaved.
%AC: This all looks pretty orientation-preserving (in fact since you get identity on the boundary; it's basically forced).
$$ h \colon W \xrightarrow{\cong} *M \# \C P^2$$
that respects the splitting on $\pi_2$ displayed in~\eqref{eq:StarSplitting}.
The star partner $*M$ is also unique up to homeomorphism, by \cite[Corollary~1.2]{StongUniqueness}.

To be more precise about the condition on $h$,  let $\iota \colon \pi_2(*\C P^2) \otimes_\Z \Z[\pi] \to   \pi_2(M \# (*\C P^2))=\pi_2(W)$ denotes the split isometric injection induced by the zigzag $*\C P^2 \leftarrow *\C P^2 \sm \mathring{D}^4 \rightarrow W$, and let $\operatorname{incl}_* \colon \pi_2(\C P^2) \to \pi_2(*M \# \C P^2)$ be defined similarly.
%fix a basis~$[\alpha]$ of $\pi_2(\C P^2)=\Z$ and $[\alpha']$ of~$\pi_2(*\C P^2)=\Z$ so that
%if $\iota \colon \pi_2(*\C P^2) \otimes_\Z \Z[\pi] \to   \pi_2(M \# (*\C P^2))=\pi_2(W)$ denotes the split isometric injection induced by the inclusion,
Then we say that $h$ \emph{respects the splitting on $\pi_2$} if for some isomorphism $f \colon \pi_2(*\C P^2) \xrightarrow{\cong} \pi_2(\C P^2)$, the following diagram commutes
\begin{equation*}
%\label{eq:StarDiagram}
\xymatrix{
\pi_2(*\C P^2) \otimes_\Z \Z[\pi] \ar@{^{(}->}[r]^-\iota \ar[d]_{f \otimes \id}^{\cong} & \pi_2(W)  \ar[d]^{h_*}_\cong \\
\pi_2(\C P^2) \otimes_\Z \Z[\pi]  \ar@{^{(}->}[r]^-{\operatorname{incl}_*} & \pi_2(*M \# \C P^2).
}
\end{equation*}
%where the bottom horizontal map is also a split isometric injection induced by the inclusion, and the left vertical map is an isomorphism that takes $[\alpha']$ to $\MP{\pm} [\alpha]$, induced by a homotopy equivalence $*\C P^2 \simeq \C P^2$.
%AC: The point is that if you give yourself the map from lower left to upper right, then FQ tell you that h induces that h induces it.
Since both horizontal maps in this diagram are split, this implies that $h_*$ induces an isomorphism~$g \colon \pi_2(M) \xrightarrow{\cong} \pi_2(*M)$, and so $h_*$ splits as follows:
\[h_* = (g_*, f_* \otimes \id) \colon \pi_2(M) \oplus (\pi_2(*\C P^2) \otimes_\Z \Z[\pi]) \xrightarrow{\cong} \pi_2(*M) \oplus (\pi_2(\C P^2) \otimes_\Z \Z[\pi]).\]
%and it is in this sense that~$h$ respects the splitting displayed in~\eqref{eq:StarSplitting}.

We recall that~$M$ and~$*M$ are orientation-preserving homotopy equivalent rel.\ boundary.
This will ensure that their automorphism invariants agree.
The argument is due to Stong~\cite[Section~2]{StongUniqueness}, and a proof can also be found in \cite[Lemma~5.7]{KPR-counterexamples}.

\begin{proposition}
\label{prop:StarHomotopyEquivalence}
If~$M$ is a topological~$4$-manifold with boundary, good fundamental group~$\pi$ and whose universal cover has nontrivial second Stiefel-Whitney class,  then~$M$ is orientation-preserving homotopy equivalent rel.\ boundary to its star partner~$*M$.
\end{proposition}

%\begin{proof}
%Set once again~$W:=M \# * \C P^2$  and use~$\alpha$ and~$\alpha'$ to denote spheres whose homotopy classes generate~$\pi_2(\C P^2)$ and~$\pi_2(*\C P^2)$.
%As we explained above, there exists a homeomorphism
%$$h \colon W \xrightarrow{\cong} *M \# \C P^2$$
%such that $h([\alpha'])=\alpha$ (here, we suppressed the inclusion maps from the notation).
%%AC:  i.e.\ such that the spheres $h \circ \alpha'$ and $\alpha$ are homotopic.
%Observe that the~$4$-complexes~$\C P^2 \cup_\alpha e^3$ and~$\C P^2 \cup_{\alpha'} e^3$ obtained by attaching~$3$-cells along $\alpha$ and $\alpha'$ have the homotopy type of $S^4$.  To see this take the standard degree one map $\C P^2 \to S^4$. The image of the 2-skeleton is null-homotopic and so the map extends to a map $\C P^2 \cup_\alpha e^3 \to S^4$. This is homotopy equivalence by the Hurewicz and Whitehead theorems.
%
%Thinking respectively of~$M$ and~$*M$ as~$M \# S^4$ and~$*M \# S^4$, the required homotopy equivalence is then obtained as the composition
%$$ M
%%\cong M \# S^4
%\simeq M \# (*\C P^2 \cup_{\alpha'} e^3)
%\cong W \cup_{ \alpha'}  e^3
%\xrightarrow{h \cup \id,\cong} (*M \#  \C P^2) \cup_{\alpha} e^3
% \cong *M \# (\C P^2 \cup_{\alpha} e^3)
%%  \simeq  *M \cong S^4
%   \simeq *M.$$
%   This homotopy equivalence is rel.\ boundary because these are interior connected sums.
%\end{proof}

We are ready to prove Theorem~\ref{thm:MainTechnical}, whose statement we now recall for the reader's convenience.
Let~$Y$ be a~$3$-manifold with an epimorphism~$\pi_1(Y) \twoheadrightarrow \Z$ whose Alexander module is torsion, and let~$(H,\lambda)$ be a form presenting~$Y$.
If~$b \in \Iso(\partial \lambda,\unaryminus\Bl_Y)/\Aut(\lambda)$ is an isometry, then there is a~$\Z$-manifold~$M$ with equivariant intersection form~$\lambda_M$,  boundary~$Y$ and with~$b_M=b$.
If the form is odd, then~$M$ can be chosen to have either~$\ks(M)=0$ or~$\ks(M)=1$.
We now conclude the proof of this theorem.
\begin{proof}[Proof of Theorem~\ref{thm:MainTechnical}]
In Proposition~\ref{prop:step4}, we proved the existence of a $\Z$~manifold~$M$ with equivariant intersection form~$\lambda_M$, boundary~$Y$ and with~$b_M = b$.
It remains to show that if~$\lambda$ is odd, then~$M$ can be chosen to have either~$\ks(M)=0$ or~$\ks(M)=1$.
This is possible by using the star partner $*M$ of $M$.
Indeed Proposition~\ref{prop:StarHomotopyEquivalence} implies that $M$ and $*M$ are homotopy equivalent rel.\ boundary and therefore Remark~\ref{rem:HomotopyEquivalence} ensures that~$b_{*M} =b_{M}$ is unchanged.
\end{proof}

\subsection{An example}
\label{sub:Example}

Remark~\ref{rem:KSProof} shows that if $M_0$ and $M_1$ are spin~$4$-manifolds with $\pi_1(M_i) \cong \Z$,  boundary homeomorphic to~$(Y,\varphi)$,  isometric equivariant intersection form, and the same automorphism invariant, then their Kirby-Siebenmann invariants agree.
The next proposition  shows that the condition on the automorphism invariant is necessary.  After the proof, we offer an extended example to illustrate the proof of Theorem \ref{thm:MainTechnicalIntro} and to show that it is possible to work with the automorphism invariants and the $\Q(t)$-valued linking numbers explicitly.

\begin{proposition}
\label{prop:KSSpin}
There are two spin~$4$-manifolds~$M_0$ and~$M_1$ with $\pi_1 \cong \Z$, equivariant intersection form isometric to~$\lambda:= (-8)$ and boundary homeomorphic to~$Y := \unaryminus L(8,1) \# (S^1 \times S^2)$ that are distinguished both by their Kirby-Siebenmann invariants and their automorphism invariants.
\end{proposition}

\begin{proof}
The manifolds~$M_0$ and~$M_1$ are obtained by boundary connect summing~$S^1 \times D^3$ to simply-connected~$4$-manifolds~$V_0$ and~$V_1$ that we now describe.
Up to homeomorphism, there are two simply-connected~$4$-manifolds~$V_0$  and~$V_1$ with intersection form~$\lambda' = (-8) \colon \Z \times \Z \to \Z$, and boundary homeomorphic to the lens space~$Y' := \unaryminus L(8,1)$. They are distinguished by Boyer's simply-connected version of the automorphism invariant~\cite[Corollary E]{BoyerRealization}.
We construct them explicitly and show that~$\ks(V_0) \neq \ks(V_1)$.

The~$(-8)$-trace on the unknot,~$V_0:=X_{-8}(U)$, gives the first of these~$4$-manifolds.
%\color{teal}
Towards describing~$V_1$,  first note that from ~$\unaryminus L(8,1)$ one can obtain  the integer homology sphere~$S_{+1}^3(T_{2,3})$ by a Dehn surgery along the framed knot~$K_1$ illustrated in Figure~\ref{fig:surgeryinlens}.  Note also that~$S_{+1}^3(T_{2,3})$ bounds a contractible topological 4-manifold~$C$. We can now build~$\unaryminus V_1$ by beginning with~$\unaryminus L(8,1)\times I$, attaching a~$+1$ framed 2-handle along~$K_1$, and capping off with~$\unaryminus C$.  The resulting manifold~$\unaryminus V_1$ has~$\partial (\unaryminus V_1)=L(8,1)$, so~$\partial V_1=\unaryminus L(8,1)$ as desired.

%To describe~$V_1$,  first note that ~$S_{+1}^3(T_{2,3})$,  the integer homology sphere obtained by~$(+1)$-surgery on the right hand trefoil~$T_{2,3}$, can be obtained by a single surgery on~$\unaryminus L(8,1)$ along the framed knot~$K_1$ illustrated in Figure~\ref{fig:surgeryinlens}.
 %The manifold~$V_1$ is now obtained capping off the trace of this surgery with a contractible manifold~$C$ with boundary~$S_{+1}^3(T_{2,3})$.

\begin{figure}[!htbp]
\center
\begin{overpic}[width=0.5\textwidth,tics=10]{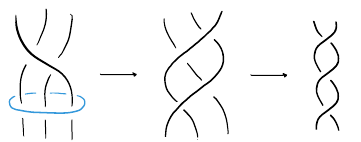}
 \put (22,5) {\textcolor{teal}{$-1$}}
  \put (19,21) {$-8$}
    \put (65,5) {$1$}
        \put (95,5) {$1$}
\end{overpic}
%\def\svgwidth{.4\linewidth}\input{Surgeryinlenscropped.pdf}
%%AC: This code doesn't work for me.
%%AC: Note that in other cases, you must have modified the file type somehow, see e.g. the code ``\def\svgwidth{.5\linewidth}\input{Fuckjprime.pdf_tex}"
%%Btw, I wonder if we should now change the name of Fuckjprime.pdf.
\caption{Peforming $-1$ surgery on the blue knot $K_1$ in the lens space $L(-8,1)$ yields the 3-manifold obtained by $+1$ surgery on the right handed trefoil in $S^3$.
Each frame of the figure should be imagined to be vertically braid closed.  The first homeomorphism indicated is a Rolfsen twist, the second is an isotopy in $S^3$. } %not an isotopy of braids
\label{fig:surgeryinlens}
\end{figure}

 The manifolds~$V_0$ and~$V_1$ are simply-connected, spin,  have boundary homeomorphic to~$\unaryminus L(8,1)$,  and intersection form isometric to~$(-8)$.
 We have that~$\ks(V_0)=0$ (because~$V_0$ is smooth),  whereas~$\ks(V_1)=\ks(C)=\mu(S_{+1}^3(T_{2,3}))=\operatorname{Arf}(T_{2,3})=1$.
 Here~$\mu$ denotes the Rochlin invariant and the relation between~$\ks$ and~$\mu$ is due to Gonz\'ales-Acu\~{n}a~\cite{GonzalezAcuna}.

 The manifolds~$M_0$ and~$M_1$ are now obtained by setting
$$M_0:=V_0 \natural (S^1 \times D^3) \quad \text{ and} \quad  M_1:=V_1 \natural (S^1 \times D^3).$$
  The manifolds~$M_0$ and~$M_1$ have~$\pi_1(M_i) \cong \Z$,  boundary homeomorphic to~$Y = \unaryminus L(8,1) \# (S^1 \times S^2)$,  and equivariant intersection form isometric to~$(-8) \colon \Z[t^{\pm 1}] \times \Z[t^{\pm 1}] \to \Z[t^{\pm 1}]$.
The additivity of the Kirby-Siebenmann invariant implies that~$\ks(M_0)=\ks(V_0)=0$ whereas~$\ks(M_1)=\ks(V_1)=1.$
The manifolds must have distinct automorphism invariants,  since otherwise by the classification (Theorem~\ref{thm:Classification}) they would be homeomorphic and hence would have the same Kirby-Siebenmann invariants.
\end{proof}

\begin{example}
\label{ex:ExampleKS}
To provide an explicit example of our realisation procedure from the proof of Theorem~\ref{thm:MainTechnicalIntro}, we describe how the manifolds~$M_0$ and~$M_1$ realise two distinct, explicit automorphism invariants.

Fix a model of~$Y := \unaryminus L(8,1) \# (S^1 \times S^2)$ as surgery on a $2$-component unlink $L_1 \cup L_2$ with framings~$(-8,0)$.
Consider the epimorphism~$\varphi \colon \pi_1(Y) \cong \Z_8 * \Z \to \Z$ given by sending the meridian~$\mu_{L_1}$ of~$L_1$ to~$0$ and the meridian~$\mu_{L_2}$ to~$1$.
Fix a lift~$\wt{\mu}_{L_1}$ of~$\mu_{L_1}$ to the infinite cyclic cover and note that it generates~$H_1(Y;\Z[t^{\pm 1}]) \cong  \Z[t^{\pm 1}]/(8)$ and satisfies~$\Bl_Y(\wt{\mu}_{L_1},\wt{\mu}_{L_1})=1/8.$
One way to see this latter equality is to use the calculation of the linking form of lens spaces.
%L(8,1) has intersection form -1/8. L(-8,1) has intersection form 1/8.

A verification shows that~$Y$ is presented by the Hermitian form
\begin{align*}
\lambda \colon \Z[t^{ \pm 1}] \times \Z[t^{ \pm 1}] &\to \Q(t)/\Z[t^{\pm 1}] \\
&(x,y) \mapsto 8x\overline{y}.
\end{align*}
Note also that multiplication by~$3$ induces an isometry of~$\Bl_Y \cong -\partial \lambda$.
Using the notation from the proof of Step 1 in Section~\ref{sub:Plan}, we let~$x_1$ be a generator of~$\Z[t^{\pm 1}]$, and we let~$x_1^* \in \Z[t^{\pm 1}]^*$ be the dual generator.
In these bases, the matrix of~$\lambda$ is~$Q = (-8)$.
We therefore obtain two elements of~$\Iso(\partial \lambda,\unaryminus\Bl_{Y})$ by considering
%Sending~$1 \in \Z[t^{\pm 1}]/(8)$ to~$\wt{\mu}_{L_1}$
%determines an identification
\begin{align*}
&b_0 \colon \Z[t^{\pm 1}]/(8) \xrightarrow{\cong} H_1(Y;\Z[t^{\pm 1}]),\, [x_1^*] \mapsto [\wt{\mu}_{L_1}], \\
&b_1  \colon  \Z[t^{\pm 1}]/(8) \xrightarrow{\cong}  H_1(Y;\Z[t^{\pm 1}]),\, [x_1^*] \mapsto 3[\wt{\mu}_{L_1}].
\end{align*}
Since~$\Aut(\lambda)=\{ \pm t^k\}_{k \in \Z}$,  it follows that~$b_0$ and~$b_1$ remain distinct in~$\Iso(\partial \lambda,\unaryminus\Bl_{Y})/\Aut(\lambda).$
That they remain in  distinct orbits of the action of~$\Homeo^+_{\varphi}(Y)$ requires the following claim.

%gives an element of~$\Iso(\partial \lambda,\unaryminus\Bl_{Y})$, and since~$\Aut(\lambda)$ and~$\Homeo^+_{\varphi}(Y)$ act by~$\pm t^n$ only, it follows that~$b_0$ and~$b_1$ lie in different orbits of~$\Iso(\partial \lambda,\unaryminus\Bl_{Y})/\Aut(\lambda) \times \Homeo^+_{\varphi}(Y)$.

%There is an isomorphism~$\Aut(\lambda) \cong \{\pm t^n \mid n \in \Z\}$
\begin{claim*}
The group~$\Homeo^+_{\varphi}(Y)$ acts on~$H_1(Y;\Z[t^{\pm 1}])$ as follows: for each $\psi \in \Homeo^+_{\varphi}(Y)$, we have that $\psi \cdot x=\pm t^k x$ for some~$k \in \Z$.
\end{claim*}
%Note that~$\Aut(\lambda) \cong \{\pm t^n \mid n \in \Z\} \cong  \Homeo^+_{\varphi}(Y)$.  The first isomorphism is straightforward algebra, since~$\Z[t^{\pm 1}]$ has only trivial units. We argue that the  second isomorphism holds next.
\begin{proof}
In~\cite[Theorem~3]{deSa-Rourke} we find the statement that every automorphism of a connected sum of 3-manifolds is a composition of slides, permutations, and automorphisms of the factors. That article was an announcement, and the theorem is actually due to Hendriks-Laudenbach~\cite[~\S5,~Th\'eor\`eme]{Hendriks-Laudenbach}.
 For our purposes the statement in~\cite{deSa-Rourke} is easier to apply, which is why we mention it.
Permutations are irrelevant here since there is a unique irreducible factor. Sliding the~$\unaryminus L(8,1)$ factor around the generator of~$S^1 \times S^2$ exactly corresponds to an action by~$t^n$.
Sliding the handle sends a generator~$t \in \pi_1(S^1 \times S^2)$ to~$g\cdot t$ where~$g \in \Z_8$.
However it acts trivially on a generator of~$\pi_1(L(8,1))$ and hence acts trivially on~$H_1(Y;\Z[t^{\pm 1}])$.
%Note that a handle slide like this can include a Gluck twist, but the Gluck twist acts trivially on the fundamental group so can be safely ignored here.
 It remains to consider automorphisms of the irreducible factor, i.e.\ of~$L(8,1)$.  Bonahon \cite{Bonahon} proved that every element of~$\Homeo^+(L(8,1))$ acts by~$\pm 1$ on~$H_1(L(8,1))$, and hence such an element acts by~$\pm 1$ on~$H_1(Y;\Z[t^{\pm 1}])$.
 Combining these conclusions, we see that every homeomorphism~$f \in \Homeo^+_{\varphi}(Y)$ acts by~$\pm t^n$ on~$H_1(Y;\Z[t^{\pm 1}])$, for some sign and some~$n \in \Z$, as asserted.
\end{proof}

The claim implies that the isometries~$b_0$ and~$b_1$ determine distinct elements in the orbit set~$\Iso(\partial \lambda,\unaryminus\Bl_{Y})/\Aut(\lambda) \times \Homeo^+_{\varphi}(Y)$.
We will show that applying the realisation process of Theorem~\ref{thm:MainTechnical} to these elements results in~$M_0$ and~$M_1$ respectively.
Following the notation of Section~\ref{sub:Plan}, for~$i=0,1$, precompose~$b_i$ with the canonical projection~$\Z[t^{\pm 1}]^* \to \Z[t^{\pm 1}]/(8)$ to get the epimorphism
$$\varpi_i \colon \Z[t^{\pm 1}]^*  \xrightarrow{}\Z[t^{\pm 1}]/(8)  \xrightarrow{b_i} H_1(Y;\Z[t^{\pm 1}]$$
%For~$i=0,1$, let~$\widetilde{K}_i \subset Y^\infty$ be a framed knot representing~$b_i(1)$ and let~$K_i \subset Y$ be its projection down to~$Y$.
For~$i=0,1$, let~$\widetilde{K}_i \subset Y^\infty$ be a framed knot representing~$\varpi_i(x_1^*)$ and let~$K_i \subset Y$ be its projection down to~$Y$.
We can assume that~$\widetilde{K}_i \subset \unaryminus L(8,1) \subset  (S^2 \times \R) \#_{k \in \Z} t^k(\unaryminus L(8,1))=Y^\infty$.
Thinking of~$Y$ as the~$(-8,0)$-framed surgery on the unlink~$L_1 \cup L_2$,  one can arrange also for~$K_i$ to be disjoint from~$L_1 \cup L_2$.
Consider the~$3$-component link~$K_i \cup L_1 \cup L_2 \subset S^3$.
Note that~$K_i \cup L_2$ is split from~$L_1$, ~$\ell k (K_0,L_1)=1$ and~$\ell k(K_1,L_1)=3.$
%because of b_0,b_1.
When we refer to a framing of~$K_i$, it will be as a knot in~$S^3$.
%Using the notation from the proof of Step 1 (see e.g.\ Section~\ref{sub:Plan}), we let~$x_1$ be a generator of~$\Z[t^{\pm 1}]$, and we let~$x_1^* \in \Z[t^{\pm 1}]^*$ be the dual generator. We have~$Q = (-8)$.
Let~$\pi_{K_1}$ (resp.~$\pi_{K_0}$) be the~$(\unaryminus 1)$-parallel of~$K_1$ (resp.~$0$-parallel of~$K_0$),  and let~$\wt{\pi}_{\wt{K}_i}$ be a lift of~$\pi_{K_i}$ to~$Y^\infty$, which is a parallel of~$\wt{K}_i$ for~$i=0,1.$

The next claim carries out by hand the first step of the plan described in Section~\ref{sub:Plan}.

\begin{claim*}
For~$i=0,1$,  the knot~$\widetilde{K}_i  \subset Y^\infty$ represents the homology class~$\varpi_i(x_1^*)$, and the parallel~$\widetilde{\pi}_{\widetilde{K}_i}$ satisfies
$$ \ell k_{\Q(t)}(\widetilde{K}_i,\widetilde{\pi}_{\widetilde{K}_i})=\tmfrac{1}{8}.$$
In particular,~$\widetilde{K}_i$ has equivariant linking matrix~$A_{\widetilde{K}_i}=\left(\tmfrac{1}{8}\right)=\unaryminus \left(\tmfrac{1}{\unaryminus 8}\right)=-Q^{-T}$ for~$i=0,1.$
%Surgery on the framed knot~$K_1 \subseteq Y$, with framing~$-1$ gives the first step in our realisation procedure, in particular satisfies the requirements in Proposition~\ref{prop:Step1}. Hence ~$M_1$ has automorphism invariant ~$b_1$.
%Similarly ~$M_0$ has automorphism invariant~$b_0$.
\end{claim*}
\begin{proof}
The assertion concerning the homology class holds by construction and so we focus on the equivariant linking number calculation.
The proofs are similar for~$M_0$ and~$M_1$, so we give the most details for~$M_1$, since that is the more complicated case, and then we sketch the easier case of~$M_0$.
We will use the equation
\begin{equation}\label{eq-lk-no-eqn}
  [\wt{\pi}_{\wt{K}_1}] = \lk_{\Q(t)}(\wt{K}_1,\wt{\pi}_{\wt{K}_1}) [\mu_{\wt{K}_1}] \in H_1(Y \sm \nu(K_1);\Q(t))
\end{equation} from Definition~\ref{def:EquivariantLinking}.
The~$\Z$-cover~$Y^\infty$ of~$Y$ is~$(S^2 \times \R) \#_{k \in \Z} t^k (\unaryminus L(8,1))$, and there is no linking between curves in different~$L(8,1)$ summands.
 Thus it suffices to investigate the~$\Q$-valued linking number of~$K_1$ and~$\pi_{K_1}$ in~$Y' := \unaryminus L(8,1)$, and consider the result as an element of~$\Q(t)$. Formally speaking, we use an isomorphism
\[H_1(Y^\infty\sm \cup_{i\in \Z} t^i \cdot \nu(\wt{K}_1)) \cong H_1(Y' \sm \nu(K_1)) \otimes_{\Z} \Z[t^{\pm 1}],\]
and then tensor both sides further by~$-\otimes_{\Z[t^{\pm 1}]} \Q(t)$.
 We compute in the right hand side and translate to a conclusion about the left hand side.

Since~$Y':=\unaryminus L(8,1)=S^3_{(-8)}(L_1)$, the manifold~$Y'\sm \nu(K_1)$ is obtained from the exterior of the~$2$-component link~$L_1 \cup K_1 \subset S^3$ by Dehn filling~$L_1$ with surgery coefficient~$-8$.
Since~$\ell k(L_1,K_1)=3$, the homology is therefore
$$ H_1(Y' \sm \nu(K_1)) \cong \frac{\Z\langle\mu_{L_1} \rangle \oplus \Z \langle\mu_{K_1} \rangle}
{\langle -8\mu_{L_1} + 3 \mu_{K_1} \rangle} \cong \Z.$$
% \xrightarrow{\varphi,\cong} \Z.~$$
We now express~$[\pi_{K_1}]$ as a multiple of~$[\mu_{K_1}]$, as required to calculate the framing of~$K_1$.
Since~$\pi_{K_1}$ is a~$(-1)$-parallel of~$K_1$ we have~$[\pi_{K_1}] = 3[\mu_{L_1}] - [\mu_{K_1}]$.
One checks that~$\bsm 1 & 3 \\ -3 & -8 \esm \bsm-8 \\ 3 \esm = \bsm 1 \\ 0 \esm,$  so one can use the invertible matrix~$\bsm 1 & 3 \\ -3 & -8 \esm$ to change coordinates to the presentation
$$\Z \xrightarrow{\bsm 1\\ 0 \esm} \Z \oplus \Z \to H_1(Y' \sm \nu(K_1)) \to 0.$$
%%%%Mark's words
%One can check that~$\bsm 1 & 3 \\ -3 & -8 \esm \bsm-8 \\ 3 \esm = \bsm 1 \\ 0 \esm,$  so using~$\bsm 1 & 3 \\ -3 & -8 \esm$ to change coordinates to the presentation~$\Z \xrightarrow{( 1 , 0)^T} \Z \oplus \Z \to H_1(Y' \sm \nu(K_1);\Z) \to 0$,
In this presentation, we compute that
\begin{align*}
[\mu_{K_1}]&=\operatorname{proj}_2 \circ \bsm 1 & 3 \\ -3 & -8 \esm \bsm 0 \\ 1 \esm  = -8 \in \Z \cong H_1(Y' \sm \nu(K_1)), \\
[\pi_{K_1}]&=\operatorname{proj}_2 \circ \bsm 1 & 3 \\ -3 & -8 \esm \bsm 3 \\ -1 \esm  = -1 \in \Z \cong H_1(Y' \sm \nu(K_1)).
\end{align*}
Hence passing to the~$\Z$-cover, tensoring up to~$\Q(t)$ coefficients, and applying  \eqref{eq-lk-no-eqn}, we see that~~$-1 = \lk_{\Q(t)}(\wt{K}_1,\wt{\pi}_{\wt{K}_1}) \cdot (-8)$ so, as asserted
$$\lk_{\Q(t)}(\wt{K}_1,\wt{\pi}_{\wt{K}_1}) = \tmfrac{1}{8} \in \Q(t).$$
As indicated above, a similar computation shows the same result for~$M_0$.
Here are some details.
The space~$Y' \sm \nu(K_0)$ is obtained from the exterior of the link~$K_0 \cup L_1 \subset S^3$ by Dehn filling $L_1$ with framing~$-8$.
Since~$\ell k(L_1,K_0)=1$,  it follows that
$$H_1(Y' \sm \nu(K_0)) \cong \frac{\Z\langle\mu_{L_1} \rangle \oplus \Z \langle\mu_{K_0} \rangle }{\langle-8\mu_{L_1} +  \mu_{K_0}\rangle} \cong \Z.$$
We now express~$[\pi_{K_0}]$ as a multiple of~$[\mu_{K_0}]$, as required to calculate the framing of~$K_0$.
Since~$\pi_{K_0}$ is a~$0$-parallel of~$K_0$ we have~$[\pi_{K_0}] = [\mu_{L_1}]$.
Use the invertible matrix~$\bsm 1 & 8 \\ 0 & 1\esm$ to change coordinates to the presentation
$$\Z \xrightarrow{\bsm 1\\ 0 \esm} \Z \oplus \Z \to H_1(Y' \sm \nu(K_1)) \to 0.$$
In this presentation, we compute that
\begin{align*}
[\mu_{K_0}]&=\operatorname{proj}_2 \circ \bsm 1 & 8 \\ 0 & 1\esm \bsm 0 \\ 1 \esm  = 8 \in \Z \cong H_1(Y' \sm \nu(K_0)), \\
[\pi_{K_0}]&=\operatorname{proj}_2 \circ \bsm 1 & 8 \\ 0 & 1\esm \bsm 1 \\ 0 \esm  = 1 \in \Z \cong H_1(Y' \sm \nu(K_0)).
\end{align*}
Hence passing to the~$\Z$ cover, tensoring up to~$\Q(t)$ coefficients,  one obtains
% and changing basis using~$\bsm 1 & 8 \\ 0 & 1\esm$ to the presentation~$\Z \xrightarrow{(0,1)^T} \Z \oplus \Z$
%
%The knot~$K_0 \subseteq Y'$ has parallel~$\pi_{K_0}$ with framing~$0$.
% we compute that in~$\Z$ we have~$[\pi_{K_0}] = 1$ and~$[\mu_{K_0}] = 8$. Hence once again
$$\lk_{\Q(t)}(\wt{K}_0,\wt{\pi}_{\wt{K}_0}) =\tmfrac{1}{8} \in \Q(t).$$
This concludes the proof of the claim.
\end{proof}

The combination of the claim with Step~$2$ of the plan from Section~\ref{sub:Plan} implies that surgery along~$K_i$ yields a~$\Z[t^{\pm 1}]$-homology~$3$-sphere for~$i=0,1.$
In order to recover the construction described during the proof of Proposition~\ref{prop:KSSpin} however,  we take~$\widetilde{K}_i$ (and therefore~$K_i  \subset \unaryminus L(8,1)$) to be the unknot for~$i=0,1$:  as described in the proposition, surgery on~$Y$ along~$K_0$ and~$K_1$ then yields~$S^1 \times S^2$ and~$(S^1 \times S^2) \# S^3_{+1}(T_{2,3})$ respectively.
The infinite cyclic covers of these manifolds have vanishing Alexander modules yielding a ``by hand" version of Step~$2$.

Step~$3$ is carried out by capping off with~$S^1 \times D^3$ and~$(S^1 \times D^3) \natural C$ respectively; both of these are homotopy~$S^1 \times D^3$s.
Thus~$M_0$ and~$M_1$ are obtained by the realisation process of our main theorem.
%
%Thus~$M_i$ is obtained by surgering~$Y$ along~$\widetilde{K}_i$ representing~$b_i(1)$ with equivariant self linking equal to the single entry of~$-Q^{-T}$ and capping off the outcome by a homotopy~$S^1 \times D^3$.
It follows that~$b_{M_0} =b_0 \neq b_1=b_{M_1}$,  as asserted.
%It follows that the construction of~$M_1$ precisely follows the algorithm set out in the proof of Theorem~\ref{thm:MainTechnicalIntro}, and hence~$M_1$ has automorphism invariant~$b_1$.
%Similarly the construction of~$M_0$ does indeed realise the automorphism invariant~$b_0$. This completes the proof of the claim.
\color{black}
\end{example}

In summary,  the Kirby-Siebenmann invariant of spin 4-manifolds is not always controlled by the boundary and the intersection form.
Rather, the automorphism invariant must be taken into account as well.

An explanation for this is that the automorphism invariant can act nontrivially on the spin structures. Using $b_0$ to fix an isometry $\partial \lambda \cong -\Bl_Y$, $b_1$ determines an automorphism of $\Bl_Y$.  If this automorphism preserved the quadratic enhancement of $\Bl_Y$ determined by a spin structure (or by the presentation of $\partial \lambda \cong \Bl_Y$ as the boundary of an even Hermitian form~\cite[p.243]{RanickiExact}, \cite[Definition~2.5]{CCP}) then the induced spin structures on $Y$ would agree. Then $M_0$ and $M_1$ would be stably homeomorphic and hence their Kirby-Siebenmann invariants would be the same; see~\cite[Proposition~4.2]{CCP}.
But when we consider an automorphism of the linking form that does not preserve the quadratic enhancement, as is the case for $b_1$ above, then the Kirby-Siebenmann invariants can be different, as with the example just given.

Finally, we note that the example just given, without adding the copies of $S^1 \times D^3$, is also compelling in the simply-connected case.
We gave it for infinite cyclic fundamental group since that is the topic of the present paper. %and we cannot go back in time and add it to Boyer's paper.

\color{black}

\section{Application to~$\Z$-surfaces in~$4$-manifolds}
\label{sec:Discs}

Recall that a \emph{$\Z$-surface} refers to a locally flat, embedded surface in a $4$-manifold whose complement has infinite cyclic fundamental group.
In this section we apply our classification of $4$-manifolds with fundamental group $\Z$ to the study of $\Z$-surfaces in simply-connected $4$-manifolds and prove Theorems~\ref{thm:SurfacesRelBoundaryIntro},~\ref{thm:SurfacesWithBoundaryIntro}, and~\ref{thm:SurfacesClosedIntro} from the introduction.
In Subsection~\ref{sub:Boundary}, we focus on~$\Z$-surfaces with boundary up to equivalence rel.\ boundary.
In the shorter Subsections~\ref{sub:SurfacesBoundaryEq} and~\ref{sub:Closed}, we respectively study surfaces with boundary up to equivalence (not necessarily rel.\ boundary) and closed surfaces. Subsection~\ref{sub:OpenQuestions} lists some open problems.

\subsection{Surfaces with boundary up to equivalence rel.\ boundary}
\label{sub:Boundary}
Let~$N$ be a simply-connected~$4$-manifold with boundary homeomorphic to $S^3$. We fix once and for all a particular homeomorphism $h \colon \partial N \cong S^3$. Let~$K \subset S^3$ be a knot. Thus $K$ and $h$ determine a knot in $\partial N$, which we also denote by $K$.
The goal of this subsection is to give an algebraic description of the set of~$\Z$-surfaces in~$N$ with boundary $K$ up to equivalence rel.\ boundary.

We begin with some conventions.
Given a properly embedded~$\Z$-surface~$\Sigma \subset N$ in a simply-connected~$4$-manifold,
%define
denote its exterior by~$N_\Sigma:=N\smallsetminus \nu(\Sigma)$.
Throughout this section, we will refer to embedded surfaces simply as $\Sigma$, and abstract surfaces as~$\Sigma_{g,b}$, where $g$ is the genus and $b$ is the number of boundary components; we may sometimes write $\Sigma_g$ when $b=0$.
Recall that throughout,~$\Sigma_{g,b}$ and $N$ will be oriented.
This data determines orientations on $S^3$, $K$, and every meridian of an embedding of~$\Sigma_{g,b}$.
Observe that the~$\pi_1(N_\Sigma) \cong \Z$ hypothesis implies that~$[\Sigma,\partial \Sigma]=0\in H_2(N,\partial N)$ by~\cite[Lemma 5.1]{ConwayPowell}, so the relative Euler number of the normal bundle of $\Sigma$, with respect to the zero-framing of $\nu (\partial N)$, vanishes~\cite[Lemma~5.2]{ConwayPowell}.
From now on, we choose a framing~$\nu(\Sigma) \cong \Sigma \times \mathring{D}^2 \cong \Sigma \times \R^2$ compatible with the orientation and with the property that for each simple closed curve~$\gamma_k \subset \Sigma$, we have~$\gamma_k \times \lbrace e_1 \rbrace \subset N  \setminus \Sigma$ is nullhomologous in~$N  \setminus \Sigma$.
We will refer to such a framing as a \emph{good framing}.
As such,
when ~$\partial\Sigma=K\subset\partial N$
we can identify the boundary of~$N_\Sigma$ as
$$\partial N_\Sigma \cong E_K \cup_\partial (\Sigma_{g,1} \times S^1)=:M_{K,g},$$
where the gluing~$\partial$ takes~$\lambda_K$ to~$\partial\Sigma\times\{ \operatorname{pt}\}$.

We call two locally flat surfaces~$\Sigma,\Sigma' \subset N$ with boundary~$K \subset \partial N \cong S^3$ \emph{equivalent rel.\ boundary} if there is an orientation-preserving homeomorphism of pairs~$(N,\Sigma) \cong (N,\Sigma')$ that is pointwise the identity on~$\partial N \cong S^3$.
Note that if~$\Sigma \subset N$ is a~$\Z$-surface with boundary~$K$,  then~$N_\Sigma$ is a $\Z$-manifold with boundary~$\partial N_\Sigma \cong M_{K,g}$~\cite[Lemma 5.4]{ConwayPowell} and~$H_1(M_{K,g};\Z[t^{\pm 1}])\cong H_1(E_K;\Z[t^{\pm 1}]) \oplus \Z^{2g}$
%{(for this decomposition, we use~\cite[Lemma 5.5]{ConwayPowell})}
is torsion because the Alexander module~$H_1(E_K;\Z[t^{\pm 1}])$ of~$K$ is torsion~\cite[Lemma~5.5]{ConwayPowell}.
Additionally, note that the equivariant intersection form~$\lambda_{N_\Sigma}$ of a surface exterior~$N_\Sigma$  must present~$M_{K,g}$.

Consequently, as we did for manifolds, it is natural to fix a form~$(H,\lambda)$ that presents~$M_{K,g}$ and to consider the set~$\operatorname{Surf(g)}^0_\lambda(N,K)$ of genus~$g$~$\Z$-surfaces in~$N$ with boundary~$K$ and~$\lambda_{N_\Sigma}\cong \lambda$.

\begin{definition}
\label{def:Surface(g)RelBoundary}
For a nondegenerate Hermitian form~$(H,\lambda)$ over $\Z[t^{\pm 1}]$ that presents~$M_{K,g}$, set
$$\operatorname{Surf(g)}^0_\lambda(N,K):=\lbrace \Z\text{-surfaces~$\Sigma \subset N$ for~$K$ with } \lambda_{N_\Sigma}\cong \lambda \rbrace/\text{ equivalence rel.~$\partial$}.$$
\end{definition}

There is an additional necessary condition for this set to be nonempty.
For conciseness, we write~$\lambda(1):=\lambda \otimes_{\Z[t^{\pm 1}]} \Z_\varepsilon$, where $\Z_\varepsilon$ denotes $\Z$ with the trivial $\Z[t^{\pm 1}]$-module structure.
This way,  if~$A(t)$ is a matrix that represents~$\lambda$, then~$A(1)$ represents~$\lambda(1)$.
Additionally, recall that if~$W$ is a~$\Z$-manifold, then~$\lambda_W(1) \cong Q_W$,  where~$Q_W$ denotes the standard intersection form of~$W$; see e.g.~\cite[Lemma 5.10]{ConwayPowell}.
Thus, if we take~$W=N_\Sigma$ and assume that~$\lambda \cong \lambda_{N_\Sigma}$, then
$$\lambda(1) \cong \lambda_{N_\Sigma}(1) \cong Q_{N_\Sigma} \cong Q_N \oplus  (0)^{\oplus 2g},$$
where the last isometry follows from a Mayer-Vietoris argument.
Thus,  for the set~$\operatorname{Surf(g)}^0_\lambda(N,K)$ to be nonempty, it is also necessary that~$\lambda(1)\cong Q_N \oplus  (0)^{\oplus 2g}$.

For the final piece of setup for the statement of the main result of the section,  we describe an action of~$\operatorname{Homeo}^+(\Sigma_{g,1},\partial)$ on the set~$\Iso(\partial \lambda,\unaryminus \Bl_{M_{K,g}})$ as follows.
First, a rel.\ boundary homeomorphism  $x \colon \Sigma_{g,1} \to \Sigma_{g,1}$ induces an isometry $x''_* \colon \Bl_{M_{K,g}} \cong  \Bl_{M_{K,g}} $ as follows. Extend~$x$ to a self homeomorphism $x'$ of $\Sigma_{g,1}\times S^1$ by defining $x'(s,\theta)=(x(s),\theta)$.
Then extend $x'$ by the identity over $E_K$; in total one obtains a self homeomorphism $x''$ of $M_{K,g}$.
Now lift this homeomorphism to the covers and take the induced map on $H_1$ to get $x''_* \colon \Bl_{M_{K,g}} \cong  \Bl_{M_{K,g}}$.
The required action is now by postcomposition; for $f \in \Iso(\partial \lambda,\unaryminus \Bl_{M_{K,g}})$, define  $x \cdot f := x''_* \circ f$.
%Basepoints are taken in the boundary of the surface.
The main result of this section proves Theorem \ref{thm:SurfacesRelBoundaryIntro} from the introduction.
The formulation of the result is different than in the introduction, but clearly equivalent.

\begin{theorem}
\label{thm:SurfacesRelBoundary}
Let~$N$ be a simply-connected~$4$-manifold with boundary~$\partial N \cong S^3$ and let~$K \subset S^3$ be a knot.
Given a nondegenerate Hermitian form~$(H,\lambda)$ over~$\Z[t^{\pm 1}]$, the following assertions are equivalent:
\begin{enumerate}
\item
the Hermitian form~$(H,\lambda)$ presents~$M_{K,g}$ and satisfies~$\lambda(1)\cong Q_N \oplus  (0)^{\oplus 2g}$;
\item the set~$\operatorname{Surf(g)}^0_\lambda(N,K)$ is nonempty and there is a bijection
$$\operatorname{Surf(g)}^0_\lambda(N,K) \approx \Iso(\partial \lambda,\unaryminus\Bl_{M_{K,g}})/(\Aut(\lambda)\times \operatorname{Homeo}^+(\Sigma_{g,1},\partial)).$$
\end{enumerate}
\end{theorem}

\begin{remark}
\label{rem:SurfacesRelBoundaryTheorem}
We collect some remarks concerning Theorem~\ref{thm:SurfacesRelBoundary}.
\begin{itemize}
\item If $(H,\lambda)$ presents $M_{K,g}$, then there is a non-canonical bijection
%AC: non-canonical in the sense that it depends on the choice of the isomorphism $\partial \lambda \cong M_{K,g}$.
 $$\frac{\Iso(\partial \lambda,\unaryminus\Bl_{M_{K,g}})}{(\Aut(\lambda)\times \operatorname{Homeo}^+(\Sigma_{g,1},\partial))} \approx \frac{ \Aut(\partial \lambda)}{(\Aut(\lambda) \times \operatorname{Homeo}^+(\Sigma_{g,1},\partial))}.$$
% The latter is sometimes more convenient for computations.
In addition,  we have the isomorphism~$\Aut(\partial \lambda)   \cong \Aut(\Bl_{M_{K,g}})\cong \Aut (\Bl_K) \oplus \operatorname{Sp}_{2g}(\Z)$ where the latter is the group of automorphisms of the symplectic intersection pairing of~$\Sigma_{g,1}$~\cite[Propositions 5.6 and 5.7]{ConwayPowell}.
  %MP originally  wrote
  %The image of $\operatorname{Homeo}^+(\Sigma_{g,1},\partial))$ is $0 \oplus \operatorname{Sp}_{2g}(\Z)$.  So one can express the quotients above as
The group $\operatorname{Homeo}^+(\Sigma_{g,1},\partial)$ acts trivially on the first summand and transitively on the second.
%AC: Trivially because we are looking at rel boundary homeos.
%Transitive means that for every x,y there is an h with h.x=y.
%We know that every x \in Sp_{2g}(\Z)=Homeo(\Sigma_{g,1},\partial)=\Aut(\Bl_{\Sigma_{g,1} \times S^1) is induced by a self-homeo h_x of \Sigma_{g,1} \times S^1.
%Thus for x,y \in Sp_2g, we  have h_{yx^{-1}}\cdot x=h_yh_{x^{-1}}x=y.
%Note that the isomorphism $\Aut(\partial \lambda)  \cong \Aut(\Bl_{M_{K,g}})$ is non-canonical.
Therefore one can express the quotients above as
 \[\Aut (\Bl_K)/ \Aut (\lambda),\]
where the action of $\Aut(\lambda)$ on~$\Aut(\Bl_K)$ arises by restricting the action of $\Aut(\lambda)$ on $\Aut(\partial \lambda)\cong \Aut(\Bl_{M_{K,g}}) \cong \Aut(\Bl_K) \oplus \operatorname{Sp}_{2g}(\Z)$ to the first summand.
We stress again that the isomorphism $\Aut(\partial \lambda) \cong \Aut(\Bl_{M_{K,g}})$ is not canonical.
 %%Don't delete.
 %the action of $F \in \Aut(\lambda)$ on $h \in \Aut(\Bl_K)$ arises by writing $\partial F=h_K \oplus h_\Sigma$ and setting $F \cdot h=h \circ h_K^{-1}$.
The set $\Aut (\Bl_K)/ \Aut (\lambda)$ was mentioned in Theorem~\ref{thm:SurfacesRelBoundaryIntro} from the introduction.
% by restricting an isometry $\partial h $ of  $\Bl_{M_{K,g}}=\Bl_K \oplus \Bl_{\Sigma_{g,1}}$}
\item The action of~$\operatorname{Homeo}^+(\Sigma_{g,1},\partial)$ on~$\Iso(\partial \lambda,\unaryminus \Bl_{M_{K,g}})$ factors through the corresponding mapping class group $\operatorname{Mod}^+(\Sigma_{g,1},\partial) := \pi_0(\operatorname{Homeo}^+(\Sigma_{g,1},\partial))$.
In particular, Theorem~\ref{thm:SurfacesRelBoundary} could have equally well been stated using  $\operatorname{Mod}^+(\Sigma_{g,1},\partial)$ instead of $\operatorname{Homeo}^+(\Sigma_{g,1},\partial)$.
%\item In the case of discs,  we have~$M_{K,g}=M_K$ and so~$\Bl_{M_{K,g}}$ is simply~$\Bl_K$, the Blanchfield form of~$K$.
%In general, there is an  isometry~$\Bl_{M_{K,g}} \cong \Bl_K  \oplus \Bl_{\Sigma_g \times S^1}$~\cite[Proposition~5.7]{ConwayPowell}.
\item  Our surface set~$\operatorname{Surf(g)}^0_\lambda(N,K)$ is defined up to equivalence, hence Theorem~\ref{thm:SurfacesRelBoundary} only gives a classification of surfaces up to equivalence (instead of ambient isotopy).
This is because we prove Theorem \ref{thm:SurfacesRelBoundary} as a consequence of Theorem \ref{thm:ClassificationRelBoundary} and the equivalence on~$\mathcal{V}^0_\lambda(M_{K,g})$ is up to \textit{any }homeomorphism rel.\ boundary, not just homeomorphisms in a prescribed isotopy class.  As a consequence,
% does not require that homeomorphisms
%induce the identity on~$H_2$. In particular, homeomorphisms in $\mathcal{V}^0_\lambda({K,g})$ need not
%be isotopic to the identity rel.\ boundary.
when $N$ admits homeomorphisms which are not isotopic to the identity rel.\ boundary,  there can be~$\Z$-surfaces that are equivalent rel.\ boundary but not ambient isotopic. Here is an example.
\black

%However ambient isotopic surfaces must be equivalent via a homeomorphism which is topologically isotopic to the identity, and in particular induces the identity on second homology.
%It was to be expected that the theorem would concern surface equivalence (instead of ambient isotopy):
%a homeomorphism between simply-connected topological~$4$-manifolds that is topologically isotopic to the identity induces the identity on second homology.
%Thus, we only get equivalence because in the definition of~$\mathcal{V}^0_\lambda(M_K)$, the homeomorphisms need not induce the identity on~$H_2$.

% of two $\Z$-surfaces that are equivalent rel.\ boundary but not ambient isotopic.
%illustrating this phenomenon.
Let~$K \subset S^3$ be a knot with nontrivial Alexander polynomial~$\Delta_K$, that bounds a~$\Z$-disc in a punctured~$\C P^2$ with intersection form represented by the~$1 \times 1$ matrix~$(\Delta_K)$.
Let~$N$ be given by the boundary connected sum with another punctured~$\C P^2$ (so that~$N$ is a punctured~$\C P^2\# \C P^2$), and
%by an abuse of notation
denote the same~$\Z$-disc considered in~$N$ by~$D$.
%Still has boundary K.
There is a self-homeomorphism~$\tau \colon N \to N$ that induces~$\bsm 0 & 1 \\ 1 & 0 \esm$ on~$H_2(N) \cong \Z^2$.
Isotope~$\tau$ to be the identity on~$\partial N \cong S^3$.
%AC: Collar neighorhood trickery.
The discs~$D$ and~$\tau (D)$ are equivalent rel.\ boundary.
%AC: Mayer-Vietoris, see end of smooth section.
But a short computation shows that the equivariant intersection forms of the
%complements
exteriors are
$\bsm\Delta_K & 0 \\ 0 & 1 \esm \text{ and } \bsm 1 & 0 \\ 0 & \Delta_K \esm$
respectively.
A straightforward computation shows that every~$\Z[t^{\pm 1}]$-isometry between these two forms augments over~$\Z$ to~$\bsm 0 & 1 \\ 1 & 0 \esm$.
It follows that there is no ambient isotopy between~$D$ and~$\tau(D)$.
\end{itemize}
\end{remark}

Theorem~\ref{thm:SurfacesRelBoundary} will be proved in three steps.
\begin{enumerate}
\item We define a map $\Theta$ from a set of equivalence classes of embeddings~$\Sigma_{g,1} \hookrightarrow N$, which we denote~$\operatorname{Emb}_\lambda^0(\Sigma_{g,1},N;K)$ and which we will define momentarily,
to the set of manifolds~$\mathcal{V}_\lambda^0(M_{K,g})$ from Definition~\ref{def:V0lambdaY}.
By Theorem~\ref{thm:ClassificationRelBoundary},~$\mathcal{V}_\lambda^0(M_{K,g})$  corresponds bijectively to the set of isometries~$\Iso(\partial \lambda,\unaryminus \Bl_{M_{K,g}})/\Aut(\lambda)$.
\item We prove that the map $\Theta$ is a bijection, by defining a map $\Psi$ in the other direction, from the set of manifolds to the set of embeddings, and showing that both $\Theta \circ \Psi$ and $\Psi \circ \Theta$ are the identity maps.
\item We describe the set of surfaces~$\operatorname{Surf(g)}^0_\lambda(N,K)$ as a quotient of~$\operatorname{Emb}_\lambda^0(\Sigma_{g,1},N;K)$ by $\operatorname{Homeo}^+(\Sigma_{g,1},\partial)$.
We show that this action and the actions of~$\operatorname{Homeo}^+(\Sigma_{g,1},\partial)$ on~$\mathcal{V}_\lambda^0(M_{K,g})$ and $\Iso(\partial \lambda,\unaryminus \Bl_{M_{K,g}})/\Aut(\lambda)$ are all compatible. Passing to orbits leads to the bijection in Theorem~\ref{thm:SurfacesRelBoundary}.
%We then use the bijection from the first step to define an action of~$\operatorname{Homeo}^+(\Sigma_{g,1},\partial)$ on~$\mathcal{V}_\lambda^0(M_{K,g})$ and $\Iso(\partial \lambda,\unaryminus \Bl_{M_{K,g}})/\Aut(\lambda)$ leading to the bijection in Theorem~\ref{thm:SurfacesRelBoundary}.
 This step is largely formal.
\end{enumerate}

\subsubsection*{Step $(1)$: From embeddings to manifolds}

\black
For the first step, we give some definitions and construct the map which will be the bijection in Theorem~\ref{thm:SurfacesRelBoundary}.

Consider the following set:
$$\operatorname{Emb}_\lambda^0(\Sigma_{g,1},N;K)=\frac{\lbrace e \colon \Sigma_{g,1} \hookrightarrow N \mid e(\Sigma_{g,1}) \text{ is a } \Z\text{-surface for~$K$ with } \lambda_{N_{e(\Sigma_{g,1})}}\cong \lambda \rbrace }{\text{ equivalence rel.~$\partial$}}.$$
Two embeddings~$e_1,e_2$ are \emph{equivalent rel.\ boundary} if there exists a homeomorphism~$\Phi \colon N \to N$ that is the identity on~$\partial N \cong S^3$ and satisfies~$\Phi \circ e_1=e_2$.

In what follows, we let~$\varphi \colon \pi_1(M_{K,g}) \twoheadrightarrow \Z$ be the epimorphism such that the induced map $\varphi' \colon H_1(M_{K,g}) \twoheadrightarrow \Z$ is the unique epimorphism that maps the meridian of~$K$ to~$1$ and the other generators to zero.
When we write~$\mathcal{V}^0_\lambda(M_{K,g})$, it is with respect to this epimorphism~$\varphi$. Recall also that we have a fixed homeomorphism $h\colon\partial N\to S^3$; whenever we say $\partial N\cong S^3$, it is with this fixed $h$.

%%%DONT DELETE
%\begin{lemma}
%\label{lem:IsotopyV0}
%Assume that  $[(W,f)]=[(W',f') \in V_\lambda^0(Y)$.
%If there is an homeomorphism $W \to W'$ such that $f'$ is isotopic to $f \circ F$, then $(W,f)$ and $(W',f')$ agree in $V_\lambda^0(Y)$.
%\end{lemma}
%\begin{proof}
%We must find a homeomorphism $H\colon W \to W$ such that $f'=F \circ f$.
%Since $f'$ is isotopic to $f \circ F$, we deduce that ${f'}^{-1} \circ f \circ F| \colon \partial W' \to \partial W'$ is isotopic to the identity, $\id_{W'}$.
%Implant the isotopy in a collar neighborhood of $\partial W'$ and extend by the identity, to obtain a homeomorphism $G:=W' \to W'$ whose restriction to the boundary is ${f'}^{-1} \circ f \circ F|$.
%Now $H:=G \circ F$ is the required homeomorphism since, by construction,  $f' \circ H| =f$.
%\end{proof}
%%%%%%%%%%%%%%%%%

In addition to our homeomorphism $h \colon \partial N \to S^3$, we fix once and for all the following data.
\begin{itemize}
%\item a homeomorphism $h \colon \partial N \to S^3$;
\item A closed tubular neighborhood $\overline{\nu}(K) \subset \partial N$.
Since we have already fixed $h$, and since we are abusively using $K$ for both the knot $K$ in $\partial N$ and for the image $h(K)$ in $S^3$, this choice of $\overline{\nu}(K) \subset \partial N$ also determines a particular neighborhood $\overline{\nu}(K) \subset S^3$. We will use~$E_K$ exclusively to denote the complement of ${\nu}(K)$ in $S^3$.
%Occasionally for emphasis we will denote the image in $S^3$ of this neighborhood as $\overline{\nu}(h(K))$ and similiarly,  we write either $(h(\lambda_K),h(\mu_K))$ or $(\lambda_{h(K)},\mu_{h(K)})$ for the longitude-meridian pair of $h(K)$. {LP:come back and strip if we dont actually use}.
\item A homeomorphism $D \colon \partial \Sigma_{g,1} \times S^1 \to \partial \overline{\nu}(K)$ that takes $\partial \Sigma_{g,1}\times \{1\}$ to the $0$-framed longitude of $K$ and $\lbrace \operatorname{pt} \rbrace \times S^1$ to the meridian of $K$ such that
$$ M_{K,g}=E_{K} \cup_D \Sigma_{g,1} \times S^1.$$
%=\partial N \setminus \nu(K)$.
\end{itemize}

These choices can change the bijection, however we are interested only in the existence of a bijection, so this is not an issue.
%a problem.

Next we define the map which will be the bijection in Theorem~\ref{thm:SurfacesRelBoundary}.

\begin{construction}
\label{cons:EmbVBijection}
We construct a map~$\Theta \colon \operatorname{Emb}_\lambda^0(\Sigma_{g,1},N;K) \to \mathcal{V}^0_\lambda(M_{K,g})$.

Let~$e \colon \Sigma_{g,1} \hookrightarrow N$ be an embedding that belongs to~$\operatorname{Emb}_\lambda^0(\Sigma_{g,1},N;K)$.
We will assign to~$e$ a pair~$(N_{e(\Sigma_{g,1})},f)$, where~$f \colon \partial N_{e(\Sigma_{g,1})} \to M_{K,g}$ is a homeomorphism.
The pair we construct will depend on several choices, but we will show that the outcome is independent of these choices up to equivalence in $\mathcal{V}_\lambda^0(M_{K,g})$.
%First some notatation:
%\begin{itemize}
%\item We write $\Sigma:=e(\Sigma_{g,1})$ for the image of our embedding.
%\end{itemize}

%To build this homeomorphism, let us first be precise about its target; we will say that the exterior of $K=e|_{\partial \Sigma_{g,1} \times \{0\}}$ has its boundary
%is identified to the boundary of $\Sigma_{g,1} \times S^1$ via a particular homeomorphism $D \colon \partial \Sigma_{g,1} \times S^1 \to \partial N \setminus \nu(K)$.

To cut down on notation we set $\Sigma:=e(\Sigma_{g,1})$ and describe the choices on which our pair $(N_{\Sigma},f)$ will a priori depend.
\begin{enumerate}
\item An embedding $\iota \colon \overline{\nu}(\Sigma) \hookrightarrow N$ of the normal bundle of $\Sigma$ such that $\iota(\overline{\nu}(\Sigma)) \cap \partial N$ agrees with our fixed tubular neighbhourhood of $K$.
\item A good framing
%{LP: Do we actually use good in the description of the bijection? so we dont need good here. But we can have it, by independence of framing (and by the boundary hypothesis) later. AC: it matters to define the coefficient system I think.} framing
%(good is defined in the conventions at the top of this section).
%This is a particular bundle isomorphism
$ \gamma \colon \overline{\nu}(\Sigma) \cong \Sigma_{g,1} \times D^2$ such that $h|\circ \iota\circ\gamma^{-1}=D$:
\begin{equation}
\label{eq:Compatible}
\xymatrix{
\partial \Sigma_{g,1} \times S^1 \ar[r]^-D \ar[d]^{\gamma^{-1}}&\partial \overline{\nu}(K) \subset E_{K}\\
\gamma^{-1}(\partial \Sigma_{g,1} \times S^1)  \ar[r]^-{\iota|}& \iota(\gamma^{-1}(\partial \Sigma_{g,1} \times S^1)) \subset \partial N \setminus \nu(K)  \ar[u]^{h|}.
}
\end{equation}
%{MP: changed something so that the $\gamma^{-1}$ part makes sense.}
In this diagram, $h|$ denotes the restriction of our fixed identification $h \colon \partial N \cong S^3$ and $D \colon \partial \Sigma_{g,1} \times S^1 \to \partial \overline{\nu}(K)$ is the homeomorphism that we fixed above.
\end{enumerate}

We also record some of the notation that stems from these choices.
\begin{itemize}
%\item We temporarily write $N_{e_\iota(\Sigma_{g,1})}:=N \setminus \iota (\overline{\nu}(e(\Sigma_{g,1}))$ for the exterior of $e(\Sigma_{g,1})$ in $N$.
\item The boundary of the surface exterior $N_\Sigma$ decomposes as
\begin{equation}
\label{eq:DecompositionBoundarySurfaceExterior}
\partial N_{\Sigma} \cong \big(\partial N \setminus \nu(K)\big) \cup \Big{(}\partial \iota(\overline{\nu}(\Sigma))\smallsetminus \left( \iota(\nu(\Sigma)) \cap \partial N \right)\Big{)}.
\end{equation}
Here the first part of this union is homeomorphic to a knot exterior, while the second is homeomorphic to $\Sigma_{g,1} \times S^1$.
\item Restricting our fixed homeomorphism $h \colon \partial N \cong S^3$ to the knot exterior part in~\eqref{eq:DecompositionBoundarySurfaceExterior}, we obtain the homeomorphism
$$h| \colon \partial N \smallsetminus \nu(K) \to E_{K} \subset M_{K,g}.$$
\item On the circle bundle part of~\eqref{eq:DecompositionBoundarySurfaceExterior}, we consider the homeomorphism
  \[\gamma| \circ \iota^{-1} \colon \Big{(}\partial \iota(\overline{\nu}(\Sigma))\smallsetminus \left( \iota(\nu(\Sigma)) \cap \partial N \right) \Big{)} \to \Sigma_{g,1} \times S^1 \subset M_{K,g}.\]
  Here by the slightly abusive notation $\iota^{-1}$, we mean that since $\iota \colon \overline{\nu}(\Sigma) \hookrightarrow N$ is an embedding,  it is  a homeomorphism onto its image, whence the inverse.
\end{itemize}

The diagram in~\eqref{eq:Compatible} ensures that $h|$ and $\gamma| \circ \iota^{-1}$ can be glued together to give rise to the homeomorphism we have been building towards:
\begin{equation}
\label{eq:BoundaryHomeoSurface}
 f_\gamma \colon \partial N_{\Sigma} \to M_{K,g},  \ \ \ \ f_\gamma:=(h|) \cup (\gamma| \circ \iota^{-1}).
 \end{equation}
\end{construction}

Set $\Theta(e):=(N_{\Sigma},f_\gamma)$.
We need to verify that $\Theta$ gives rise to a map $ \operatorname{Emb}_\lambda^0(\Sigma_{g,1},N;K) \to \mathcal{V}^0_\lambda(M_{K,g})$.
% is a well-defined map from $ \operatorname{Emb}_\lambda^0(\Sigma_{g,1},N;K)$ to $\mathcal{V}^0_\lambda(M_{K,g})$.
In other words, we need to check that modulo homeomorphisms rel.\ boundary,~$\Theta(e)$
%does not depend on choices $(1)$ and $(2)$. More precisely, we will check that $\Theta$
does not depend on the embedding $\iota \colon \overline{\nu}(\Sigma) \hookrightarrow N$
%(such that $\iota(\overline{\nu}(\Sigma)) \cap \partial N$ agrees with our fixed tubular neighborhood~$\overline{\nu}(K)$),
nor on the particular choice of the good framing~$\gamma$ subject to the condition in~\eqref{eq:Compatible}.
We also have to verify that equivalent embeddings produce equivalent manifolds.
\black
\begin{itemize}
\item
First we show that the construction is independent of $\gamma$ and $\iota$.
Pick another embedding $\iota' \colon \overline{\nu}(e(\Sigma_{g,1})) \hookrightarrow N$ of the normal bundle and another good framing $\gamma' \colon \overline{\nu}(e(\Sigma_{g,1})) \cong \Sigma_{g,1} \times D^2$ with the same hypothesis about compatibility with $D$.
This leads to boundary homeomorphisms $f_\gamma:=(h|) \cup (\gamma|  \circ \iota^{-1})$ and $f_{\gamma'}:=(h|) \cup (\gamma'| \circ {\iota'}^{-1})$ and  we must show that the following pairs are equivalent rel.\ boundary:
\begin{equation}
\label{eq:WantEquivalent}
 (N_{e_\iota(\Sigma_{g,1})},f_\gamma) \text{ and } (N_{e_{\iota'}(\Sigma_{g,1})},f_{\gamma'}) .
\end{equation}
 For a moment we are keeping track of the embeddings $\iota$ and $\iota'$ in our  notation for exteriors.
More explicitly, we set $N_{e_\iota(\Sigma_{g,1})}:=N \setminus \iota(\nu(e(\Sigma_{g,1})))$ and similarly for $\iota'$.

By uniqueness of tubular neighourhoods~\cite[Theorem~9.3D]{FreedmanQuinn}, there is an isotopy of embeddings
$\Gamma_t \colon \Sigma_{g,1}\times D^2 \hookrightarrow N$ such that $\Gamma_0=\iota \circ \gamma^{-1}$ and $\Gamma_1=\iota' \circ \gamma'^{-1}$ that fixes a neighborhood of $\partial \Sigma_{g,1} \times D^2$.
Then by the Edwards-Kirby isotopy extension theorem~\cite{KirbyEdwards1971}, there is an isotopy of homeomorphisms $F_t\colon N \to N$ with $F_1\circ \iota \circ \gamma^{-1}=\iota' \circ \gamma'^{-1}$ and $F_0=\id_N$ and such that $F_t$ is the identity on a neighborhood of the boundary $\partial N$ for every $t \in [0,1]$.
We will argue that this $F_1$ restricted to the exteriors $N_{e_{\iota}(\Sigma_{g,1})}$ and $N_{e_{\iota'}(\Sigma_{g,1})}$ gives a rel.\  boundary homeomorphism between the pairs in~\eqref{eq:WantEquivalent}.

We wish to argue that
%after suitable further isotopies,
the restriction of $F_1$ to the surface exteriors identifies $(N_{e_{\iota}(\Sigma_{g,1})},f_\gamma)$ with $(N_{e_{\iota'}(\Sigma_{g,1})},f_{\gamma'})$ as elements of $\mathcal{V}^0_\lambda(M_{K,g})$.
Consider the following diagram:
$$
\xymatrix @C+1.5cm{
M_{K,g} \ar[d]^=& \partial N_{e_\iota(\Sigma_{g,1})} \ar[l]_{f_\gamma=(h|) \cup(\gamma|  \circ \iota^{-1})} \ar[r]^{\subset}  \ar[d]^{F_1}& N_{e_\iota(\Sigma_{g,1})} \ar[r]^{\subset} \ar[d]^{F_1}& N \ar[d]^{F_1} \\
%%%%%%%%
M_{K,g} & \partial N_{e_{\iota'}(\Sigma_{g,1})} \ar[r]^{\subset} \ar[l]_{f_{\gamma'}=(h|) \cup (\gamma'|  \circ {\iota'}^{-1})}& N_{e_{\iota'}(\Sigma_{g,1})} \ar[r]^{\subset}& N.
}
$$
The right two squares certainly commute, while the left square commutes because the homeomorphism $F_1 \colon N \to N$ is rel.\ boundary and because, by construction, $\gamma| \circ \iota^{-1}=F_1\circ \gamma'|   \circ {\iota'}^{-1}$.
In total, we have:
%\color{red}
%\begin{equation}
%\label{eq:VerificationTheta}
% F_1 \circ f_{\gamma'}
%=F_1 \circ \left( (h|) \cup (\gamma'|  \circ {\iota'}^{-1}) \right)
%=(F_1 \circ  h|) \cup (F_1\circ \gamma'|  \circ {\iota'}^{-1})
%%=(F_1 \circ  h|) \cup (\gamma \circ \iota^{-1})
%=h| \cup (\gamma \circ \iota^{-1})=f_{\gamma}.
%\end{equation}
%\color{black}

\begin{equation}
\label{eq:VerificationTheta}
 f_{\gamma'} \circ F_1
= \left( (h|) \cup (\gamma'|  \circ {\iota'}^{-1}) \right) \circ F_1
=(h| \circ F_1 ) \cup (\gamma'|  \circ {\iota'}^{-1}\circ F_1)
%=(h| \circ F_1) \cup (\gamma \circ \iota^{-1})
=h| \cup (\gamma \circ \iota^{-1})=f_{\gamma}.
\end{equation}

%%%%%%%%%%%
%%DON'T DELETE:
%% it's what we initially wrote when we didn't fix a tubular neighbhoorhood of $K$.
%%%%%%%%
%Thus,  if we could show that $F_1 \circ h|$ is isotopic to $h|$ via an isotopy which fixes $\partial(\nu_\iota(K))$, then $ f_{\gamma'}\circ F_1$ would be isotopic to $f_\gamma$ and the verification  that $ (N_{e_\iota(\Sigma_{g,1})},f_\gamma) \text{ and } (N_{e_{\iota'}(\Sigma_{g,1})},f_{\gamma'})$ are equivalent rel.\ boundary would then be concluded by appealing to Lemma~\ref{lem:IsotopyV0}.
%
%We therefore claim that $F_1 \circ h|$ is isotopic to $h|$.
%When we consider $F_1$ restricted to $\partial N\setminus\nu_L(K)$, where $\nu_L(K)$ neighborhood of $K$ containing both $\nu_\iota(K)$ and $\nu_{\iota'}(K)$, we have that $F_1$ is the identity {LP:need to know something like isotopy extension is finitely supported}. Therefore, the only differences between $h|$ and $F_1\circ h|$ can be as maps with domain $\nu_L(K)\setminus\nu_\iota(K)\cong T^2\times I$.
%Since the only homeomorphisisms of $T^2\times I$ come from twisting the $I$ factor about some $S_1$ action on $T^2$ \cite{}, \textcolor{red}{some magic happens} {LP: if you restrict to disks, i can finish here, because twists are isotopic to the identity when performed on a compressible torus. AC: Can we fix this by imposing that the tubular neighbhoorhoods of the surfaces agree with a fixed one on the boundary?} and we find that $f_\gamma$ and $f_{\gamma'}\circ F_1$ are isotopic.
%%%%%%%%%%%%%%%%%%%%
%%%%%%%%%%%%

%-------------------------------------------------------------------------

\item  We now show that the map $\Theta$ from Construction~\ref{cons:EmbVBijection} is well defined up to  rel.\ boundary homeomorphisms of $N$; recall that this is the equivalence relation on the domain $\operatorname{Emb}_\lambda^0(\Sigma_{g,1},N;K)$.
Assume that~$e,e' \colon \Sigma_{g,1} \hookrightarrow N$ are embeddings that are homeomorphic rel.\ boundary via a homeomorphism~$F \colon N \to N$.
%Use any good framings $\gamma' \colon \overline{\nu}({e'(\Sigma_{g,1})}) \to \Sigma_{g,1} \times D^2$ for $e'(\Sigma_{g,1})$ and the good framing $\gamma:=F \circ \gamma'$ for ${e(\Sigma_{g,1})}$.
Pick good framings $\gamma,\gamma'$ for $\overline{\nu}(e(\Sigma_{g,1}))$ and $\overline{\nu}(e'(\Sigma_{g,1}))$ as well as an embedding $\iota' \colon \overline{\nu}(e'(\Sigma_{g,1})) \hookrightarrow N$.
We now consider the embedding $ \iota:= F^{-1} \circ \iota' \circ (\gamma')^{-1} \circ \gamma$. The following diagram commutes:
\begin{equation}
\label{eq:DiagramVerificationThetaRelBoundary}
\xymatrix{
\Sigma_{g,1} \times D^2 \ar[r]^-{\gamma^{-1},\cong} \ar[d]^=& \overline{\nu}(e(\Sigma_{g,1})) \ar[r]^-{\iota,\cong}& \iota(\overline{\nu}(e'(\Sigma_{g,1}))) \ar[r]^-{\subset } \ar[d]^{F|}& N  \ar[d]^F\\
%%%%%%
\Sigma_{g,1} \times D^2 \ar[r]^-{{\gamma'}^{-1},\cong}& \overline{\nu}(e'(\Sigma_{g,1})) \ar[r]^-{\iota',\cong}& \iota'(\overline{\nu}(e(\Sigma_{g,1}))) \ar[r]^-{\subset } & N.
}
\end{equation}
As in Construction~\ref{cons:EmbVBijection}, the choice of framings leads to boundary homeomorphisms
\begin{align*}
& f=(h|)\cup(\gamma| \circ \iota^{-1}) \colon \partial N_{e_{\iota}(\Sigma_{g,1})} \xrightarrow{\cong} M_{K,g}, \\
& f'=(h|)\cup(\gamma'| \circ {\iota'}^{-1} ) \colon \partial N_{e'_{\iota'}(\Sigma_{g,1})} \xrightarrow{\cong} M_{K,g}.
\end{align*}
As in~\eqref{eq:VerificationTheta},  using the diagram from~\eqref{eq:DiagramVerificationThetaRelBoundary} and the fact that $F$ is a rel.\ boundary homeomorphism, we deduce that $F|=f'^{-1}\circ f$ and that $F$ restricts to a rel.\ boundary homeomorphism
%AC: The diagram tells you its id on the \Sigma_{g,1} \times S^1 part also.
$$F|\colon N_{e_{\iota}(\Sigma_{g,1})}\to N_{e'_{\iota'}(\Sigma_{g,1})}.$$
We conclude that
%$F$ induces a rel.\ boundary homeomorphism, which shows that
$(N_{e(\Sigma_{g,1})},f)$ is equivalent to $(N_{e'(\Sigma_{g,1})},f')$ in $\mathcal{V}^0_\lambda(M_{K,g})$.
\end{itemize}
This concludes the verification that the map $\Theta$ from Construction~\ref{cons:EmbVBijection} is well defined.

\begin{remark}
\label{rem:OmitEmbedding}
From now on,  we continue to use the notation $\Sigma:=e(\Sigma_{g,1})$ and we omit the choice of an embedding $\iota \colon \overline{\nu}(\Sigma_{g,1}) \hookrightarrow N$ from the notation since we have shown that $\Theta(e)$ is independent of the choice of embedding $\iota$ up to equivalence in $\mathcal{V}^0_\lambda(M_{K,g})$.
In practice this means that we will simply write $\overline{\nu}(\Sigma) \subset N$.
Since we omit~$\iota$ from the notation, we also allow ourselves to think of (the inverse of) a good framing $\gamma$ as giving an embedding
$$ \gamma^{-1} \colon \Sigma_{g,1} \times D^2 \hookrightarrow \overline{\nu}(\Sigma) \subset N.$$
Similarly,  given a choice of such a good framing, we now write the homeomorphism from~\eqref{eq:BoundaryHomeoSurface} as
\begin{equation}
\label{eq:BoundaryHomeoSurfaceNoEmbedding}
 f_\gamma \colon \partial N_{\Sigma} \to M_{K,g},  \ \ \ \ f_\gamma:=(h|) \cup (\gamma|),
 \end{equation}
once again omitting $\iota$ from the notation.
We sometimes also omit the choice of the framing $\gamma$ from the notation, writing instead $\Theta(e)=(N_{\Sigma},f)$.
\end{remark}

\subsubsection*{Step $(2)$: From manifolds to embeddings}

We set up some notation aimed towards proving that~$\Theta$ is a bijection when the form $\lambda$ is even, and that $\Theta$ is a bijection when $\lambda$ is odd and the Kirby-Siebenmann is fixed.
Set~$\varepsilon:=\ks(N)$ and write~$\mathcal{V}^{0,\varepsilon}_\lambda(M_{K,g})$ for the subset of those manifolds in~$\mathcal{V}^{0}_\lambda(M_{K,g})$ whose Kirby-Siebenmann invariant  equals~$\varepsilon$.
Observe that by additivity of the Kirby-Siebenmann invariant (see e.g.~\cite[Theorem 8.2]{FriedlNagelOrsonPowell}), if $\lambda$ is odd and~$\Sigma\subset N$ is a $\Z$-surface,  then~$\ks(N_\Sigma)=\ks(N)=\varepsilon$, so the image of $\Theta$ lies in $\mathcal{V}^{0,\varepsilon}_\lambda(M_{K,g})$.
The next proposition is the next step in the proof of Theorem~\ref{thm:SurfacesRelBoundary}.

\begin{proposition}
\label{prop:EmbVBijections}
Let~$N$ be a simply-connected~$4$-manifold with boundary~$\partial N \cong S^3$,  let~$K \subset S^3$ be a knot and let~$(H,\lambda)$ be  a nondegenerate Hermitian form with $\lambda(1) \cong Q_N \oplus (0)^{2g}.$
\begin{enumerate}
\item If~$\lambda$ is even, then the map~$\Theta$ from Construction~\ref{cons:EmbVBijection} determines a bijection
$$\operatorname{Emb}_\lambda^0(\Sigma_{g,1},N;K) \to \mathcal{V}^0_\lambda(M_{K,g}).$$
\item If~$\lambda$ is odd, then the map~$\Theta$ from Construction~\ref{cons:EmbVBijection} determines a bijection
$$\operatorname{Emb}_\lambda^0(\Sigma_{g,1},N;K) \to \mathcal{V}^{0,\varepsilon}_\lambda(M_{K,g}),$$
where $\varepsilon=\ks(N)$.
\end{enumerate}
\end{proposition}

\begin{proof}
We construct an inverse~$\Psi$ to the assignment~$\Theta \colon e \mapsto (N_{e(\Sigma_{g,1})},f)$ from Construction~\ref{cons:EmbVBijection}; this will in fact take up most of the proof.
Let~$(W,f)$ be a pair, where~$W$ is a~$4$-manifold with fundamental group~$\pi_1(W)\cong\Z$,  equivariant intersection form~$\lambda_W\cong \lambda$ and,  in the odd case,  Kirby-Siebenmann invariant~$\ks(W)=\varepsilon$, and~$f  \colon \partial W \cong M_{K,g}$ is a homeomorphism.

The inverse~$\Psi(W,f)$ is an embedding $\Sigma_{g,1} \hookrightarrow N$ defined as follows.
Glue~$\Sigma_{g,1} \times D^2$ to~$W$ via the homeomorphism~$f^{-1}|_{\Sigma_{g,1}\times S^1}$.
%a homeomorphism which identifies~$\Sigma_{g,1}\times\partial D^2$ with the~$\Sigma_{g,1}\times S^1$
% in the definition of~$M_{K,g}$
This produces a~$4$-manifold~$\widehat{W}$ with boundary~$\partial \widehat{W}=(\partial W \setminus f^{-1}(\Sigma_{g,1}\times S^1)) \cup (\partial \Sigma_{g,1} \times D^2)$, together with an embedding
$$ \times \lbrace 0 \rbrace  \colon \Sigma_{g,1}\hookrightarrow \widehat{W} \ \ \ \ x \mapsto  (x,0) \in \Sigma_{g,1}\times \{0\} \subset \Sigma_{g,1} \times D^2.$$
% given by~$e(x)= (x,0) \in \Sigma_{g,1}\times \{0\} \subset \Sigma_{g,1} \times D^2$.
Note for now that~$\partial \Sigma_{g,1}\times \{0\}\subset\partial\widehat{W}$  bounds a genus~$g$~$\Z$-surface in $\widehat{W}$ (with exterior~$W$).
%; we will identify this curve later.

We will use the homeomorphism~$f \colon \partial W \to  M_{K,g}$ to define a homeomorphism~$f' \colon  \partial \widehat{W} \to \partial N~$ and then use  Freedman's classification of compact simply-connected 4-manifolds with~$S^3$ boundary,  to deduce that this homeomorphism extends to a homeomorphism~$F\colon \widehat{W} \to N$. We will then take our embedding to be
$$\Psi(W,f)  :=F\circ  (\times \lbrace 0 \rbrace)
\colon \Sigma_{g,1}\hookrightarrow N.$$
The next paragraphs flesh out the details of this construction.
Namely, firstly we build $f' \colon \partial \widehat{W} \to \partial N$ and secondly we argue it extends to a homeomorphism $F \colon \widehat{W} \to N$.
\begin{itemize}
\item Towards building this $f'$, first observe that we get a natural homeomorphism $\partial\widehat{W}\to S^3$ as follows.
Restricting~$f$  gives a homeomorphism~$f| \colon \partial W \setminus f^{-1}(\Sigma_{g,1}\times S^1) \cong S^3\setminus \nu(K)$.
Recall that the homeomorphism $D \colon \partial \Sigma_{g,1} \times S^1 \to \partial \overline{\nu}(K)$ sends $\partial\Sigma_{g,1}\times\{\operatorname{pt}\}$ to $\lambda_K$ and~$\{\cdot\}\times\partial D^2$ to $\mu_K$, where $\lambda_K$ and $\mu_K$ respectively denote the Seifert longitude and meridian of $K\subset S^3$.
Since~$\mu_K$ bounds a disc in $\overline{\nu}(K)$, this homeomorphism extends to a homeomorphism
\color{black}
\begin{equation}
\label{eq:varphi}
\vartheta\colon \partial\Sigma_{g,1}\times D^2\to \overline{\nu}(K).
\end{equation}
%%%Arxiv version below
%We can also define a homeomorphism
%\begin{equation}
%\label{eq:varphi}
%\vartheta\colon \partial\Sigma_{g,1}\times D^2\to \overline{\nu}(K)
%\end{equation}
% which sends $\partial\Sigma_{g,1}\times\{\operatorname{pt}\}$ to $\lambda_K$ and $\{\cdot\}\times\partial D^2$ to $\mu_K$, where $\lambda_K$ and $\mu_K$ denote the Seifert longitude and meridian of $K\subset S^3$.
%%%Don't delete:
%%Think of $S^3 \setminus h(\overline{\nu}(K))$ as $M_{K,g}\setminus \Sigma_{g \times S^1}$ in $M_{K,g}$ we have those identifications where the meridian and longitude show up....that's where stuff is mapped to by $f$.
%Specifying the image of these curves determines a map~$\partial\Sigma_{g,1}\times \partial D^2\to \overline{\nu}(K)$ (up to isotopy) which then extends to~$\vartheta\colon \partial\Sigma_{g,1}\times D^2 \to \overline{\nu}(K)$ because~$\mu_K$ bounds a disc in $\overline{\nu}(K)$.
Note that $\vartheta$ is well defined up to isotopy.
Consider the following diagram:
$$
\xymatrix{
\partial W \setminus f^{-1}(\Sigma_{g,1}^\circ \times S^1) \ar[r]^-{f|,\cong}& S^3\setminus \nu(K) \\
%%%%%
\partial \Sigma_{g,1} \times S^1 \ar[u]_{f^{-1}|_{\partial \Sigma_{g,1} \times S^1}} \ar[d]^{\subset} \ar[r]^{D,\cong}& \partial \overline{\nu} (K)  \ar[u]_{\subset} \ar[d]^{\subset}  \\
%%%%%%
\partial \Sigma_{g,1} \times D^2 \ar[r]^{\vartheta,\cong}& \overline{\nu}(K).
}
$$
The bottom square commutes by definition of $\vartheta$, whereas the top square commutes because~$f|$ is obtained by restricting $f \colon \partial W \to M_{K,g}=(S^3\setminus \nu(K)) \cup_D \Sigma_{g,1} \times S^1$.
The commutativity of this diagram implies that~$f$ and~$\vartheta$ combine to a homeomorphism
$$ f|\cup\vartheta \colon =\partial\widehat{W} \to S^3.$$
\color{black}
% The union~$f|\cup\vartheta \colon \partial\widehat{W} \to S^3$ is continuous on $\partial\Sigma_{g,1}\times\partial D^2$ because the gluing used to define~$\widehat{W}$ was~$f^{-1}$.
Then~$h^{-1} \circ (f|\cup\vartheta)$ gives the required homeomorphism
$$f':= h|^{-1} \circ (  f|  \cup \vartheta) \colon \partial \widehat{W} \to \partial N.$$ Further, we observe that $f'(\partial\Sigma_{g,1})=K$.

\item To prove that this homeomorphism extends to a homeomorphism~$\widehat{W} \cong N$, we will appeal to Freedman's theorem that for every pair of simply-connected topological~$4$-manifolds with boundary homeomorphic to~$S^3$, the same intersection form,
% the same type
and the same Kirby-Siebenmann invariant, every homeomorphism between the boundaries extends to a homeomorphism between the 4-manifolds~\cite{Freedman}.
%BoyerUniqueness.
 We check now that the hypotheses are satisfied.

%\color{black}
First, we argue that~$\widehat{W}$ is simply-connected.
The hypothesis that~$W$ lies in~$\mathcal{V}^0_\lambda(M_{K,g})$ implies that there is an isomorphism~$\widehat{\varphi} \colon \pi_1(W) \xrightarrow{\cong} \Z$ such that~$\varphi=\widehat{\varphi}\circ\kappa$, where~$\kappa$ is the inclusion induced map~$\pi_1(M_{K,g})\to\pi_1(W)$ (see Definition \ref{def:V0lambdaY}). Since we required that~$\varphi(\mu_K)$ generates~$\Z$, we must have that~$\kappa(\mu_K)$ generates~$\pi_1(W)\cong\Z$.
Since gluing~$\Sigma_{g,1} \times D^2$ along~$\Sigma_{g,1} \times S^1$ has the effect of killing~$\kappa(\mu_K)$, we conclude that~$\widehat{W}$ is simply-connected as claimed.

Next we must show that~$Q_{\widehat{W}}$ is isometric to~$Q_N$.
A Mayer-Vietoris argument establishes the isometry~$Q_{\widehat{W}}\oplus (0)^{\oplus 2g}\cong Q_W$.
%AC: kind of straightforwards but I wrote it up in a note.
It then follows from our assumption on the Hermitian form~$(H,\lambda)$ that  we have the isometries
$$   Q_{\widehat{W}} \oplus (0)^{\oplus 2g} \cong Q_W \cong \lambda_W(1) \cong \lambda(1) \cong Q_N \oplus  (0)^{\oplus 2g}.$$
This implies that~$Q_{\widehat{W}} \cong Q_N$ because both forms are nonsingular (indeed~$\partial \widehat{W}\cong \partial N \cong S^3$).
%AC: The details are in a note but basically write out the isometry as a 2x2 block matrix and then use the fact that it is an isometry to deduce that the upper left block itself preserves the forms.  Proving bijectivity is then ok.

In the even case, we deduce that both~$\widehat{W}$ and~$N$ are spin.
In the odd case, using the additivity of the Kirby-Siebenmann invariant (see e.g.~\cite[Theorem 8.2]{FriedlNagelOrsonPowell}), we have~$\ks(\widehat{W})=\ks(W)=\varepsilon=\ks(N)$.

Therefore~$\widehat{W}$ and~$N$ are simply-connected topological~$4$-manifolds with boundary~$S^3$, with the same intersection form
%, the same type
and the same Kirby-Siebenmann invariant. Freedman's classification of simply-connected~$4$-manifolds with boundary~$S^3$ now ensures that the homeomorphism~$f'  \colon \partial   \widehat{W} \to \partial N$ extends to a homeomorphism $F\colon \widehat{W} \to N$ that induces the isometry $Q_{\widehat{W}} \cong Q_N$ and fits into the following commutative diagram
\begin{equation}
\label{eq:DiagramForf'}
\xymatrix{
(\partial W \setminus f^{-1}(\Sigma_{g,1}\times S^1)) \cup (\partial \Sigma_{g,1} \times D^2) \ar[r]^-{=}\ar[d]^-{h|^{-1}  \circ  (f|  \cup \vartheta|)}& \partial \widehat{W} \ar[r]^\subset \ar[d]^{f'}& \widehat{W}\ar[d]^-{F} \\
%%%%%%%%%
(\partial N\setminus \nu(K)) \cup \overline{\nu}(K)\ar[r]^-{=}&  \partial N \ar[r]^\subset &N.
}
\end{equation}
\end{itemize}
As mentioned above,  we obtain an embedding as
\begin{equation}
\label{eq:DefOfPsi}
\Psi(W,f):=\Big{(}e\colon \Sigma_{g,1} \xrightarrow{\times \lbrace 0\rbrace} \widehat{W} \xrightarrow{F,\cong}N\Big{)}.
\end{equation}
% where the first embedding~$\times \lbrace 0\rbrace \colon \Sigma_{g,1}\hookrightarrow \widehat{W}=W \cup_f (\Sigma_{g,1} \times D^2)$ was given by~$\iota(x)= (x,0) \in \Sigma_{g,1}\times \{0\} \subset \Sigma_{g,1} \times D^2$.
 This concludes the construction of our embedding~$\Psi(W,f)$.

We must check that this construction gives rise to a map~$\Psi \colon \mathcal{V}^0_\lambda(M_{K,g})\to \operatorname{Emb}_\lambda^0(\Sigma_{g,1},N;K).$
In other words, we verify that,  up to homeomorphisms of $N$ rel.\ boundary,  the embedding $e$ from~\eqref{eq:DefOfPsi} depends neither on the choice of isometry~$Q_{\widehat{W}} \cong Q_N$ nor the choice of $\vartheta$ from \eqref{eq:varphi} nor the homeomorphism~$\widehat{W} \cong N$ extending our boundary homeomorphism nor on the homeomorphism rel.\ boundary type of~$(W,f)$.

\begin{itemize}
\item  The precise embedding~$e$ depends on the homeomorphism~$\widehat{W} \cong N$ chosen to extend a given $f'$.
 This homeomorphism in turn depends on the choice of isometry~$Q_{\widehat{W}} \cong Q_N$.
 However for any two choices $F_1$ and $F_2$ of homeomorphisms~$\widehat{W} \cong N$ extending $f'$,  the resulting embeddings are equivalent rel. boundary, as can be seen by composing one choice of homeomorphism with the inverse of the other:
$$
\xymatrix@R0.5cm{
\Sigma \ar[r]^{[\times 0]} \ar[d]^=&\widehat{V} \ar[r]^{F_1}\ar[d]^=& W \ar[d]^{F_2 \circ F_1^{-1}}\\
%%%
\Sigma \ar[r]^{[\times 0]} &\widehat{V} \ar[r]^{F_2}& W.
}
 $$
So the equivalence class of the surface~$\Psi(W,f)$ does not depend on the choice of isometry~$Q_{\widehat{W}} \cong Q_N$ nor on the choice of homeomorphism~$\widehat{W} \cong N$ realizing this isometry and extending $f'$.
\item Next, we show that the definition is independent of the choice of $\vartheta \colon  \partial \Sigma_{g,1} \times D^2 \to \overline{\nu}(K)$ within its isotopy class.
If $\vartheta_0,\vartheta_1 \colon \partial \Sigma_{g,1} \times D^2 \to \overline{\nu}(K)$ are isotopic, then so are the resulting homeomorphisms $f_0':=(f| \cup \vartheta_0|),f_1':=(f| \cup \vartheta_1|) \colon \partial \widehat{W} \to \partial N$ via an isotopy $f_s'$.

\begin{claim*}
There is an isotopy $F_s \colon \widehat{W} \to N$ extending $f_s'$.
\end{claim*}
\begin{proof}
Pick a homeomorphism $F_0 \colon  \widehat{W} \to N$ extending $f_0'$; when we constructed~$\Psi(W,f)$, we argued that such an $F_0$ exists.
There are collars $\partial \widehat{W} \times [0,1]$ and $\partial N \times [0,1]$ such that $F_0|_{\partial \widehat{W} \times [0,1]}=f_0' \times [0,1]$.
Here it is understood that the boundaries of $\widehat{W} $ and $N$ are respectively given by $\partial \widehat{W} \times \lbrace 0 \rbrace$ and $\partial N \times \lbrace 0 \rbrace$.

The idea is to implant the isotopy $f_s'$ between $f_0',f_1'$ in these collars in order to obtain an isotopy between $F_0$ and a homeomorphism $F_1$ that retricts to $f_1'$ on the boundary.
To carry out this idea, consider the restriction
$$ F_0| \colon \widehat{W} \setminus (\partial \widehat{W} \times [0,1]) \to N \setminus (\partial N \times [0,1]).$$
%and let $f_s' \colon \partial \widehat{W} \to \partial N$ be an isotopy of homeomorphisms between $f_0'$ and $f_1'$.
Define an isotopy of homeomorphisms between the collars via the formula
%{LP:Cleaner definition wins}
%{AC: This is Mark's formula, I had come up with $G_s(x,t)=(f_{(1-t)s}'(x),t)$. Do they do the same? MP: Seems to be they do more or less the same thing. }
%AC: Mine also has
%G_s(x,1)=(f_0',1)
% $G_0(x,t)=(f_0(x),t)$
%G_1(x,0)=(f_1,0).
\begin{align*}
G_s \colon \partial \widehat{W} \times [0,1] &\to \partial N \times [0,1] \\
(x,t) &\mapsto
(f_{(1-t)s}'(x),t).
\end{align*}
%\begin{align*}
%G_s \colon \partial \widehat{W} \times [0,1] &\to \partial N \times [0,1] \\
%(x,t) &\mapsto
%\begin{cases}
%(f_{s-t}'(x),t) &\quad  t \in [0,s]  \\
%(f_{0}'(x),t) &\quad  t \in [s,1].
%\end{cases}
%\end{align*}
Since $G_s(x,1)=(f_0'(x),1)$ for every $s$,  we obtain the required isotopy as $F_s:=G_s \cup F_0$.
By construction $F_i$ restricts to $f_i'$ on the boundary for $i=0,1$, thus concluding the proof of the claim.
%because G_0(x,t)=f_0'(x,t)$ (in particular good for t=0) and $G_1(x,0)=(f'_{1},0)$.
%\color{black}
\end{proof}
%The previous paragraph ensures that up to equivalence rel.\ boundary,
Thanks to the claim, we can use $F_0$ and $F_1$ to define the embeddings  $e_0:=F_0 \circ (\times \lbrace 0 \rbrace)$ and $e_1:=F_1 \circ (\times \lbrace 0 \rbrace)$.
This way,  $F_1 \circ F_0^{-1} \colon N \to N$ is an equivalence rel.\ boundary between $e_0$ and $e_1$ so that the definition of $\Psi$ is independent of the choice of $\vartheta$ within its isotopy class.

\item Next we check the independence of the rel.\ boundary homeomorphism type of~$(W,f)$.
If we have~$(W_1,f_1)$ and~$(W_2,f_2)$ that are equivalent rel.\ boundary, then there is a homeomorphism~$\Phi \colon W_1 \to W_2$ that satisfies~$f_2 \circ \Phi| =f_1$.
This homeomorphism extends to~$\widehat{\Phi}:=\Phi \cup \id_{\Sigma_{g,1} \times D^2} \colon \widehat{W}_1 \to \widehat{W}_2$ and therefore to a homeomorphism~$N \to N$ that is, by construction rel.\ boundary.
%AC: To see that $\Phi$ extends use $f_2 \circ \Phi| =f_1$ and the fact that $\widehat{W}_i=W_i \cup_{f_i} \Sigma_{g,1} \times D^2$
A formal verification using this latter homeomorphism then shows that the embeddings~$\Psi(W_1,f_1)$ and~$\Psi(W_2,f_2)$ are equivalent rel.\ boundary.
%AC: I uploaded the verification to the dropbox.
\end{itemize}

Now we prove that the maps~$\Theta$ and~$\Psi$ are mutually inverse.

\begin{itemize}
\item First we prove that~$\Psi \circ \Theta=\id$.
Start with an embedding~$e \colon \Sigma_{g,1} \hookrightarrow N$ and write~$\Theta(e)=(N_{e(\Sigma_{g,1})},f)$ with~$f=(h|) \cup (\gamma|) \colon \partial N_{e(\Sigma_{g,1})} \to M_{K,g}$ the homeomorphism described in Construction~\ref{cons:EmbVBijection}.
Then~$\Psi(\Theta(e))$ is an embedding
$$ \Sigma_{g,1} \xrightarrow{\times \lbrace 0 \rbrace} N_{e(\Sigma_{g,1})} \cup_f (\Sigma_{g,1} \times D^2) \xrightarrow{F,\cong} N.$$
We showed that the equivalence class of this embedding is independent of the homeomorphism~$F$ that extends~$f$.
It suffices to show that we can make choices so that $\Psi(\Theta(e))$ recovers $e$. This can be done explicitly as follows.
Choose $\vartheta:=h\circ\gamma^{-1}  \colon  \partial \Sigma_{g,1} \times D^2 \to \overline{\nu}(K)$.
Then we have
$f'=\id_{\partial N \setminus \nu(K)}  \cup (h^{-1} \circ (h \circ \gamma^{-1}))=\id_{\partial N \setminus \nu(K)} \cup \gamma|^{-1}$ where the notation is as in~\eqref{eq:DiagramForf'} (with $W=N_{e(\Sigma_{g,1})}$).
We already know an extension of $f'$, namely $\id_{N_{e(\Sigma_{g,1})}} \cup \gamma^{-1}$, which we take to be $F$.
Thus $\Psi(\Theta(e))=\gamma^{-1}|_{\Sigma_{g,1} \times \lbrace 0 \rbrace} \colon \Sigma_{g,1} \hookrightarrow N$ which, by definition of a normal bundle,  agrees with the initial embedding $e$.

\item Next we prove that~$\Theta \circ \Psi=\id$.
This time we start with a pair~$(W,f)$ consisting of a 4-manifold~$W$ and a homeomorphism~$f \colon  \partial W \to M_{K,g}$.
Then~$\Psi(W,f)$ is represented by an embedding ~$e \colon \Sigma_{g,1} \xrightarrow{ \times \lbrace 0 \rbrace} \widehat{W} \xrightarrow{F,\cong} N$.
Recall that we write~$h \colon \partial N \to S^3$ for our preferred homeomorphism and that by construction,  on the boundaries,~$F$ restricts to
$$h|^{-1} \circ (f|  \cup \vartheta) \colon \partial \widehat{W} \to \partial N$$
where (the isotopy class of)~$\vartheta \colon \partial \Sigma_{g,1} \times D^2 \to \overline{\nu}(K)$ satisfies the properties listed below equation~\eqref{eq:varphi}.

We frame $\Sigma_{g,1} \times \lbrace 0 \rbrace \subset \widehat{W}$ via the unique homeomorphism $\operatorname{fr}\colon\overline{\nu}(\Sigma_{g,1}\times\{0\})\to\Sigma_{g,1}\times D^2$ that makes the following diagram commute:
$$
\xymatrix{ \overline{\nu}(\Sigma_{g,1}\times \lbrace 0 \rbrace)
 \ar[rr]^{\operatorname{fr}}\ar[dr]^{\operatorname{incl} }&&  \Sigma_{g,1} \times D^2  \ar[dl]_{\operatorname{incl} } \\
&\widehat{W}=W\cup(\Sigma_{g,1} \times D^2 ).&
}
$$
We then frame $e(\Sigma_{g,1}) \subset N$ via
%$\Sigma \subset N$ via
$$\gamma:= \operatorname{fr} \circ F^{-1}|  \colon \overline{\nu}(e(\Sigma_{g,1})) \cong \Sigma_{g,1} \times D^2.$$
This framing is good thanks to the definition of $\varphi \colon \pi_1(M_{K,g}) \to \Z$ as the unique epimorphism that maps the meridian of $K$ to $1$ and the other generators to zero: indeed this implies that the curves on $\Sigma_{g,1} \times \lbrace 0\rbrace$ are nullhomologous in $W$ and therefore the same thing holds for $e(\Sigma_{g,1})\subset N$.
%AC: Use that $\varphi$ factors through $\pi_1(W)$.
It can be verified that this framing satisfies the condition from~\eqref{eq:Compatible}.
%%Don't delete
%%AC: Here is the verification.
%Going down (with intermediate step through $\widehat{W}$) then right, then up we have
%\Sigma_{g,1} \times S^1  \to h(f( \Sigma_{g,1} \times S^1)). Here, I used the definition of the framing and the defiintion of $F$ (extending $f'$ which itself extends $f \colon W \to M_{h(K),g}$).
%Going right using $D$ gives $D(\partial \Sigma_{g,1} \times S^1)$.
%We must therefore why $h(f( \Sigma_{g,1} \times S^1))=D(\partial \Sigma_{g,1} \times S^1)$.
%That's because when we built $\widehat{W}$ we used $f$ for the gluing so $f( \Sigma_{g,1} \times S^1)$ can be thought of as \partial \overline{\nu}(K).
%%Here I went a bit quick because technically $f$ lands in $M_{h(K),g}$ but ultimately we're thinking of its homeomorphic image in $\partial N$ (which is why there was no $h$).
%So down then right then up is in fact $h(\overline{\nu}(K))\subset M_{h(K},g}=E_{h(K)}\cup_D \Sigma_{g,1} \times S^1$ so under this identification we get  D(\partial \Sigma_{g,1} \times S^1).

We then obtain~$\Theta(\Psi(W,f))=(N_{\Sigma} := N \setminus \nu(e(\Sigma_{g,1})),h| \cup \gamma|)$, where, as dictated by Construction~\ref{cons:EmbVBijection}, the boundary homeomorphism is~$h| \cup \gamma| \colon \partial N_\Sigma \to M_{K,g}$.
Here we are making use of the fact that up to equivalence, we can choose any framing in the definition of~$\Theta$.

We have to prove that~$(N_{\Sigma},h| \cup \gamma|)$ is homeomorphic rel.\ boundary to~$(W,f)$.
We claim that the restriction of~$F \colon \widehat{W} \to N$ gives the required homeomorphism.
To see this, consider the following diagram
% which stems from the definition of $F$:
$$
\xymatrix @C+0.3cm{
M_{K,g} \ar[d]^=& (\partial W \setminus f^{-1}(\Sigma_{g,1} \times S^1)) \cup (f^{-1}(\Sigma_{g,1} \times S^1))\ar[l]_-{f,\cong} \ar[r]^-{=}  \ar[d]^{f':=(h|^{-1} \circ f|) \cup F| }& \partial W \ar[r]^{\subset} \ar[d]^{F|}&  W \ar[r]^{\subset} \ar[d]^{F|}& \widehat{W} \ar[d]^F \\
%%%%%%%%%%%%%
M_{K,g} & (\partial N \setminus \nu(K)) \cup (\partial \overline{\nu}(\Sigma) \setminus (\nu(\Sigma) \cap \partial N)))  \ar[r]^-{=} \ar[l]_-{h| \cup \gamma|}& \partial N_\Sigma \ar[r]^{\subset}& N_\Sigma \ar[r]^{\subset} &N.
}
$$
The right two squares certainly commute. In the second-from-left square, we have just expanded out $\partial W$ and $\partial N_\Sigma$, as well as written $F|$ explicitly on the regions where we have an explicit description from the construction of $\Psi$. So this square commutes.

It remains to argue that the left square commutes.
By construction $F|_{\partial \widehat{W}}=f'=  h^{-1} \circ (f| \cup \vartheta)$.
%$(h^{-1} \circ f|) \cup \vartheta)$.
Thus on the knot exteriors, we have that $F|=h^{-1} \circ f|$ and so the left portion of the square commutes on the knot exteriors.

Now it remains to prove that $\gamma| \circ F|=f$.
By definition of $\gamma=\operatorname{fr} \circ F^{-1}$,  we must show that $\operatorname{fr}|=f|$ on  $f^{-1}(\Sigma_{g,1} \times S^1)$.
First note that $\operatorname{fr}$ has domain $\overline{\nu}(\Sigma_{g,1} \times \lbrace 0 \rbrace)  \subset \widehat{W}=W\cup (\Sigma_{g,1} \times D^2)$, so it appears we are attempting to compare maps which have different domains. However, the definition of $\widehat{W}$ identifies the portion of the boundary of $\overline{\nu}(\Sigma_{g,1})$ that we are interested in with $f^{-1}(\Sigma_{g,1}\times S^1)\subset \partial W$ via $f^{-1}|\circ \operatorname{fr|}$, so it makes sense to compare $f$ on $f^{-1}(\Sigma_{g,1} \times S^1)$ with $\operatorname{fr|}$ on $\operatorname{fr|}^{-1}\circ f|_{f^{-1}(\Sigma_{g,1} \times S^1)}$. These maps are tautologically equal.
%To obtain the effect on $f^{-1}(\Sigma_{g,1} \times S^1) \subset \partial W$, we use the identification of that space with $\Sigma_{g,1} \times S^1 \subset \widehat{W}$ via $f$ so that thought this way $\operatorname{fr}$ is precisely $f$.
%We must therefore understand the effect of $F|$ on $f^{-1}(\Sigma_{g,1} \times S^1)$.
%On $\partial \Sigma_{g,1} \times D^2 \subset \widehat{W}$, the map $F|$ agrees with $h|^{-1} \circ \vartheta|$.
%To apply this to a subspace of $\partial W$, we therefore use $f$ for the identification so $\gamma \circ F|_{f^{-1}(\Sigma_{g,1} \times S^1)}$
Therefore the left hand side of the diagram commutes and  this concludes the proof that~$\Theta \circ \Psi=\id$.
\end{itemize}

We have shown that~$\Theta$ and~$\Psi$ are mutually inverse, and so both are bijections.
This completes the proof of Proposition~\ref{prop:EmbVBijections}.
\end{proof}

\subsubsection*{Step $(3)$: From embeddings to submanifolds}
Now we deduce a description of~$\operatorname{Surf(g)}^0_\lambda(N,K)$ from Proposition~\ref{prop:EmbVBijections}.
Note that $ \operatorname{Surf(g)}^0_\lambda(N,K)$ arises as the orbit set
%%%%%
%%Don't delete
%{AC: Here and throughout: one could upgrade to a mapping class group action.
%Indeed one could also write $\operatorname{Surf(g)}^0_\lambda(N,K)= \operatorname{Emb}_\lambda^0(\Sigma_g,N;K)/\operatorname{MCG}^+(\Sigma_g,\partial)$
%Here is the argument.
%If $x_0,x_1 \colon \Sigma_{g,1} \to \Sigma_{g,1}$ are isotopic homeos then $e \circ x_0,e\circ x_1 \colon \Sigma_{g,1} \hookrightarrow N$ are isotopic embeddings and are therefore equivalent by isotopy extension (which gives a homeo $G \colon N \to N$ taking $e \circ x_0$ to $e \circ x_0$.  Does isotopy extension ensure that $G$ can be taken to be $\id$ on $N$ for the same reason that uniqueness  of tubular ngbhs is rel.\ boundary.}
%Thus the next two propositions could  be stated using mapping class groups instead of Homeo, but as we remarked after the main statement, we already knew that.
%%%%
$$ \operatorname{Surf(g)}^0_\lambda(N,K)= \operatorname{Emb}_\lambda^0(\Sigma_{g,1},N;K)/\operatorname{Homeo}^+(\Sigma_{g,1},\partial),$$
where the left action of~$x \in \operatorname{Homeo}^+(\Sigma_{g,1},\partial)$ on~$e\in \operatorname{Emb}_\lambda^0(\Sigma_{g,1},N;K)$ is defined by $x \cdot e=e \circ x^{-1}$.
There is a surjective map~$\operatorname{Emb}_\lambda^0(\Sigma_{g,1},N;K) \to  \operatorname{Surf(g)}^0_\lambda(N,K)$ that maps an embedding~$e \colon \Sigma_{g,1} \hookrightarrow N$ onto its image.
One then verifies that this map descends to a bijection on the orbit set.
%AC: Assume e,e' \colon \Sigma_{g,1} \hookrightarrow agree in the quotient.  This means that there is a rel.\ boundary homeo h \colon N \to N such that h \circ e =e'.  This precisely means that the surfaces e(\Sigma_{g,1}) and e'(\Sigma_{g,1}) are equivalent rel.\ boundary.

Next, we note that~$\operatorname{Homeo}^+(\Sigma_g,\partial)$ acts on the sets~$\mathcal{V}^0_\lambda(M_{K,g})$ and~$\mathcal{V}^{0,\varepsilon}_\lambda(M_{K,g})$
as follows.
%by postcomposition.
A rel.\ boundary homeomorphism~$x \colon \Sigma_{g,1} \to \Sigma_{g,1}$ extends to a self homeomorphism $x'$ of $\Sigma_{g,1}\times S^1$ by defining $x'(s,\theta)=(x(s),\theta)$.
Then extend $x'$ by the identity over $E_K$; in total one obtains a self homeomorphism $x''$ of $M_{K,g}$.
The required action is now by postcomposition: for $(W,f)$ representing an element of $\mathcal{V}^0_\lambda(M_{K,g})$ or~$\mathcal{V}^{0,\varepsilon}_\lambda(M_{K,g})$, define $x \cdot (W,f):=(W,x'' \circ f )$.

%More precisely, given a pair~$(W,f)$ with~$f \colon \partial W \to M_{K,g}$ a homeomorphism,
%extend a homeomorphism~$x \in \operatorname{Homeo}^+(\Sigma_{g,1},\partial)$ to a self homeomorphism $x'$ of $\Sigma_{g,1}\times S^1$ by defining $x'(s,\theta)=(x(s),\theta)$.
%Then extend $x'$ by the identity over $E_K$; in total one obtains a self homeomorphism $x''$ of $M_{K,g}$.

%extend a homeomorphism~$x \in \operatorname{Homeo}^+(\Sigma_{g,1},\partial)$ by the identity to obtain a homeomorphism~$M_{K,g} \to M_{K,g}$ (that we also denote by $x$) and define the action by~$x \cdot (W,f)=(W,x \circ f )$.

The following proposition is now a relatively straightforward consequence of Proposition~\ref{prop:EmbVBijections}.

%\color{black}
\begin{proposition}
\label{prop:SurfBijectionCorrected}
Let~$N$ be a simply-connected~$4$-manifold with boundary~$\partial N \cong S^3$,  let~$K \subset S^3$ be a knot and let~$(H,\lambda)$ be  a nondegenerate Hermitian form with $\lambda(1) \cong Q_N \oplus (0)^{2g}.$
\begin{enumerate}
\item If~$\lambda$ is even, then the map~$\Theta$ from Construction~\ref{cons:EmbVBijection} descends to a bijection
$$\operatorname{Surf(g)}^0_\lambda(N,K) \to \mathcal{V}^0_\lambda(M_{K,g})/\operatorname{Homeo}^+(\Sigma_{g,1},\partial).$$
\item If~$\lambda$ is odd, then the map~$\Theta$ from Construction~\ref{cons:EmbVBijection} descends to a bijection
$$\operatorname{Surf(g)}^0_\lambda(N,K) \to \mathcal{V}^{0,\varepsilon}_\lambda(M_{K,g})/\operatorname{Homeo}^+(\Sigma_{g,1},\partial),$$
where $\varepsilon=\ks(N)$.
\end{enumerate}
\end{proposition}

\begin{proof}
Thanks to Proposition~\ref{prop:EmbVBijections}, it is enough to check that~$\Theta(x \cdot e)=x \cdot \Theta(e)$ for~$x \in \Homeo(\Sigma_{g,1},\partial)$ and~$e \colon \Sigma_{g,1} \hookrightarrow N$ an embedding representing an element of~$\operatorname{Emb}_\lambda^0(\Sigma_{g,1},N;K)$.
By definition of~$\Theta$, we know that~$\Theta(x \cdot e)$ is~$(N_{e(x^{-1}(\Sigma_{g,1}))},f_{e\circ x^{-1}})$ and~$x \cdot \Theta(e)=(N_{e(\Sigma_{g,1})},x'' \circ f_{e})$ where the~$f_e,f_{e \circ x^{-1}}$ are homeomorphisms from the boundaries of these surface exteriors to~$M_{K,g}$ that can be constructed,  up to equivalence rel.\ boundary,  using any choice of good framing; recall Construction~\ref{cons:EmbVBijection}.
In what follows, we will make choices of framings so that the pairs~$\Theta(x \cdot e)=(N_{e(x^{-1}(\Sigma_{g,1}))},f_{e \circ x^{-1}})$ and~$x \cdot \Theta(e)=(N_{e(\Sigma_{g,1})},x'' \circ f_e)$ are equivalent rel.\ boundary.

Pick a good framing $\gamma \colon \overline{\nu}(e(\Sigma_{g,1})) \cong \Sigma_{g,1} \times D^2$ so that $\Theta(e)=(N_{e(\Sigma_{g,1})},f_e)=(N_{e(\Sigma_{g,1})},h| \cup \gamma|)$.
%Since $\gamma^{-1} \colon \Sigma_{g,1} \hookrightarrow N$ gives an embedding of the normal bund
Since $\gamma^{-1} \colon \Sigma_{g,1} \times D^2 \hookrightarrow N$ satisfies $\gamma^{-1}|_{\Sigma_{g,1} \times \lbrace 0 \rbrace}=e$, we deduce that $\gamma^{-1} \circ (x^{-1} \times \id_{D^2})$ gives an embedding of the normal bundle of $e \circ x^{-1}$.
We can therefore choose the inverse $\gamma_{e \circ x}:=(x  \times \id_{D^2}) \circ \gamma$ as a good framing for the embedding $e \circ x^{-1}$.
Using this choice of good framing to construct $f_{e \circ x^{-1}}$, we have~$\Theta(e \circ x^{-1})=(N_{e \circ x^{-1}(\Sigma_{g,1})},h| \cup ((x \times \id_{D^2}) \circ \gamma|))$.
Using these observations and the fact that $x$ is rel.\ boundary,  we obtain
\begin{align*}
  \Theta(x \cdot e)
&=\Theta(e \circ x^{-1})
=(N_{e \circ x^{-1}(\Sigma_{g,1})},h| \cup ((x \times \id_{D^2}) \circ \gamma|))\\
&=(N_{e \circ x^{-1}(\Sigma_{g,1})},x'' \circ (h| \cup \gamma|))
  =x \cdot (N_{e(\Sigma_{g,1})},f_e)
=x \cdot  \Theta(e).
\end{align*}
%AC: I guess that we are using N_{ey{\Sigma_{g,1})}=N_{e{\Sigma_{g,1}}} for any surface homeomorphism $y$; just set theoretic.
%\color{black}
This proves that the pairs~$\Theta(x \cdot e)=(N_{e(x^{-1}(\Sigma_{g,1}))},f_{e \circ x^{-1}})$ and~$x \cdot \Theta(e)=(N_{e(\Sigma_{g,1})},f_e)$ are equivalent rel.\ boundary and thus concludes the proof of the proposition.
\end{proof}

We now deduce our description of the surface set, thus proving the main result of this section.

\begin{proof}[Proof of Theorem \ref{thm:SurfacesRelBoundary}]
%We know that embeddings correspond to the manifold by the long proposition.
%By the second proposition we deduce that submanifolds correspond to the manifold set moded out by homeo.
%It therefore only remains to show that the map~$b$ induces a bijection between manifolds modulo homeo and  bAut modulo homeo.
%The combination of Propositions~\ref{prop:EmbVBijections} and
We have already argued the $(2) \Rightarrow (1)$ direction below Definition~\ref{def:Surface(g)RelBoundary}, and so we focus on the converse.
Since we assumed that $\lambda(1)\cong Q_N \oplus  (0)^{\oplus 2g}$,  we can apply Proposition~\ref{prop:SurfBijectionCorrected} to deduce that
%shows that
 if~$\lambda$ is even then the map~$\Theta$ from Construction~\ref{cons:EmbVBijection} induces a bijection
$$\operatorname{Surf(g)}^0_\lambda(N,K) \to \mathcal{V}^0_\lambda(M_{K,g})/\operatorname{Homeo}^+(\Sigma_{g,1},\partial)$$
while if~$\lambda$ is odd,  for~$\varepsilon:=\ks(N)$, the map~$\Theta$ induces a bijection
$$\operatorname{Surf(g)}^0_\lambda(N,K) \to \mathcal{V}^{0,\varepsilon}_\lambda(M_{K,g})/\operatorname{Homeo}^+(\Sigma_{g,1},\partial).$$
Since we assumed that $(H,\lambda)$ presents $M_{K,g}$,
%Thus 
the theorem will follow from Theorem~\ref{thm:ClassificationRelBoundary} once we show that the map $b \colon V_\lambda^0(M_{K,g}) \to \Iso(\partial \lambda,\unaryminus \Bl_{M_{K,g}})/\Aut(\lambda)$ from Construction~\ref{cons:Invariant} intertwines the~$\Homeo^+(\Sigma_{g,1},\partial)$-actions, i.e. satisfies~$b_{x\cdot (W,f)}=x \cdot b_{(W,f)}$ for every~$x \in \Homeo^+(\Sigma_{g,1},\partial)$ and for every pair~$(W,f)$ representing an element of~$V_\lambda^0(M_{K,g})$.

This follows formally from the definitions of the actions:
on the one hand,  for some isometry $F \colon \lambda \cong  \lambda_W$, we have~$b_{x\cdot (W,f)}=b_{(W,x''\circ f)}=x''_*\circ  f_* \circ D_W \circ \partial F$;  on the other hand,  we have~$x \cdot b_{(W,f)}$ is~$x \cdot (f_* \circ D_W \circ \partial F)$ and this gives the same result.
This concludes the proof of Theorem~\ref{thm:SurfacesRelBoundary}.
\end{proof}

\subsection{Surfaces with boundary up to equivalence}
\label{sub:SurfacesBoundaryEq}

The study of surfaces up to equivalence (instead of equivalence rel.\  boundary) presents additional challenges:
% the last paragraph in the proof of Theorem~\ref{thm:SurfacesRelBoundary}
while there is still a map $\Theta \colon \operatorname{Emb}_\lambda(\Sigma_{g,1},N;K) \to \mathcal{V}_\lambda(M_{K,g})$, the proof of  Proposition~\ref{prop:EmbVBijections} (in which we constructed an inverse $\Psi$ of $\Theta$)  breaks down because if~$W$ and~$W'$ are homeomorphic~$\Z$-fillings of~$M_{K,g}$, it is unclear whether we can always find a homeomorphism~$W \cup (\Sigma_{g,1} \times D^2) \cong W' \cup (\Sigma_{g,1} \times D^2)$.
We nevertheless obtain the following result.

\begin{theorem}
\label{thm:SurfacesWithBoundary}
Let~$N$ be a simply-connected~$4$-manifold with boundary~$\partial N \cong S^3$,  let~$K$ be a knot such that every isometry of~$\Bl_K$ is realised by an orientation-preserving homeomorphism~$E_K \to E_K$
%\textcolor{red}{that is the identity on the boundary, }
and let~$(H,\lambda)$ be a nondegenerate Hermitian form over~$\Z[t^{\pm 1}]$.
 The following assertions are equivalent:
 \begin{enumerate}
\item the Hermitian form~$\lambda$ presents~$M_{K,g}$ and~$\lambda(1)\cong Q_N \oplus  (0)^{\oplus 2g}$;
\item up to equivalence, there exists a unique genus~$g$ surface~$\Sigma \subset N$ with boundary~$K$ and whose exterior has equivariant intersection form~$\lambda$, i.e.~$|\operatorname{Surf(g)}_\lambda(N,K)|=1$.
\end{enumerate}
 \end{theorem}
 \begin{proof}
 We already proved the fact that the second statement implies the first, so we focus on the converse.
We can apply Theorem~\ref{thm:SurfacesRelBoundary} to deduce that~$\operatorname{Surf(g)}^0_\lambda(N,K)$ is nonempty, this implies in particular that~$\operatorname{Surf(g)}_\lambda(N,K)$ is nonempty.
Since this set is nonempty, we assert that the hypothesis on~$K$ ensures we can apply~\cite[Theorem 1.3]{ConwayPowell} to deduce that~$|\operatorname{Surf(g)}_\lambda(N,K)|=1$.

In contrast to Theorem~\ref{thm:SurfacesWithBoundary}, the statement of~\cite[Theorem 1.3]{ConwayPowell}  contains the additional condition that the orientation-preserving homeomorphism~$f \colon E_K \to E_K$ be the identity on~$\partial E_K$.
We show that this assumption is superfluous, so that we can apply~\cite[Theorem~1.3]{ConwayPowell} without assuming that $f|_{\partial E_K}=\id_{\partial E_K}$.

First, note that since~$f$ realises an isometry of~$\Bl_K$, it is understood that $f$ preserves a basepoint~$x_0$ and satisfies~$f([\mu_K])=[\mu_K]$, where~$[\mu_K] \in \pi_1(E_K,x_0)$ is the based homotopy class of a meridian of~$K$.
An application of the Gordon-Luecke theorem~\cite{GordonLuecke} now implies that~$f|_{\partial E_K}$ is isotopic to~$\id_{\partial E_K}$; this isotopy can be assumed to be basepoint preserving by~\cite[page~57]{FarbMargalit}.
%AC: +punctured=marked, page 47
Implanting this basepoint preserving isotopy in a collar neighborhood of~$\partial E_K$ implies that~$f$ itself is basepoint preserving isotopic to a homeomorphism $E_K \to E_K$ that restricts to the identity on~$\partial E_K$. This completes the proof that the extra assumption in the statement of \cite[Theorem~1.3]{ConwayPowell} can be assumed to hold without loss of generality.
%AC: Still preserves oriented meridians because isotopy preserves such homotopical conditions.
%\end{remark}
 \end{proof}

\subsection{Closed surfaces}
\label{sub:Closed}
We now turn our attention to closed~$\Z$-surfaces.
Let~$X$ be a closed simply-connected~$4$-manifold and let~$\Sigma \subset X$ be a closed~$\Z$-surface with genus~$g$, whose normal bundle we frame as in the case with boundary.
With this framing, we can now identify the boundary of~$X_\Sigma:=X \setminus \nu(\Sigma)$ as
$$\partial X_\Sigma \cong  \Sigma_g \times S^1.$$
Two such surfaces~$\Sigma$ and~$\Sigma'$ are \emph{equivalent} if there exists an orientation-preserving homeomorphism~$(X,\Sigma) \cong (X,\Sigma')$.
%AC: \vartheta is the analogous one to the bounded case. Trivial on the surface part.
Again as in the case of surfaces with boundary,~$X_\Sigma$ is a $\Z$ manifold and~$H_1( \Sigma_g \times S^1;\Z[t^{\pm 1}]) \cong \Z^{2g}$ is torsion.
Additionally, note that the equivariant intersection form~$\lambda_{X_\Sigma}$ of a surface exterior~$X_\Sigma$ must present~$ \Sigma_g \times S^1$.

\begin{definition}
\label{def:Surface(g)Closed}
For a nondegenerate Hermitian form~$(H,\lambda)$ over~$\Z[t^{\pm 1}]$ presenting~$\Sigma_g \times S^1$, set
$$\operatorname{Surf(g)}_\lambda(X):=\lbrace \Z\text{-surface~$\Sigma \subset X$ with } \lambda_{X_\Sigma}\cong \lambda \rbrace /\text{ equivalence}.$$
\end{definition}

As for~$\Z$-surfaces with nonempty boundary,  in order for~$\operatorname{Surf(g)}_\lambda(X)$ to be nonempty it is additionally necessary that~$\lambda(1)\cong Q_X \oplus  (0)^{\oplus 2g}$.
It was proved in~\cite[Theorem 1.4]{ConwayPowell}  that whenever~$\operatorname{Surf(g)}_\lambda(X)$ is nonempty, it contains a single element.
We improve this statement to include an existence clause.

\begin{theorem}
\label{thm:SurfacesClosed}
Let~$X$ be a closed simply-connected~$4$-manifold.
Given a nondegenerate Hermitian form~$(H,\lambda)$ over~$\Z[t^{\pm 1}]$, the following assertions are equivalent:
\begin{enumerate}
\item the Hermitian form~$\lambda$ presents~$\Sigma_g \times S^1$ and~$\lambda(1)\cong Q_X \oplus  (0)^{\oplus 2g}$;
\item there exists a unique~$($up to equivalence$)$ genus~$g$~$\Z$-surface~$\Sigma \subset X$ whose exterior has equivariant intersection form~$\lambda$; i.e.~$|\operatorname{Surf(g)}_\lambda(X)|=1$.
 \end{enumerate}
 \end{theorem}

 \begin{proof}
We have already argued that $(2) \Rightarrow (1)$ and so we focus on the converse.
Use~$U \subset S^3$ to denote the unknot and use~$N$ to denote the simply-connected~$4$-manifold with boundary~$S^3$ obtained from~$X$ by removing a small open $4$-ball.
Note that~$M_{U,g}=\Sigma_g \times S^1$ and that~$Q_N=Q_X$.
Since the Blanchfield form of~$U$ is trivial, Theorem~\ref{thm:SurfacesWithBoundary} applies; this shows us that item~$(1)$ in Theorem \ref{thm:SurfacesClosed} is equivalent to the existence of a unique (up to equivalence) genus~$g$ surface~$\Sigma \subset N$ with boundary~$U$ and equivariant intersection form~$\lambda$, in other words:
$$|\operatorname{Surf(g)}_\lambda(N,U)|=1.$$
%\color{purple}
Pick an embedded $4$-ball $B \subset X$ so that $N=X\setminus \mathring{B}$ and~$U \subset \partial N \cong S^3.$
Capping off a representative of the unique element in~$\operatorname{Surf(g)}_\lambda(N,U)$ is now readily seen to give an element of $\operatorname{Surf(g)}_\lambda(X)$.
Since~\cite[Theorem 1.4]{ConwayPowell} shows that~$|\operatorname{Surf(g)}_\lambda(X)| \in \{0,1\}$, we deduce that~$|\Surf(g)_\lambda(X)|=1$, as required.
%\color{black}
%
%------------------------------------------------------------
%
%OLD VERSION. TO BE COMMENTED?!
%
%------------------------------------------------------------
%
%Since~\cite[Theorem 1.4]{ConwayPowell} shows that~$|\operatorname{Surf(g)}_\lambda(X)| \in \{0,1\}$, it suffices to show that $\Surf(g)_\lambda(X)$ surjects onto $\Surf(g)_\lambda(N,U)$: this will imply $|\Surf(g)_\lambda(X)|=1$.
%%It remains to prove that this is equivalent to~$|\operatorname{Surf(g)}_\lambda(X)|=1$.
%%We will prove this by demonstrating that the sets~$\operatorname{Surf(g)}_\lambda(X)$ and~$\operatorname{Surf(g)}_\lambda(N,U)$ are in bijective correspondence.
%Given a closed genus~$g$~$\Z$-surface~$\Sigma \subset X$,  a~$\Z$-surface~$\mathring{\Sigma} \subset N$ with boundary~$U$ can be obtained by removing a~$(\mathring{D}^4,\mathring{D}^2)$-pair from~$(X,\Sigma)$.
%Because~$\lambda_{X_\Sigma} \cong \lambda_{N_{\mathring{\Sigma}}}$ and because an equivalence from~$\Sigma$ to~$\Sigma'$ in~$X$, restricts to an equivalence from~$\mathring{\Sigma}$ to~$\mathring{\Sigma}'$ in~$N$, this puncturing operation gives rise to a map
%\begin{equation}
%\label{eq:ClosedToBoundary}
% \operatorname{Surf(g)}_\lambda(X) \to \operatorname{Surf(g)}_\lambda(N,U).
% \end{equation}
%The surjectivity of this map is straightforward:  a pair~$(N,\Sigma)$ where~$\Sigma$ has boundary~$U$  can be capped off by a pair~$(D^4,D^2)$ to get a closed surface in~$X$.
%As explained above,  $(1)$ is equivalent to $|\operatorname{Surf(g)}_\lambda(N,U)|=1$ which implies $|\Surf(g)_\lambda(X)|=1$, as required.
\end{proof}

\subsection{Problems and open questions}
\label{sub:OpenQuestions}

We conclude with some problems in the theory of~$\Z$-surfaces, both in the closed case and in the case with boundary.
In what follows, we set
$$\mathcal{H}_2:=\begin{pmatrix}
0&t-1 \\ t^{-1}-1&0
\end{pmatrix}.$$
We start with closed surfaces in closed manifolds where the statements are a little cleaner.
\begin{problem}
\label{prob:Closed}
Fix a closed, simply-connected 4-manifold~$X$.
 Characterise the nondegenerate Hermitian forms~$(H,\lambda)$ over~$\Z[t^{\pm 1}]$ that arise as~$\lambda_{X_\Sigma}$ where~$\Sigma \subset X$ is a closed~$\Z$-surface of genus~$g$.
\end{problem}

It is known that if~$\lambda$ is as in Problem~\ref{prob:Closed}, then it must present~$\Sigma_g \times S^1$, that~$\lambda(1) \cong Q_X  \oplus (0)^{\oplus 2g}$ and that~$\lambda \oplus \mathcal{H}_2^{\oplus n} \cong Q_X \oplus \mathcal{H}_2^{\oplus (g+n)}$ for some~$n\geq  0$.
The necessity of the first two conditions was mentioned in Subsection~\ref{sub:Closed} while the necessity of third was proved in~\cite[Proposition 1.6]{ConwayPowell}.

Here is what is known about Problem~\ref{prob:Closed}:
\begin{itemize}
\item  if~$X=S^4$ and~$g\neq 1,2$,  then~$\lambda \cong  \mathcal{H}_2^{\oplus g}$~\cite[Section 7]{ConwayPowell};
\item for $X=\C P^2$ and $g=0$,  the equivariant intersection form is necessarily the form $(x,y)\mapsto x\overline{y}$ and it follows that $\Z$-spheres in $X$ are unique up to isotopy~\cite[Proposition A.1]{ConwayOrson};
\item if~$b_2(X) \geq |\sigma(X)| +6$, then~\cite[Theorem 7.2]{Sunukjian} implies that~$\lambda \cong Q_X \oplus \mathcal{H}_2^{\oplus g}$.
\end{itemize}

This leads to the following question, a positive answer  to which would solve Problem~\ref{prob:Closed}.
\begin{question}
\label{question:Closed}
Let~$X$ be a closed simply-connected~$4$-manifold and let~$(H,\lambda)$ be a nondegenerate Hermitian form  over~$\Z[t^{\pm 1}]$.
Is it the case that if~$\lambda$ presents~$\Sigma_g \times S^1$,~$\lambda(1) \cong Q_X  \oplus (0)^{\oplus 2g}$ and~$\lambda \oplus \mathcal{H}_2^{\oplus n} \cong Q_X \oplus \mathcal{H}_2^{\oplus (g+n)}$  for some~$n\geq  0$, then~$\lambda \cong Q_X \oplus \mathcal{H}_2^{\oplus g}$?
\end{question}
If the answer to Question~\ref{question:Closed} were positive, then using Theorem~\ref{thm:SurfacesClosed} one could completely classify closed~$\Z$-surfaces in closed simply-connected~$4$-manifolds: for every~$g \geq 0$, in a closed simply-connected~$4$-manifold~$X$, there would exist a unique~$\Z$-surface of genus~$g$ in~$X$ up to equivalence.

Next, we discuss the analogous (but more challenging) problem for surfaces with boundary.
\begin{problem}
\label{prob:Boundary}
Fix a simply-connected 4-manifold~$N$ with boundary $S^3$.
Characterise the nondegenerate Hermitian forms~$(H,\lambda)$ over~$\Z[t^{\pm 1}]$ that arise as~$\lambda_{N_\Sigma}$ where~$\Sigma \subset N$ is a~$\Z$-surface of genus~$g$ with boundary a fixed knot~$K$.
For brevity, we call such forms~$(N,K,g)$-\emph{realisable}.
\end{problem}

It is known that if~$\lambda$ is~$(N,K,g)$-realisable, then it must present~$M_{K,g}$,  satisfy~$\lambda(1) \cong Q_N  \oplus (0)^{\oplus 2g}$ as well as~$\lambda \oplus \mathcal{H}_2^{\oplus n} \cong Q_N \oplus \mathcal{H}_2^{\oplus (g+n)}$ for some~$n\geq  0$.
The necessity of the first two conditions was mentioned in Subsection~\ref{sub:Boundary} while the necessity of third was proved in~\cite[Proposition~1.6]{ConwayPowell}.

Here is what is known about Problem~\ref{prob:Boundary}:
\begin{itemize}
  \item if~$N=D^4, g\neq 1,2$ and~$K$ has Alexander polynomial one, then~$\lambda \cong  \mathcal{H}_2^{\oplus g}$~\cite[Section~7]{ConwayPowell};
  \item  for~$N=\C P^2 \setminus \mathring{D}^4$ and~$g=0$, the equivariant intersection form $\lambda$ is necessarily the form~$(x,y) \mapsto x\Delta_K\overline{y}$.
  After this article appeared, the classification~$\Z$-discs in~$\C P^2 \setminus \mathring{D}^4$ was studied in~\cite{ConwayDaiMiller}.
\end{itemize}

We conclude by listing consequences of further solutions to Problem~\ref{prob:Boundary}.
\begin{enumerate}
\item Using Theorem~\ref{thm:SurfacesRelBoundary},  a solution to Problem~\ref{prob:Boundary} would make it possible to fully determine the classification of properly embedded~$\Z$-surfaces in a simply-connected~$4$-manifold~$N$ with  boundary~$S^3$ up to equivalence rel.\ boundary: for every~$g \geq 0$, there would be precisely one~$\Z$-surface of genus~$g$ in~$N$ with boundary~$K$ for every element of~$\Aut(\Bl_K)/\Aut (\lambda)$, where~$\lambda$ ranges across all~$(N,K,g)$-realisable forms.
\item  If one dropped the rel.\ boundary condition, then one might conjecture that for every~$g \geq 0$, in a simply-connected~$4$-manifold~$N$ with boundary~$S^3$, there is precisely one~$\Z$-surface of genus~$g$ with boundary~$K$ for every element of~$\Aut(\partial \lambda)/\left( \Aut (\lambda) \times \Homeo^+(E_K,\partial)\right)$, where~$\lambda$ ranges across~$(N,K,g)$-realisable forms.
If the conjecture were true, then a solution to Problem~\ref{prob:Boundary} would provide a complete description of the set of properly embedded~$\Z$-surfaces in a simply-connected~$4$-manifold~$N$ with  boundary~$S^3$, up to equivalence.
%MP: The theorem we stated is clearly a corollary of that conjecture, I think, it says that if the action of Homeo E_K is transitive we get a unique surface. Also this seems like the natural generalisation of Aut(Bl_K)/Aut(lambda)?
\end{enumerate}

\section{Ubiquitous exotica}\label{sec:ubiq}
In this section we demonstrate the failure of our topological classification to hold in the smooth setting.
In Subsection \ref{sub:exoticbackground} we set up some preliminaries we will require about Stein 4-manifolds and corks.
In Subsection \ref{sub:exoticproofs} we give the proofs of Theorems \ref{thm:exoticmanifolds} and \ref{thm:exoticdiscs} from the introduction.
 In this section,  all manifolds and embeddings are understood to be smooth.

\subsection{Background on Stein structures and corks}\label{sub:exoticbackground}
We will be concerned with arranging that certain  compact 4-manifolds with boundary admit a Stein structure.
The unfamiliar reader can think of this as a particularly nice symplectic structure. Abusively, we will say that any smooth 4-manifold which admits a Stein structure is Stein.
The reason for this sudden foray into geometry is to take advantage of restrictions on the genera of smoothly embedded surfaces representing certain homology classes in Stein manifolds.
 %there are particularly easy to compute lower bounds on the genus function of a Stein 4-manifold,
These restrictions will aid us in demonstrating that two 4-manifolds are not diffeomorphic.
In this section, we will recall both a combinatorial condition for ensuring that a 4-manifold is Stein and the  restrictions on smooth representatives of certain homology classes in Stein manifolds.
%bounds on the genus function.
We use the conventions and setup of \cite{GompfStein} throughout.

\begin{figure}\center
\def\svgwidth{.4\linewidth}%% Creator: Inkscape 1.0.1 (c497b03c, 2020-09-10), www.inkscape.org
%% Accompanies image file 'gompfform.pdf' (pdf, eps, ps)
\begingroup%
  \makeatletter%
  \providecommand\color[2][]{%
    \errmessage{(Inkscape) Color is used for the text in Inkscape, but the package 'color.sty' is not loaded}%
    \renewcommand\color[2][]{}%
  }%
  \providecommand\transparent[1]{%
    \errmessage{(Inkscape) Transparency is used (non-zero) for the text in Inkscape, but the package 'transparent.sty' is not loaded}%
    \renewcommand\transparent[1]{}%
  }%
  \providecommand\rotatebox[2]{#2}%
  \newcommand*\fsize{\dimexpr\f@size pt\relax}%
  \newcommand*\lineheight[1]{\fontsize{\fsize}{#1\fsize}\selectfont}%
  \ifx\svgwidth\undefined%
    \setlength{\unitlength}{294.66419851bp}%
    \ifx\svgscale\undefined%
      \relax%
    \else%
      \setlength{\unitlength}{\unitlength * \real{\svgscale}}%
    \fi%
  \else%
    \setlength{\unitlength}{\svgwidth}%
  \fi%
  \global\let\svgwidth\undefined%
  \global\let\svgscale\undefined%
  \makeatother%
  \begin{picture}(1,0.43214134)%
    \lineheight{1}%
    \setlength\tabcolsep{0pt}%
    \put(0,0){\includegraphics[width=\unitlength,page=1]{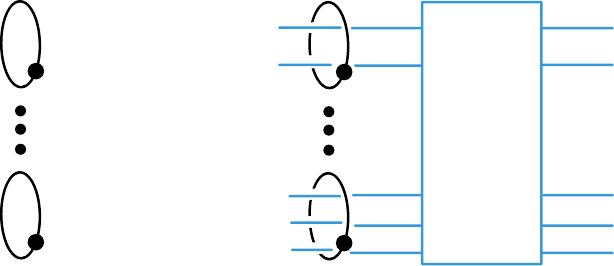}}%
    \put(0.75791972,0.19572261){\color[rgb]{0.20392157,0.59607843,0.85882353}\makebox(0,0)[lt]{\lineheight{1.25}\smash{\begin{tabular}[t]{l}$T$\end{tabular}}}}%
  \end{picture}%
\endgroup%

\caption{The left hand side shows a handle diagram for a boundary connected sum of~$S^1 \times D^3$.
On the right hand side, the tangle diagram~$T$ satisfies the conventions of a front diagram. }
\label{fig:gompfform}
\end{figure}

We begin by recalling a criterion to ensure that a handle diagram with an unique 0-handle and no 3 or 4-handles describes a Stein 4-manifold.
Recall that we can describe~$\natural_{i=1}^r S^1\times B^3$ using the dotted circle notation for 1-handles as in the left frame of Figure~\ref{fig:gompfform}. It is not hard to show that any link in~$\#_{i=1}^r S^1\times S^2$ can be isotoped into the position shown in the right frame of Figure \ref{fig:gompfform}, where inside the tangle marked~$T$ we require that the diagram meet the conventions of a front diagram for the standard contact structure on~$S^3$.
For details on front diagrams, see \cite{EtnyreLec}; stated briefly this amounts to isotoping the diagram so that all vertical tangencies are replaced by cusps and so that at each crossing the more negatively sloped strand goes over.
 We note that front diagrams require oriented links; we can choose orientations on our 2-handle attaching spheres arbitrarily,  since orienting the link does not affect the 4-manifold.
 Thus any handle diagram with a unique 0-handle and no 3 or 4-handles can be isotoped into the form of the right frame of Figure~\ref{fig:gompfform}; we say that such a diagram is in \emph{Gompf standard form}.

For a diagram in Gompf standard form, let~$L_i^T$ denote the tangle diagram obtained by restricting the~$i^{th}$ component~$L_i$ of the diagram of~$L$ to~$T$.
For a diagram in Gompf standard form, the \textit{Thurston-Bennequin number}~$TB(L_i)$ of~$L_i$  is defined as
$$TB(L^D_i)=w(L^T_i)-c(L_i^T)$$
where~$w(L_i^T)$ denotes the writhe of the tangle and~$c(L_i^T)$ denotes the number of left cusps.

In this setup,  the following criterion is helpful to prove that handlebodies are Stein.

\begin{theorem}[\cite{Eliash,GompfStein}, see also Theorem 11.2.2 of \cite{GSbook}]\label{thm:Stein}
A smooth 4-manifold~$X$ with boundary is Stein if and only if it admits a handle diagram in Gompf standard form such that the framing~$f_i$ on each 2-handle attaching curve~$L_i$ has~$f_i=TB(L_i)-1$.
\end{theorem}

\begin{figure}[!htbp]
\center
\def\svgwidth{.25\linewidth}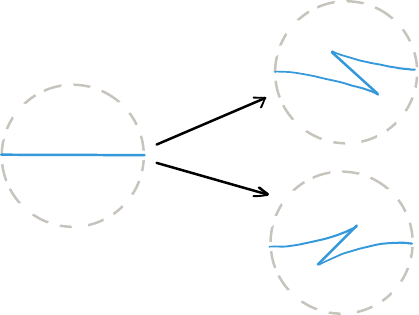
\caption{Stabilising a front diagram.}
\label{fig:stab}
\end{figure}

\begin{remark}\label{rem:largeTB}
The `if' direction of the Theorem \ref{thm:Stein} holds under the weaker hypothesis that each 2-handle attaching curve~$L_i$ has~$f_i\le TB(L_i)-1$.
To see this, observe that any 2-handle~$L_i$ can be locally isotoped via the \textit{stabilisations} demonstrated in Figure~$\ref{fig:stab}$ and observe that stabilisation preserves the condition on~$T$ and lowers the Thurston-Bennequin number of~$L_i$ by one.
The claim now follows since we can stabilise any 2-handle in a diagram in Gompf standard form to lower its Thurston-Bennequin number without changing the smooth 4-manifold described.
\end{remark}

We will also make use of the following special case of the adjunction inequality for Stein manifolds.

\begin{theorem}[\cite{LiscaMatic}]\label{thm:adjunct}
In a Stein manifold~$X$,  any homology class~$\alpha\in H_2(X)$ with~$\alpha\cdot\alpha =-1$ cannot be represented by a smoothly embedded sphere.
\end{theorem}
\begin{proof}
The proof can be deduced by combining~\cite[Theorem 3.2]{LiscaMatic} with~\cite{Brussee,FriedmanMorgan}; further exposition can be found in~\cite[Theorems 1.2 and 1.3]{AkbulutMatveyev}.
% (which is a combination of results due to).
\end{proof}

\begin{figure}[!htbp]
\center
\def\svgwidth{.75\linewidth}%% Creator: Inkscape 1.0.1 (c497b03c, 2020-09-10), www.inkscape.org
%% Accompanies image file 'Cork.pdf' (pdf, eps, ps)
\begingroup%
  \makeatletter%
  \providecommand\color[2][]{%
    \errmessage{(Inkscape) Color is used for the text in Inkscape, but the package 'color.sty' is not loaded}%
    \renewcommand\color[2][]{}%
  }%
  \providecommand\transparent[1]{%
    \errmessage{(Inkscape) Transparency is used (non-zero) for the text in Inkscape, but the package 'transparent.sty' is not loaded}%
    \renewcommand\transparent[1]{}%
  }%
  \providecommand\rotatebox[2]{#2}%
  \newcommand*\fsize{\dimexpr\f@size pt\relax}%
  \newcommand*\lineheight[1]{\fontsize{\fsize}{#1\fsize}\selectfont}%
  \ifx\svgwidth\undefined%
    \setlength{\unitlength}{425.64898682bp}%
    \ifx\svgscale\undefined%
      \relax%
    \else%
      \setlength{\unitlength}{\unitlength * \real{\svgscale}}%
    \fi%
  \else%
    \setlength{\unitlength}{\svgwidth}%
  \fi%
  \global\let\svgwidth\undefined%
  \global\let\svgscale\undefined%
  \makeatother%
  \begin{picture}(1,0.1459456)%
    \lineheight{1}%
    \setlength\tabcolsep{0pt}%
    \put(0,0){\includegraphics[width=\unitlength,page=1]{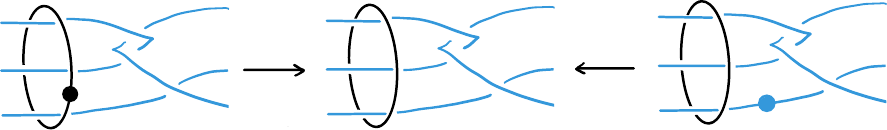}}%
    \put(0.28986421,0.07720293){\color[rgb]{0,0,0}\makebox(0,0)[lt]{\lineheight{1.25}\smash{\begin{tabular}[t]{l}$\delta$\end{tabular}}}}%
    \put(0.67165116,0.07992998){\color[rgb]{0,0,0}\makebox(0,0)[lt]{\lineheight{1.25}\smash{\begin{tabular}[t]{l}$\delta'$\end{tabular}}}}%
    \put(0.36349452,0.02675249){\color[rgb]{0,0,0}\makebox(0,0)[lt]{\lineheight{1.25}\smash{\begin{tabular}[t]{l}$0$\end{tabular}}}}%
    \put(0.55984206,0.0062996){\color[rgb]{0.20392157,0.59607843,0.85882353}\makebox(0,0)[lt]{\lineheight{1.25}\smash{\begin{tabular}[t]{l}$0$\end{tabular}}}}%
    \put(0.19441742,0.0062996){\color[rgb]{0.20392157,0.59607843,0.85882353}\makebox(0,0)[lt]{\lineheight{1.25}\smash{\begin{tabular}[t]{l}$0$\end{tabular}}}}%
    \put(0.74142132,0.03164004){\color[rgb]{0,0,0}\makebox(0,0)[lt]{\lineheight{1.25}\smash{\begin{tabular}[t]{l}$0$\end{tabular}}}}%
  \end{picture}%
\endgroup%

\caption{Two fillings of the boundary of the Akbulut cork, with boundary homeomorphism~$\delta'^{-1}\circ \delta$.
Here and throughout the rest of the paper, all handle diagrams drawn in this horizontal format should be braid closed.
}
\label{fig:cork}
\end{figure}

In order to handily construct pairs of homeomorphic 4-manifolds, we will make use of \textit{cork twisting}. Define~$C$ to be the contractible 4-manifold in the left frame of Figure~\ref{fig:cork}, which is commonly refered to as the \textit{Akbulut cork}. Observe that~$\partial C$ admits another contractible filing~$C'$ given by the right frame of Figure~\ref{fig:cork}, and that there is a natural homeomorphism~$\tau:=(\delta')^{-1} \circ \delta \colon \partial C \to \partial C'$ demonstrated in the figure.
Using the work of Freedman \cite{Freedman}, the homeomorphism~$\tau$ extends to a homeomorphism~$T \colon C\to C'$.
 As a result, for any 4-manifold~$W$ with~$\iota \colon C\hookrightarrow W$, one can construct a new~$4$-manifold~$W':= W\smallsetminus \iota(C)\cup _{(\iota|_\partial) \circ \tau^{-1}} C'$ and,  combining the identity homeomorphism~$\id_{W\smallsetminus \iota(C)}$ with~$T$,  one sees that~$W$ and~$W'$ are homeomorphic.

Historically, the literature has been concerned with two types of exotic phenomena. If smooth 4-manifolds~$X,X'$ with boundary admit a homeomorphism~$F \colon X\to X'$ but no diffeomorphism $G \colon X\to X'$ such that~$G|_\partial$ is isotopic to~$F|_\partial$, we call~$X$ and~$X'$ \textit{relatively exotic}.  If smooth 4-manifolds~$X,X'$ admit a homeomorphism $F \colon X\to X'$ but no diffeomorphism $G \colon X\to X'$ we  call~$X$ and~$X'$ \textit{absolutely exotic}.
It is easier to build relatively exotic pairs in practice. Fortunately, work of Akbulut and Ruberman shows that all relative exotica contains absolute exotica.

\begin{theorem}[Theorem A of \cite{AR16}]\label{thm:absolute}
Let $M$ and $M'$ be smooth 4-manifolds and let $F \colon M\to M'$ be a homeomorphism whose restriction to the boundary is a diffeomorphism that does not extend to a diffeomorphism $M\to M'$.
Then $M$ $($resp.\ $M')$ contains a smooth
%4-manifold
codimension $0$ submanifold~$V$ $($resp.\ $V')$ which is orientation-preserving homotopy equivalent to $M$ $($resp.\ $M')$ such that $V$ is homeomorphic but not diffeomorphic to $V'$.
\end{theorem}

If $\partial M$ and $\partial M'$ are nonempty, then $V$ and $V'$ necessarily also have nonempty boundaries since they are codimension zero submanifolds of manifolds with boundary.
We remark that Akbulut-Ruberman's theorem is only stated when $M$ is diffeomorphic to $M'$ (hence by applying a reference identification, they can in fact just call both manifolds $M$). However their proof works verbatim when $M$ and $M'$ are just homeomorphic smooth manifolds, which is the hypothesis we take above.
%Additionally, Akbulut-Ruberman do not include the emphasis that the homotopy equivalence is orientation preserving, but this follows immediately from their proof.
% The orientations of $V$ and $V'$ can just be chosen to be the ones that make the inclusions o.p. So there wouldn't be any need to emphasise it.

\subsection{Proof of Theorems \ref{thm:exoticmanifolds} and \ref{thm:exoticdiscs}}\label{sub:exoticproofs}

\black We prove Theorem~\ref{thm:exoticmanifolds} from the introduction, which for convenience
we state again here in more detail:

\begin{theorem}
\label{thm:ExoticmanifoldsNotIntro}
For every Hermitian form~$(H,\lambda)$ over~$\Z[t^{\pm 1}]$ there exists a pair of smooth $\Z$-manifolds~$M$ and~$M'$ with boundary and fundamental group~$\Z$, such that:
\begin{enumerate}
%\item  The equivariant intersection forms~$\lambda_M$ and~$\lambda_{M'}$ are isometric to~$\lambda$.
\item  there is a homeomorphism~$F \colon  M\to M'$;
%inducing some~$f \colon \partial M\to \partial M'$;
\item   $F$ induces an isometry $\lambda_M \cong \lambda_{M'}$, and both forms are isometric to~$\lambda$;
%This condition holds for our examples because they are built using a cork-twist.
\item  there is no diffeomorphism from~$M$ to~$M'$.
\end{enumerate}
In other words, every Hermitian form~$(H,\lambda)$ over~$\Z[t^{\pm 1}]$ is exotically realisable.
\end{theorem}

%
%
%For every Hermitian form~$(H,\lambda)$ over~$\Z[t^{\pm 1}]$ there exist a pair of smooth 4-manifolds~$M$ and~$M'$ with ribbon boundary and~ fundamental group~$\Z$ and the following properties:
%\begin{enumerate}
%AC: I used to just be written.
%\item  The equivariant intersection forms~$\lambda_M$ and~$\lambda_{M'}$ are isometric to~$\lambda$.
%\item  there is a homeomorphism~$F \colon  M\to M'$ inducing some~$f \colon \partial M\to \partial M'$;
%\item   $F$ induces an isometry $\lambda_M \cong \lambda_{M'}$, and both forms are isometric to~$\lambda$;
%\item  there is no diffeomorphism from~$M$ to~$M'$ inducing~$f$.
%\end{enumerate}
%In other words, every Hermitian form~$(H,\lambda)$ over~$\Z[t^{\pm 1}]$ is exotically realisable.

\begin{proof}
%[Proof of Theorem \ref{thm:exoticmanifolds}]
Let~$A(t)$ be a matrix representing the given form~$\lambda$, so that~$A(1)$ is an integer valued matrix. Choose any framed link~$L=\cup L_i\subset S^3$ with linking matrix~$A(1)$ and let~$M_1$ be the 4-manifold obtained from~$D^4$ by attaching~$A(1)_{ii}$-framed 2-handles to~$D^4$ along~$L_i$. Let~$M_2$ be the 4-manifold obtained from~$M_1$ by attaching a 1-handle (which we will think of as removing the tubular neighborhood of a trivial disc for an unknot split from~$L$).
Thus~$\pi_1(M_2)\cong \Z$ and both the integer valued intersection form~$Q_{M_2}$ and the equivariant intersection form~$\lambda_{M_2}$ are represented by a matrix for~$\lambda(1)$.

\begin{figure}[!htbp]
\center
\def\svgwidth{.5\linewidth}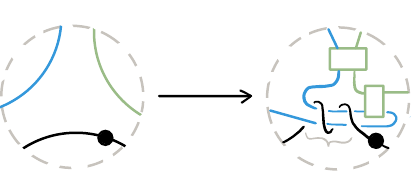
\caption{Arbitrary Hermitian forms can be realised as equivariant intersection forms by repeatedly performing the following local move, which we illustrate for~$k=2$.}
\label{fig:J'}
\end{figure}

Now we will modify the handle diagram of~$M_2$ in a way which will preserve the fundamental group and integer valued intersection form, but will result in an~$M_3$ with equivariant intersection form~$\lambda_{M_3}\cong\lambda$. For pairs~$i,j$ with~$i < j$,  for each monomial~$\ell t^k$ in the polynomial~$A(t)_{ij}$, perform the local modification exhibited in Figure \ref{fig:J'}.
Observe (for later use) that this move does not change the
%link type or the framings
framed link type of the link of attaching spheres of 2-handles.
%AC: homotoping 2-handles among each other doesn't change relator so the $\ell$ can be removed. One the $\ell$-box is removed the $k$ box unwinds using Reidemeister ones.
Furthermore, the modification does not change the fundamental group or the integer valued intersection form of~$M_2$. We exhibit in Figure \ref{fig:coverex} what the cover looks like locally after the modification.

\begin{figure}[!htbp]
\center
\def\svgwidth{.99\linewidth}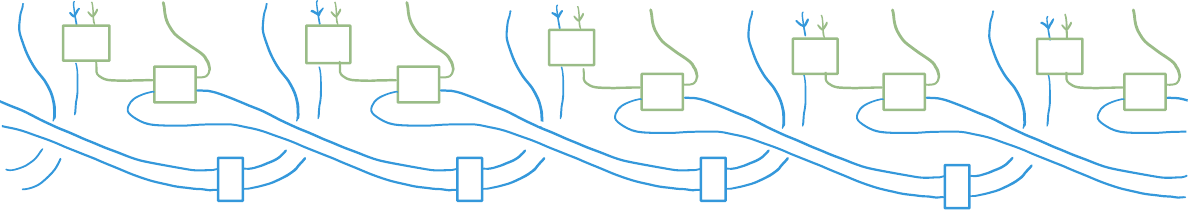
\caption{A local picture of the cover after our local modification with $k=2$. When $k>0$ the twist parameter $\epsilon$ is $1-k$, when $k<0$ it is $-k-1$.}\label{fig:coverex}
\end{figure}

Recall from Remark \ref{rem:EquivariantIntersections} that for elements~$[\widetilde{a}],[\widetilde{b}]\in H_2(M_2,\Z[t^{\pm1}])$ the equivariant intersection form satisfies
$$ \lambda_{M_2}([\widetilde{b}],[\widetilde{a}])=\sum_k (\widetilde{a} \cdot_{M_3^\infty} t^k \widetilde{b}  ) t^{-k}.$$
Thus we see that after each iteration of the local move we have that~$\lambda_{M_2'}(t)_{ij}=\lambda_{M_2}(t)_{ij}-\ell +\ell t^k$ and~$\lambda_{M_2'}(t)_{ji}=\lambda_{M_2}(t)_{ji}-\ell+\ell t^{-k}$.

For pairs~$i=j$, for each monomial~$\ell t^k$ with~$k>0$ in the polynomial~$A(t)_{ii}$, again perform the local modification in Figure \ref{fig:J'}. In this case, one finds that
\begin{equation}\label{eq:diagonalterms}
\lambda_{M_2'}(t)_{ii}=\lambda_{M_2}(t)_{ii}-2\ell +\ell t^k+\ell t^{-k}.
\end{equation}
The non-constant terms of~\eqref{eq:diagonalterms} are straightforward to deduce. The constant term is computed by considering a parallel of $\Hi$ downstairs which is 0-framed in the modification region, lifting the framing curve into the cover, and then computing the linking of the lift of the framing with $\widetilde{\Hi}$.

Once these modifications are complete, we obtain a 4-manifold~$M_3$ with~$\lambda_{M_3}$ agreeing with~$\lambda$ everywhere except \emph{a priori} on the constant terms of each~$A(t)_{ij}$.
Observe however that since these local modifications do not change the integer valued intersection form~$\lambda(1)$, we have that~$\lambda_{M_3}$ must also agree with~$\lambda$ on the constant terms of each~$A(t)_{ij}$. Thus, when we are finished, we have a smooth 4-manifold~$M_3$ with no $3$-handles,~$\pi_1(M_3) \cong \Z$ and~$\lambda_{M_3}\cong \lambda$.

\begin{figure}[!htbp]
\center
\def\svgwidth{.35\linewidth}%% Creator: Inkscape 1.0.1 (c497b03c, 2020-09-10), www.inkscape.org
%% Accompanies image file 'Equivdouble.pdf' (pdf, eps, ps)
\begingroup%
  \makeatletter%
  \providecommand\color[2][]{%
    \errmessage{(Inkscape) Color is used for the text in Inkscape, but the package 'color.sty' is not loaded}%
    \renewcommand\color[2][]{}%
  }%
  \providecommand\transparent[1]{%
    \errmessage{(Inkscape) Transparency is used (non-zero) for the text in Inkscape, but the package 'transparent.sty' is not loaded}%
    \renewcommand\transparent[1]{}%
  }%
  \providecommand\rotatebox[2]{#2}%
  \newcommand*\fsize{\dimexpr\f@size pt\relax}%
  \newcommand*\lineheight[1]{\fontsize{\fsize}{#1\fsize}\selectfont}%
  \ifx\svgwidth\undefined%
    \setlength{\unitlength}{144.29059982bp}%
    \ifx\svgscale\undefined%
      \relax%
    \else%
      \setlength{\unitlength}{\unitlength * \real{\svgscale}}%
    \fi%
  \else%
    \setlength{\unitlength}{\svgwidth}%
  \fi%
  \global\let\svgwidth\undefined%
  \global\let\svgscale\undefined%
  \makeatother%
  \begin{picture}(1,0.40862685)%
    \lineheight{1}%
    \setlength\tabcolsep{0pt}%
    \put(0,0){\includegraphics[width=\unitlength,page=1]{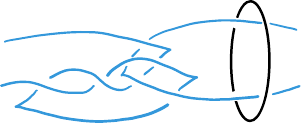}}%
    \put(0.0355963,0.30691709){\color[rgb]{0.20392157,0.59607843,0.85882353}\makebox(0,0)[lt]{\lineheight{1.25}\smash{\begin{tabular}[t]{l}$K$\end{tabular}}}}%
    \put(0.92289981,0.13484337){\color[rgb]{0,0,0}\makebox(0,0)[lt]{\lineheight{1.25}\smash{\begin{tabular}[t]{l}$0$\end{tabular}}}}%
  \end{picture}%
\endgroup%

\caption{The knot $K$ in $S^1 \times S^2$.
A handle diagram for the $4$-manifold $X$ is obtained from this diagram by dotting the black unknot and attaching a $0$-framed $2$-handle to $K$.
}
\label{fig:equivdouble}
\end{figure}

Next we will modify the 2-handles of our handle diagram ~$\sH$ of~$M_3$ to get a Stein 4-manifold~$M_4$ with the same fundamental group and equivariant intersection form as~$M_3$.
We will do this by getting the handle diagram into a form where we can apply Eliashberg's Theorem~\ref{thm:Stein}, which requires arranging that each 2-handle has a suitably large Thurston-Bennequin number.
To begin, isotope~$\sH$ into Gompf standard form, so that we think of the~$2$-handles of~$\sH$ as a Legendrian link in the standard tight contact structure on~$S^1\times S^2$. If any of the 2-handle attaching curves do not have any cusps, stabilise once so that they do.
Let~$A_3(t)$ be the equivariant linking matrix of~$\sH$; note that $A_3(t)=A(t)$ is a matrix representing the equivariant intersection form $\lambda$.
Let~$K$ be the knot in~$S^1\times S^2$ exhibited in Figure \ref{fig:equivdouble}.
Observe that if we use~$K$ to describe a 4-manifold~$X$ via attaching a~$0$-framed 2-handle to~$S^1\times B^3$ along~$K$, then~$\pi_1(X)\cong \Z$ and the equivariant intersection form~$\lambda_X$ is represented by the size one matrix~$(0)$.
Observe further that~$K$ has a Legendrian representative~$\mathcal{K}$ (illustrated in Figure~\ref{fig:equivdouble}) in the standard tight contact structure on~$S^1\times S^2$ with~$\TB(\mathcal{K})=1$.
In our handle diagram~$\sH$ of~$M_3$, let~$\mathring{K}$ be a copy of~$K$ in~$S^1\times S^2$ which is split from all of the 2-handles of~$\sH$,  as depicted in the left frame of Figure \ref{fig:connectsum}.

\begin{figure}[!htbp]
\center
\def\svgwidth{.85\linewidth}
%% Creator: Inkscape 1.0.1 (c497b03c, 2020-09-10), www.inkscape.org
%% Accompanies image file 'cts.pdf' (pdf, eps, ps)
\begingroup%
  \makeatletter%
  \providecommand\color[2][]{%
    \errmessage{(Inkscape) Color is used for the text in Inkscape, but the package 'color.sty' is not loaded}%
    \renewcommand\color[2][]{}%
  }%
  \providecommand\transparent[1]{%
    \errmessage{(Inkscape) Transparency is used (non-zero) for the text in Inkscape, but the package 'transparent.sty' is not loaded}%
    \renewcommand\transparent[1]{}%
  }%
  \providecommand\rotatebox[2]{#2}%
  \newcommand*\fsize{\dimexpr\f@size pt\relax}%
  \newcommand*\lineheight[1]{\fontsize{\fsize}{#1\fsize}\selectfont}%
  \ifx\svgwidth\undefined%
    \setlength{\unitlength}{512.02931213bp}%
    \ifx\svgscale\undefined%
      \relax%
    \else%
      \setlength{\unitlength}{\unitlength * \real{\svgscale}}%
    \fi%
  \else%
    \setlength{\unitlength}{\svgwidth}%
  \fi%
  \global\let\svgwidth\undefined%
  \global\let\svgscale\undefined%
  \makeatother%
  \begin{picture}(1,0.27704334)%
    \lineheight{1}%
    \setlength\tabcolsep{0pt}%
    \put(0,0){\includegraphics[width=\unitlength,page=1]{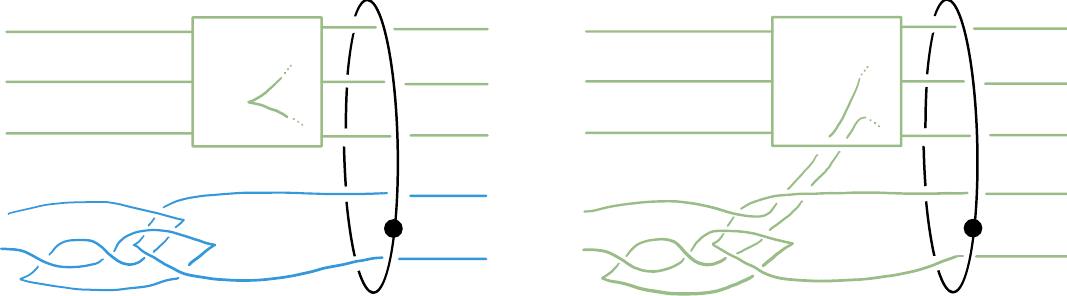}}%
    \put(0.02445311,0.12138328){\color[rgb]{0.60392157,0.7372549,0.52941176}\makebox(0,0)[lt]{\lineheight{1.25}\smash{\begin{tabular}[t]{l}$\mathcal{H_1}$\end{tabular}}}}%
    \put(0.56304764,0.11837071){\color[rgb]{0.60392157,0.7372549,0.52941176}\makebox(0,0)[lt]{\lineheight{1.25}\smash{\begin{tabular}[t]{l}$\mathcal{H_1}'$\end{tabular}}}}%
    \put(0.39788387,0.00278475){\color[rgb]{0.20392157,0.59607843,0.85882353}\makebox(0,0)[lt]{\lineheight{1.25}\smash{\begin{tabular}[t]{l}$\mathring{K}$\end{tabular}}}}%
  \end{picture}%
\endgroup%

\caption{The connect sum band can be taken with a sufficiently positive slope that choosing it to pass under any strands in the tangle~$T$ causes the diagram to remain in Gompf standard form.}\label{fig:connectsum}
\end{figure}

Now for any handle~$\Hi$ of~$\sH$ with $A_3(1)_{ii}>\TB(\Hi)-2$ form~$\Hi'$ by taking the connected sum of~$\Hi$ with~a split copy of~$\mathring{K}$ in the manner depicted in Figure \ref{fig:connectsum}.  Frame~$\Hi'$ using the same diagrammatic framing instruction that was used to frame~$\sH_i$.
One computes readily from the right frame of Figure \ref{fig:connectsum} that~$\TB(\Hi')=\TB(\Hi)+1$. Repeat this process until~$A_3(1)_{ii}\le \TB(\Hi)-2$ for all 2-handles.
Let~$M_4$ be the resulting 4-manifold.
Then~$M_4$ is Stein by Theorem~\ref{thm:Stein} and Remark~\ref{rem:largeTB}.
Further, since~$X$ contributes neither to the equivariant intersection form nor to~$\pi_1$, we have that~$M_4$ has the same equivariant intersection form and~$\pi_1$ as~$M_3$.
We record (for later use) the observation that the link in~$S^3$ consisting of the attaching spheres of the 2-handles is unchanged by these modifications; one can see this by ignoring the 1-handle in Figure~\ref{fig:connectsum} and doing a bit of isotopy.

\begin{figure}[!htbp]
\centering
\def\svgwidth{.6\linewidth}%% Creator: Inkscape 1.0.1 (c497b03c, 2020-09-10), www.inkscape.org
%% Accompanies image file 'Insertcork.pdf' (pdf, eps, ps)
\begingroup%
  \makeatletter%
  \providecommand\color[2][]{%
    \errmessage{(Inkscape) Color is used for the text in Inkscape, but the package 'color.sty' is not loaded}%
    \renewcommand\color[2][]{}%
  }%
  \providecommand\transparent[1]{%
    \errmessage{(Inkscape) Transparency is used (non-zero) for the text in Inkscape, but the package 'transparent.sty' is not loaded}%
    \renewcommand\transparent[1]{}%
  }%
  \providecommand\rotatebox[2]{#2}%
  \newcommand*\fsize{\dimexpr\f@size pt\relax}%
  \newcommand*\lineheight[1]{\fontsize{\fsize}{#1\fsize}\selectfont}%
  \ifx\svgwidth\undefined%
    \setlength{\unitlength}{293.95761108bp}%
    \ifx\svgscale\undefined%
      \relax%
    \else%
      \setlength{\unitlength}{\unitlength * \real{\svgscale}}%
    \fi%
  \else%
    \setlength{\unitlength}{\svgwidth}%
  \fi%
  \global\let\svgwidth\undefined%
  \global\let\svgscale\undefined%
  \makeatother%
  \begin{picture}(1,0.39367377)%
    \lineheight{1}%
    \setlength\tabcolsep{0pt}%
    \put(0,0){\includegraphics[width=\unitlength,page=1]{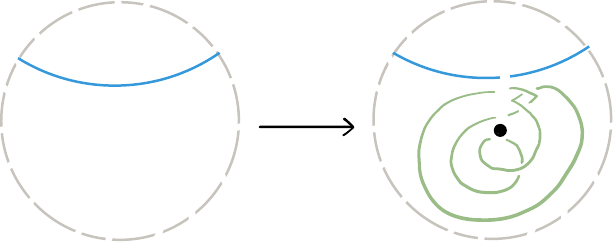}}%
    \put(0.37087899,0.31065946){\color[rgb]{0.20392157,0.59607843,0.85882353}\makebox(0,0)[lt]{\lineheight{1.25}\smash{\begin{tabular}[t]{l}$\mathcal{H}_1$\end{tabular}}}}%
    \put(0.97016522,0.323539){\color[rgb]{0.20392157,0.59607843,0.85882353}\makebox(0,0)[lt]{\lineheight{1.25}\smash{\begin{tabular}[t]{l}$\mathcal{H}_1$\end{tabular}}}}%
    \put(0.63548558,0.11980408){\color[rgb]{0.60392157,0.7372549,0.52941176}\makebox(0,0)[lt]{\lineheight{1.25}\smash{\begin{tabular}[t]{l}$0$\end{tabular}}}}%
    \put(0,0){\includegraphics[width=\unitlength,page=2]{Insertcork.pdf}}%
    \put(0.83609593,0.32829823){\color[rgb]{0.78039216,0.76470588,0.74117647}\makebox(0,0)[lt]{\lineheight{1.25}\smash{\begin{tabular}[t]{l}$\gamma$\end{tabular}}}}%
  \end{picture}%
\endgroup%

\caption{The local modification performed on the handle $\mathcal{H}_1$ of the manifold $M_3$. }
\label{fig:insertcork}
\end{figure}

Now we will make a final modification to~$M_4$ to get a 4-manifold~$M_5 =: M$ which we can cork twist to get~$M'$. Choose any 2-handle, without loss of generality we choose $\sH_1$, and perform the local modification described in Figure \ref{fig:insertcork}; the resulting 4-manifold is our~$M$.
% Once you isotope this handle diagram for $M$ into Gompf standard form, you will find that this modification lowers $\TB(\sH_1)$ by $1$, see the blue and green handles of Figure \ref{fig:stillstein}.\black

%LP:upside down,  M is obtained from Y \times I by adding 2 and 3 handles.
One can readily check that this local modification does not impact~$\pi_1$ or the equivariant intersection form. Further, this local diagram can be readily converted to Gompf standard form, (see the blue and green handles of Figure \ref{fig:stillstein})  where we have~$A_3(1)_{ii}\le \TB(\Hi) -1$ for all 2-handles, hence~$M$ is Stein. By construction,~$M$ contains a copy of the Akbulut cork~$C$.
Because $M$ has no 3-handles, $\pi_1(\partial M)$ surjects~$\pi_1(M)$.

\begin{figure}[!htbp]
\center
\def\svgwidth{.25\linewidth}%% Creator: Inkscape 1.0.1 (c497b03c, 2020-09-10), www.inkscape.org
%% Accompanies image file 'Stillstein.pdf' (pdf, eps, ps)
\begingroup%
  \makeatletter%
  \providecommand\color[2][]{%
    \errmessage{(Inkscape) Color is used for the text in Inkscape, but the package 'color.sty' is not loaded}%
    \renewcommand\color[2][]{}%
  }%
  \providecommand\transparent[1]{%
    \errmessage{(Inkscape) Transparency is used (non-zero) for the text in Inkscape, but the package 'transparent.sty' is not loaded}%
    \renewcommand\transparent[1]{}%
  }%
  \providecommand\rotatebox[2]{#2}%
  \newcommand*\fsize{\dimexpr\f@size pt\relax}%
  \newcommand*\lineheight[1]{\fontsize{\fsize}{#1\fsize}\selectfont}%
  \ifx\svgwidth\undefined%
    \setlength{\unitlength}{612bp}%
    \ifx\svgscale\undefined%
      \relax%
    \else%
      \setlength{\unitlength}{\unitlength * \real{\svgscale}}%
    \fi%
  \else%
    \setlength{\unitlength}{\svgwidth}%
  \fi%
  \global\let\svgwidth\undefined%
  \global\let\svgscale\undefined%
  \makeatother%
  \begin{picture}(1,0.99997038)%
    \lineheight{1}%
    \setlength\tabcolsep{0pt}%
    \put(0,0){\includegraphics[width=\unitlength,page=1]{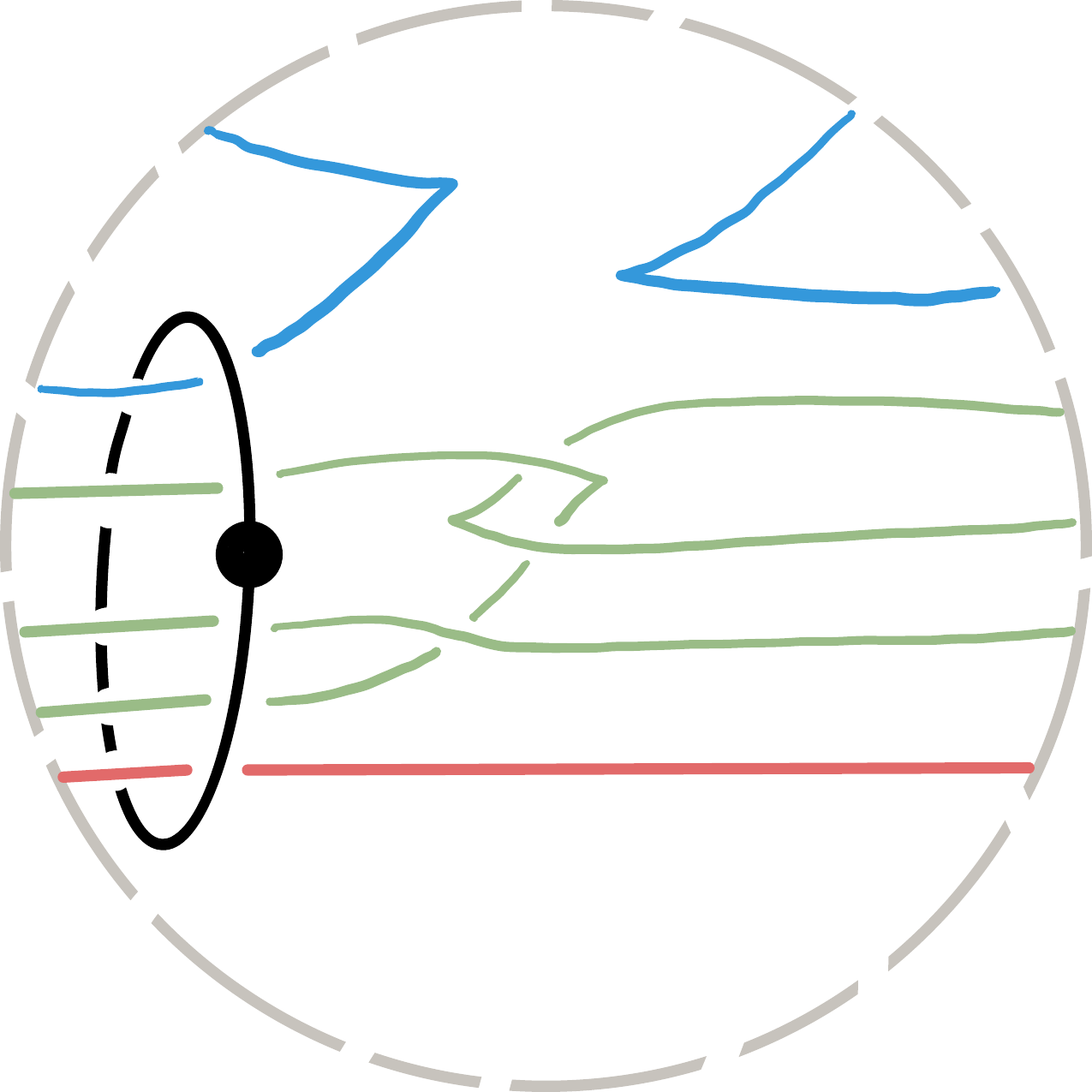}}%
    \put(0.82568693,0.90333897){\color[rgb]{0.20392157,0.59607843,0.85882353}\makebox(0,0)[lt]{\lineheight{1.25}\smash{\begin{tabular}[t]{l}$\mathcal{H}_1$\end{tabular}}}}%
    \put(0.62732386,0.21868841){\color[rgb]{0.88627451,0.41568627,0.41568627}\makebox(0,0)[lt]{\lineheight{1.25}\smash{\begin{tabular}[t]{l}$-1$\end{tabular}}}}%
    \put(0.3640463,0.61213089){\color[rgb]{0.60392157,0.7372549,0.52941176}\makebox(0,0)[lt]{\lineheight{1.25}\smash{\begin{tabular}[t]{l}$-1$\end{tabular}}}}%
    \put(0.95763988,0.24562361){\color[rgb]{0.88627451,0.41568627,0.41568627}\makebox(0,0)[lt]{\lineheight{1.25}\smash{\begin{tabular}[t]{l}$\gamma$\end{tabular}}}}%
  \end{picture}%
\endgroup%

\caption{A handle diagram for the manifold $W$ in Gompf standard form.}\label{fig:stillstein}
\end{figure}

Now define~$M'$ to be the 4-manifold obtained from~$M$ by twisting~$C$. Since there is a homeomorphism~$T\colon C\to C$ extending the twist homeomorphism~$\tau \colon \partial C\to\partial C$, there is a natural homeomorphism~$F \colon M\to M'$; let~$f$ denote the restriction~$f\colon \partial M\to \partial M'$.

It remains to show that $M$ and $M'$ are not diffeomorphic. We will begin by showing the relative statement, i.e.  there is no diffeomorphism~$G\colon M\to M'$ such that~$G|_\partial = f$.
It would be convenient if at this point we could distinguish $M$ and $M'$ directly by showing that one is Stein and one is not. Unfortunately, both are Stein. So instead we will consider auxiliary manifolds $W$ and $W'$ constructed as follows.
Suppose for a contradiction that there were such a diffeomorphism~$G$.
Construct a 4-manifold~$W$ by attaching a~$(-1)$-framed 2-handle to~$M$ along~$\gamma$ (where~$\gamma$ is the curve in~$\partial M$ marked in Figure \ref{fig:insertcork}) and a second 4-manifold~$W'$ from~$M'$ by attaching a 2-handle to~$M'$ with attaching sphere and framing given by~$(f(\gamma),-1)$. \footnote{The $(-1)$-framing instruction for $f(\gamma)$ requires a diagram of $f(\gamma)$ in $\partial M'$.
Because $f$ is a dot-zero homeomorphism, we can use the exact same diagram as we used for $\gamma$ in $\partial M$.} Notice that the image under $f$ of a~$(-1)$-framing curve for $\gamma$ is in fact a~$(-1)$-framing curve for $f(\gamma)$.
% is still $(-1)$-framed in the diagram since the homeomorphism~$f$ is obtained by performing a dot-zero exchange.
The diffeomorphism $G$ extends to give a diffeomorphism~$\widehat{G}\colon W\to W'$.
In Figure~$\ref{fig:stillstein}$, we have exhibited the natural handle diagram for~$W$ in Gompf standard form, from which Theorem \ref{thm:Stein} implies that~$W$ admits a Stein structure.
%AC: That figure is obtained by an isotopy (and closing/opening some braids) from the previous one (and replacing \gamma by a -1 framed 2-handle also denoted $\gamma$.

We will finish showing that $f$ does not extend by demonstrating that~$W'$ does not admit any Stein structure, thus~$W$ cannot be diffeomorphic to~$W'$. Since~$W'$ is obtained from~$W$ by reversing the dot and the zero on the handles of~$C$,~$f(\gamma)$ is just a meridian of a 2-handle of~$M'$. Thus the final 2-handle of~$W'$ is attached along a curve which bounds a disc in~$M'$, implying that there is a~$(-1)$-framed sphere embedded in~$W'$.
But the adjunction inequality for Stein manifolds (recall Theorem~\ref{thm:adjunct}) indicates that no 4-manifold which admits a Stein structure can contain an embedded sphere with self-intersection~$-1$.
Hence,~$W$ is not diffeomorphic to~$W'$, thus there cannot be a diffeomorphism~$ G \colon M\to M'$ extending~$f$.

Now we would like to extend this to a statement about absolute exotica.
To do so, we apply Theorem \ref{thm:absolute} to our $M, M'$,  and $f$ to produce a pair of smooth 4-manifolds $V$ and $V'$ (both of which have nonempty boundary) which are homeomorphic but not diffeomorphic.
Since $V$ and~$V'$ are orientation-preserving homotopy equivalent to $M$ and $M'$ respectively, the equivariant intersection forms $\lambda_V$ and $\lambda_{V'}$ are also isometric to $\lambda$, and both $V$ and $V'$ have fundamental group $\Z$.
Since~$V$ and $V'$ are homeomorphic, so are $\partial V$ and $\partial V'$.
\end{proof}

Next, we prove Theorem~\ref{thm:exoticdiscs} from the introduction, again stated here in more detail.
If one wants to show that any $2$-handlebody~$N$ with boundary~$S^3$ contains a pair of exotic $\Z$-discs one can run the same proof, where in the first line~$\mathcal{H}'$ is chosen to be a handle diagram for~$N$; this was mentioned in Remark~\ref{rem:Smooth}.

%,  this proof will show that any $2$-handlebody contains a pair of exotic $\Z$-discs.}

\begin{theorem}
\label{thm:ExoticDiscsMain}
 For every Hermitian form~$(H,\lambda)$ over~$\Z[t^{\pm 1}]$ such that~$\lambda(1)$ is realised as the intersection form of a smooth simply-connected 4-dimensional 2-handlebody~$N$ with~$\partial N\cong S^3$, there exists a pair of smooth~$\Z$-discs~$D$ and~$D'$ in~$N$ with the same boundary and the following properties:
 \begin{enumerate}
 \item the equivariant intersection forms~$\lambda_{N_D}$ and~$\lambda_{N_{D'}}$ are isometric to~$\lambda$;
 \item $D$ is topologically isotopic to~$D'$ rel.\ boundary;
 \item $D$ is not smoothly equivalent to~$D'$ rel.\ boundary.
 \end{enumerate}
\end{theorem}

\begin{proof}
Let~$\sH'$ be a handle diagram for a 2-handlebody with~$S^3$ boundary and such that~$Q_N$ isometric to~$\lambda(1)$.
Let~$D$ be the standard disc for a local unknot in~$\partial N$, and as usual let~$N_D$ be its exterior, which has handle diagram~$\sH:=\sH'\cup 1-$handle.
%AC: Here local means a little unknot far from all the other handles.

Akin to the proof of Theorem \ref{thm:exoticmanifolds}, we will now modify the linking of the handles of~$\sH$ to get a Stein manifold with equivariant intersection form $\lambda$.
However, we also want to do so in such a way that the manifold presented by~$\sH$ is still~$N_{D'}$ for some smooth disc~$D'$ properly embedded in~$N$.

We claim that if we modify only the linking of the 2-handles with the 1-handle, and not the linking of the 2-handles with each other nor the knot type or framing of the 2-handles, we will have that~$\sH$ presents such an~$N_{D'}$.
To prove the claim, first observe that $X$ is the exterior of  a disc in~$N$ if and only if~$N$ can be obtained from~$X$ by adding on a single 2-handle.
Observe that adding a 0-framed 2-handle to the meridian of a 1-handle in dotted circle notation allows us to erase both the new 2-handle and the 1-handle.
Thus, if our modifications only change the way the 2-handles of $N$ link the new one-handle, we will still have the property that after a single 2-handle addition we obtain $N$, thus our manifold is the exterior of a disc embedded in $N$.
This concludes the proof of the claim.

%AC: What I initially wrote.
%The claim now follows because after peforming the specified modifications, adding a 2-handle~$\mathcal{h}$ to~$\sH$ along a 0-framed meridian to the 1-handle gives the same resulting 4-manifold, namely~$N$,  as ignoring the 1-handle of~$\sH$.
%%, and because of the types of modifications we did, this is just the handle diagram~$\sH'$, thus the 4-manifold is~$N$.
%Thus the manifold presented by~$\sH$ is obtained from~$N$ by removing the cocore disc to~$\mathcal{H}$, so the manifold presented by~$\sH$ is~$N_{D'}$ for some smooth disc~$D'$ properly embedded in~$N$.

\begin{figure}[!htbp]
\center
\def\svgwidth{.6\linewidth}%% Creator: Inkscape 1.0.1 (c497b03c, 2020-09-10), www.inkscape.org
%% Accompanies image file 'Kylesdisks.pdf' (pdf, eps, ps)
\begingroup%
  \makeatletter%
  \providecommand\color[2][]{%
    \errmessage{(Inkscape) Color is used for the text in Inkscape, but the package 'color.sty' is not loaded}%
    \renewcommand\color[2][]{}%
  }%
  \providecommand\transparent[1]{%
    \errmessage{(Inkscape) Transparency is used (non-zero) for the text in Inkscape, but the package 'transparent.sty' is not loaded}%
    \renewcommand\transparent[1]{}%
  }%
  \providecommand\rotatebox[2]{#2}%
  \newcommand*\fsize{\dimexpr\f@size pt\relax}%
  \newcommand*\lineheight[1]{\fontsize{\fsize}{#1\fsize}\selectfont}%
  \ifx\svgwidth\undefined%
    \setlength{\unitlength}{352.89690399bp}%
    \ifx\svgscale\undefined%
      \relax%
    \else%
      \setlength{\unitlength}{\unitlength * \real{\svgscale}}%
    \fi%
  \else%
    \setlength{\unitlength}{\svgwidth}%
  \fi%
  \global\let\svgwidth\undefined%
  \global\let\svgscale\undefined%
  \makeatother%
  \begin{picture}(1,0.23260694)%
    \lineheight{1}%
    \setlength\tabcolsep{0pt}%
    \put(0,0){\includegraphics[width=\unitlength,page=1]{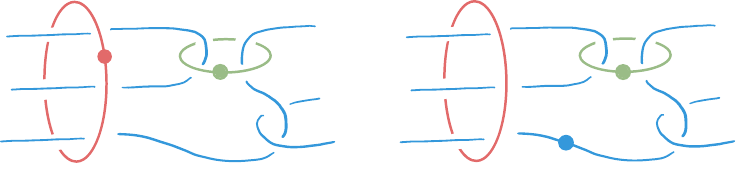}}%
    \put(0.25985072,0.04524716){\color[rgb]{0.20392157,0.59607843,0.85882353}\makebox(0,0)[lt]{\lineheight{1.25}\smash{\begin{tabular}[t]{l}$0$\end{tabular}}}}%
    \put(0.56617936,0.05511492){\color[rgb]{0.88627451,0.41568627,0.41568627}\makebox(0,0)[lt]{\lineheight{1.25}\smash{\begin{tabular}[t]{l}$0$\end{tabular}}}}%
  \end{picture}%
\endgroup%

\caption{In both frames the red and blue handles give a nonstandard handle diagram for~$D^4$, and in both frames the green knot~$K\subset S^3$ bounds a disc disjoint from the 1-handle; these are our two discs~$\Sigma$ and~$\Sigma'$ for~$K$ in~$D^4$. The handle diagrams here present~$D^4_\Sigma$ and~$D^4_{\Sigma'}$.}\label{fig:kylesdiscs}
\end{figure}

Now observe that all of the modifications we performed in the proof of Theorem~\ref{thm:exoticmanifolds} to get from~$M_2$ to~$M_4$ modified only the linking of the 2-handles with the 1-handle, and not the linking of the 2-handles with each other nor the knot type or framing of the 2-handles.
Thus we can again perform those same modifications to our~$\sH$ to obtain a smooth~$\Z$-disc~$D'$ properly embedded in~$N$ such that the resulting~$\sH$ is a handle diagram for~$N_{D'}$ in Gompf standard form satisfying Eliashberg's criteria and such that the equivariant intersection form of the exterior is~$\lambda_{N_{D'}}\cong\lambda$. Notice in particular that~$N_{D'}$ is Stein.

Now let~$\Sigma$ and~$\Sigma'$ be the pair of slice discs for~$K$ in~$D^4$ exhibited in Figure \ref{fig:kylesdiscs}. These discs were constructed following the techniques of \cite{Hayden}. It is elementary to check from the exhibited handle diagrams that both discs have~$\pi_1(D^4_\Sigma)=\pi_1(D^4_{\Sigma'})=\Z$ and are ribbon.
%AC: For the LHS you get \pi_1(N_\Sigma)=<g,r \ | \ rr^{-1}grg^{-1}>=<g>. For the RHS one you first have eto do an isotopy.
It is then a consequence of \cite[Theorem 1.2]{ConwayPowellDiscs} that~$\Sigma$ is topologically isotopic to~$\Sigma'$ rel.\ boundary.

\begin{figure}[!htbp]
\center
\def\svgwidth{.8\linewidth}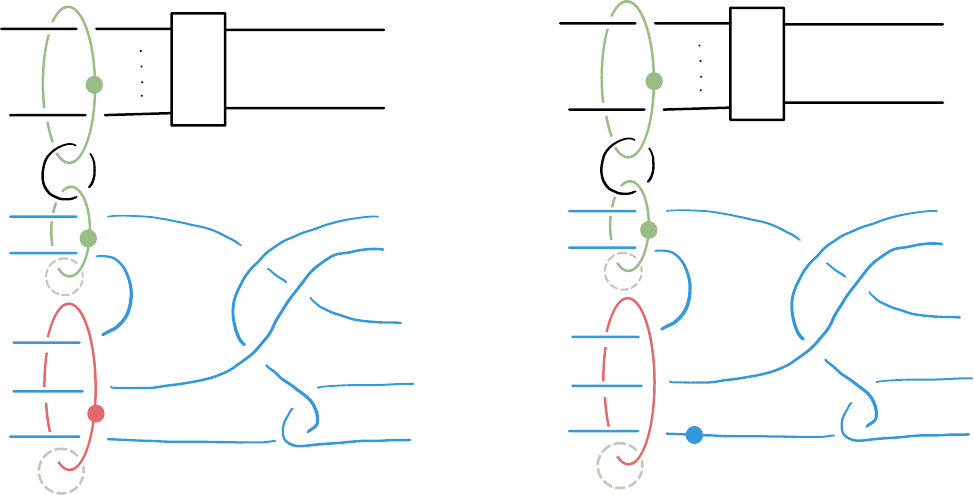
%AC: I removed the space from .  pdf to .pdf
\caption{The left frame gives a handle diagram for~$N_R$, and the right for~$N_{R'}$. The top black 2-handles and tangle~$T$ represent the handle diagram of~$N_{D'}$ in Gompf standard form which we already constructed.}\label{fig:RandRprime}
\end{figure}

We will construct discs~$R$ and~$R'$ in~$N$ by taking the boundary connect sum of pairs~$(N,R):=(N,D')\natural(D^4,\Sigma)$ and~$(N,R'):=(N,D')\natural(D^4,\Sigma')$.
We demonstrate natural handle decompositions for~$N_R$ and~$N_{R'}$ in Figure \ref{fig:RandRprime}. It is straightforward to confirm that~$\pi_1(N_R)\cong\pi_1(N_{R'})\cong\Z$.
%AC: We have a decomposition of N_{R} as  $N_\Sigma \cup N_D$ along \mu x D^2 so the two meridians become identified.  \pi_1(N_\Sigma) contributes no other relations so we are done. Lisa drew the decomposition in a note
Further, since~$\Sigma$ is topologically isotopic to~$\Sigma'$ in~$D^4$ rel.\ boundary,~$R$ is topologically isotopic in~$N$ to~$R'$ rel.\ boundary.
Since~$\Sigma$ and~$\Sigma'$ are~$\Z$-discs in~$D^4$, their exteriors are aspherical~\cite[Lemma~2.1]{ConwayPowell} and so both~$\lambda_{N_\Sigma}$ and~$\lambda_{N_{\Sigma'}}$ are trivial.
It is then not hard to show that band summing~$D'$ with~$\Sigma$ or~$\Sigma'$ does not change the equivariant intersection form, so~$\lambda_{N_R}\cong\lambda_{N_{R'}}\cong\lambda_{N_{D'}}$.
%AC: Again, use the decomposition of N_{R} as  $N_\Sigma \cup N_D$ along \mu x D^2 and the fact that $N_\Sigma$ and \mu x D^2 have no  H_2 and no H_1.

It remains to show that~$R$ is not smoothly equivalent to~$R'$ rel.\ boundary.
If~$R$ were equivalent to~$R'$ rel.\ boundary then there would be a diffeomorphism~$F \colon N_R \to N_{R'}$ which is the identity on the boundary.
Let~$\gamma$ and~$\delta$ be the curves in~$\partial N_R = \partial N_{R'}$ demonstrated in Figure~\ref{fig:RandRprime}, and let~$W$ (similarly~$W'$) be formed from~$N_R$ by attaching~$(-1)$-framed 2-handles along~$\gamma$ and~$\delta$.
%%%
%do not delete.
%AC: Foreshadowing the next paragraph,  both $\gamma$ and $\delta$ will play a role in showing that $W$ and $W'$ are not diffeomorphic: the $2$-handle attached along $\gamma$ will imply that $W$ is Stein, whereas the $2$-handle attached along $\delta$ will imply that $W'$ is not Stein.
%%%

\begin{figure}[!htbp]
\center
\def\svgwidth{.4\linewidth}%% Creator: Inkscape 1.0.1 (c497b03c, 2020-09-10), www.inkscape.org
%% Accompanies image file 'Steindisk.pdf' (pdf, eps, ps)
\begingroup%
  \makeatletter%
  \providecommand\color[2][]{%
    \errmessage{(Inkscape) Color is used for the text in Inkscape, but the package 'color.sty' is not loaded}%
    \renewcommand\color[2][]{}%
  }%
  \providecommand\transparent[1]{%
    \errmessage{(Inkscape) Transparency is used (non-zero) for the text in Inkscape, but the package 'transparent.sty' is not loaded}%
    \renewcommand\transparent[1]{}%
  }%
  \providecommand\rotatebox[2]{#2}%
  \newcommand*\fsize{\dimexpr\f@size pt\relax}%
  \newcommand*\lineheight[1]{\fontsize{\fsize}{#1\fsize}\selectfont}%
  \ifx\svgwidth\undefined%
    \setlength{\unitlength}{208.08597565bp}%
    \ifx\svgscale\undefined%
      \relax%
    \else%
      \setlength{\unitlength}{\unitlength * \real{\svgscale}}%
    \fi%
  \else%
    \setlength{\unitlength}{\svgwidth}%
  \fi%
  \global\let\svgwidth\undefined%
  \global\let\svgscale\undefined%
  \makeatother%
  \begin{picture}(1,0.64824431)%
    \lineheight{1}%
    \setlength\tabcolsep{0pt}%
    \put(0,0){\includegraphics[width=\unitlength,page=1]{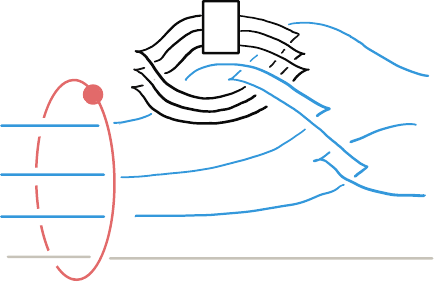}}%
    \put(0.28171775,0.06628105){\color[rgb]{0.78039216,0.76470588,0.7254902}\makebox(0,0)[lt]{\lineheight{1.25}\smash{\begin{tabular}[t]{l}$-1$\end{tabular}}}}%
    \put(0.8991181,-0.00172428){\color[rgb]{0.78039216,0.76470588,0.7254902}\makebox(0,0)[lt]{\lineheight{1.25}\smash{\begin{tabular}[t]{l}$\delta$\end{tabular}}}}%
    \put(0.91588256,0.50973274){\color[rgb]{0.20392157,0.59607843,0.85882353}\makebox(0,0)[lt]{\lineheight{1.25}\smash{\begin{tabular}[t]{l}$0$\end{tabular}}}}%
    \put(0.48653646,0.56676512){\color[rgb]{0.03921569,0.03921569,0.03921569}\makebox(0,0)[lt]{\lineheight{1.25}\smash{\begin{tabular}[t]{l}$T$\end{tabular}}}}%
  \end{picture}%
\endgroup%

\caption{The black 2-handles here have both framing and~$\TB$ one less than they had in Figure \ref{fig:RandRprime}; since we had already arranged that the tangle~$T$ in Figure \ref{fig:RandRprime} satisfied the framing criteria of Theorem \ref{thm:Stein}, this handle diagram also satisfies the criteria.}\label{fig:NRstein}
\end{figure}

If a diffeomorphism~$F \colon N_R \to N_{R'}$ extending the identity exists, then~$W$ is diffeomorphic to~$W'$. Observe that~$W'$ does not admit a Stein structure, because the 2-handle along~$\delta$ naturally introduces a~$(-1)$-framed 2-sphere embedded in~$W'$, which violates the Stein adjunction inequality in Theorem~\ref{thm:adjunct}. However,~$W$ admits the handle decomposition given in Figure \ref{fig:NRstein}, which is in Gompf standard form, so Theorem~\ref{thm:Stein} ensures that~$W$ admits a Stein structure.
Therefore~$W$ is not diffeomorphic to~$W'$, so there can be no such~$F$, so~$R$ is not smoothly equivalent to~$R'$ rel.\ boundary.
\end{proof}

\begin{remark}
In the above proof,~$R$ is smoothly isotopic to~$R'$ \emph{not} rel.\ boundary, because~$\Sigma$ is smoothly isotopic to~$\Sigma'$ not rel.\ boundary. If we wanted to produce~$R$ and~$R'$ which are not smoothly isotopic (without a boundary condition), we could have instead used a~$\Sigma$ and~$\Sigma'$ which are not isotopic rel.\ boundary and run a similar argument. Such~$\Sigma$ and~$\Sigma'$ are produced in \cite{Hayden}; we have not pursued this here because the diagrams are somewhat more complicated.
\end{remark}

\section{Nontrivial boundary automorphism set}
\label{sec:NonTrivialbAut}

We prove that there are examples of pairs $(Y,\varphi)$ for which the set of 4-manifolds with fixed boundary $Y$ and equivariant intersection form, up to homeomorphism, can have arbitrarily large cardinality.
This was alluded to in Example~\ref{ex:LargeStableClassIntro}.
The main step in this process is to find a sequence of Hermitian forms~$(H_i,\lambda_i)$ for which~$\big\{\big|\Aut(\partial \lambda_i)/\Aut(\lambda_i)\big|\big\}$ is unbounded.
The most direct way to achieve this is when~$H$ has rank~$1$.
Indeed, in this case,~$\Aut(\partial \lambda)/\Aut(\lambda)$ can be described in terms of certain units of~$\Z[t^{\pm 1}]/\lambda$, as we now make precise.
\medbreak

Given a ring~$R$ with involution~$x \mapsto \overline{x}$,  the group of \emph{unitary units}~$U(R)$ refers to those~$u \in R$ such that~$u \overline{u}=1$.
For example, when~$R=\Z[t^{\pm 1}]$, all units are unitary and are of the form~$\pm t^{k}$ with~$k \in \Z$.

In what follows, we make no distinction between rank one
%non-degenerate
Hermitian forms and
%nontrivial
symmetric Laurent polynomials.
The next lemma follows by unwinding the definition of~$\Aut(\partial \lambda)$; see also~\cite[Remark 1.16]{ConwayPowell}.

\begin{lemma}
\label{lem:UnitaryUnits}
If~$\lambda \in \Z[t^{\pm 1}]$ is a symmetric Laurent polynomial, then
$$\Aut(\partial \lambda)/\Aut(\lambda)=U(\Z[t^{\pm 1}]/\lambda)/U(\Z[t^{\pm 1}]).$$
\end{lemma}

%Proposition~\ref{prop:UnitaryUnits} gives a concrete way of finding examples of ma

Given a symmetric Laurent polynomial~$P \in \Z[t^{\pm 1}]$,  use~$n_P$ to denote the number of ways~$P$ can be written as an unordered product~$ab$ of symmetric polynomials~$a,b\in \Z[t^{\pm 1}]$ such that there exists~$x,y\in \Z[t^{\pm 1}]$ with~$ax+by=1$, where the factorisations~$ab$ and~$(-a)(-b)$ are deemed equal.

%A symmetric polynomial~$\lambda$ is \emph{$n$-large} if is a product~$P=p_1\cdots p_n$ of~$n$ coprime irreducible factors such that the~$2^{n-1}$ ways to factor~$P$ as a product~$P=ab$ of two polynomials

\begin{lemma}
\label{lem:NonTrivialbAut}
If~$P \in \Z[t^{\pm 1}]$ is a symmetric Laurent polynomial, then~$U(\Z[t^{\pm 1}]/2P)/U(\Z[t^{\pm 1}])$ contains at least~$n_P$ elements.
\end{lemma}
\begin{proof}
A first verification shows that if~$P$ factorises as~$P=ab$ where~$a,b\in \Z[t^{\pm 1}]$ are symmetric polynomials and satisfy~$ax+by=1$, then~\[\Phi(a,b):=-ax+by\] is a unitary unit in~$\Z[t^{\pm 1}]/2P$, i.e.\ belongs to~$U(\Z[t^{\pm 1}]/2P)$:
\begin{align*}
(-ax+by)\overline{(-ax+by)}
&=a\overline{a}x \overline{x}+b\overline{b}y\overline{y}-ax\overline{b}\overline{y}-\overline{a}\overline{x}by=a\overline{a}x \overline{x}+b\overline{b}y\overline{y}-ab(x\overline{y}+\overline{x}y)\\
&\equiv a\overline{a}x \overline{x}+b\overline{b}y\overline{y}+ab(x\overline{y}+\overline{x}y)
=(ax+by)\overline{(ax+by)} =1.
\end{align*}
It can also be verified that~$\Phi(a,b)$ depends neither on the ordering of~$a,b$ nor on the choice of~$x,y$. The former check is immediate from the definition of~$\Phi$ because~$-1 \in U(\Z[t^{\pm 1}])$.
We verify that the assignment does not depend on the choice of~$x,y$.
Assume that~$ax+by=1=ax'+by'$ for~$x,x',y,y'\in \Z[t^{\pm 1}]$.
We deduce that~$ax'=1=ax$ mod~$b$ and~$by'=1=by$ mod~$a$.
But now~$x' \equiv  (ax)x'=x(ax')=x$ mod~$b$ and similarly~$y'=y$ mod~$a$ so that~$x'=x+k b$ and~$y'=y+\ell a$ for~$k,l \in \Z[t^{\pm 1}]$.
Expanding~$ax'+by'=1$, it follows that~$k=-l$.
Therefore
$$ -ax'+by'=-a(x+kb)+b(y-k a) \equiv -ax+by.$$

We will prove that if~$\Phi(a,b)=v \cdot \Phi(a',b')$ for some unit~$v \in U(\Z[t^{\pm 1}])$, then~$(a,b)=\pm (a',b')$ or~$(a,b)=\pm (b',a')$.  It then follows that for any two ways~$(a,b)$ and~$(a',b')$  of factorising~$P$, distinct up to sign and up to reordering, the resulting elements~$\Phi(a,b)$ and~$\Phi(a',b')$ are distinct in~$U(\Z[t^{\pm 1}]/2P)/U(\Z[t^{\pm 1}])$, from which the proposition follows.

Assume that~$x,x',y,y' \in \Z[t^{\pm 1}]$ are such that~$ax+by=1=a'x'+b'y'$ and~$-ax+by=-a'x'+b'y'$ mod~$2P$.
Add~$2ax+2a'x'v$ to both sides of the congruence~$-ax+by=v(-a'x'+b'y')$ mod~$2P$.
Using that~$ax+by=1$ and~$a'x'+b'y'=1$,  we obtain the congruence
\begin{equation}\label{eq:aaaaaaah}
2ax+v=2a'x'v+1 \text{ mod } 2P.
\end{equation}
Similarly, add~$-2by+2a'x'v$ to both sides of~$-ax+by=v(-a'x'+b'y')$ mod~$2P$.
Using that~$ax+by=1$ and~$a'x'+b'y'=1$,  we obtain the equation
\begin{equation}
\label{eq:bbbbbbh}
-2by+v=2a'x'v-1 \text{ mod } 2P.
\end{equation}
We deduce from the previous two equations that~$v+1$ and~$v-1$ are divisible by~$2$.
Since~$v=\pm t^k$, we deduce that~$\pm t^k \pm 1$ is divisible by~$2$ and so~$v=\pm 1$.

First, we treat the case where the unit is~$v=1$.
\begin{claim}
We have~$(i)$~$a$ divides~$a'$, and~$(ii)$~$a'$ divides~$a$.
\end{claim}
\begin{proof}
As~$v=1$, ~\eqref{eq:aaaaaaah} implies that~$2ax=2a'x'$ mod~$2P$.
Writing~$2P=2ab$,  and simplifying the~$2$s, we deduce that~$a$ divides~$a'x'$.
Similarly, writing~$2P=2a'b'$, and simplifying the~$2$s, we deduce that~$a'$ divides~$ax$.
Next,  multiply the equations~$1=ax+by$ (resp.\ $1=a'x'+b'y'$) by~$a$ (resp.\ $a'$) to obtain
\begin{align*}
a&=a^2x+aby \\
a'&={a'}^2x'+a'b'y'.
\end{align*}
Since~$a'$ divides~$ax$ and~$ab=P=a'b'$,  it follows that~$a'$ divides~$a$.
The same reasoning with the second equation shows that~$a$ divides~$a'$.
This concludes the proof of the claim.
\end{proof}
Using the claim we have~$a=ua'$ for some unit~$u$; this unit is necessarily symmetric since both~$a$ and~$a'$ are symmetric.
It follows that~$a'b'=ab=ua'b$ with $u=\pm 1$.
We deduce~$b'=ub$ and therefore~$b=b'/u$.
Thus~$(a,b)=u \cdot (a',b')$ as required, in the case~$v=1$.

 Next, we treat the case where the unit is~$v=-1$.
\begin{claim}
We have~$(i)$~$b$ divides~$a'$, and~$(ii)$~$a'$ divides~$b$.
\end{claim}
\begin{proof}
As~$v=-1$, ~\eqref{eq:bbbbbbh} implies that~$-2by=2a'x'$ mod~$2P$.
Writing~$2P=2ab$,  and simplifying the~$2$s, we deduce that~$b$ divides~$a'x'$.
Similarly, writing~$2P=2a'b'$, and simplifying the~$2$s, we deduce that~$a'$ divides~$by$.
Next,  multiply the equations~$1=ax+by$ (resp.\ $1=a'x'+b'y'$) by~$b$ (resp.\ $a'$) to obtain
\begin{align*}
b&=abx+b^2y \\
a'&={a'}^2x'+a'b'y'.
\end{align*}
Since~$a'$ divides~$by$ and~$ab=P=a'b'$,  it follows that~$a'$ divides~$b$.
The same reasoning with the second equation shows that~$b$ divides~$a'$.
This concludes the proof of the claim.
\end{proof}
Using the claim we have~$b=ua'$ for some unit~$u$; this unit is necessarily symmetric since both~$b$ and~$a'$ are symmetric.
It follows that~$a'b'=ab=uaa'$ with $u=\pm 1$.
We deduce~$b'=ua$ and therefore~$a=b'/u$.
Thus~$(a,b)=u \cdot (b',a')$ as required, in the case that~$v=-1$.
This completes the proof that~$\Phi(a,b)=v \cdot \Phi(a',b')$ implies~$(a,b)=\pm (a',b')$ or~$(a,b)=\pm (b',a')$, which completes the proof of the proposition.
%AC: If \pm t^k -1=2q(t) and k>0, then plugging in zero, we get~$0-1=2\cdot q(0)$ which is impossible.
%The conclusion now follows from tedious but explicit case work involving the equations~$ax+by=1=a'x'+b'y', -ax+by=-a'x'+b'y'$ mod~$2P$ and~$P=ab=a'b'$.
\end{proof}

Over~$\Z$,  it is not difficult to show that if~$N$ is an integer that can be factored as a product of~$n$ distinct primes, then~$U(\Z/N)/U(\Z)$ contains precisely~$2^{n-1}$ elements.
Using Lemma~\ref{lem:NonTrivialbAut},  the next example shows that
%the same
a similar
%AC: It's not the same lower bound as we lost a factor $2$. For $2P$ "the same" bound would give $2^n$ instead of $2^{n-1}$. Please don't change it back.
lower bound (which is not in general sharp) holds over~$\Z[t^{\pm 1}]$.
%Note however, that this estimate is far from optimal as it actually turns out that if~$N$ is not a prime, then~$U(\Z[t^{\pm 1}]/N)/U(\Z[t^{\pm 1}])$ is infinite, but this will not be pursued here.

\begin{example}
\label{ex:Integer}
The reader can check that if~$P$ is an integer than can be factored as a product~$p_1\cdots p_n$ of~$n$ distinct primes, then~$n_P=2^{n-1}$.
Lemma~\ref{lem:NonTrivialbAut} implies that~$U(\Z[t^{\pm 1}]/2P)/U(\Z[t^{\pm 1}])$ contains at least~$2^{n-1}$ elements.
%Note however, that this estimate is far from optimal as it actually turns out that if~$N$ is not a prime, then an application of the chinese remainder theorem$^\copyright$ shows that~$U(\Z[t^{\pm 1}]/N)/U(\Z[t^{\pm 1}])$ is infinite, but this will not be pursued here,  as I was not allowed to write it: Diarmuid Crowley has the sole rights to the chinese remainder theorem$^\copyright$.
%AC: Note that over \Z there are~$2^n$, but that's life.
\end{example}

\begin{remark}
 In order to produce examples, there is no need to restrict~$P$ an integer. Take~$P = q_1 \cdots q_n$, where the~$q_i$ are symmetric Laurent polynomials such that for every~$i,j$, there exists~$x,y\in \Z[t^{\pm 1}]$ with~$q_ix+q_jy=1$. The latter condition implies, via a straightforward induction on $n$,  that there exists such~$x,y$ for any pair of polynomials~$q_{i_1} \cdots q_{i_k}$ and~$q_{i_{k+1}} \cdots q_{i_n}$ with~$\{i_1,\dots,i_n\} = \{1,\dots,n\}$ obtained from factoring~$P$. Then by applying~$\Phi$ we can obtain examples of~$P$ such that~$U(\Z[t^{\pm 1}]/2P)/U(\Z[t^{\pm 1}])$ has cardinality at least~$2^{n-1}$.
However, this level of generality is not strictly necessary, as Example~\ref{ex:Integer}, in which~$P$ is an integer, suffices to prove Proposition~\ref{prop:LargeStableClass} below.
 \black
\end{remark}

We now prove the main result of this section that was mentioned in Example~\ref{ex:LargeStableClassIntro} from the introduction: there are examples of pairs~$(Y,\varphi)$ for which the set of 4-manifolds with fixed boundary~$Y$ and equivariant intersection form, up to homeomorphism, can have arbitrarily large cardinality.  Recall that $\mathcal{V}_\lambda^0(Y)$ and~$\mathcal{V}_\lambda(Y)$ were defined in Definitions~\ref{def:V0lambdaY} and~\ref{def:VlambdaY} respectively.
%\footnote{Added this sentence since the definition no longer appears in the introduction, might be harder to find these definitions. }

\begin{proposition}
\label{prop:LargeStableClass}
For every~$m\ge 0$, there is a pair~$(Y,\varphi)$ and a Hermitian form~$(H,\lambda)$ so that~$\mathcal{V}_\lambda^0(Y)$ and~$\mathcal{V}_\lambda(Y)$ have at least~$m$ elements.
\end{proposition}

\begin{figure}[!htbp]
\center
\def\svgwidth{.6\linewidth}%% Creator: Inkscape 1.0.1 (c497b03c, 2020-09-10), www.inkscape.org
%% Accompanies image file 'Trivialsym.pdf' (pdf, eps, ps)
\begingroup%
  \makeatletter%
  \providecommand\color[2][]{%
    \errmessage{(Inkscape) Color is used for the text in Inkscape, but the package 'color.sty' is not loaded}%
    \renewcommand\color[2][]{}%
  }%
  \providecommand\transparent[1]{%
    \errmessage{(Inkscape) Transparency is used (non-zero) for the text in Inkscape, but the package 'transparent.sty' is not loaded}%
    \renewcommand\transparent[1]{}%
  }%
  \providecommand\rotatebox[2]{#2}%
  \newcommand*\fsize{\dimexpr\f@size pt\relax}%
  \newcommand*\lineheight[1]{\fontsize{\fsize}{#1\fsize}\selectfont}%
  \ifx\svgwidth\undefined%
    \setlength{\unitlength}{465.82251842bp}%
    \ifx\svgscale\undefined%
      \relax%
    \else%
      \setlength{\unitlength}{\unitlength * \real{\svgscale}}%
    \fi%
  \else%
    \setlength{\unitlength}{\svgwidth}%
  \fi%
  \global\let\svgwidth\undefined%
  \global\let\svgscale\undefined%
  \makeatother%
  \begin{picture}(1,0.34537686)%
    \lineheight{1}%
    \setlength\tabcolsep{0pt}%
    \put(0,0){\includegraphics[width=\unitlength,page=1]{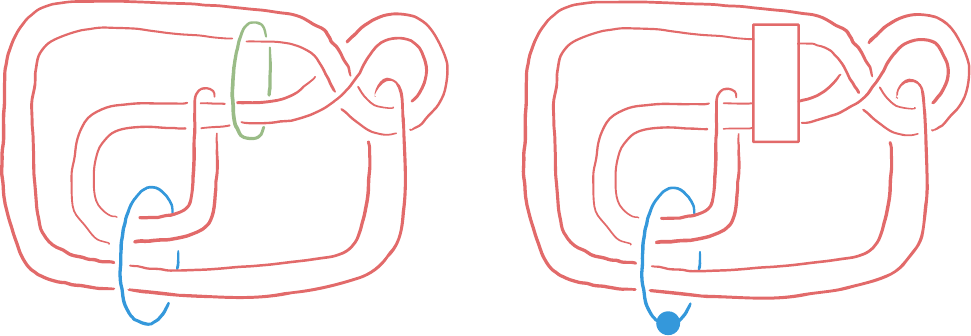}}%
    \put(0.180437,0.00178887){\color[rgb]{0.20392157,0.59607843,0.85882353}\makebox(0,0)[lt]{\lineheight{1.25}\smash{\begin{tabular}[t]{l}$0$\end{tabular}}}}%
    \put(0.23195895,0.10805209){\color[rgb]{0.88627451,0.41568627,0.41568627}\makebox(0,0)[lt]{\lineheight{1.25}\smash{\begin{tabular}[t]{l}$0$\end{tabular}}}}%
    \put(0.7708674,0.11058216){\color[rgb]{0.88627451,0.41568627,0.41568627}\makebox(0,0)[lt]{\lineheight{1.25}\smash{\begin{tabular}[t]{l}$n$\end{tabular}}}}%
    \put(0.78926802,0.25341707){\color[rgb]{0.88627451,0.41568627,0.41568627}\makebox(0,0)[lt]{\lineheight{1.25}\smash{\begin{tabular}[t]{l}$n$\end{tabular}}}}%
    \put(0.20306267,0.27393418){\color[rgb]{0.60392157,0.7372549,0.52941176}\makebox(0,0)[lt]{\lineheight{1.25}\smash{\begin{tabular}[t]{l}$\gamma$\end{tabular}}}}%
  \end{picture}%
\endgroup%

\caption{Left frame: the complement of~$\gamma$ is a hyperbolic 3-manifold~$Z$ with trivial mapping class group.
Right frame: This $\Z$-manifold~$W_n$ has equivariant intersection form~$(n)$ and, for~$n$ sufficiently large, boundary~$\partial  W_n$ with trivial mapping class group.}\label{fig:trivialsym}
\end{figure}

\begin{proof}
Since the cardinality of~$\mathcal{V}_\lambda^0(Y)$ is greater than that of~$\mathcal{V}_\lambda(Y)$, it suffices to prove that the latter set can be made arbitrarily large.
However since proof involving~$\mathcal{V}_\lambda^0(Y)$ is substantially less demanding, we include it as a quick warm up.

Set~$\lambda:=2P$ where~$P$ is an integer than can be factored as a product~$p_1\cdots p_k$ of~$k$ distinct  primes with~$2^{k-1} \geq m$.
Example~\ref{ex:Integer} and Proposition~\ref{lem:NonTrivialbAut} imply that~$U(\Z[t^{\pm 1}]/\lambda)/U(\Z[t^{\pm 1}])$ has at least~$2^{k-1}$ elements.
By Proposition~\ref{lem:UnitaryUnits}, this means that~$\Aut(\partial \lambda)/\Aut(\lambda)$ has at least~$2^{k-1}$ elements.
As in the proof of Theorem~\ref{thm:ExoticmanifoldsNotIntro},  construct a smooth~$\Z$-manifold~$W$ with equivariant intersection form~$\lambda$.
In our setting, where~$\lambda:=2P$, the manifold produced will be~$X_{\lambda}(U)\natural (S^1\times D^3)$, where~$X_{\lambda}(U)$ is the manifold obtained by attaching a~$\lambda$-framed 2-handle to~$D^4$ along the unknot~$U$.
%Construct a smooth~$4$-manifold~$W$ with~$\pi_1(W) \cong \Z$, ribbon boundary, and equivariant intersection form~$\lambda$; this can be done as in the proof of Theorem~\ref{thm:ExoticmanifoldsNotIntro} (recall in particular Figure~\ref{fig:J'}).
Let~$Y'$ be the boundary of this~$4$-manifold and let~$\varphi \colon \pi_1(Y') \to \pi_1(W) \cong \Z$ be the inclusion induced map.
Since~$\lambda$ presents~$Y'$, Theorem~\ref{thm:ClassificationRelBoundary} implies that~$\mathcal{V}_\lambda^0(Y')$ has at least~$2^{k-1}\geq m$ elements, as required.

We now turn to the statement involving~$\mathcal{V}_\lambda(Y)$.
\begin{claim*}
There is an integer~$N>0$ so that for any~$n >N$, there exists a smooth~$\Z$-manifold~$W_n$ with equivariant intersection form~$(n)$ and such that~$\partial W_n$ has trivial mapping class group.
\end{claim*}
\begin{proof}
Let~$L$ be the~$3$-component link in the left frame of Figure \ref{fig:trivialsym} and let~$Z$ be the~$3$-manifold obtained from~$L$ by~$0$-surgering both the red and blue components, and removing a tubular neighborhood of the green component~$\gamma$.
Using verified computations in Snappy inside of Sage, we find that~$Z$ is hyperbolic and has trivial mapping class group.\footnote{Transcripts of the computation are available at \cite{data}.}
By Thurston's hyperbolic Dehn surgery theorem \cite[Theorem 5.8.2]{Thurston},
%and \cite{NeumannZagier}[Theorem 1A]
there exists~$N > 0$ such that for~$n > N$, the manifold~$Z_n$ obtained by~$-1/n$ filling~$\gamma$ is hyperbolic and has trivial symmetry group; for the mapping class group part of this statement, see for example~\cite[Lemma~2.2]{DHL}.

Let~$W_n$ be the 4-manifold described in the right frame of Figure \ref{fig:trivialsym} and observe that~$\partial W_n\cong Z_n$.
It is not difficult to verify that~$W_n$ is a $\Z$-manifold with equivariant intersection form~$(n)$.  This concludes the proof of the claim.
\end{proof}

We conclude the proof of the proposition.
Fix~$m \geq 0$ and choose an integer~$P$ such that
\begin{itemize}
\item~$P$ can be factored as a product~$p_1\cdots p_k$ of~$k$ distinct primes with~$2^{k-1} \geq m$. %with~$n/2 \leq 2^{k-1}$.
\item~$2P>N$ where~$N$ is as in the claim.
\end{itemize}
Since~$2P>N$,  the claim implies that~$Y:=\partial W_{2P}$ has trivial mapping class group.
The proof is now concluded as in the warm up, but we spell out the details.
As we already mentioned,~$W_{2P}$ has equivariant intersection form~$\lambda:=2P$.
Example~\ref{ex:Integer} and Proposition~\ref{lem:NonTrivialbAut} imply that~$U(\Z[t^{\pm 1}]/\lambda)/U(\Z[t^{\pm 1}])$ has at least~$2^{k-1}$ elements.
By Proposition~\ref{lem:UnitaryUnits}, this means that~$\Aut(\partial \lambda)/\Aut(\lambda)$ has at least~$2^{k-1}$ elements.
Since~$Y$ has trivial mapping class group, either of Theorem~\ref{thm:ClassificationRelBoundary} or Theorem~\ref{thm:Classification} implies that~$\mathcal{V}_\lambda(Y)=\mathcal{V}_\lambda^0(Y)$ has at least~$2^{k-1} \geq m$ elements.
\end{proof}
%\color{black}

\bibliography{BiblioRealisation}
\bibliographystyle{alpha}
\end{document}